\documentclass[12pt]{article}
\usepackage{graphicx}
\usepackage{amsmath}
\usepackage{amsfonts}
\usepackage{enumerate}
\usepackage{latexsym}
\usepackage{epsfig}
\usepackage{float}
\usepackage{color}
\usepackage{amscd}
\usepackage{amssymb}
\usepackage{epstopdf}

\newtheorem{theorem}{Theorem}[section]
\newtheorem{lemma}[theorem]{Lemma}
\newtheorem{coroll}[theorem]{Corollary}

\def\proofbox{\begin{picture}(6.5,6.5)
\put(0,0){\framebox(6.5,6.5){}}\end{picture}}
\newenvironment{proof}{\noindent{\it Proof.\quad}}{\hfill\proofbox}

\begin{document}

\def\CC{\mathcal C} \def\Id{{\rm Id}}

\title{Superinjective Simplicial Maps of the Two-sided Curve Complexes on Nonorientable Surfaces}
\author{Elmas Irmak and Luis Paris}

\maketitle

\renewcommand{\sectionmark}[1]{\markright{\thesection. #1}}

\thispagestyle{empty}
\maketitle
\begin{abstract} Let $N$ be a compact, connected, nonorientable surface of genus $g$ with $n$ boundary components with $g \geq 5$,
$n \geq 0$. Let $\mathcal{T}(N)$ be the two-sided curve complex of $N$. If $\lambda :\mathcal{T}(N) \rightarrow \mathcal{T}(N)$ is
a superinjective simplicial map, then there exists a homeomorphism $h : N \rightarrow N$ unique up to isotopy such that $H(\alpha) = \lambda(\alpha)$
for every vertex $\alpha$ in $\mathcal{T}(N)$ where $H=[h]$.
\end{abstract}

\maketitle

{\small Key words: Mapping class groups, simplicial maps, nonorientable surfaces

}

\section{Introduction}
Let $N$ be a compact, connected, nonorientable surface of genus $g$ with $n$ boundary components. Mapping class group,
$Mod_N$, of $N$ is defined to be the group of isotopy classes of all self-homeomorphisms of $N$. The \textit{complex of curves},
$\mathcal{C}(N)$, on $N$ is an abstract simplicial complex defined as follows: A simple closed curve on $N$ is called nontrivial
if it does not bound a disk, a M\"{o}bius band, and it is not isotopic to a boundary component of $N$. The vertex set of $\mathcal{C}(N)$
is the set of isotopy classes of nontrivial simple closed curves on $N$. A set of vertices forms a simplex in $\mathcal{C}(N)$ if
they can be represented by pairwise disjoint simple closed curves. Let $\mathcal{T}(N)$ be the subcomplex of $\mathcal{C}(N)$ which
is spanned by all 2-sided curves on $N$. The geometric intersection number $i([a], [b])$ of two vertices $[a]$, $[b]$ in
$\mathcal{T}(N)$ is the minimum number of points of $x \cap y$ where $x \in [a]$ and $y \in [b]$. A simplicial map
$\lambda : \mathcal{T}(N) \rightarrow \mathcal{T}(N)$ is called superinjective if the following condition holds:
if $\alpha, \beta$ are two vertices in $\mathcal{T}(N)$ such that $i(\alpha,\beta) \neq 0$, then
$i(\lambda(\alpha),\lambda(\beta)) \neq 0$.

The main result is the following:

\begin{theorem} \label{A} Let $N$ be a compact, connected, nonorientable surface of genus $g$ with $n$ boundary components with $g \geq 5$,
$n \geq 0$. If $\lambda :\mathcal{T}(N) \rightarrow \mathcal{T}(N)$ is a superinjective simplicial map, then there exists
a homeomorphism $h : N \rightarrow N$ unique up to isotopy such that $H(\alpha) = \lambda(\alpha)$
for every vertex $\alpha$ in $\mathcal{T}(N)$ where $H=[h]$.\end{theorem}

An application of our theorem will be given by the authors in the proof of the following theorem given in \cite{IrP}.

\begin{theorem} (Irmak-Paris) \label{other-result} Let $N$ be a compact, connected, nonorientable surface of genus $g$ with $n$ boundary
components with $g \geq 5$, $n \geq 0$. Let $K$ be a finite index subgroup of $Mod_N$. If $f: K \to Mod_N$ is an injective homomorphism,
then $f$ is induced by a homeomorphism of $N$, (i.e. for some $g \in Mod_N$, $f(k) = gkg^{-1}$ for all $k \in K$).\end{theorem}

Here are some known results for compact, connected, orientable surfaces: Ivanov proved that any automorphism of complex of curves
is induced by a homeomorphism of the surface if the genus is at least two, \cite{Iv1}. By using this result he gave the classification
of isomorphisms between two finite index subgroups of mapping class groups, \cite{Iv1}. Korkmaz and Luo extended Ivanov's results to
small genus, see \cite{K1}, \cite{L}. Ivanov-McCarthy proved that injective homomorphisms between mapping class groups are induced
by homeomorphisms \cite{IvMc}. The first author proved that superinjective simplicial maps of complex of curves are induced by
homeomorphisms when the genus is at least two. By using this result she gave the classification of injective homomorphisms from
finite index subgroups of the extended mapping class group to the whole group for genus at least two, \cite{Ir1},
\cite{Ir2}, \cite{Ir3}. Behrstock-Margalit and Bell-Margalit proved these results for small genus cases \cite{BhM}, \cite{BeM}.
Shackleton proved that locally injective simplicial maps of complex of curves are induced by homeomorphisms and obtained similar
results (strong local co-Hopfian results) for mapping class groups in \cite{Sh}. Aramayona-Leininger proved the existence of
finite rigid sets for locally injective simplicial maps, \cite{AL1}. They proved that there is an exhaustion of the curve
complex by a sequence of finite rigid sets, \cite{AL2}.

On compact, connected, nonorientable surfaces Atalan-Korkmaz proved that any automorphism of complex of curves
is induced by a homeomorphism of the surface \cite{AK}. The first author proved that any injective simplicial map of the whole complex of
curves is induced by a homeomorphism of the surface \cite{Ir4}. However, we do not know if the results of \cite{Ir4} can be used to prove
Theorem \ref{other-result}, as we do not know how to ``detect" one-sided curves with the mapping class group.

\section{Superinjective simplicial maps on $\mathcal{T}(N)$}

In this section we assume that $\lambda : \mathcal{T}(N) \rightarrow \mathcal{T}(N)$ is a superinjective simplicial map.
We will prove some properties of $\lambda$. First we will give some definitions.

A set $P$ of pairwise disjoint, nonisotopic, nontrivial 2-sided simple closed curves on $N$ is called a $P$-$S$ decomposition, if each
component of the surface $N_P$, obtained by cutting $N$ along $P$, is a pair of pants or a projective plane with two boundary components.
Let $a$ and $b$ be two distinct elements in a $P$-$S$ decomposition $P$. Then $a$ is called {\it adjacent} to $b$ w.r.t. $P$ iff
there exists a pair of pants in $P$ which has $a$ and $b$ on its boundary or there exists a projective plane with two boundary components
having $a$ and $b$ on its boundary. Let $P$ be a $P$-$S$ decomposition of $N$. Let $[P]$ be the set of isotopy classes of elements
of $P$. Note that $[P]$ is a maximal simplex of $\mathcal{T}(N)$. Every maximal simplex $\sigma$ of $\mathcal{T}(N)$ is equal to
$[P]$ for some $P$-$S$ decomposition $P$ of $N$. There are different dimensional maximal simplices in $\mathcal{T}(N)$
(see Figure \ref{Fig0}). In the figure we see cross signs. This means that the interiors of the disks with cross signs inside
are removed and the antipodal points on the resulting boundary components are identified.

\begin{figure}
\begin{center}
\epsfxsize=2.7in \epsfbox{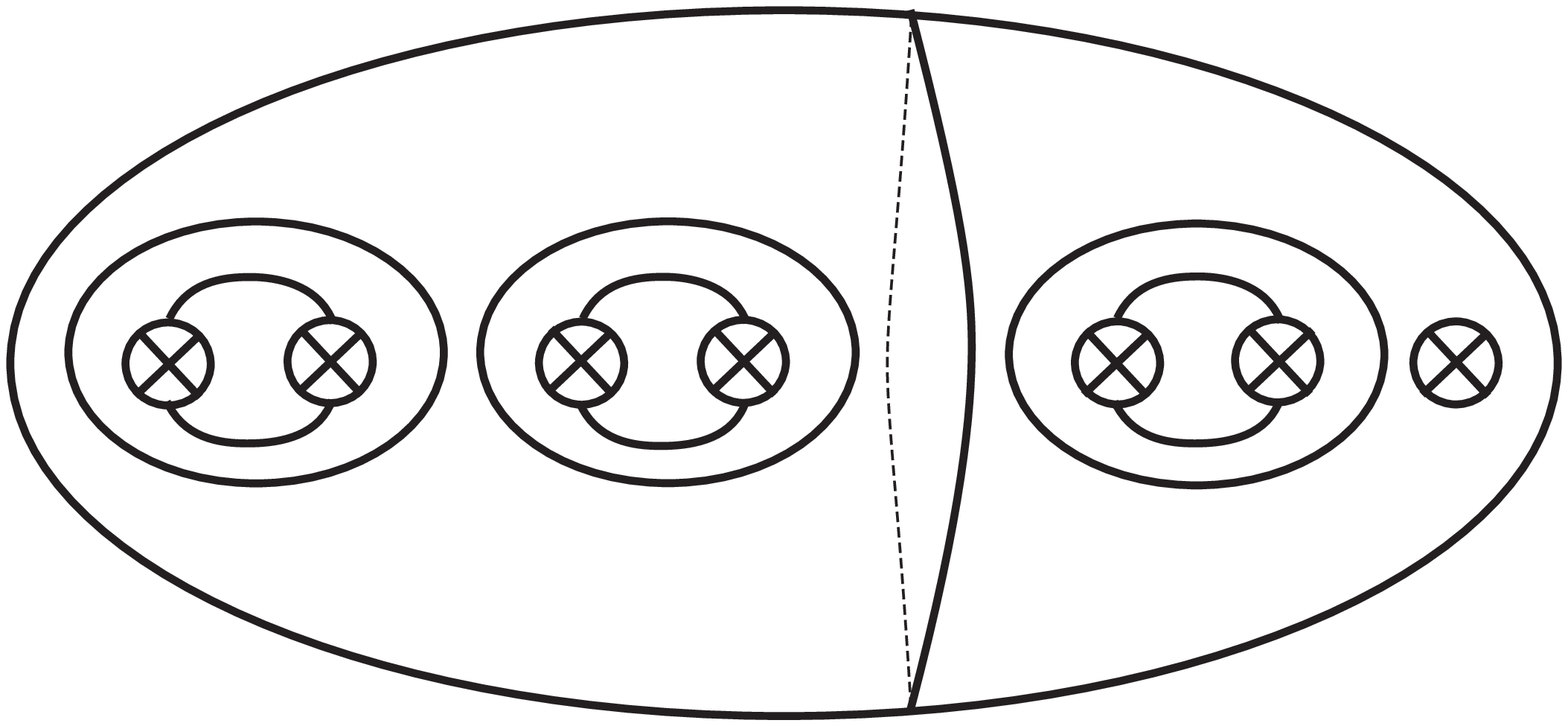} \hspace{0.in}
\epsfxsize=2.7in \epsfbox{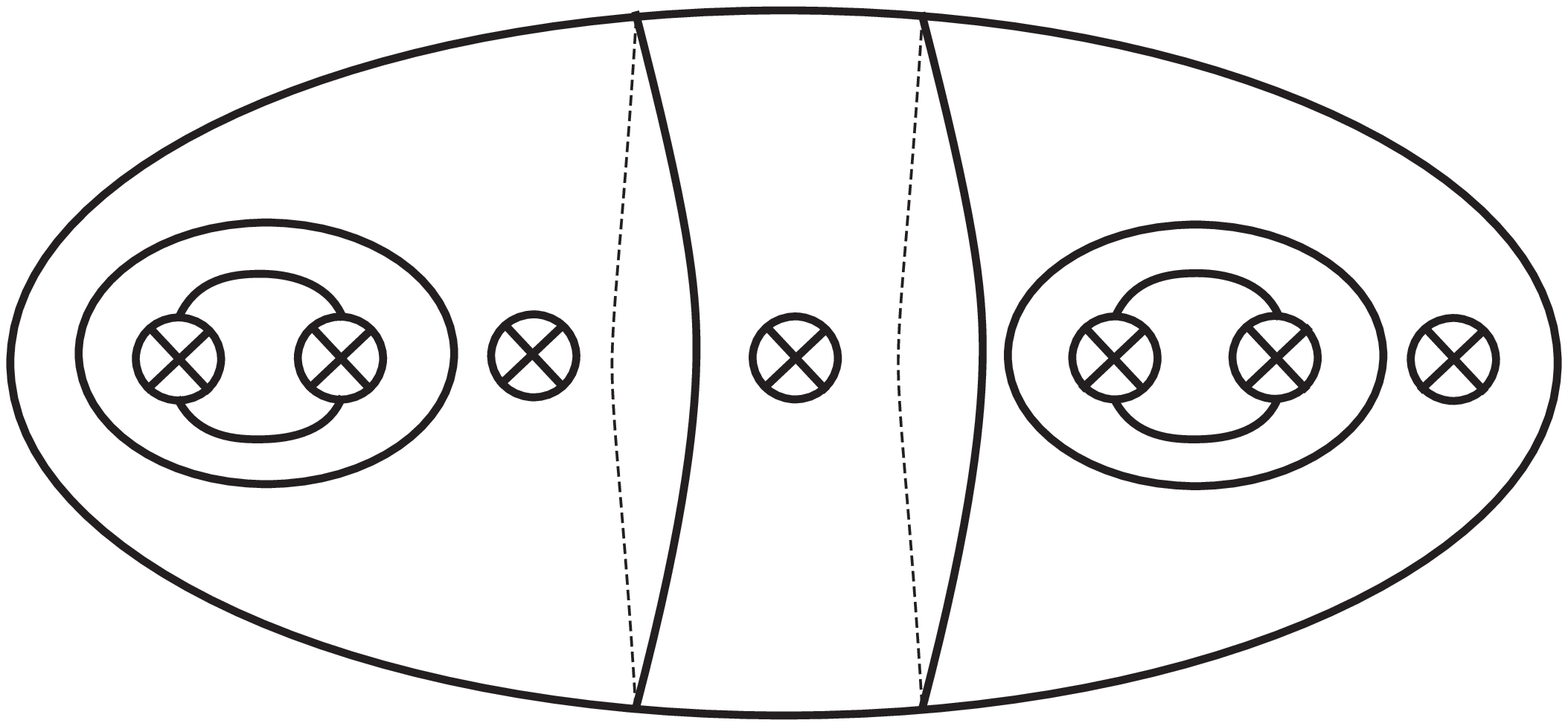}
\caption{Some maximal simplices in $\mathcal{T}(N)$}
\label{Fig0}
\end{center}
\end{figure}

\begin{lemma}
\label{tp} Let $g \geq 2$. Suppose that $(g, n) \neq (2, 0)$. Every top dimensional maximal simplex in $\mathcal{T}(N)$ has
dimension $3r+n-3$ if $g=2r+1$, and $3r+n-4$ if $g=2r$.
\end{lemma}

\begin{proof} Let $P$ be a $P$-$S$ decomposition which corresponds to a top dimensional maximal simplex, $\Delta$, in $\mathcal{T}(N)$.
It is easy to see that if $g$ is odd, then there is exactly one projective plane with two boundary components in $N_P$, and all the other
pieces of $N_P$ are pairs of pants. If $g$ is even then all the pieces in $N_P$ are pairs of pants.

Let $m$ be the number of curves in $P$, and $s$ be the number of pieces in $N_P$. Since each piece in $N_P$ has Euler characteristic $-1$
and the Euler characteristic of $N$ is $2-g-n$, $s = g+n-2$. Suppose $g=2r+1$, $r \geq 1$. Then there are $3s -1$ boundary components in $N_P$.
Since each curve in $P$ corresponds to two boundary components in $N_P$, we have $3s -1 = 2m +n$. This gives $3(g+n-2) -1 = 3(2r+n-1) -1 = 2m +n$. So, $m = 3r + n -2$. Now suppose $g=2r$, $r \geq 1$. Then there are $3s$ boundary components in $N_P$. Since each curve in $P$ corresponds to
two boundary components in $N_P$, we have $3s = 2m +n$. This gives $3(g+n-2) = 3(2r+n-2) = 2m +n$. So, $m = 3r + n -3$. Hence, $\Delta$ has dimension
$3r+n-3$ if $g=2r+1$, and $3r+n-4$ if $g=2r$.\end{proof}

\begin{lemma}
\label{inj} Suppose that $g \geq 4$. Then $\lambda : \mathcal{T}(N) \rightarrow \mathcal{T}(N)$ is injective.\end{lemma}

\begin{proof} Let $[a]$ and $[b]$ be two distinct vertices in $\mathcal{T}(N)$. If $i([a], [b]) \neq 0$, then
$i(\lambda([a]), \lambda([b])) \neq 0$ since $\lambda$ is superinjective. This implies that $\lambda([a]) \neq \lambda([b])$ as both are
isotopy classes of 2-sided curves. If $i([a], [b]) = 0$, we consider the following two cases: (i) If $a$ bounds a one-holed Klein bottle, $K$,
and $b$ is inside $K$, then we choose a vertex $[c]$ of $\mathcal{T}(N)$ such that $i([c], [b]) = 0$, and
$i([c], [a]) \neq 0$, and hence $i(\lambda([c]), \lambda([b])) = 0$, $i(\lambda([c]), \lambda([a])) \neq 0$. (ii) In all other cases,
we choose a vertex $[c]$ of $\mathcal{T}(N)$ such that $i([c], [a]) = 0$,
$i([c], [b]) \neq 0$, and hence $i(\lambda([c]), \lambda([a])) = 0$, $i(\lambda([c]), \lambda([b])) \neq 0$. In both cases we see
that $\lambda([a]) \neq \lambda([b])$. Hence, $\lambda$ is injective.\end{proof}\\

Since a superinjective map is injective it sends top dimensional maximal simplices to top dimensional maximal simplices. In the following
lemma we will see that adjacency is preserved w.r.t. top dimensional maximal simplices.

\begin{lemma}
\label{adjacent} Suppose that $g \geq 4$. Let $P$ be a top dimensional $P$-$S$ decomposition on $N$. Let $a, b \in P$ such that $a$ is
adjacent to $b$ w.r.t. $P$. There exists $a' \in \lambda([a])$ and $b' \in \lambda([b])$ such that $a'$ is adjacent to $b'$ w.r.t. $P'$ where
$P'$ is a set of pairwise disjoint curves representing $\lambda([P])$ containing $a', b'$.
\end{lemma}

\begin{proof} Let $P$ be a top dimensional $P$-$S$ decomposition on $N$. Let $a, b \in P$ such that $a$ is adjacent to $b$ w.r.t. $P$. We can
find a 2-sided simple closed curve $c$ on $N$ such that $c$ intersects only $a$ and $b$ nontrivially (with nonzero geometric intersection) and
$c$ is disjoint from all the other curves in $P$. Let $P'$ be a set of pairwise disjoint curves representing $\lambda([P])$. Since $\lambda$
is injective by Lemma \ref{inj}, $\lambda$ sends top dimensional maximal simplices of $\mathcal{T}(N)$ to top dimensional maximal simplices
of $\mathcal{T}(N)$. So, $P'$ corresponds to a top dimensional maximal simplex. Assume that $\lambda([a])$ and $\lambda([b])$ do not have
adjacent representatives w.r.t. $P'$. Since $i([c], [a]) \neq 0$ and $i([c], [b]) \neq 0$, we have $i(\lambda([c]), \lambda([a])) \neq 0$
and $i(\lambda([c]), \lambda([b])) \neq 0$ by superinjectivity. Since $i([c], [e]) = 0$ for all $e \in P \setminus \{a, b\}$, we have
$i(\lambda([c]), \lambda([e])) = 0$ for all $e \in P \setminus \{a, b\}$. But this is not possible because $\lambda([c])$ has to intersect
geometrically essentially with some isotopy class other than $\lambda([a])$ and $\lambda([b])$ in $\lambda([P])$ to be able to make essential
intersections with $\lambda([a])$ and $\lambda([b])$ since $\lambda([P])$ is a top dimensional maximal simplex. This gives a contradiction to the assumption
that $\lambda([a])$ and $\lambda([b])$ do not have adjacent representatives w.r.t. $P'$.\end{proof}

\begin{lemma}
\label{piece1} Suppose that $g \geq 5, n \geq 0 $. Let $x, y, z$ be nontrivial nonseparating simple closed curves on $N$ which bound a
pair of pants on a genus two orientable subsurface $S$ of $N$. There exist nonseparating curves
$x' \in \lambda([x]), y' \in \lambda([y]), z' \in \lambda([z])$ such that $x', y', z'$ bound a pair of pants on $N$.\end{lemma}

\begin{proof} Suppose $g = 5, n = 0$. We can complete $\{x, y, z\}$ to a curve configuration $Q = \{x, y, z, t, w\}$ on $N$ as shown in
Figure \ref{fig2} (i). Let $P = \{x, w, z, t\}$, $R = \{x, y, z, t\}$. Let $P', R'$ be sets of pairwise disjoint curves representing
$\lambda([P]), \lambda([R])$ respectively. $P, R$ correspond to top dimensional maximal simplices. Since $\lambda$ is injective by
Lemma \ref{inj}, $\lambda$ sends top dimensional maximal simplices of $\mathcal{T}(N)$ to top dimensional maximal simplices of $\mathcal{T}(N)$.

\begin{figure}[htb]
\begin{center}
\hspace{0.2cm} \epsfxsize=2.07in \epsfbox{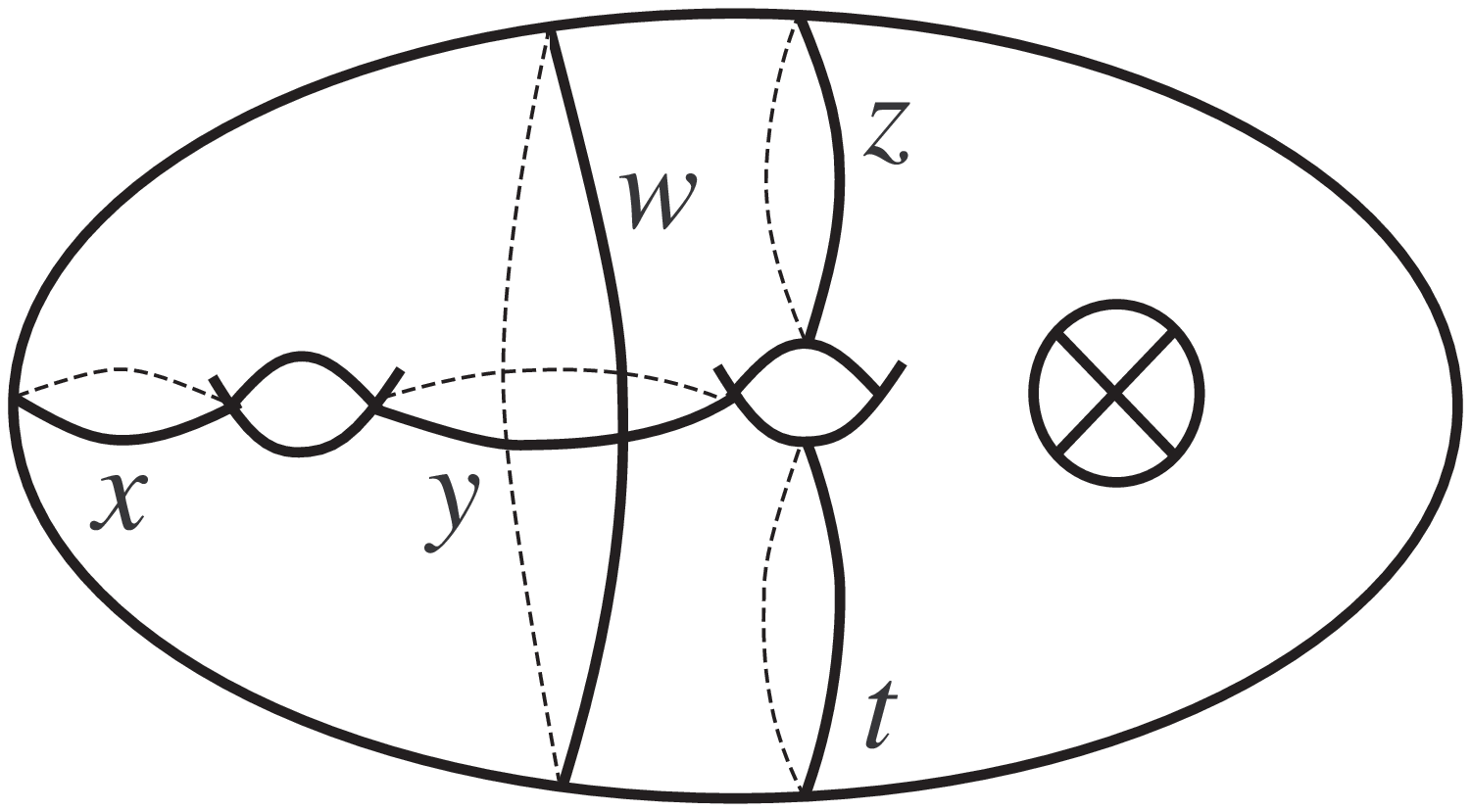} \hspace{0.1cm}
\epsfxsize=2.07in \epsfbox{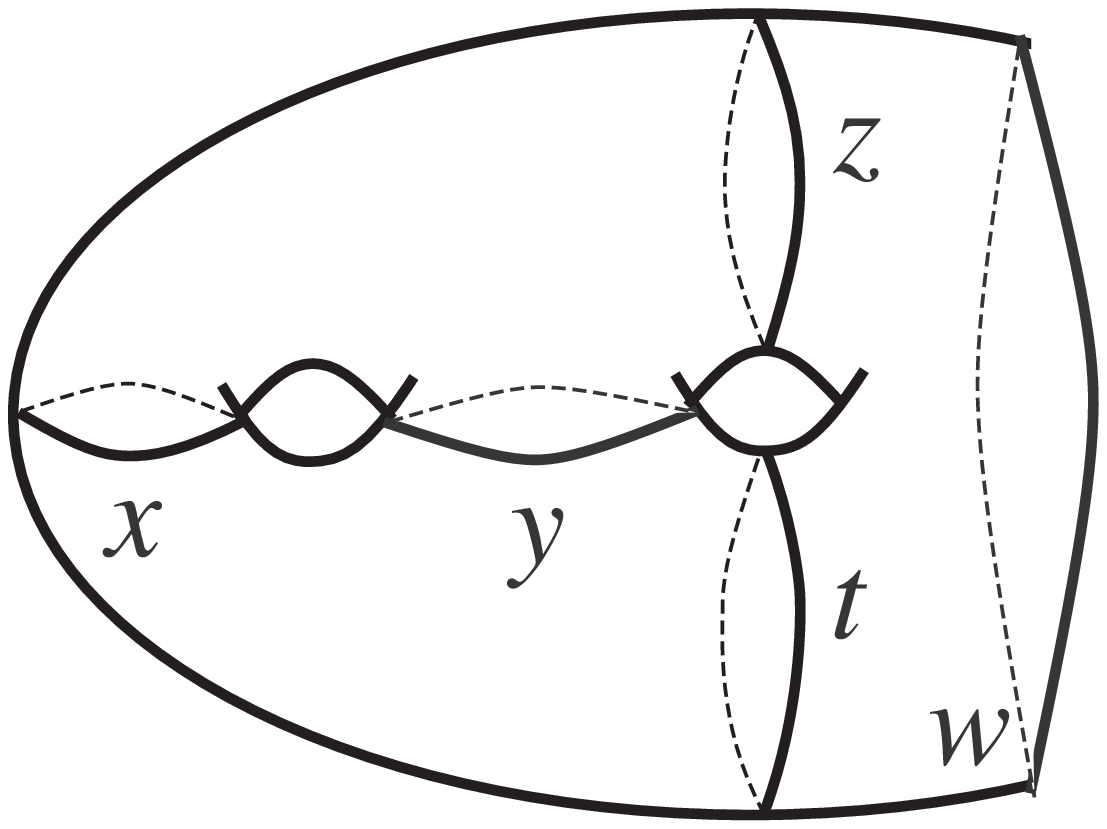} \hspace{-1.2cm}
\epsfxsize=2.07in \epsfbox{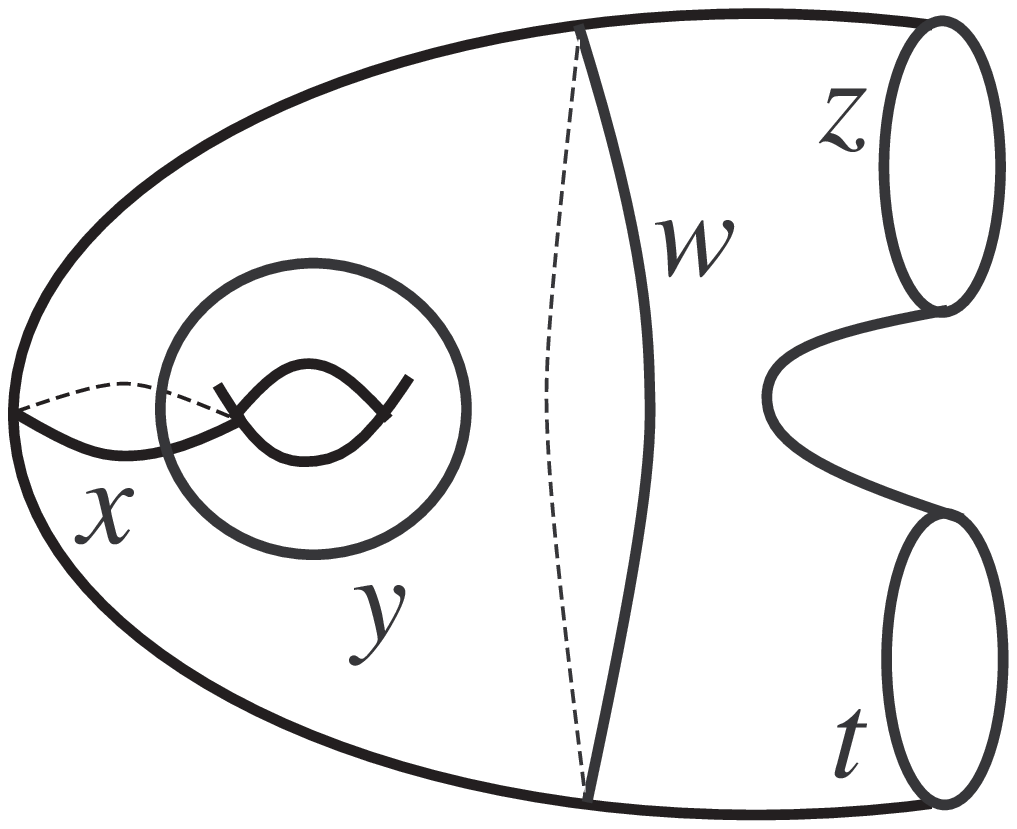}

\hspace{0.2cm}   (i) \hspace{3.9cm} (ii) \hspace{3.9cm} (iii)

\caption{Curve configuration I} \label{fig2}
\end{center}
\end{figure}

Let $x' \in \lambda([x]), z' \in \lambda([z]),  t' \in \lambda([t]), w' \in \lambda([w])$ be disjoint representatives in $P'$. Since $w$ is adjacent
to $x, z, t$ w.r.t. $P$, by Lemma \ref{adjacent} $w'$ is adjacent to $x', z', t'$ w.r.t. $P'$. So, there exist two pairs of pants, say $P_1$, $P_2$
on the surface $N$ having $w'$ as one of their boundary component such that each of $x', z', t'$ is a boundary component of $P_1$ or $P_2$. Note that
$P_1$ and $P_2$ are not necessarily essential and one of the boundary components of $P_1$ or $P_2$ can bound a M\"{o}bius band on the surface.
The curve $x$ is not adjacent to $z$ and $t$ w.r.t. $P$. We can find 2-sided simple closed curves $c$ and $d$ on $N$ such that $x, z, c, d$ are pairwise
nonisotopic, $c$ intersects only $x$ nontrivially and is disjoint from all the other curves in $P$, $d$ intersects only $z$ nontrivially
and is disjoint from all the other curves in $P$, and $c$ and $d$ are disjoint. Since $\lambda$
is injective $\lambda([x]), \lambda([z]), \lambda([c]), \lambda([d])$ are pairwise distinct. We also have
$i(\lambda([c]), \lambda([x])) \neq 0$, $i(\lambda([c]), \lambda([u])) = 0$ for all $u \in P \setminus \{x\}$,
$i(\lambda([d]), \lambda([z])) \neq 0$, $i(\lambda([d]), \lambda([v])) = 0$ for all $v \in P \setminus \{z\}$, and
$i(\lambda([c]), \lambda([d])) = 0$ by superinjectivity. This is possible only when $\lambda([x])$ and $\lambda([z])$ have representatives
which are not adjacent w.r.t. $P'$. So, $x'$ is not adjacent to $z'$ w.r.t. $P'$. Similarly, $x'$ is not adjacent to $t'$ w.r.t. $P'$.
This implies that one of the pair of pants, say $P_1$, has $w', z', t'$ as its boundary components. Then, $x'$ and $w'$ are boundary
components of $P_2$. Let $m$ be the other boundary component of $P_2$. Since $P'$ corresponds to a top dimensional maximal simplex,
and $x'$ is not adjacent to any of $z'$ and $t'$ we see that there exists a projective plane with two boundary components having $z', t'$
on its boundary. Now it is easy to see that $m$ is isotopic to $x'$ as $P'$ corresponds to a top dimensional maximal simplex. 
Let $y' \in \lambda([y])$ in $R'$ such that $y'$ has minimal intersection with $x', t', w', z'$. Since there exists a
projective plane with two boundary components having $z', t'$ on its boundary, and $z'$ should be adjacent to $x'$ and $y'$ w.r.t. $R'$,
we see that there exists a pair of pants that has $x', y', z'$ as its boundary components. We see that $x', z'$ are
nonseparating. To see that $y'$ is also nonseparating, it is enough to see that since $t'$ should be adjacent to $x'$ and $y'$ w.r.t. $R'$, and $t'$
and $z'$ are the boundary components of a projective plane with two boundary components disjoint from $x'$ and $y'$,  
there exists a pair of pants that has $x', y', t'$ as its boundary components.

Suppose that $g = 5, n \geq 1$ or $g \geq 6, n \geq 0$. We can complete $\{x, y, z\}$ to a curve configuration $Q = \{x, y, z, t, w\}$
on $S$ as shown in Figure \ref{fig2} (ii). Then we complete $Q$ to a top dimensional $P$-$S$ decomposition $P$ on $N$ in anyway we like.
Let $P'$ be a set of pairwise disjoint curves representing $\lambda([P])$. $P'$ corresponds to a top dimensional maximal simplex on $N$.
The curve $x$ is not adjacent to $w$ w.r.t. $P$. We can find 2-sided simple closed curves $c$ and $d$ on $N$ such that $x, w, c, d$ are
pairwise nonisotopic, $c$ intersects only $x$
nontrivially and is disjoint from all the other curves in $P$, $d$ intersects only $w$ nontrivially and is
disjoint from all the other curves in $P$, and $c$ and $d$ are disjoint. Since $\lambda$ is injective
by Lemma \ref{inj}, $\lambda$ sends top dimensional maximal simplices of $\mathcal{T}(N)$ to top dimensional maximal simplices of
$\mathcal{T}(N)$, and
$\lambda([x]), \lambda([w]), \lambda([c]), \lambda([d])$ are pairwise distinct. We also have
$i(\lambda([c]), \lambda([x])) \neq 0$, $i(\lambda([c]), \lambda([u])) = 0$ for all $u \in P \setminus \{x\}$,
$i(\lambda([d]), \lambda([w])) \neq 0$, $i(\lambda([d]), \lambda([v])) = 0$ for all $v \in P \setminus \{w\}$, and
$i(\lambda([c]), \lambda([d])) = 0$ by superinjectivity. This is possible
only when $\lambda([x])$ and $\lambda([w])$ have representatives which are not adjacent w.r.t. $P'$. Similar argument shows that
$\lambda([y])$ and $\lambda([w])$ have representatives which are not adjacent w.r.t. $P'$.

Let $x' \in \lambda([x]), y' \in \lambda([y]), z' \in \lambda([z]),  t' \in \lambda([t]), w' \in \lambda([w])$ be disjoint representatives.
Since $z$ is adjacent to $x, y, t, w$ w.r.t. $P$, by Lemma \ref{adjacent} $z'$ is adjacent to $x', y', t', w'$
w.r.t. $P'$. So, there exist two pairs of pants in $P'$ having $z'$ as one of its boundary components. The other boundary components
of these pairs of pants are $x', y', t', w'$. By the above arguments we know that $w'$ is not adjacent to any of $x'$ and $y'$. This implies that
one of the pairs of pants has $x', y', z'$ as its boundary components, and all of them are nonseparating.\end{proof}

\begin{coroll} \label{nonsep} Suppose that $g \geq 5, n \geq 0 $. Let $x$ be a 2-sided nonseparating simple closed curve with nonorientable
complement on $N$. Then $\lambda([x])$ is the isotopy class of a 2-sided nonseparating simple closed curve on $N$.\end{coroll}

\begin{proof} It follows from the proof of Lemma \ref{piece1}.
\end{proof}

\begin{lemma}
\label{sep1} Suppose that $g \geq 5, n \geq 0 $. Let $w$ be a separating curve on $N$ such that it separates a torus with one hole
having the curves $x, y$ as shown in Figure \ref{fig2} (iii). Then there exist $x' \in \lambda([x]), y' \in \lambda([y]), w'  \in \lambda([w])$
such that $w'$ is a separating curve on $N$ and it separates a torus with one hole containing $x', y'$.\end{lemma}

\begin{proof} We complete $\{w\}$ to a curve configuration $Q = \{x, w, z, t\}$ on $N$ as shown in Figure \ref{fig2} (iii). Then we complete
$Q$ to a top dimensional $P$-$S$ decomposition $P$ on $N$ in anyway we like. Let $P'$ be a set of pairwise disjoint curves representing
$\lambda([P])$. $P'$ corresponds to a top dimensional maximal simplex on $N$.

Let $x' \in \lambda([x]), w' \in \lambda([w]), z' \in \lambda([z]),  t' \in \lambda([t])$ be disjoint representatives in $P'$. Since $w$
is adjacent to $x, z, t$ w.r.t. $P$, by Lemma \ref{adjacent} $w'$ is adjacent to $x', z', t'$ w.r.t. $P'$. So, there exist two pairs of pants,
say $P_1$, $P_2$ on the surface $N$ having $w'$ as one of their boundary components such that each of $x', z', t'$ is a boundary component of $P_1$ or
$P_2$. Note that $P_1$ and $P_2$ are not necessarily essential and one of the boundary components of $P_1$ or $P_2$ can bound a M\"{o}bius band on
the surface. We know that $x$ is not adjacent to $z$ and $t$ w.r.t. $P$. Consider the curve $y$ shown in the figure. Let $y' \in \lambda([y])$ such that
$y'$ intersects minimally with elements in $P'$. To see that $x'$ is not adjacent to $z'$ w.r.t. $P'$ we consider
the following: there exists a curve $v$ such that $v$ intersects $z$ nontrivially and $v$ is disjoint from all the
other elements in $P$ and $y$. The curve $y$ intersects $x$ nontrivially and $y$ is disjoint from all the
other elements in $P$. Let $v' \in \lambda([v])$ such that $v'$ intersects minimally with elements in $P'$ and $y'$. Since $\lambda$ is superinjective,
we have that $v'$ intersects $z'$ nontrivially and $v'$ is disjoint from all the
other elements in $P'$ and $y'$. The curve $y'$ intersects $x'$ nontrivially, and $y'$ is disjoint from all the
other elements in $P'$. This implies that $x'$ cannot be adjacent to $z'$ w.r.t. $P'$. With a similar argument we can see that $x'$
is not adjacent to $t'$ w.r.t. $P'$, and $x'$ is only adjacent
to $w'$ w.r.t. $P'$. Since $x'$ is not adjacent to $z'$ and $t'$ w.r.t. $P'$, one of the
pairs of pants, say $P_1$, has $w', z', t'$ as its boundary components. Then, $P_2$ has $x'$, $w'$ and a third boundary component, say $m$.
By using Corollary \ref{nonsep}, we see that $x'$ is nonseparating. Since $x'$ is nonseparating and $x'$ is only adjacent to $w'$ w.r.t. $P'$, $x'$ is isotopic to $m$.
Now there are two choices, either $w'$ separates a torus with one hole like $w$, or $w'$ separates a Klein bottle with one hole.
Suppose $w'$ separates a Klein bottle with one hole. Since $y$ intersects $x$ nontrivially, and $y$ is disjoint from $w$, $y'$ intersects
$x'$ nontrivially, and $y'$ is disjoint from $w'$. This would imply that $y'$ is in the Klein bottle bounded by $w'$. But since $y'$ is
not isotopic to $x'$ and $x'$ is the unique two sided curve up to isotopy in that Klein bottle (see \cite{Sc}), we get a contradiction.
So, $w'$ separates a torus with one hole which has $x'$ and $y'$ in it.\end{proof}

\begin{figure}
\begin{center}
\epsfxsize=2.7in \epsfbox{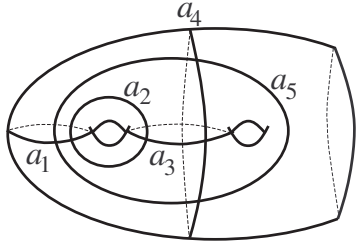} \caption{Curves
intersecting once}
\label{fig5-e}
\end{center}
\end{figure}

\begin{lemma}
\label{intone} Let $g \geq 5, n \geq 0$. Let $\lambda : \mathcal{T}(N) \rightarrow \mathcal{T}(N)$ be a superinjective
simplicial map. Let $\alpha_1$, $\alpha_2$ be two vertices of $\mathcal{T}(N)$. If $i(\alpha_1, \alpha_2)=1$,
then $i(\lambda(\alpha_1), \lambda(\alpha_2))=1$.
\end{lemma}

\begin{proof} Let $\alpha_1$, $\alpha_2$ be two vertices of $\mathcal{T}(N)$ such that $i(\alpha_1, \alpha_2)=1$. Let $a_1, a_2$ be minimally intersecting 
representatives of $\alpha_1$, $\alpha_2$ respectively. We complete $a_1, a_2$ to a curve configuration $\{a_1, a_2, a_3, a_4, a_5\}$ as shown in 
Figure \ref{fig5-e}. Let $\alpha_i = [a_i]$ for $i= 3, 4, 5$. We have $i(\alpha_{i},
\alpha_{j})=0$ if and only if the curves $a_i, a_ j$ are disjoint. Since $\lambda$ is superinjective we have $i(\lambda(\alpha_{i}),
\lambda(\alpha_{j}))=0$ if and only if the curves $a_i, a_ j$ are disjoint. By Lemma \ref{sep1} there exist 
$a_1' \in \lambda([a_1]), a_2' \in \lambda([a_2]), a_4'  \in \lambda([a_4])$
such that $a_4'$ is a separating curve on $N$ and it separates a torus with one hole containing the curves $a_1', a_2'$. Then by using the intersection information 
that $i(\lambda(\alpha_{i}), \lambda(\alpha_{j}))=0$ if and only if the curves $a_i, a_ j$ are disjoint, as in the proof of Ivanov's Lemma 1 given in \cite{Iv1}, 
we see that $i(\lambda(\alpha_1), \lambda(\alpha_2))=1$.\end{proof}
 
\begin{lemma}
\label{piece2-a} Let $g \geq 6$. Suppose that $g$ is even. Let $R \subset N$ be a projective plane with two boundary components, $a, b$
where $a, b$ are both nonseparating simple closed curves on $N$ such that the complement of $R$ is nonorientable. There exist
$a' \in \lambda([a])$ and $b' \in \lambda([b])$ such that $a'$ and $b'$ are nonseparating and they are the boundary components of
a projective plane with two boundary components on $N$.\end{lemma}

\begin{figure}[htb]
\begin{center}
\epsfxsize=3.1in \epsfbox{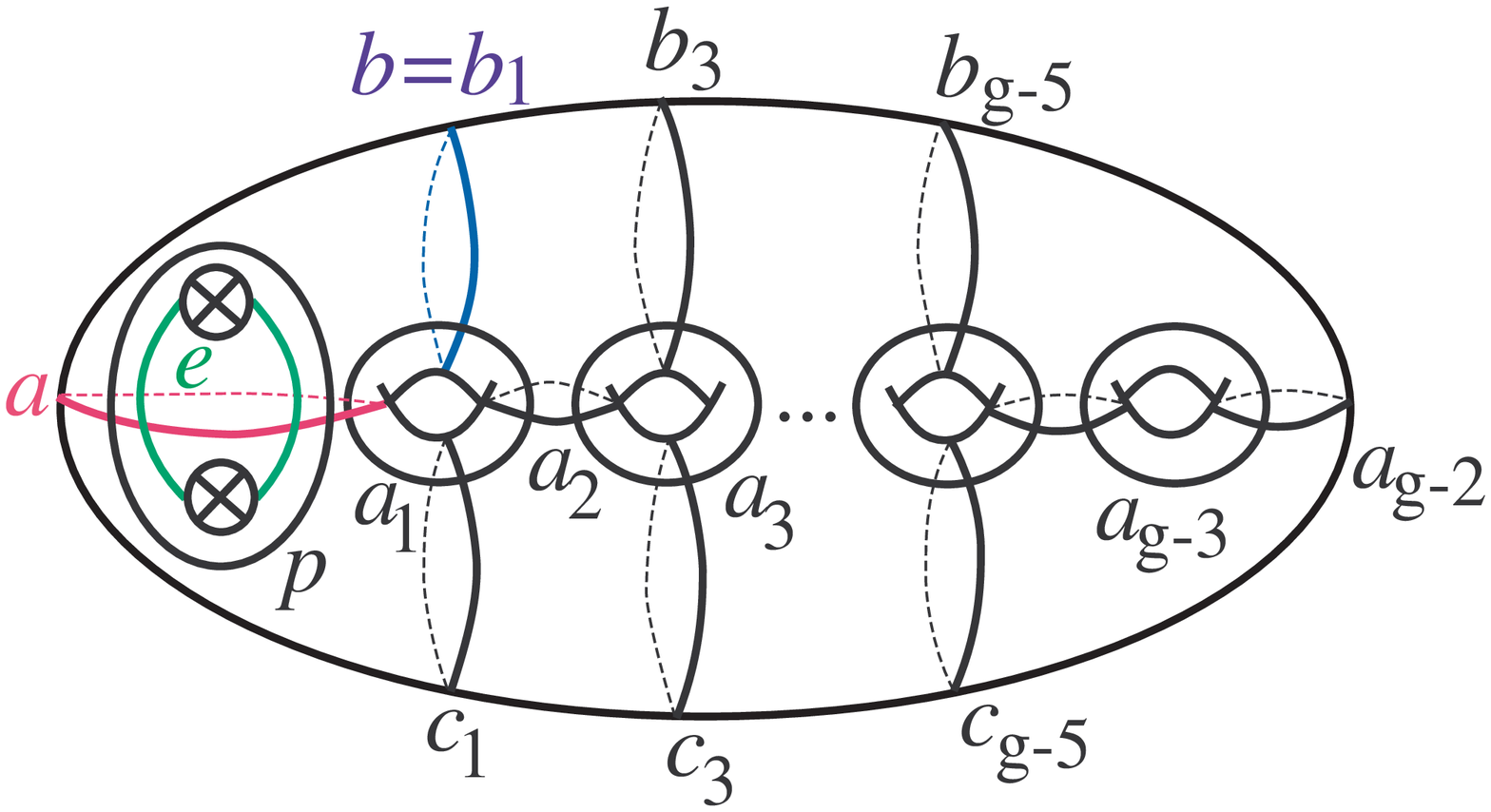} \hspace{-0.3cm} \epsfxsize=2.87in \epsfbox{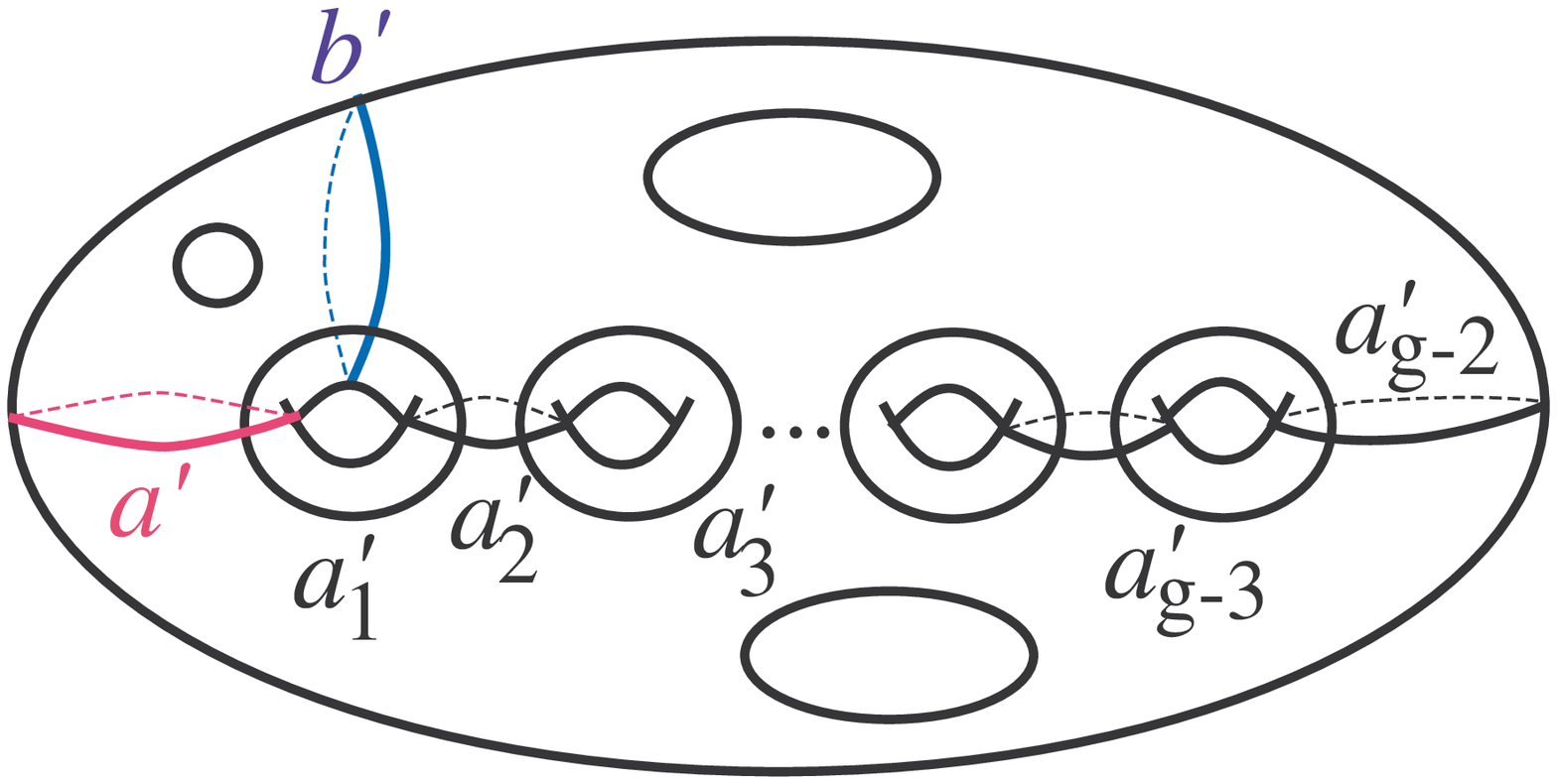}

\hspace{-0.9cm} (i) \hspace{6.5cm} (ii)

\epsfxsize=3.1in \epsfbox{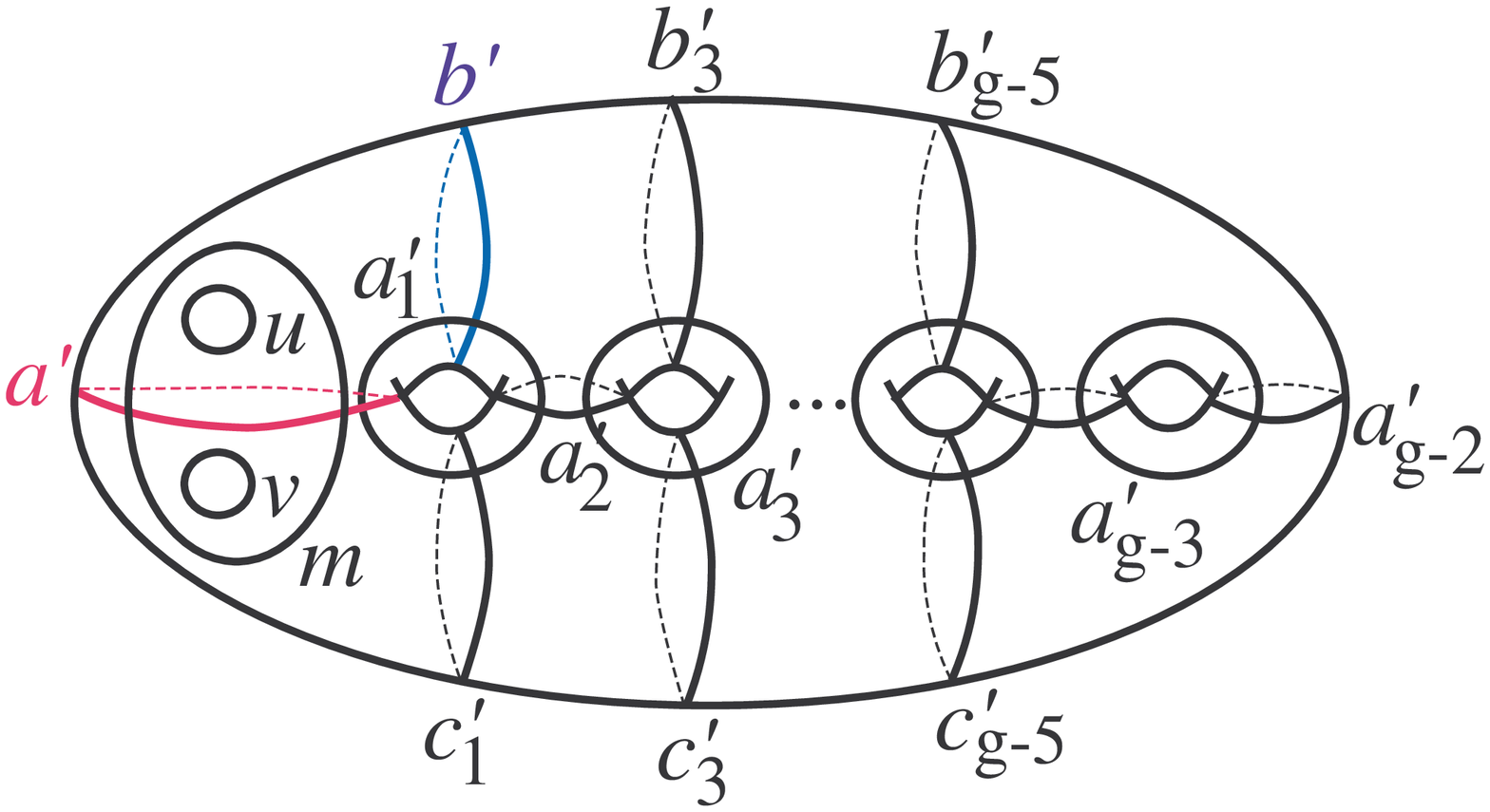}

\hspace{-1cm} (iii)

\caption{Curves on a closed surface of even genus, $g \geq 6$} \label{fig2a}
\end{center}
\end{figure}

\begin{proof} {\bf Case (i):} Assume that $N$ is closed. We complete $a, b$ to a curve configuration $\mathcal{C} = \{a, b, e, p, a_1,$
$a_2, \cdots, a_{g-2}, b_1, b_3,$ $b_5, \cdots, b_{g-5}, c_1, c_3, \cdots, c_{g-5} \}$ as shown in Figure \ref{fig2a} (i). Let
$P = \{e, p, a_2, a_4, a_6, \cdots, $ $a_{g-2}, b, b_3, b_5, \cdots, b_{g-5}, c_1, c_3, \cdots, c_{g-5}\}$. $P$ is a top dimensional
$P$-$S$ decomposition on $N$.

We have $i([a], [a_1]) =1, i([b], [a_1]) =1, i([a], [a_i]) =0, i([b], [a_i]) =0,$ for all $i= 2, 3, \cdots, g-2$, $i([a_j], [a_{j+1}]) =1$ for all
$j= 1, 2, \cdots, g-3$, and $i([a_j], [a_k]) =0$ for all $k \neq j \pm 1$. Let
$a' \in \lambda([a]), b' \in \lambda([b]), e' \in \lambda([e]), p' \in \lambda([p]), a'_1  \in \lambda([a_1]),$
$a'_2  \in \lambda([a_2]), \cdots, a_{g-2}' \in \lambda([a_{g-2}]), b'_3  \in \lambda([b_3]), b'_5  \in \lambda([b_5]),$
$ \cdots, b_{g-5}' \in \lambda([b_{g-5}]), c'_1  \in \lambda([c_1]), c'_3  \in \lambda([c_3]), \cdots,$ $ c_{g-5}' \in \lambda([c_{g-5}])$
be minimally intersecting representatives. By Lemma \ref{intone} geometric intersection number one is preserved. So,
$i([a'], [a'_1]) =1, i([b'], [a'_1]) =1, i([a'], [a'_i]) =0,  i([b'], [a'_i]) =0$ for all $i= 2, 3, \cdots, g-2, i([a'_j], [a'_{j+1}]) =1$ for all
$j= 1, 2, \cdots, g-3,$ and $i([a'_j], [a'_k]) =0$ for all $k \neq j \pm 1$. A regular neighborhood of $a' \cup b'\cup a_1' \cup \cdots \cup a_{g-2}'$
is an orientable surface of genus $\frac{g-2}{2}$ with three boundary components as shown in Figure \ref{fig2a} (ii). If three nonseparating
nontrivial curves (for example $b, a_2, b_3$) in $P$ bound a pair of pants, then it is easy to see that they bound a pair of pants
on a genus two orientable subsurface $S$ of $N$. So, by Lemma \ref{piece1}, they correspond to curves ($b', a_2', b_3'$ in our example)
which bound a pair of pants. We also know that geometric intersection zero and one are preserved. These imply that there is a subsurface
$R$ of $N$ such that $R$ is an orientable surface of genus $\frac{g-2}{2}$ with two boundary components $u, v$, and $R$ has
$a', b', a_1', \cdots, a_{g-2}', b_3', \cdots, b_{g-5}', c_1', \cdots, c_{g-5}'$ on it as shown in Figure \ref{fig2a} (iii).
Let $m$ be the curve that is shown in the figure. We will show that $p'$ is isotopic to $m$. Let
$P'= \{e', p', b', a'_2, a'_4, a'_6, \cdots, $ $a'_{g-2}, b'_3, b'_5, \cdots, b'_{g-5}, c'_1, c'_3, \cdots, c'_{g-5}\}$. We know that $P'$
is a top dimensional $P$-$S$ decomposition on $N$.
The curve $b$ is adjacent to four curves $p, c_1, b_3, a_2$ w.r.t. $P$. By Lemma \ref{adjacent} $b'$ is adjacent to four curves
$p', c'_1, b'_3, a'_2$ w.r.t. $P'$. Since there is already a pair of pants which has $b', a'_2, b_3'$ on its boundary,
there exists a pair of pants having $b', p', c_1'$ on its boundary. Since $p$ is disjoint from each of $b, a_1, c_1$, we see that $p'$
is disjoint from each of $b', a'_1, c'_1$. We also know that $p'$ is disjoint from all $a'_i, b'_j, c'_k$ in the figure. Since $p'$ is
disjoint from each of $b', a'_i, b'_j, c'_k$, and there exists a pair of pants having $b', p', c_1'$ on its boundary, we see that $p'$ is
isotopic to $m$. Since $p$ is adjacent to $e$ w.r.t. $P$, $p'$ is adjacent to $e'$  w.r.t. $P'$. So, there exists a pair of pants $Q$ on $N$
containing $p'$ and $e'$ on its boundary. Let $x$ be the third boundary component of $Q$. Since there are no other curves of $P'$ which lie
on the side of $p'$ containing $e'$, $x$ must be isotopic to $e'$. Since $N$ is a closed nonorientable surface of genus $g$ and one side of
$p'$ is an orientable surface of genus $\frac{g-2}{2}$, we see that $p'$ bounds a nonorientable surface of genus two with one boundary
component on the other side.

Now, consider the curves $u, v$ in Figure \ref{fig2a} (iii). They are two sided curves and they do lie in the Klein bottle with one boundary
component, $K$, having $m$ as the boundary. None of them can be bounding a disk because $a'$ is not isotopic to any of $b'$ or $c_1'$. So,
the curves $u, v$ cannot be isotopic to $m$. Hence, they either both have to bound M\"{o}bius bands or they have to be both nonseparating,
isotopic to each other and to $e'$, as up to isotopy that is the only nonseparating two sided curve in $K$, see \cite{Sc}. The second case gives a
contradiction as $a'$ should intersect nontrivially with $e'$, as $a$ intersects nontrivially with $e$. So, the first case happens, i.e.
$u, v$ both bound M\"{o}bius bands. Hence, $a'$ and $b'$ are nonseparating and they are the boundary components of a projective plane with
two boundary components on $N$.

{\bf Case (ii):} Assume that $N$ has boundary. We complete $a, b$ to a curve configuration
$\mathcal{C} = \{a, b, e, p, a_1,$ $a_2, \cdots, a_{g-3}, b_3, b_5, \cdots, $ $b_{g-3}, c_1, c_3, \cdots, c_{g-3}, d_1,$
$ d_2, \cdots, d_{n-1} \}$ as shown in Figure \ref{fig1b} (i). Let $P = \{e, p, b, a_2, a_4, a_6, \cdots, $
$a_{g-4}, b_3, b_5, \cdots,$ $ b_{g-3}, c_1, c_3, \cdots,$ $ c_{g-3}, d_1, d_2, \cdots, d_{n-1} \}$. $P$ is a top dimensional
$P$-$S$ decomposition on $N$. Let $a' \in \lambda([a]), b' \in \lambda([b]), e' \in \lambda([e]), p' \in \lambda([p]),$
$ a'_1  \in \lambda([a_1]), a'_2  \in \lambda([a_2]), \cdots, $ $ a_{g-3}' \in \lambda([a_{g-3}]),$
$ b'_3  \in \lambda([b_3]), b'_5  \in \lambda([b_5]), \cdots, b_{g-3}' \in \lambda([b_{g-3}]), c'_1  \in \lambda([c_1]),$
$ c'_3  \in \lambda([c_3]), \cdots, c_{g-3}' \in \lambda([c_{g-3}]), d'_1 \in \lambda([d_1]),$ $ d'_2 \in \lambda([d_2]),$
$ \cdots, d_{n-1}' \in \lambda([d_{n-1}])$ be minimally intersecting representatives.

\begin{figure}[htb]
\begin{center}
\epsfxsize=3.25in \epsfbox{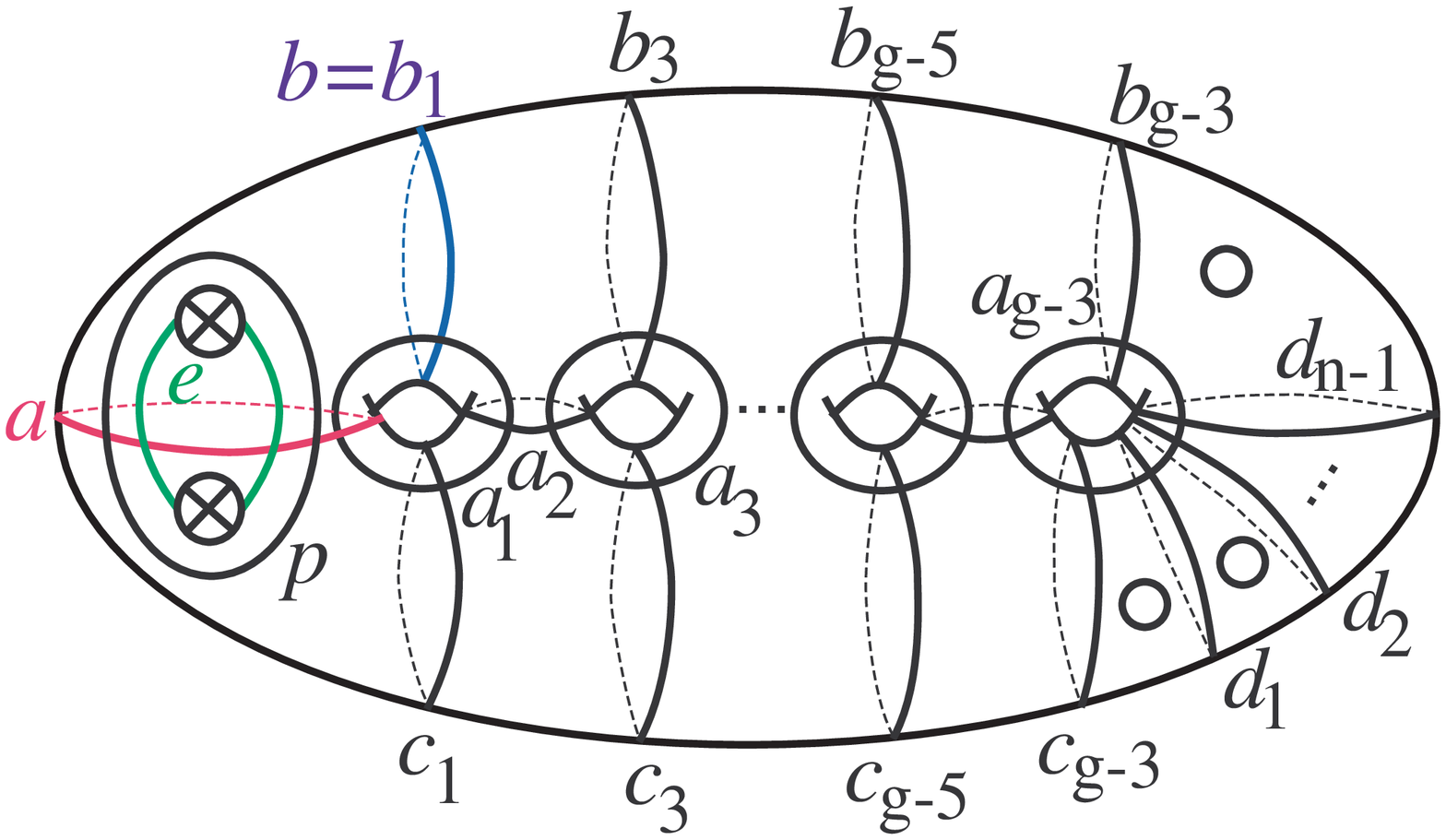} \hspace{-1.3cm} \epsfxsize=3.25in \epsfbox{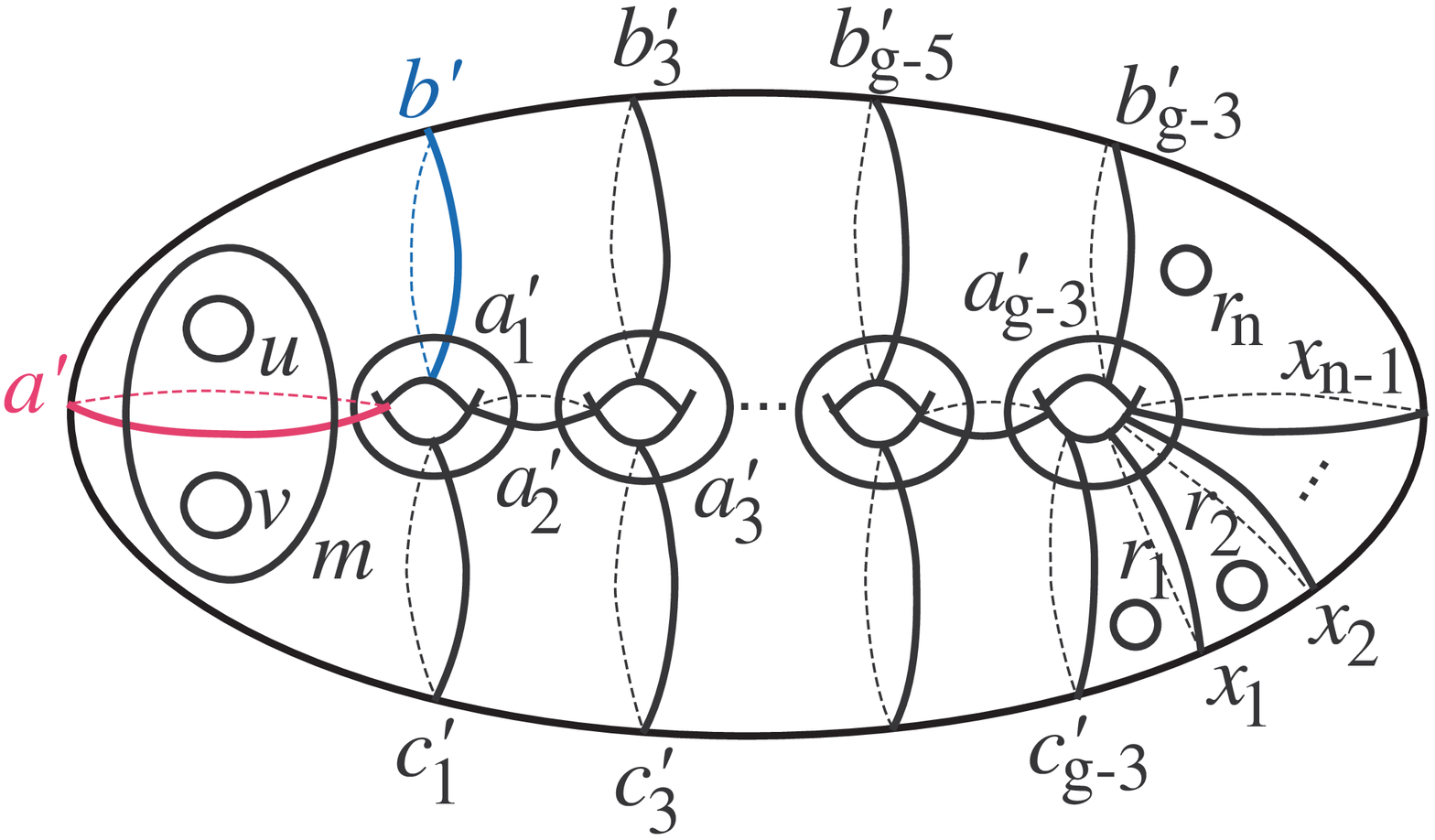}

\hspace{-1cm} (i) \hspace{6.5cm} (ii)
\caption{Curves on a surface with boundary, even genus, $g \geq 6$} \label{fig1b}
\end{center}
\end{figure}

By using that geometric intersection number zero and one is preserved we can see that a regular neighborhood of the union of all the
elements in $\{ a', b', a'_1,$ $a'_2, \cdots, a'_{g-3}, b'_3,$ $b'_5, \cdots,$ $b'_{g-3}, c'_1, c'_3, \cdots, c'_{g-3}, d'_1, d'_2,$
$ \cdots, d'_{n-1} \}$ is an orientable surface of genus $\frac{g-2}{2}$ with several boundary components. As in the first case, by
using that if three nonseparating nontrivial curves bound a pair of pants then they correspond to curves which bound a pair of pants,
we see that there is a subsurface $R$ of $N$ such that $R$ is an orientable surface of genus $\frac{g-2}{2}$ with $n+2$ boundary
components, and $R$ has
$a', b', a_1', \cdots, a_{g-3}', b_3', \cdots, b_{g-3}',$ $c_1', \cdots, c_{g-3}'$ on it as shown in Figure \ref{fig1b} (ii).
It is easy to see that the set $\{x_1, \cdots, x_{n-1}\}$, where the curves $x_i$ are as shown in the figure, correspond to the set
$\{d'_1, d'_2, \cdots, d'_{n-1} \}$.

Let $u, v, r_1, r_2, \cdots, r_n$ be the boundary components of $R$, and $m$ be the curve shown in the figure. We will show that $p'$
is isotopic to $m$. We have a top dimensional $P$-$S$ decomposition $P'= \{e', p', b', a'_2, a'_4, a'_6, \cdots, $
$a'_{g-4}, b'_3, b'_5, \cdots, b'_{g-3},$ $ c'_1, c'_3, \cdots, c'_{g-3}, d'_1, d'_2,$ $ \cdots, d'_{n-1}\}$. The curve $b$ is adjacent
to four curves $p, c_1, b_3, a_2$ w.r.t. $P$. By Lemma \ref{adjacent} $b'$ is adjacent to four curves $p', c'_1, b'_3, a'_2$ w.r.t. $P'$.
Since there is already a pair of pants which has $b', a'_2, b_3'$ on its boundary,
there exists a pair of pants having $b', p', c_1'$ on its boundary. Since $p$ is disjoint from each of $b, a_1, c_1$, we see that $p'$
is disjoint from each of $b', a'_1, c'_1$. We also know that $p'$ is disjoint from all $a'_i, b'_j, c'_k, d'_m$ in the figure. Since $p'$
is disjoint from each of $b', a'_i, b'_j, c'_k, d'_m$ and there exists a pair of pants having $b', p', c_1'$ on its boundary, we see that
$p'$ is isotopic to $m$.
Since $p$ is adjacent to $e$ w.r.t. $P$, $p'$ is adjacent to $e'$  w.r.t. $P'$. So, there exists a pair of pants $Q$ on $N$ containing $p'$
and $e'$ on its boundary. Let $x$ be the third boundary component of $Q$. The curve $e$ is not adjacent to $d_1$ w.r.t. $P$. There are two
2-sided nonisotopic disjoint curves $\xi_1, \xi_2$ such that $\xi_1$ intersects only $e$ and $p$ nontrivially in $P$, and $\xi_2$ intersects
only $d_1$ nontrivially in $P$. Let $\xi'_1, \xi'_2$ be the corresponding disjoint curves which have minimal intersection with the elements
of $P'$. Then $\xi'_1$ intersects only $e', p'$ nontrivially in $P'$, and $\xi'_2$ intersects only $d'_1$ nontrivially in $P'$. So, $e'$
cannot be adjacent to $d'_1$ w.r.t. $P'$. Similar arguments show that $e'$ cannot be adjacent to any of
$b'_{g-3}, c'_{g-3}, d'_2, d'_3, \cdots, d'_{n-1} $ w.r.t. $P'$. This implies that $e'$ is adjacent to only $p'$ in $P'$. It is easy to
see that $e'$ is a nonseparating curve. If $e'$ was a separating curve then since we have all the curves in $P'$ on the surface and $P'$
is a a top dimensional $P$-$S$ decomposition, our surface $N$ would have more than $n$ boundary components and that gives a contradiction.
Now we can see that $x$ is isotopic to $e'$, all the $r_i$'s are isotopic to boundary components of
$N$, and $p'$ bounds a nonorientable surface of genus 2 with one boundary component on its side containing $e'$.

Now, similar to the first case, consider the curves $u, v$ in Figure \ref{fig1b} (ii). They are 2 sided curves and they do lie in the Klein
bottle with one boundary component, $K$, having $m$ as the boundary. As in the first case, none of them can be bounding a disk because $a'$
is not isotopic to any of $b'$ or $c_1'$. So, the curves $u, v$ cannot be isotopic to $m$. Hence, they either have to both bound M\"{o}bius
bands or they have to be both nonseparating and isotopic to each other and to $e'$. The second case gives a contradiction as $a'$ should
intersect nontrivially with $e'$, as $a$ intersects nontrivially with $e$. So, the first case happens, i.e. $u, v$ both bound M\"{o}bius
bands. Hence, $a'$ and $b'$ are nonseparating and they are the boundary components of a projective plane with two boundary components on $N$.
Since $d_1$ is adjacent to $d_2$ w.r.t. $P$, $d'_1$ is adjacent to $d'_2$ w.r.t. $P'$. By using other similar adjacency relations, now it is
easy to see that $x_i$ is isotopic to $d'_i$ for each $i$.\end{proof}\\

In the proof of Case (ii) in the above lemma we also proved that if two nonseparating curves in $P$ cut a pair of pants $M$ on $N$
such that the third boundary component of $M$ is a boundary component of $N$ and the complement of $M$ is nonorientable, then the corresponding
nonseparating curves in $P'$ are also the boundary components of a pair of pants such that the third boundary component is a boundary
component of $N$. Therefore we have also proved the following:

\begin{lemma} \label{piece2-aa} Suppose that $g \geq 6, n \geq 1$ and $g$ is even. Let $a, b$ be two nonseperating curves on $N$ such that together with
a boundary component of $N$ they bound a pair of pants $P$ on $N$, and the complement of $P$ is connected and nonorientable. There exists $a'  \in \lambda([a])$
and $b'  \in \lambda([b])$ such that $a'$ and $b'$ are nonseparating and together with a boundary component of $N$ they bound a pair of pants
on $N$.\end{lemma}

\begin{figure}[htb]
\begin{center}
\epsfxsize=3.1in \epsfbox{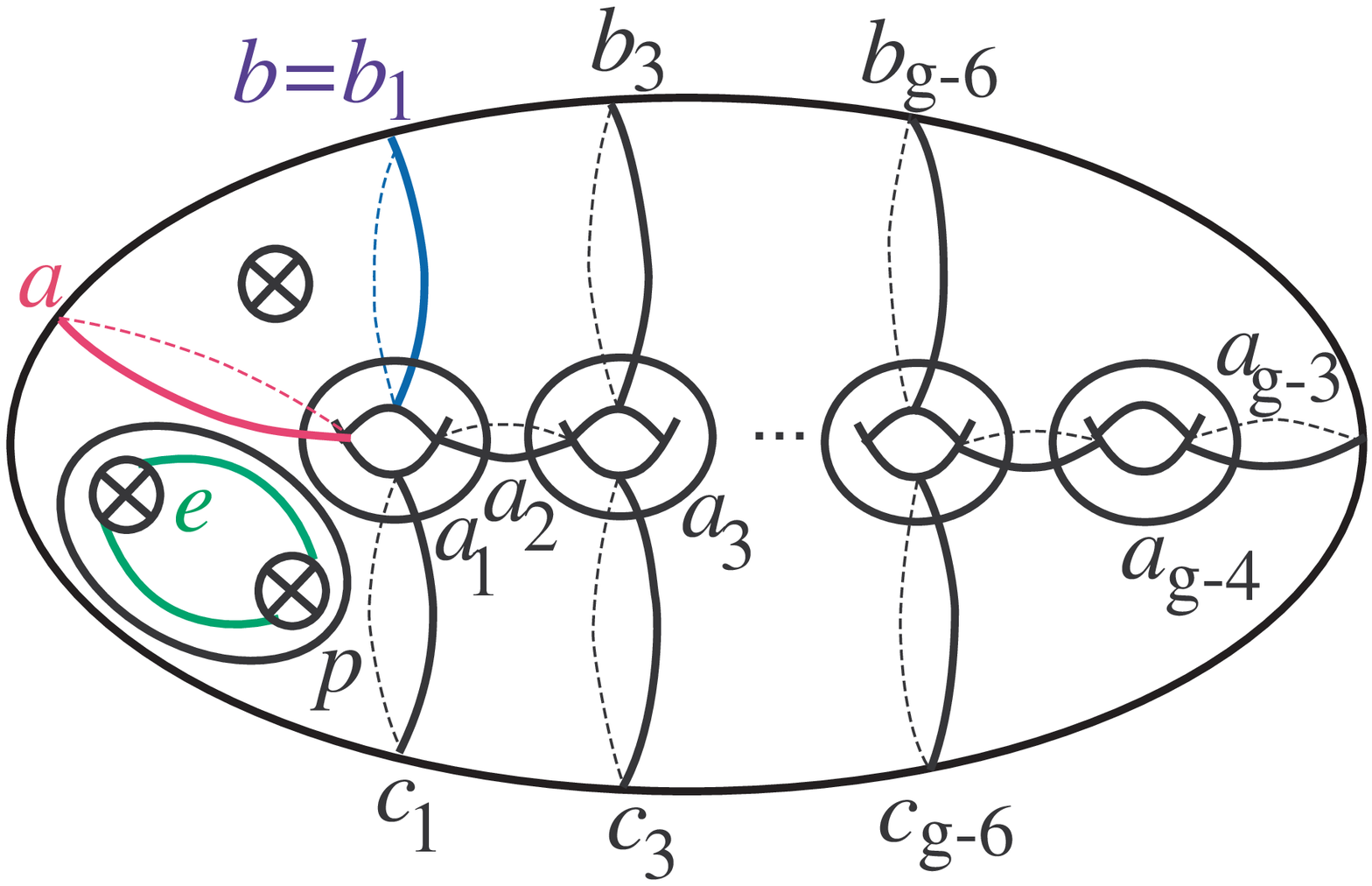} \hspace{-1cm}
\epsfxsize=3.1in \epsfbox{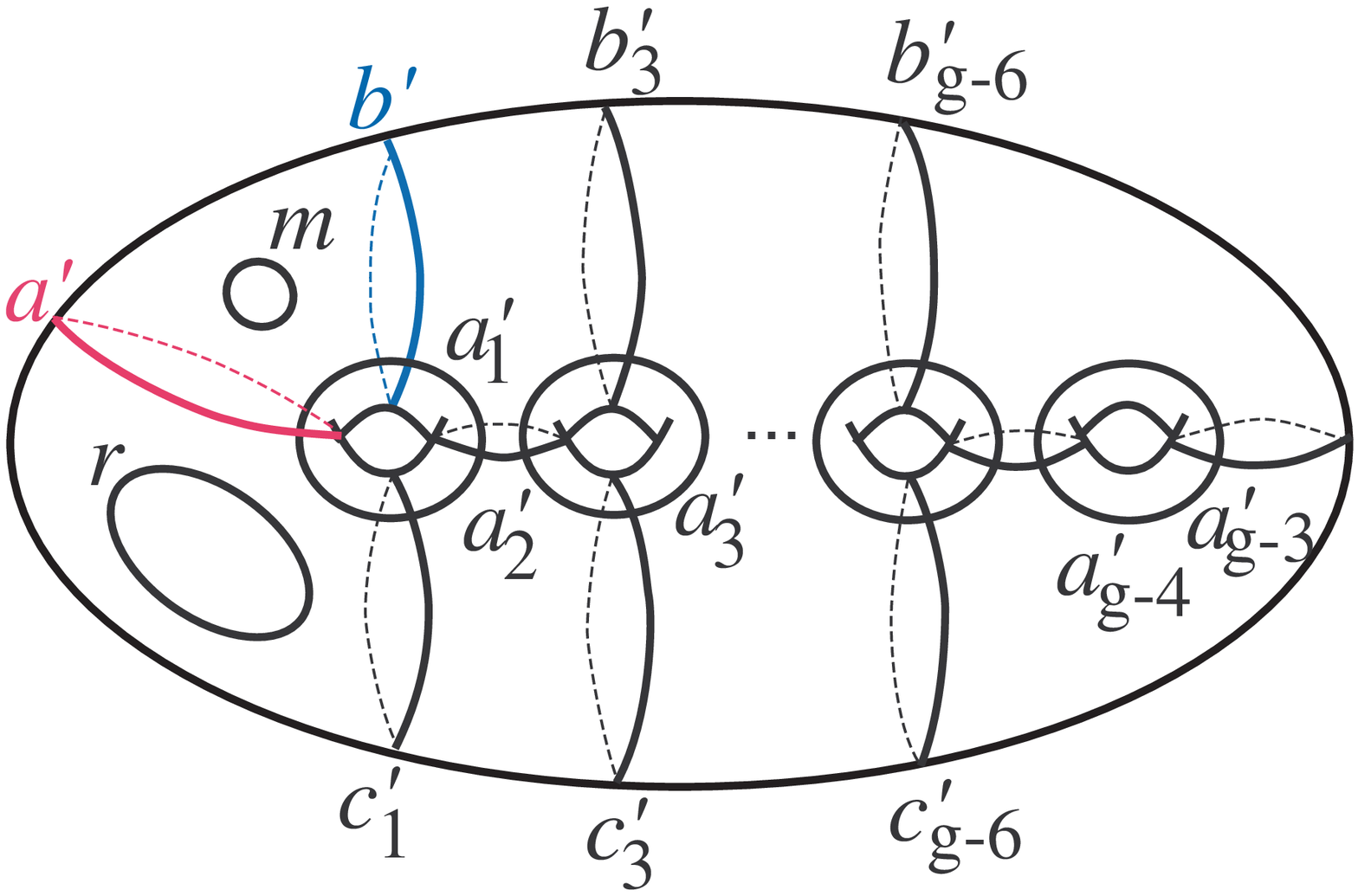}

\hspace{-1cm} (i) \hspace{6.5cm} (ii)
\caption{Curves on a closed surface of odd genus, $g \geq 7$} \label{fig3}
\end{center}
\end{figure}

\begin{lemma}
\label{piece2-b} Let $g \geq 5$. Suppose that $g$ is odd. Let $R \subset N$ be a projective plane with two boundary components, $a, b$
where $a, b$ are both nonseparating simple closed curves on $N$ such that the complement of $R$ is nonorientable. There exist
$a' \in \lambda([a])$ and $b'  \in \lambda([b])$ such that $a'$ and $b'$ are nonseparating and they are the boundary components of a
projective plane with two boundary components on $N$.\end{lemma}

\begin{proof} {\bf Case (i):} Assume that $N$ is closed and $g \geq 7$. We complete $a, b$ to a curve configuration
$\mathcal{C} = \{a,  b, e, p, a_1,$ $a_2, \cdots, a_{g-3}, b_3,$ $b_5, \cdots, b_{g-6}, c_1, c_3, \cdots, c_{g-6} \}$ as shown in
Figure \ref{fig3} (i). We have a top dimensional $P$-$S$ decomposition $P = \{a, b, e, p, a_2,$ $ a_4,$ $\cdots, $
$a_{g-3}, b_3, \cdots, b_{g-6}, c_1,$ $ c_3, \cdots, c_{g-6}\}$ on $N$.
Let $a' \in \lambda([a]), b' \in \lambda([b]), e' \in \lambda([e]), p' \in \lambda([p]), a'_1  \in \lambda([a_1]),$
$a'_2  \in \lambda([a_2]), \cdots, a_{g-3}' \in \lambda([a_{g-3}]), b'_3  \in \lambda([b_3]), b'_5  \in \lambda([b_5]), \cdots,$
$ b_{g-6}' \in \lambda([b_{g-6}]), c'_1  \in \lambda([c_1]), c'_3  \in \lambda([c_3]), \cdots, c_{g-6}' \in \lambda([c_{g-6}])$
be minimally intersecting representatives. Let $\mathcal{C'}$ be the set of these representatives. Let
$P' = \{a', b', e', p', a'_2, a'_4, \cdots, a'_{g-3}, b'_3, b'_5, \cdots, $ $ b'_{g-6}, c'_1, c'_3, \cdots, c'_{g-6}\}$.

As in the proof of Lemma \ref{piece2-a}, there is a subsurface $R$ of $N$ such that $R$ is an orientable surface of genus $\frac{g-3}{2}$
with two boundary components $m, r$, and $R$ has $a', b', a'_1, a'_2,$
$ \cdots, a'_{g-3}, b'_3, b'_5, \cdots, b'_{g-6}, c'_1, c'_3, \cdots, c'_{g-6}$ on it as shown in Figure \ref{fig3} (ii). The
curve $c_1$ is adjacent to four curves $a, p, a_2, c_3$ w.r.t. $P$. By Lemma \ref{adjacent} $c'_1$ is adjacent to four curves
$a', p', a'_2, c'_3$ w.r.t. $P'$. Since there is already a pair of pants which has $c'_1, a'_2, c'_3$ on its boundary,
there exists a pair of pants having $c'_1, a', p'$ on its boundary. Since $p$ is disjoint from each of $a, a_1, c_1$, we see that
$p'$ is disjoint from each of $a', a'_1, c'_1$. We also know that $p'$ is disjoint from all $a'_i, b'_j, c'_k$ in the figure.
Since $p'$ is disjoint from each of $b', a'_i, b'_j, c'_k$, and there exists a pair of pants having $c'_1, a', p'$ on its boundary,
we see that $p'$ is isotopic to $r$. Since $p$ is adjacent to $e$ w.r.t. $P$, $p'$ is adjacent to $e'$ w.r.t. $P'$. So, there exists a
pair of pants containing $p'$ and $e'$ as boundary components. Let $x$ be the other boundary component of this pair of pants. We see
that $e'$ is a nonseparating curve on $N$ by using Corollary \ref{nonsep}. By using the technique given in the proof of
Lemma \ref{piece2-a} (i), we see that since $e$ is not adjacent to $b$ w.r.t. $P$, $e'$ is not adjacent to $b'$ w.r.t. $P'$.
These imply that $x$ is isotopic to $e'$. Hence, $p'$ either bounds a torus with one boundary component or a nonorientable surface
of genus two with one boundary component. In either case, $m$ has to bound a M\"{o}bius band. Hence, $a'$ and $b'$ are nonseparating
and they are the boundary components of a projective plane with two boundary components on $N$.

\begin{figure}[htb]
\begin{center}
\epsfxsize=3.25in \epsfbox{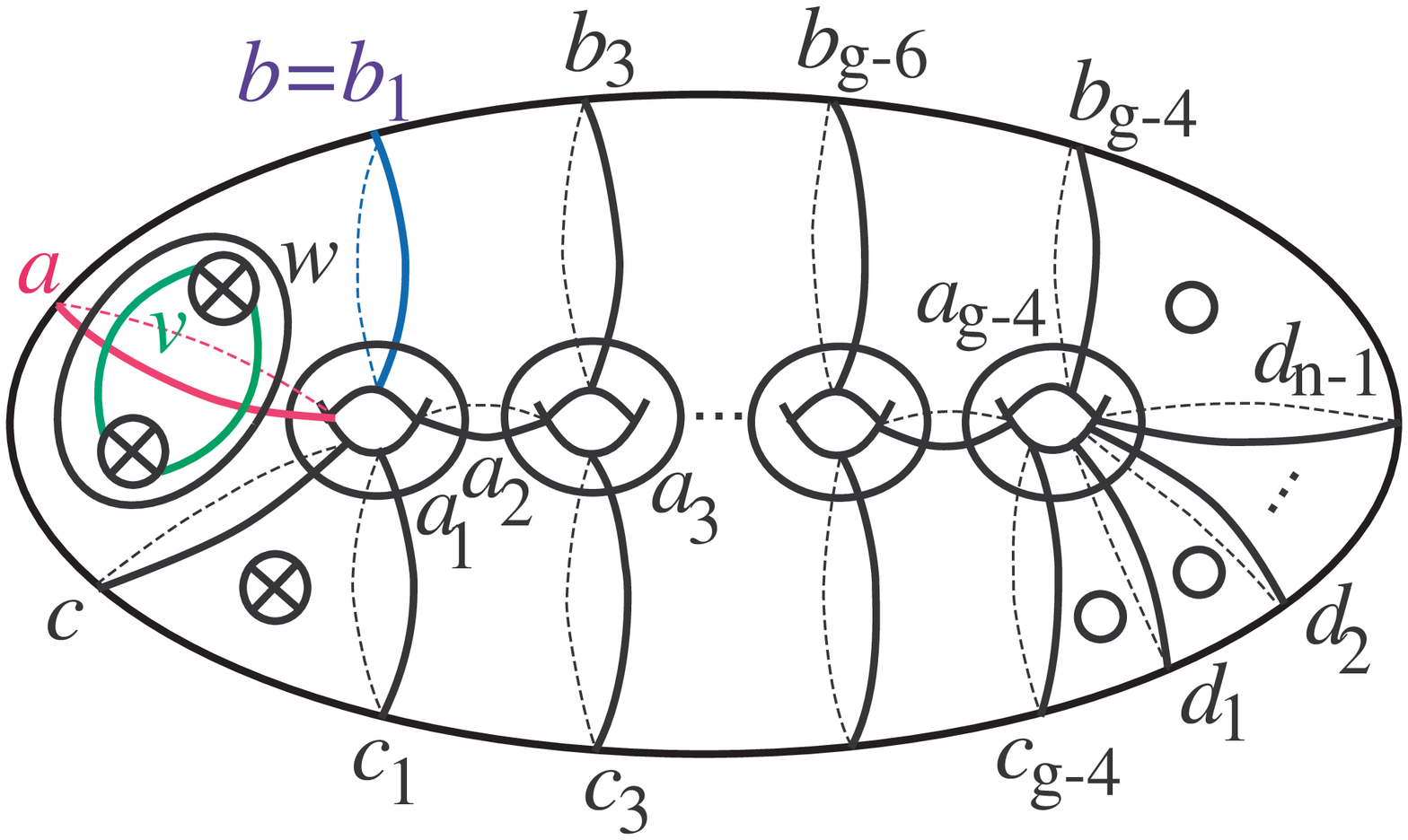} \hspace{-1.5cm} \epsfxsize=3.25in \epsfbox{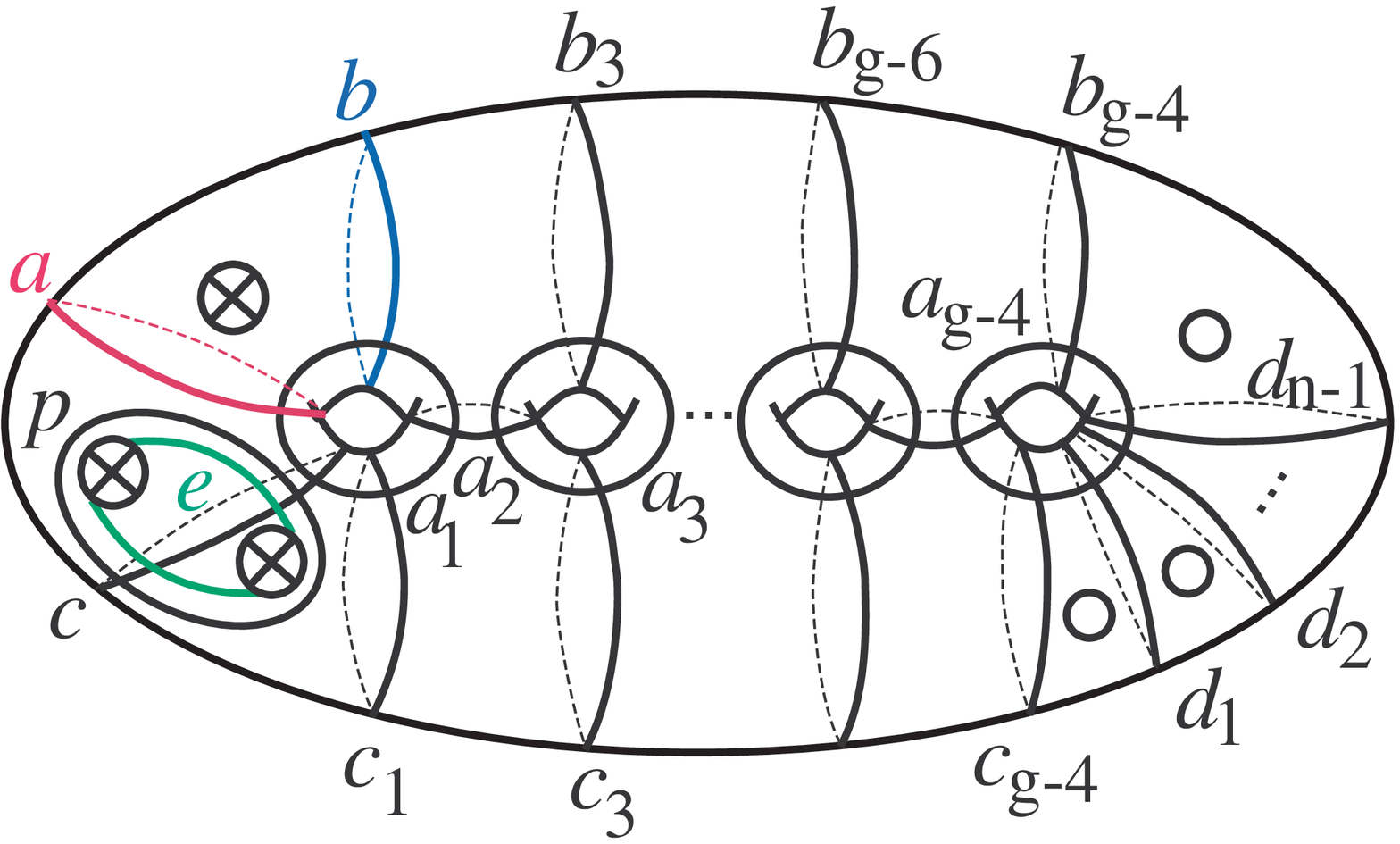}

\hspace{-1cm} (i) \hspace{6.5cm} (ii)
\caption{Curves on a surface with boundary, odd genus, $g \geq 7$} \label{fig1c}
\end{center}
\end{figure}

{\bf Case (ii):} Assume that $g \geq 7$ and $n \geq 1$. We complete $a, b$ to a curve configuration
$\mathcal{C} = \{a, b, c, e, p, v, w, a_1,$ $a_2, \cdots, a_{g-4}, b_3,$ $b_5, \cdots, $ $b_{g-4}, c_1, c_3, \cdots, c_{g-4}, d_1,$
$d_2, \cdots, $ $d_{n-1}\}$ as shown in Figure \ref{fig1c} (i), (ii). We have a top dimensional $P$-$S$ decomposition
$P = \{v, w, b, c, a_2, a_4, $ $ \cdots, $ $a_{g-5}, b_3, b_5, \cdots, b_{g-4}, c_1, c_3, \cdots, c_{g-4}, d_1, d_2, \cdots, d_{n-1}\}$ on $N$.

We choose minimally intersecting representatives of the elements in $\{\lambda([x]): x \in \mathcal{C}\}$. Let
$\mathcal{C'}= \{a', b', c', e', p', v', w', a'_1,$ $a'_2, \cdots, a'_{g-4}, b'_3,$ $b'_5, \cdots, $
$b'_{g-4}, c'_1, c'_3, \cdots, c'_{g-4}, d'_1,$ $d'_2, \cdots, $ $d'_{n-1}\}$ be
the set of these elements. We use the notation that $x' \in \lambda([x])$ for all $x \in \mathcal{C}$. As in the proof of
Lemma \ref{piece2-a}, by using that geometric intersection number zero, nonzero and one are preserved, adjacency w.r.t. $P$ is preserved and if
three nonseparating nontrivial curves bound a pair of pants then they correspond to curves which bound a pair of pants, we see that
there is a subsurface $R$ of $N$ such that $R$ is an orientable surface of genus $\frac{g-3}{2}$ with $n+3$ boundary components, and
$R$ has $a', b', c', a'_1, a'_2, \cdots, a'_{g-4}, b_3', b'_5, \cdots, b_{g-4}', c_1', c'_3, \cdots, c_{g-4}', d'_1, d'_2, \cdots, d'_{n-1}$
on it as shown in Figure \ref{fig1d}, $\{a', c'\} = \{x, y\}$ and $\{d'_1, d'_2, \cdots, d'_{n-1} \} = \{v_1, v_2, \cdots, v_{n-1} \}$
where $x, y, v_1, v_2, \cdots, v_{n-1}$ are curves as shown in the figure. By using that adjacency with respect to top dimensional $P$-$S$ decompositions and intersection zero and one are preserved 
it is easy to see that $d'_i = v_i$ for all $i = 1, 2, \cdots, n-1$. Now by using that intersection
zero and nonzero are preserved, we see that $x = a', y= c'$.

\begin{figure}[htb]
\begin{center}
\hspace{2cm} \epsfxsize=3.3in \epsfbox{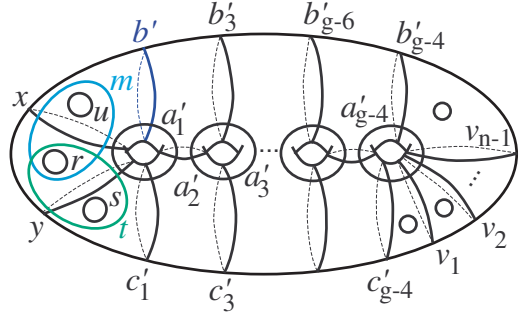}
\caption{Curve configuration II} \label{fig1d}
\end{center}
\end{figure}

$P' = \{v', w', b', c', a'_2, a'_4, \cdots, $ $a'_{g-5}, b'_3, b'_5, \cdots, b'_{g-4}, c'_1, c'_3, \cdots, c'_{g-4}, d'_1, d'_2, \cdots, d'_{n-1}\}$
is a top dimensional $P$-$S$ decomposition on $N$. Let $u, r, s$ be three of the boundary components of $R$, and let $m, t$ be the curves
as shown in the figure. Since $a', b', c', c'_1$ are pairwise nonisotopic, none of $u, r, s$ bounds a disk on $N$. We will show that $w'$
is isotopic to $m$. The curve $b$ is adjacent to four curves $w, c, a_2, b_3$ w.r.t. $P$. By Lemma \ref{adjacent} $b'$ is adjacent to four
curves $w', c', a'_2, b'_3$ w.r.t. $P'$. Since there is already a pair of pants which has $b', a'_2, b_3'$ on its boundary, there exists a
pair of pants having $w', b', c'$ on its boundary. Since $w$ is disjoint from each of $b, a_1, c$, we see that $w'$ is disjoint from each of
$b', a'_1, c'$. We also know that $w'$ is disjoint from all $a'_i, b'_j, c'_k, d'_m$ in the figure. Since $w'$ is disjoint from each of
$b', a'_i, b'_j, c'_k, d'_m$, and there exists a pair of pants having $w', b', c'$ on its boundary, we see that $w'$ is isotopic to $m$.
Since $w$ is adjacent to $v$ w.r.t. $P$, $w'$ is adjacent to $v'$  w.r.t. $P'$. So, there exists a pair of pants $Q$ (not necessarily essential,
i.e. one of he boundary components of $Q$ could bound a M\"{o}bius band) on $N$ containing $w'$
and $v'$ on its boundary. Let $z$ be the third boundary component of $Q$. As before we can see that since $v$ is not adjacent to any of the
curves $c_1, c_{g-4}, d_i$ for $i = 1, \cdots, n-1$ w.r.t. $P$,  $v'$ is not adjacent to $c'_1, c'_{g-4}, d'_i$ for $i = 1, \cdots, n-1$ w.r.t.
$P'$. Since $v$ is a nonseparating curve with nonorientable complement, $v'$ is a nonseparating curve by Corollary \ref{nonsep}. Now we see
that $z$ is isotopic to $v'$, and $w'$ either bounds a torus with one boundary component or a nonorientable surface of genus two with one
boundary component. This implies that either $s$ bounds a M\"{o}bius band or $s$ is isotopic to a boundary component of $N$. By using the
curves, $e, p$ and using similar arguments we see that $p'$ is isotopic to $t$, and $p'$ either bounds a torus with one boundary component
or a nonorientable surface of genus 2 with one boundary component. In either case, we see that $s$ can't be isotopic to a boundary component
of $N$. Hence, $s$ bounds a M\"{o}bius band. Then, $t$ bounds a nonorientable surface of genus two with one boundary component $K$, and
hence $r$ bounds a M\"{o}bius band. Similarly, $u$ bounds a M\"{o}bius band. This implies that each of the other boundary components of
$R$ is isotopic to a boundary component of $N$. Hence, $a'$ and $b'$ are nonseparating and they are the boundary components of a projective
plane with two boundary components on $N$. We proved that in all cases if two curves in $P$ are the boundary components of a pair of pants
such that the third boundary component is a boundary component of $N$, then the corresponding curves in $P'$ are also the boundary
components of a pair of pants such that the third boundary component is a boundary component of $N$.

\begin{figure}[htb]
\begin{center}
\hspace{0.7cm} \epsfxsize=2.5in \epsfbox{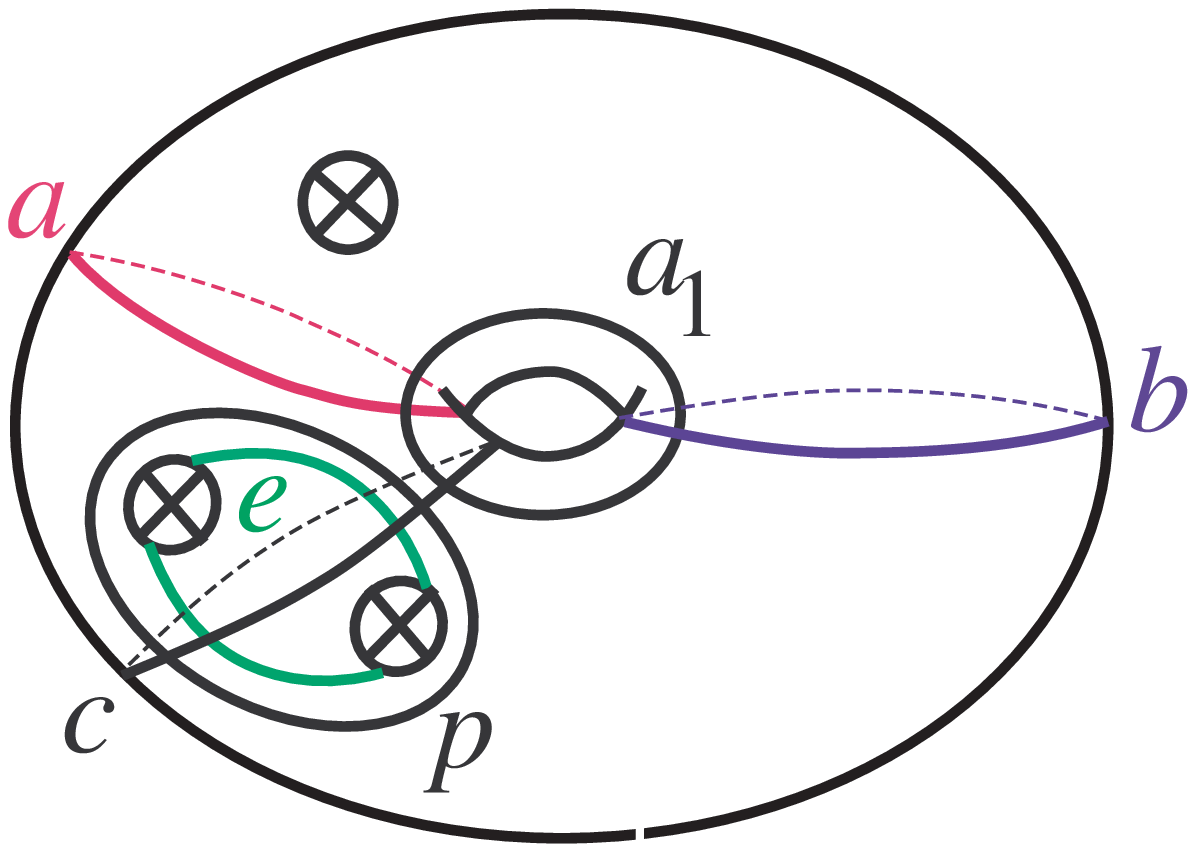} \hspace{-1cm}
\hspace{0.7cm} \epsfxsize=2.3in \epsfbox{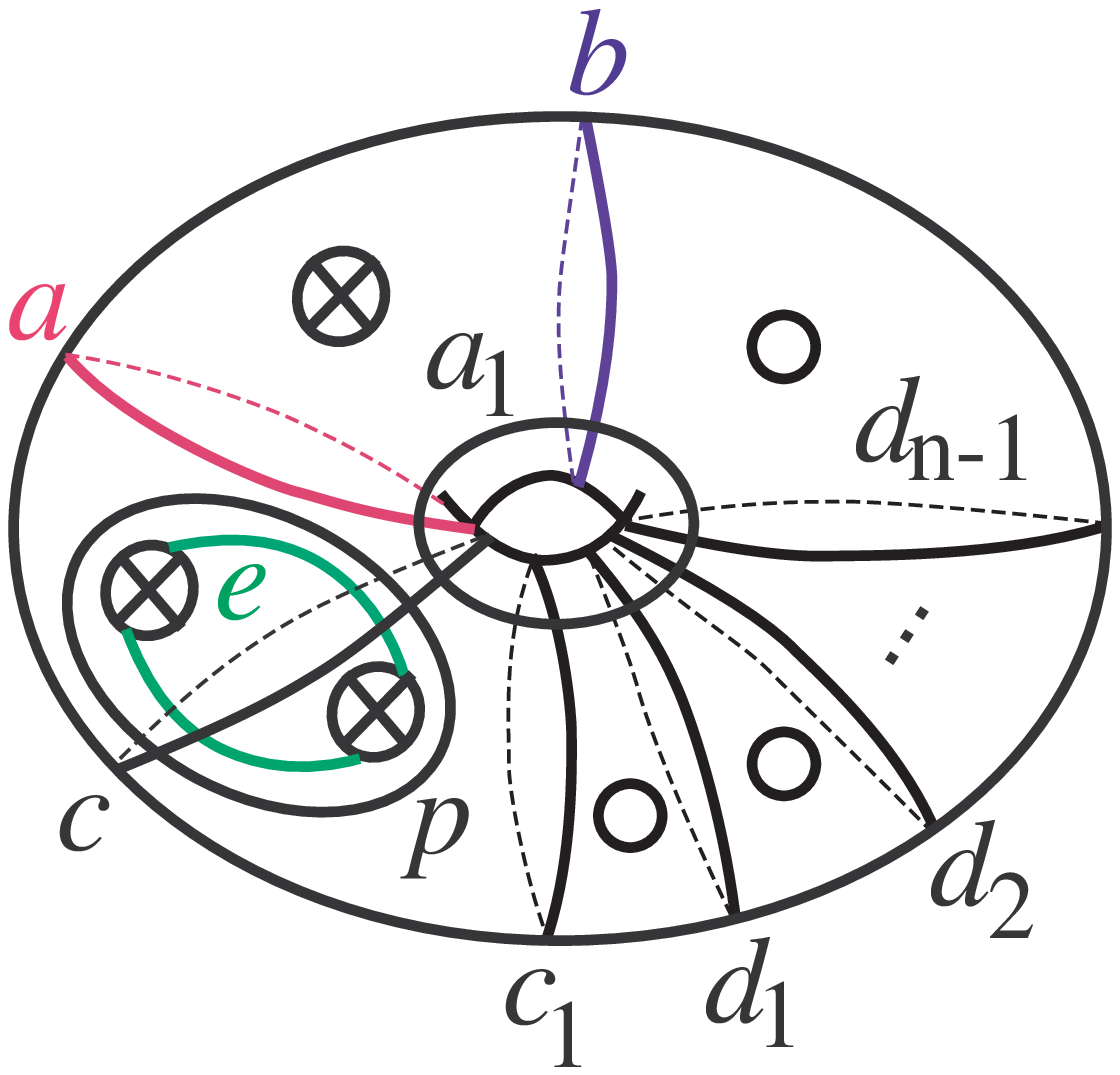}

\hspace{-0.7cm} (i) \hspace{5.5cm} (ii)

\hspace{0.8cm} \epsfxsize=2.3in \epsfbox{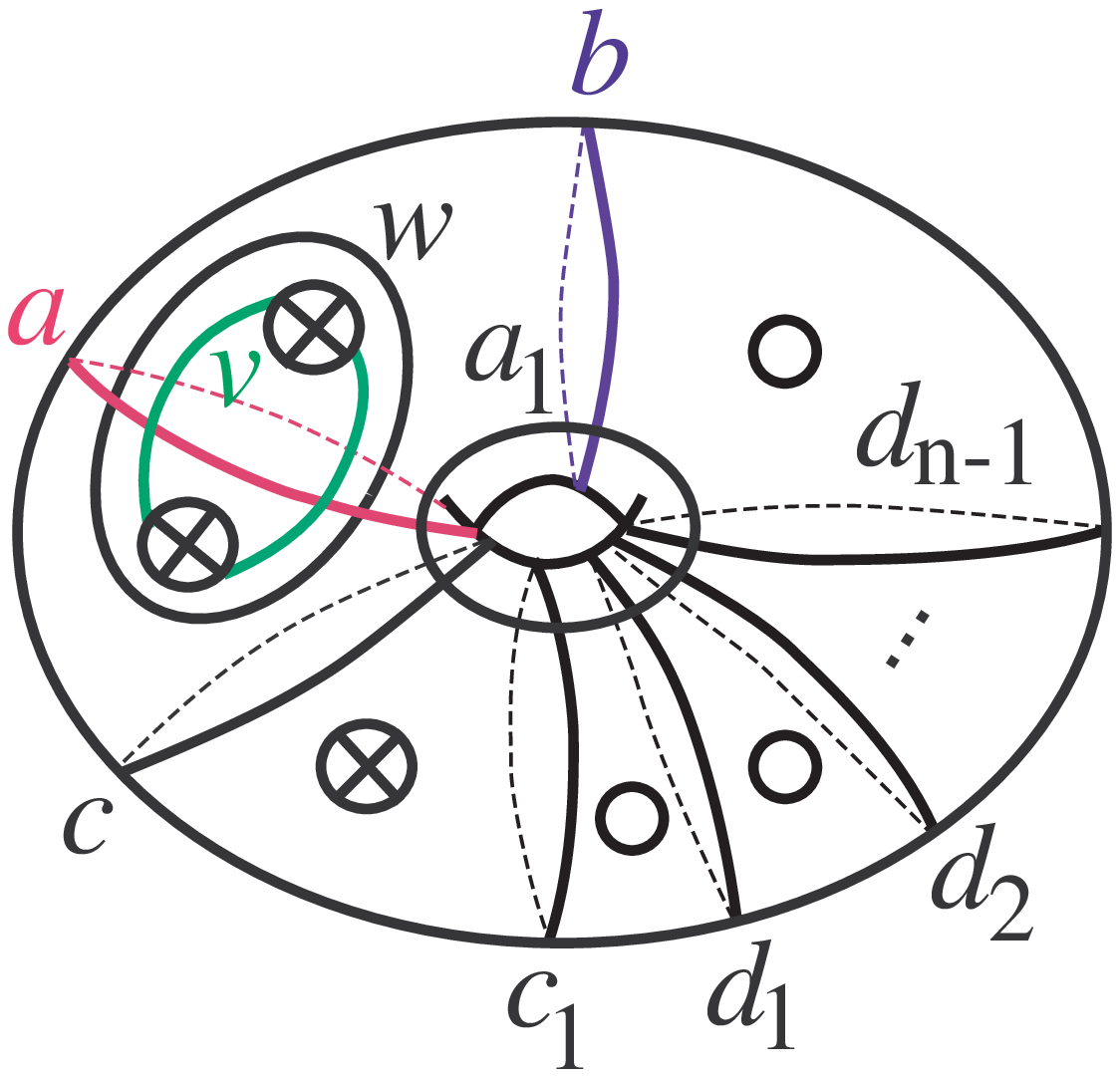}\hspace{-1cm}
\hspace{1.4cm} \epsfxsize=2.3in \epsfbox{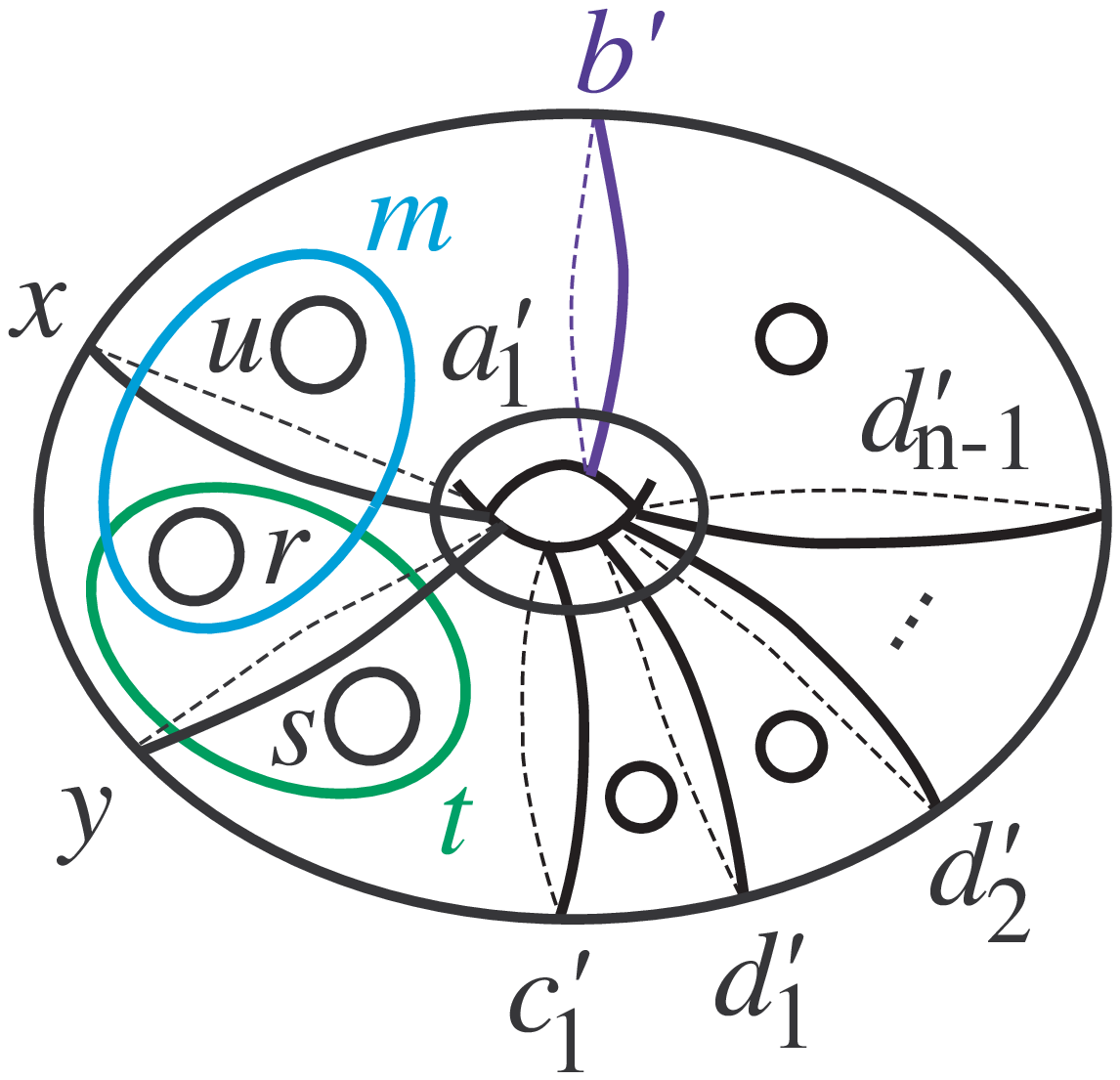}

\hspace{-0.7cm} (iii) \hspace{5.5cm} (iv)

\caption{Curves on a surface of genus 5} \label{fig3-c}
\end{center}
\end{figure}

\begin{figure}[htb]
\begin{center}
\hspace{0.1cm} \epsfxsize=1.8in \epsfbox{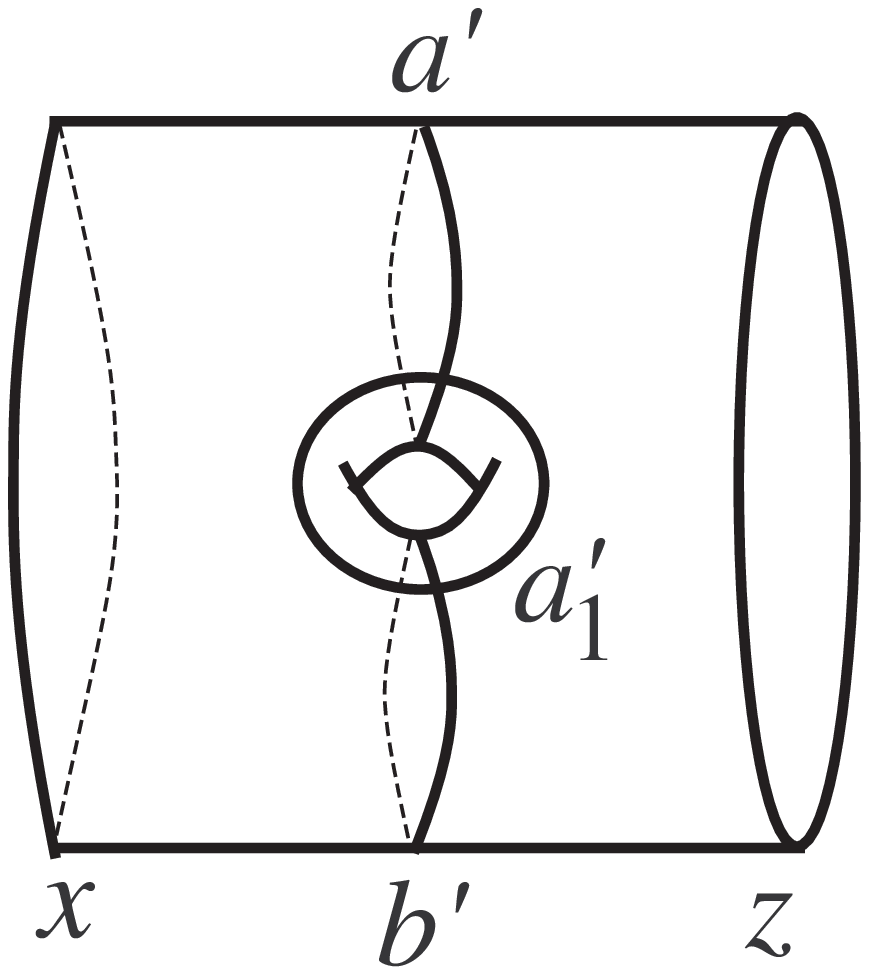} \hspace{-0.99cm}
\epsfxsize=2.4in \epsfbox{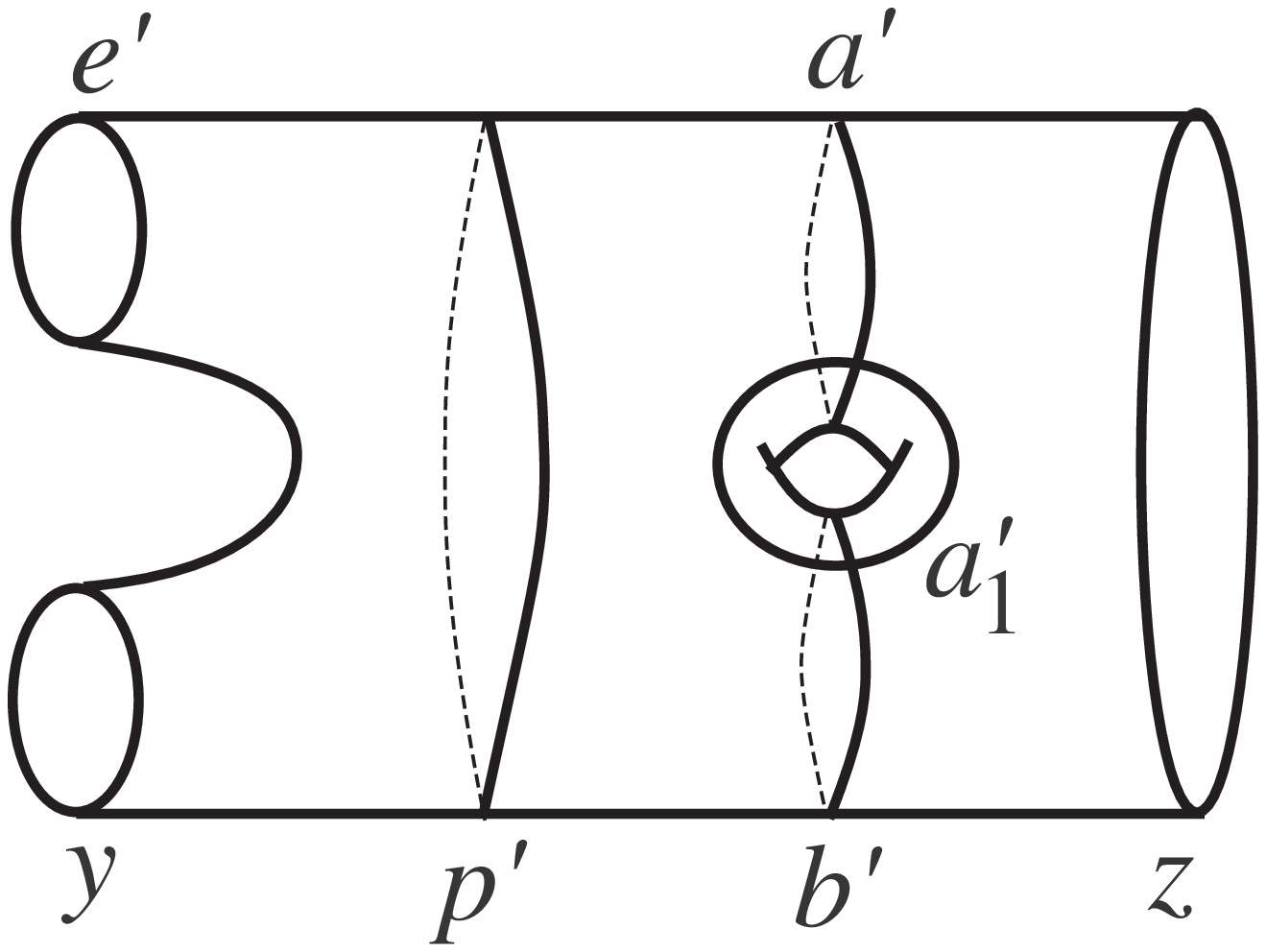} \hspace{-0.99cm}
\epsfxsize=2.4in \epsfbox{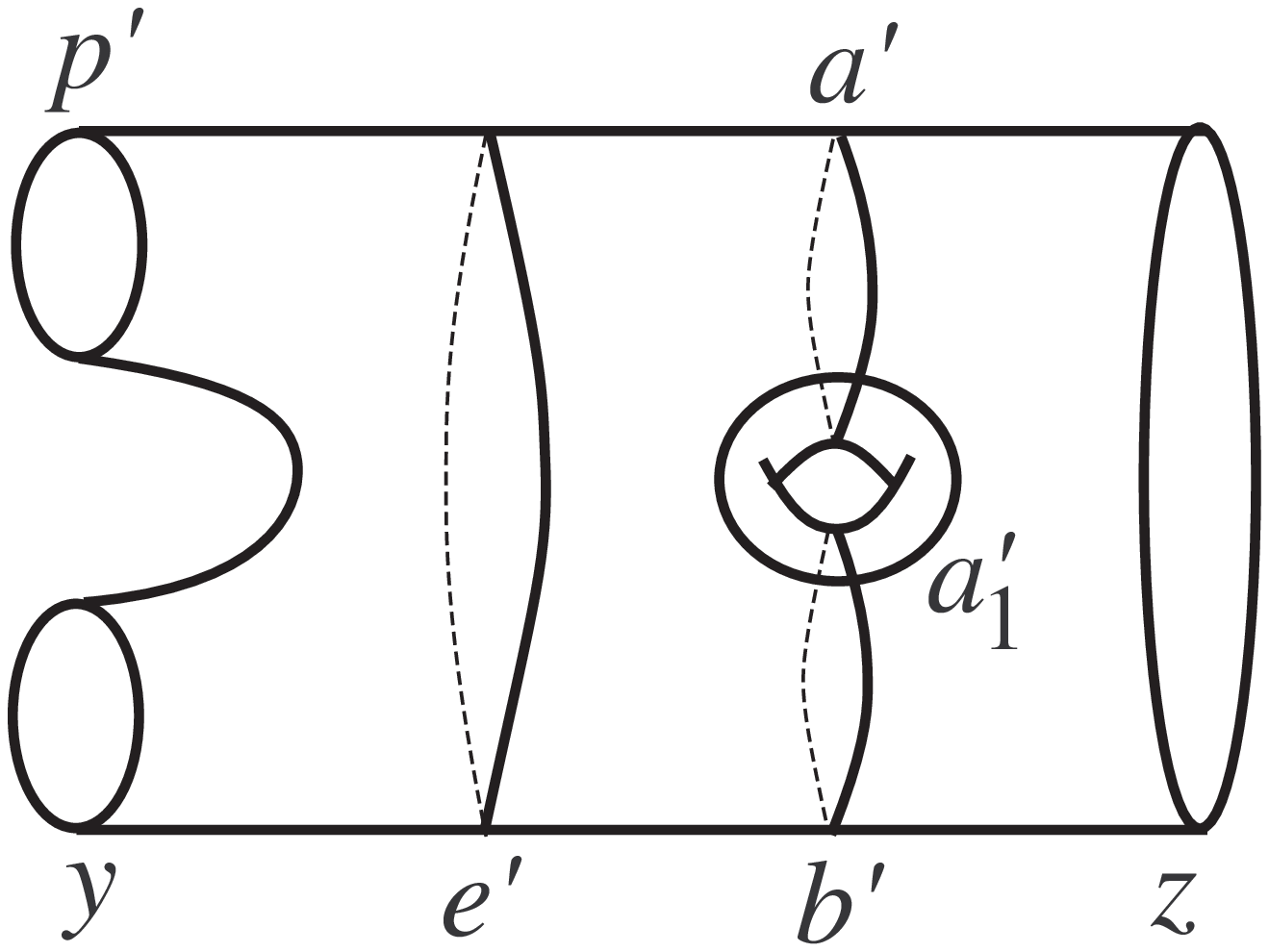}

\hspace{-1.6cm} (i) \hspace{4.2cm} (ii) \hspace{4.2cm} (iii)
\caption{Curves on a surface of genus 5} \label{fig3-c-2}
\end{center}
\end{figure}

{\bf Case (iii):} Assume that $N$ is closed and $g = 5$. We complete $a, b$ to a curve configuration
$\mathcal{C} = \{a, a_1, b, c, e, p\}$ as shown in Figure \ref{fig3-c} (i). We have a top dimensional $P$-$S$ decomposition
$P = \{a, b, e, p\}$ on $N$.
Let $a' \in \lambda([a]), a'_1  \in \lambda([a_1]), b' \in \lambda([b]), c' \in \lambda([c]), e' \in \lambda([e]), p' \in \lambda([p])$
be minimally intersecting representatives. Let $\mathcal{C'}$ be the set of these representatives. Let
$P' = \{a', b', e', p'\}$. Since geometric intersection zero and one are preserved, there is a subsurface $R$ of $N$ such that $R$ is an orientable
surface of genus one with two boundary components $x, z$, and $R$ has $a', b', a'_1$ on it as shown in Figure \ref{fig3-c-2} (i). Since
$e'$ and $p'$ are disjoint from $a' \cup b' \cup a_1'$ and $P' = \{a', b', e', p'\}$ is a top dimensional $P$-$S$ decomposition, we see
that one of the boundary components, say $x$ of $R$, is isotopic to $p'$ or $e'$. {\bf (a)} Suppose $x$ is isotopic to $p'$, see
Figure \ref{fig3-c-2} (ii). Since $p$ and $e$ are adjacent to each other w.r.t. $P$, we have $p'$ and $e'$ are adjacent to each other w.r.t. $P'$.
Then there exists a pair of pants $Q$ on the surface $N$ (not necessarily essential, i.e. one of the boundary components of it may bound
a M\"{o}bius band on the surface)
having $p'$ and $e'$ on its boundary as shown in Figure \ref{fig3-c-2} (ii). Let $y$ be the third boundary component of $Q$ as shown in the figure.
Since $P' = \{a', b', e', p'\}$ is a top dimensional $P$-$S$ decomposition, we have two choices: $e'$ is isotopic to $y$ or $z$. Consider the curve $c$.
Since $c$ is disjoint from $a$ and $b$, $c$ intersects $a_1$ once, and $c$ intersects each of $e$ and $p$ nontrivially,
$c'$ is disjoint from $a'$ and $b'$, $c'$ intersects $a'_1$ once, and $c'$ intersects each of $e'$ and $p'$ nontrivially.
This would give a contradiction if $e'$ is isotopic to $z$. So, $e'$ is isotopic to $y$. This implies that $z$ has to bound a M\"{o}bius band
which finishes the proof. {\bf (b)} If $x$ is isotopic to $e'$ (see Figure \ref{fig3-c-2} (iii)), then with a similar argument we get the result.

{\bf Case (iv):} Assume that $g = 5$ and $n \geq 1$. We complete $a, b$ to a curve configuration $\mathcal{C} = \{a, b, c, e, p, v, w, a_1, c_1,
d_1, d_2, \cdots, $ $d_{n-1}\}$ as shown in Figure \ref{fig3-c} (ii), (iii). We have a top dimensional $P$-$S$ decomposition
$P = \{v, w, b, c, c_1, d_1, d_2, \cdots, d_{n-1}\}$ on $N$. We choose minimally intersecting representatives of the elements
in $\{\lambda([x]): x \in \mathcal{C}\}$. Let $\mathcal{C'}= \{a', b', c', e', p', v', w', a'_1, d'_1, d'_2, \cdots, d'_{n-1}\}$ be
the set of these elements. We use the notation that $x' \in \lambda([x])$ for all $x \in \mathcal{C}$. As in the proof of
Lemma \ref{piece2-a}, by using that geometric intersection number zero and one is preserved, adjacency w.r.t. $P$ is preserved we see that
there is a subsurface $R$ of $N$ such that $R$ is an orientable surface of genus one with $n+3$ boundary components, $R$ has 
$a', b', c', a'_1, c'_1, d'_1, d'_2, \cdots, d'_{n-1}$ on it and $\{a', c'\} = \{x, y\}$
where all the curves are as shown in Figure \ref{fig3-c} (iv). By using that intersection zero and nonzero are preserved, it is easy to see that $x = a', y= c'$.
We have a top dimensional $P$-$S$ decomposition $P' = \{v', w', b', c', c'_1, d'_1, d'_2, \cdots, d'_{n-1}\}$ on $N$. Let $u, r, s$ be three of
the boundary components of $R$, and let $m, t$ be the curves
as shown in the figure. Since $a', b', c', c'_1$ are pairwise nonisotopic, none of $u, r, s$ bounds a disk on $N$. We will show that $w'$
is isotopic to $m$. It is easy to see that $v', w'$ are not adjacent to any of the curves $c'_1, d'_1, d'_2, \cdots, d'_{n-1}$, and
$v', w'$ are both disjoint from each of the curves $b', c', a'_1$. We also know that $P'$ is a top dimensional $P$-$S$ decomposition on $N$.
All this information implies that $m$ is isotopic to $v'$ or $w'$. Suppose $m$ is isotopic to $v'$. Since $w$ is adjacent to $v$ w.r.t. $P$,
$w'$ is adjacent to $v'$  w.r.t. $P'$. So, there exists a pair of pants $Q$ (not necessarily essential) on $N$ containing $w'$ and $v'$ on
its boundary. Let $z$ be the third boundary component of $Q$. Since $P'$ is a top dimensional $P$-$S$ decomposition on $N$, and $w'$ is not
adjacent to $c'_1, d'_i$ for $i = 1, \cdots, n-1$ w.r.t. $P'$, we have $z$ is isotopic to $w'$. This would imply that $v'$ is a separating
curve, which gives a contradiction by Corollary \ref{nonsep}. So, $m$ is not isotopic to $v'$. Hence, $m$ is isotopic to $w'$. With similar
ideas we could see that in this case $z$ is isotopic to $v'$, and hence $w'$ either bounds a torus with one boundary component or a nonorientable
surface of genus two with one boundary component. This implies that either $s$ bounds a M\"{o}bius band or $s$ is isotopic to a boundary
component of $N$. By using the
curves, $e, p$ and using similar arguments we see that $p'$ is isotopic to $t$, and $p'$ either bounds a torus with one boundary component
or a nonorientable surface of genus 2 with one boundary component. In either case, we see that $s$ can't be isotopic to a boundary component
of $N$. Hence, $s$ bounds a M\"{o}bius band. Then, $t$ bounds a nonorientable surface of genus two with one boundary component $K$, and
hence $r$ bounds a M\"{o}bius band. Similarly, $u$ bounds a M\"{o}bius band. This implies that each of the other boundary components of
$R$ is isotopic to a boundary component of $N$. This finishes the proof.\end{proof}\\

The proof of the following lemma is clearly included in the proof of Lemma \ref{piece2-b}.

\begin{lemma}
\label{piece2-bb} Suppose that $g \geq 5, n \geq 1$ and $g$ is odd. Let $a, b$ be two nonseperating curves on $N$ such that together with
a boundary component of $N$ they bound a pair of pants $P$ on $N$, and the complement of $P$ is connected and nonorientable. There exist
$a' \in \lambda([a])$ and $b'  \in \lambda([b])$ such that $a'$ and $b'$ are nonseparating and together with a boundary component of
$N$ they bound a pair of pants on $N$.\end{lemma}

\begin{figure}[htb]
\begin{center}

\hspace{.2cm} \epsfxsize=3.4in \epsfbox{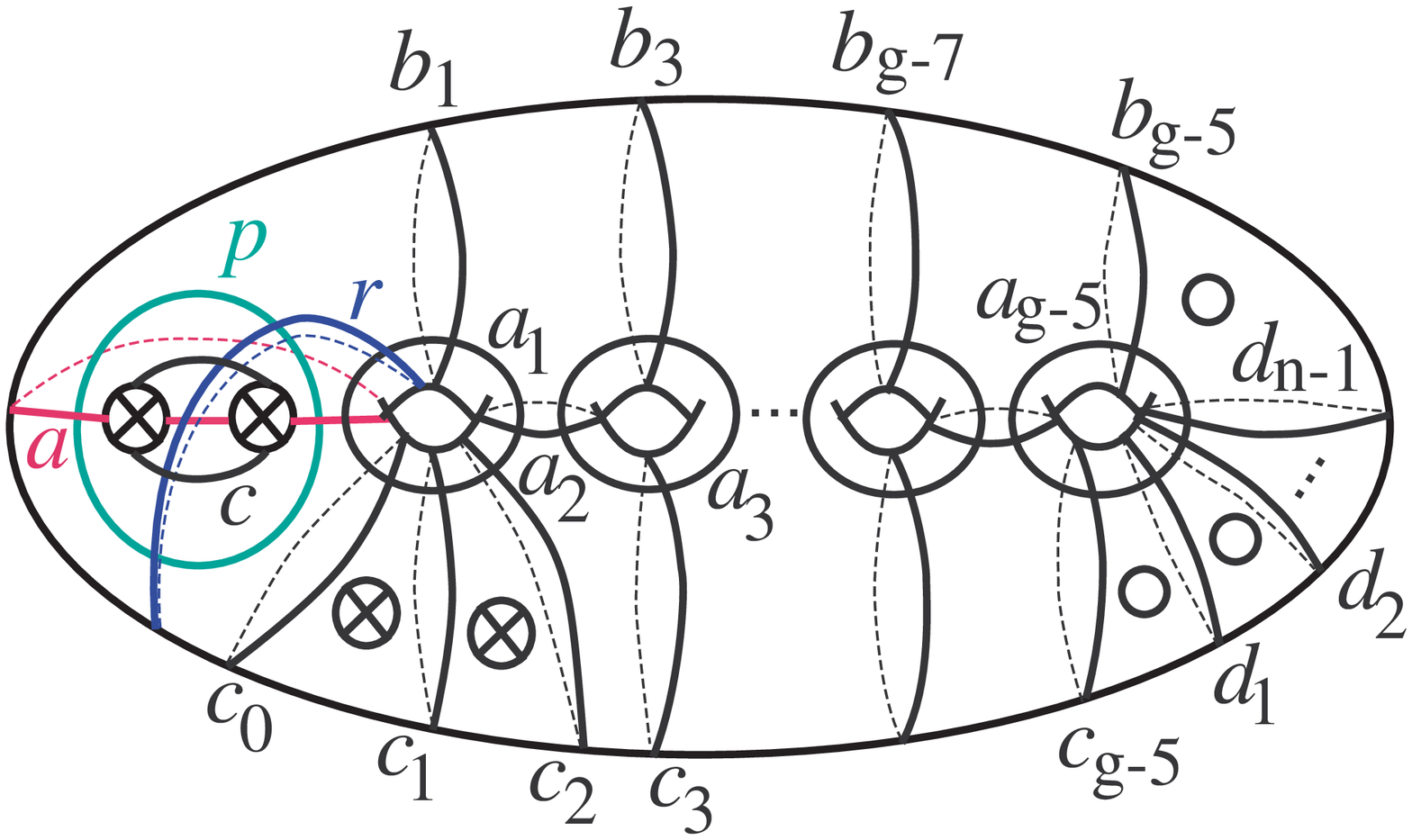}

\hspace{-1.3cm}(i)

\epsfxsize=3.17in \epsfbox{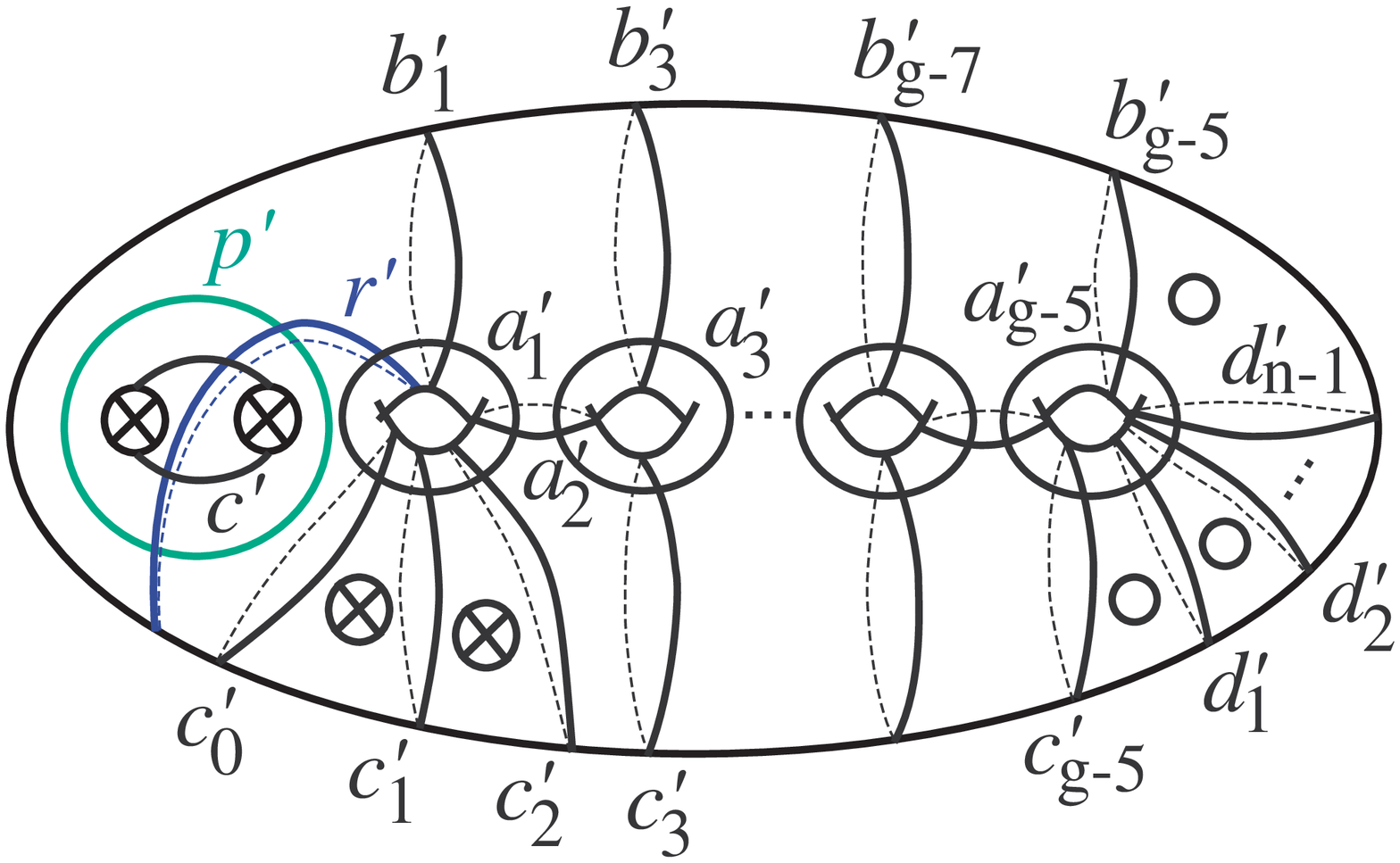} \hspace{-1.5cm} \epsfxsize=3.17in \epsfbox{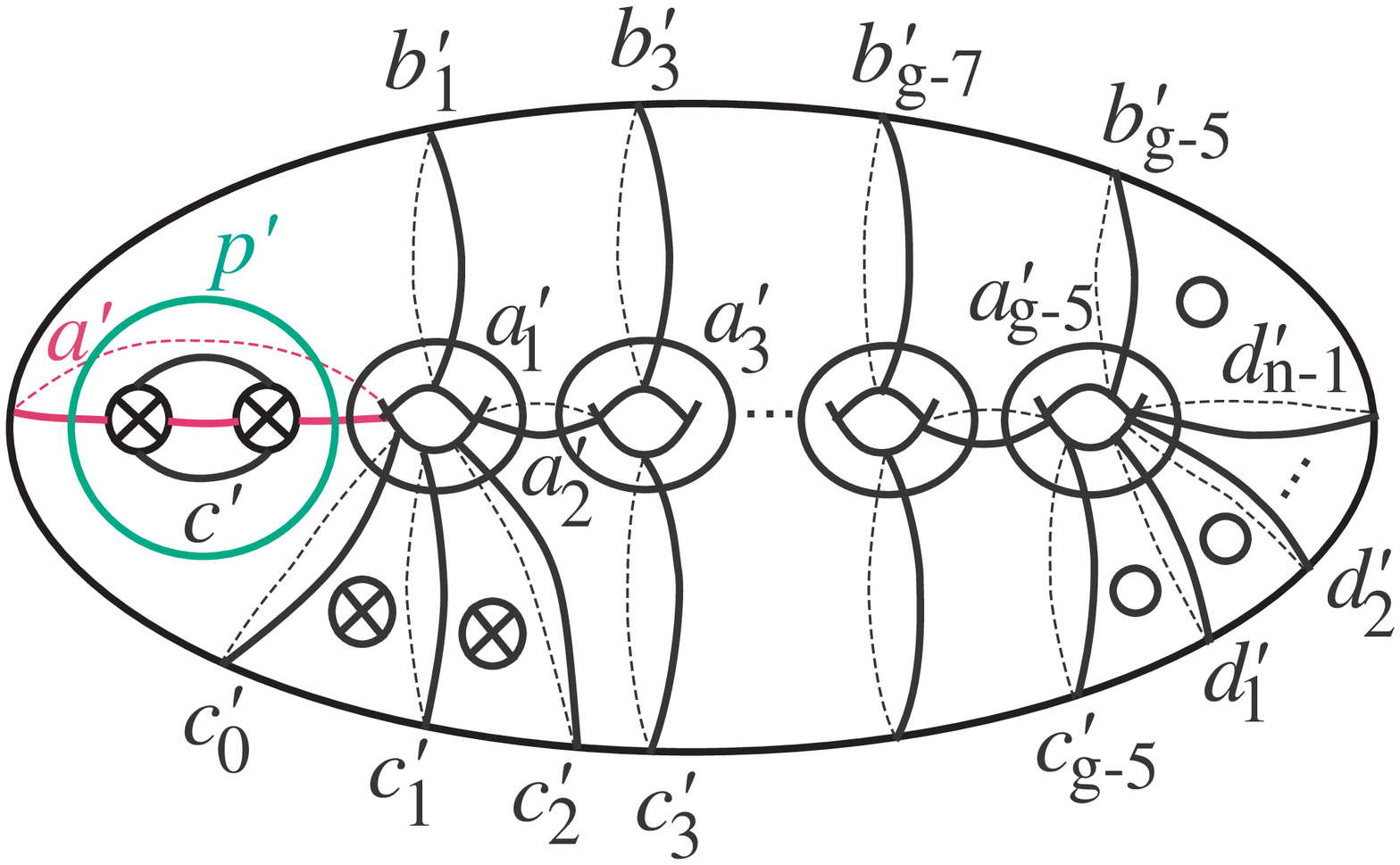}

\hspace{-1cm} (ii) \hspace{6.5cm} (iii)

\caption{Curve configuration III} \label{fig1b-n10}
\end{center}
\end{figure}

\begin{lemma}
\label{int-twi} Suppose that $g \geq 6$ and $g$ is even. Let $a, p, r$ be curves as shown in Figure \ref{fig1b-n10} (i). There
exist $a' \in \lambda([a]), p' \in \lambda([p]), r' \in \lambda([r])$ such that $i([p'], [a'])=2$ and $i([p'], [r'])=2$. \end{lemma}

\begin{proof} We will give the proof when $g \geq 8$ and $n \geq 2$. The proof for the remaining cases will be similar.

Let $a, p, r, c$ be as shown in Figure \ref{fig1b-n10} (i). Consider the curve configuration
$\mathcal{B} = \{a_1, a_2, \cdots, a_{g-5}, b_1, b_3, \cdots, b_{g-5}, c_0, c_1, c_2, c_3, c_5, \cdots,$
$ c_{g-5}, d_1, d_2, \cdots, d_{n-1}, r\}$ as shown in Figure \ref{fig1b-n10}. Let $a'_1  \in \lambda([a_1]), a'_2  \in \lambda([a_2]),$
$\cdots, a_{g-5}' \in \lambda([a_{g-5}]), b'_1  \in \lambda([b_1]), b'_3  \in \lambda([b_3]),$
$ \cdots, b_{g-5}' \in \lambda([b_{g-5}]), c'_0  \in \lambda([c_0]), c'_1  \in \lambda([c_1]),   \cdots, c_{g-5}' \in \lambda([c_{g-5}]),$
$ d'_1  \in \lambda([d_1]), d'_2  \in \lambda([d_2]), \cdots, $ $d_{n-1}' \in \lambda([d_{n-1}]), r' \in \lambda([r])$ be minimally
intersecting representatives. By Lemma \ref{intone} geometric intersection one is preserved. So, a regular neighborhood of union of
all the elements in $\mathcal{B'}= \{a'_1, a'_2, \cdots, a_{g-5}', b'_1, b'_3, \cdots,$
$ b'_{g-5}, c'_0, c'_1, \cdots, c'_{g-5}, d'_1, d'_2, \cdots,$ $d'_{n-1}, r'\}$ is an orientable surface of genus $\frac{g-4}{2}$
with several boundary components.

By Lemma \ref{piece2-a}, if two curves in $\mathcal{B}$ separate a twice holed projective plane on $N$, then the corresponding
curves in $\mathcal{B'}$ separate a twice holed projective plane on $N$. By Lemma \ref{piece2-aa}, if two curves in $\mathcal{B}$
are the boundary components of a pair of pants where the third boundary component of the pair of pants is a boundary component of $N$,
then the corresponding curves in $\mathcal{B'}$ are the boundary components of a pair of pants where the third boundary component
of the pair of pants is a boundary component of $N$. By Lemma \ref{piece1}, if three nonseparating curves in $\mathcal{B}$ bound a
pair of pants on $N$, then the corresponding curves in $\mathcal{B'}$ bound a pair of pants on $N$. These imply that the curves
in $\mathcal{B'}$ are as shown in Figure \ref{fig1b-n10} (ii).
 
Let $p' \in \lambda([p])$, $c' \in \lambda([c])$ such that $p'$ and $c'$ have minimal intersection with the elements of $\mathcal{B'}$.
We have that $p'$ and $c'$ are both disjoint from all the curves in $\mathcal{B'}$ except for $r'$. Since we also know that $c'$ is a
nonseparating curve by Corollary \ref{nonsep}, we see that $p'$ and $c'$ are as shown in Figure \ref{fig1b-n10} (ii). Hence, $i([p'], [r'])=2$.
So, we have $\{a'_1, a'_2, \cdots, a'_{g-5}, b'_1, b'_3, \cdots,$ $ b'_{g-5}, c'_0, c'_1, \cdots, c'_{g-5}, d'_1, d'_2, \cdots, d'_{n-1}, p', c'\}$
as shown in Figure \ref{fig1b-n10} (iii).

Let $a' \in \lambda([a])$ such that $a'$ has minimal intersection with the elements of $\mathcal{B'} \cup \{c', p'\}$.
By Lemma \ref{intone} geometric intersection number one is preserved. Since $a$ intersects $a_1$ only once, and $a$ is disjoint
from $b_1$ and $c_0$, $a'$ intersects $a'_1$ only once, and $a'$ is disjoint from $b'_1$ and $c'_0$. So, there exists a unique arc up to
isotopy of $a'$ in the pair of pants, say $Q$, bounded by $p', b_1', c_0'$, starting and ending on $p'$ and intersecting the arc of $a_1'$
in $Q$ only once. Since $a$ is disjoint from $c$, $a'$ is disjoint from $c'$. So, there exists a unique arc up to isotopy of $a'$ in the
Klein bottle with one hole bounded by $p'$ (we are considering isotopy where the end points of the arcs can move on the boundary
during the isotopy). So, the elements in the set $\{a', c', p'\} \cup \mathcal{B'} \setminus \{r'\}$ are as
shown in Figure \ref{fig1b-n10} (iii) up to an action of a power of Dehn twist about $p'$. Hence, $i([p'], [a'])=2$.\end{proof}\\

\begin{lemma}
\label{int-twi-2} Suppose that $g \geq 6$ and $g$ is even. Let $p, x, y$ be curves as shown in Figure \ref{fig1ba} (i).
There exist $p' \in \lambda([p]), x' \in \lambda([x]), y' \in \lambda([y])$ such that $i([p'], [x'])=2$ and $i([p'], [y'])=2$. \end{lemma}

\begin{proof} The proof is similar to the proof of Lemma \ref{int-twi}. We will give the proof when $g \geq 8$ and $n \geq 2$.
The proof for the remaining cases will be similar.

We consider the curve configuration $\mathcal{B} = \{a_1, a_2, \cdots, a_{g-5}, b_1, b_3, \cdots, b_{g-5}, c_0, c_1, c_2, c_3, $
$ \cdots, c_{g-5}, d_1, d_2, \cdots, d_{n-1}, p, r, x, y\}$ as shown in Figure \ref{fig1ba} (i). Let
$a'_1  \in \lambda([a_1]), a'_2  \in \lambda([a_2]),$
$\cdots, a_{g-5}' \in \lambda([a_{g-5}]), b'_1  \in \lambda([b_1]), b'_3  \in \lambda([b_3]),$
$ \cdots, b_{g-5}' \in \lambda([b_{g-5}]), c'_0  \in \lambda([c_0]), c'_1  \in \lambda([c_1]), \cdots, c_{g-5}' \in \lambda([c_{g-5}]),$
$ d'_1  \in \lambda([d_1]), d'_2  \in \lambda([d_2]), \cdots, $ $d_{n-1}' \in \lambda([d_{n-1}]),$
$ p' \in \lambda([p]), r' \in \lambda([r]), x' \in \lambda([x]), y' \in \lambda([y])$ be minimally intersecting representatives.

\begin{figure}[htb]
\begin{center}

\hspace{1cm}  \epsfxsize=3.17in \epsfbox{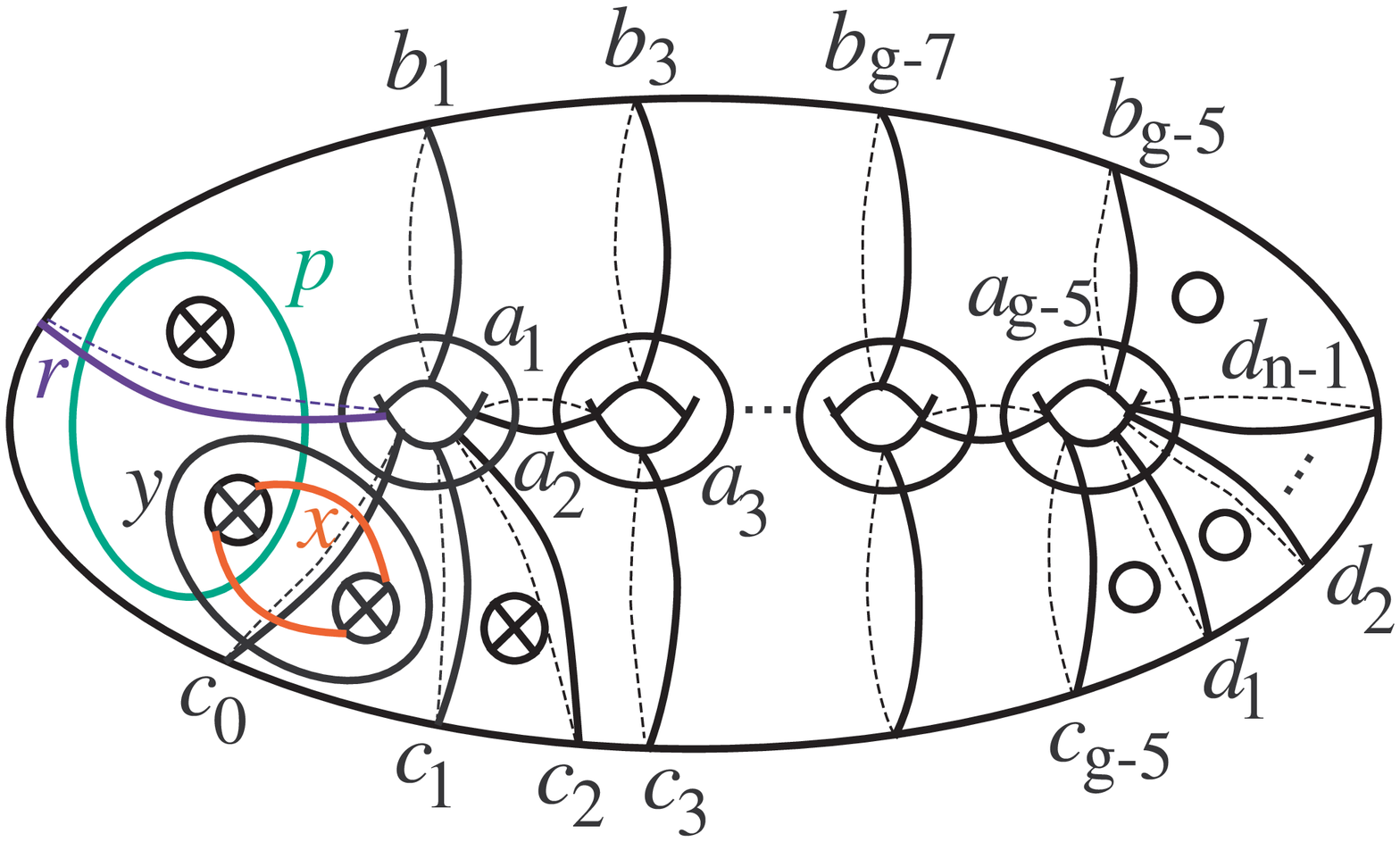}

\hspace{-0.7cm}(i)
 
\hspace{0.5cm} \epsfxsize=3.17in \epsfbox{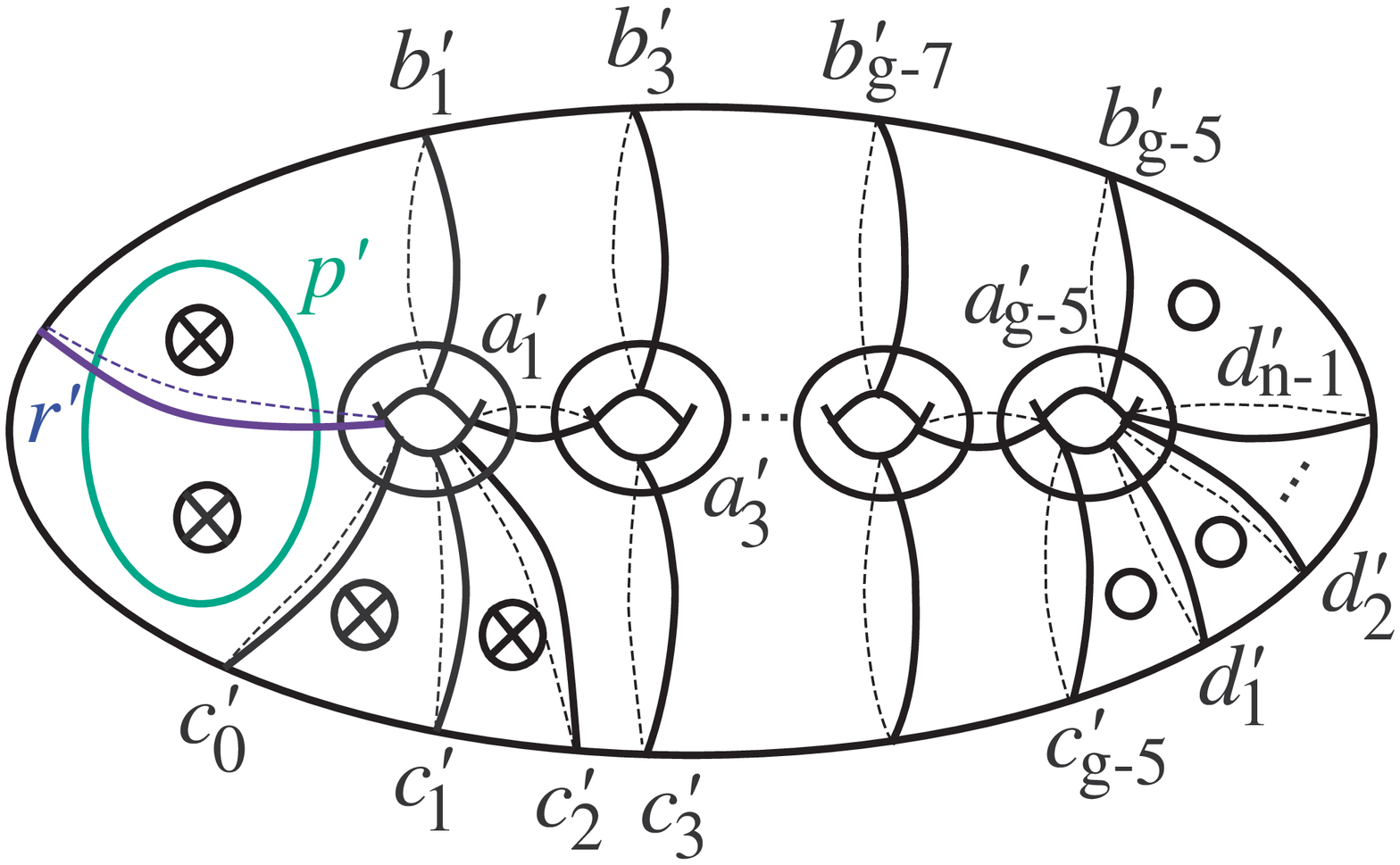} \hspace{-1.5cm} \epsfxsize=3.17in \epsfbox{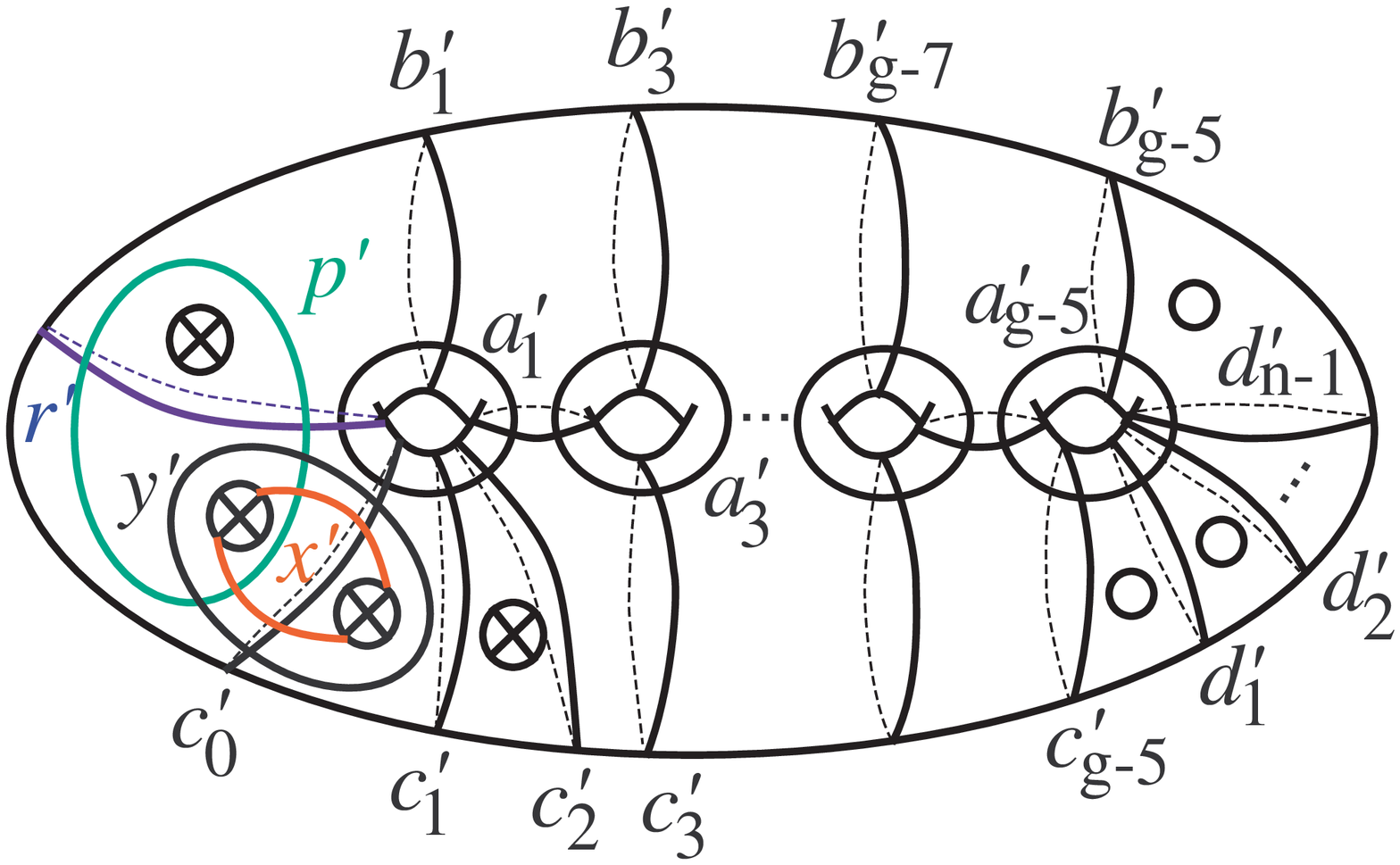}

\hspace{-0.7cm} (ii) \hspace{6.5cm} (iii)

\caption{Curve configuration IV} \label{fig1ba}
\end{center}
\end{figure}

As in the proof of Lemma \ref{int-twi}, the curves in $\{a'_1, a'_2, \cdots, a_{g-5}', b'_1, b'_3, \cdots,$ $ b_{g-5}', c'_0, c'_1,$
$\cdots, c_{g-5}', d'_1, d'_2, \cdots, d'_{n-1}, p', r'\}$ will be as shown in Figure \ref{fig1ba} (ii). Since there exists a
homeomorphism sending $p$ to $y$ and $p'$ bounds a Klein bottle with one hole, we know that $y'$ bounds a Klein bottle with one hole. The curve
$y$ is disjoint from each of $r, a_1, c_1$ and it has nontrivial intersection with $c_0$. So, $y'$ is disjoint from each of $r', a'_1, c'_1$
and it has nontrivial intersection with $c'_0$. All this information about $y'$ implies that $y'$ is as shown in Figure \ref{fig1ba} (iii). 
Since $x$ is disjoint from each of $r, a_1, c_1$ and it has nontrivial intersection with $c_0$ and $\lambda$ is injective, $x'$ has to
be in the Klein bottle bounded by $y'$ and should be as shown in the figure. So, we get $i([p'], [x'])=2$ and $i([p'], [y'])=2$.\end{proof}\\
 
\begin{figure}[htb]
\begin{center}
\hspace{-0.2cm}  \epsfxsize=2.86in \epsfbox{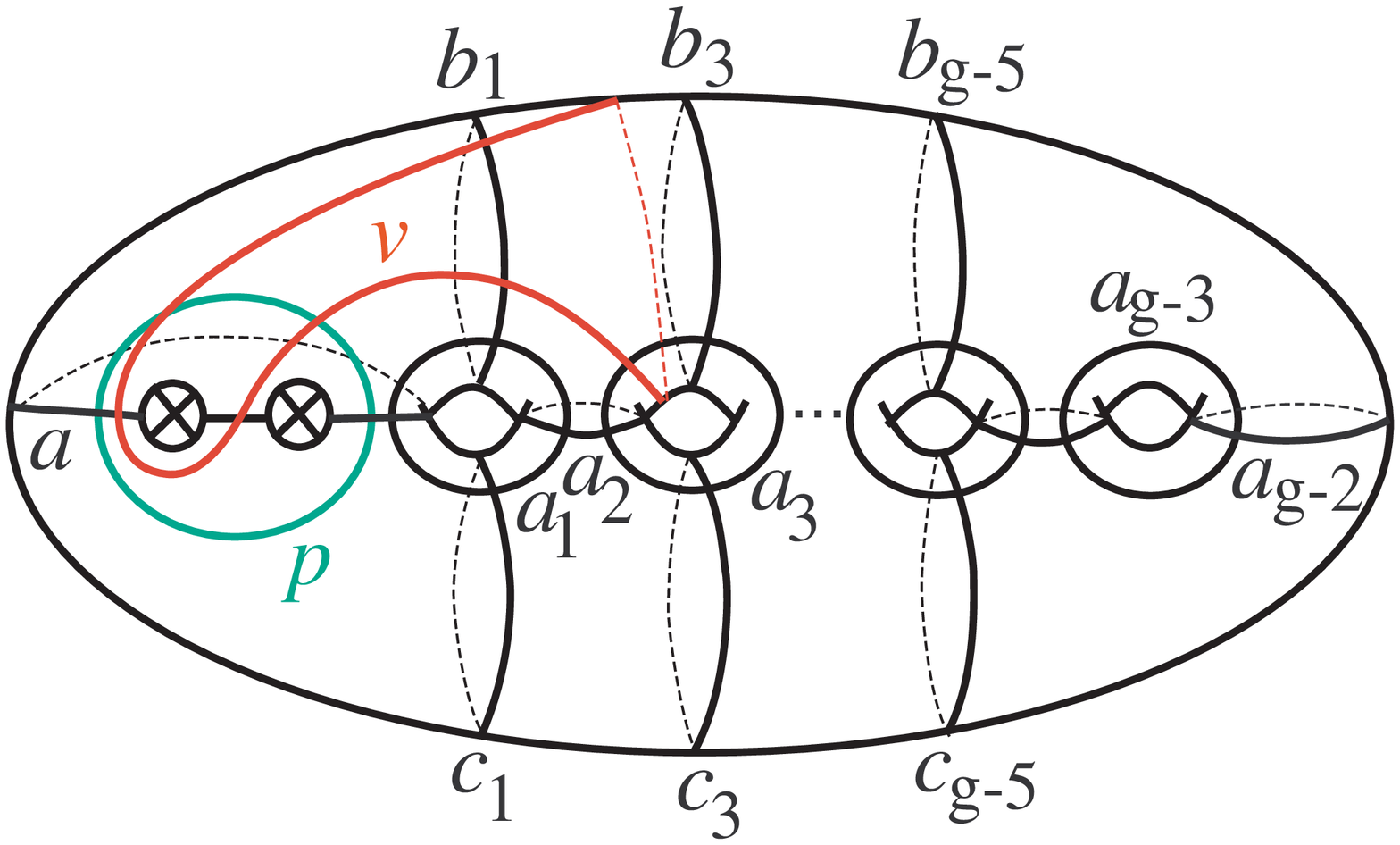} \hspace{-0.5cm}  \epsfxsize=3.17in \epsfbox{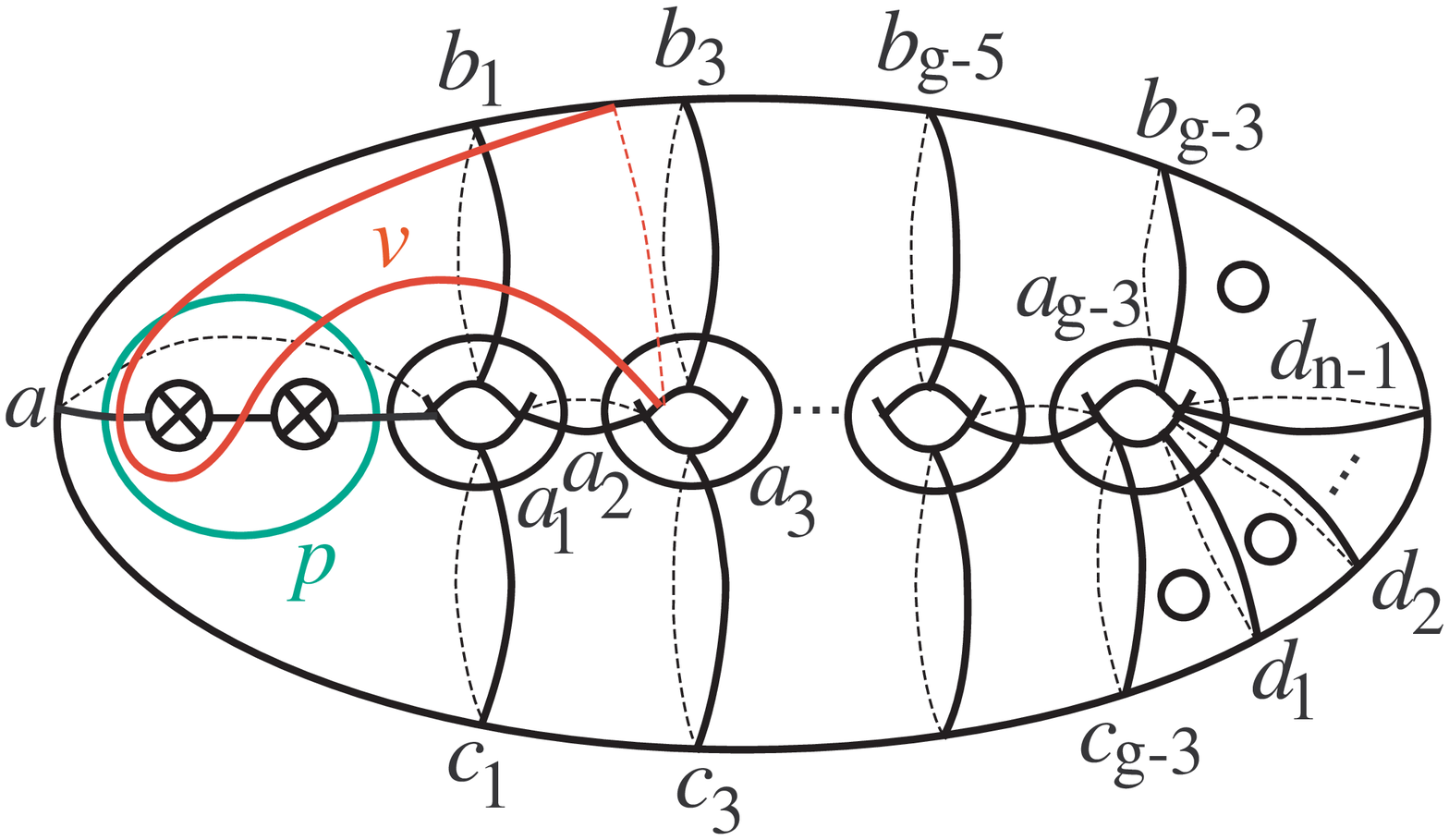}  \\

\hspace{-1cm} (i) \hspace{6.5cm} (ii)

\epsfxsize=2.95in \epsfbox{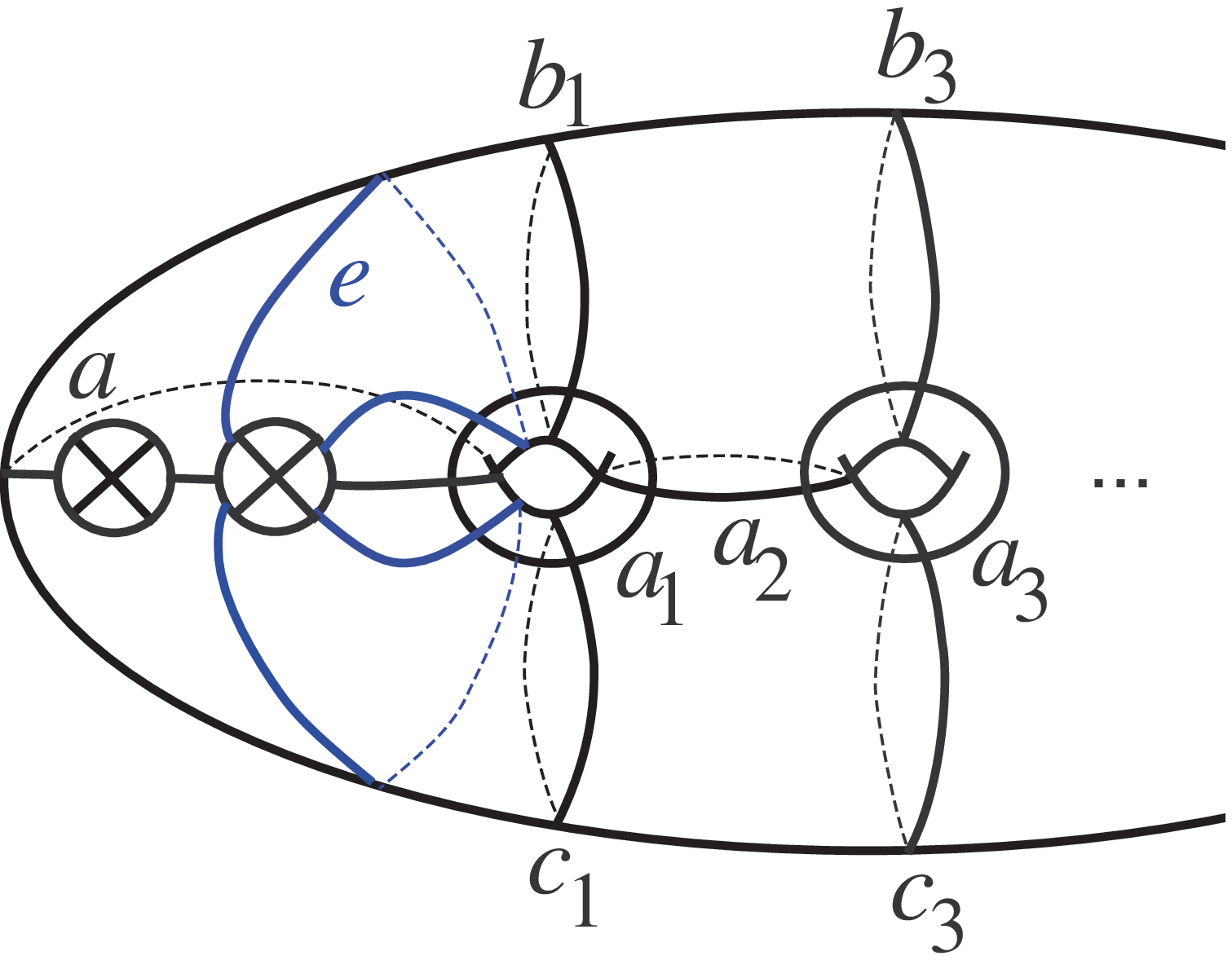}  \hspace{-0.5cm}  \epsfxsize=2.95in \epsfbox{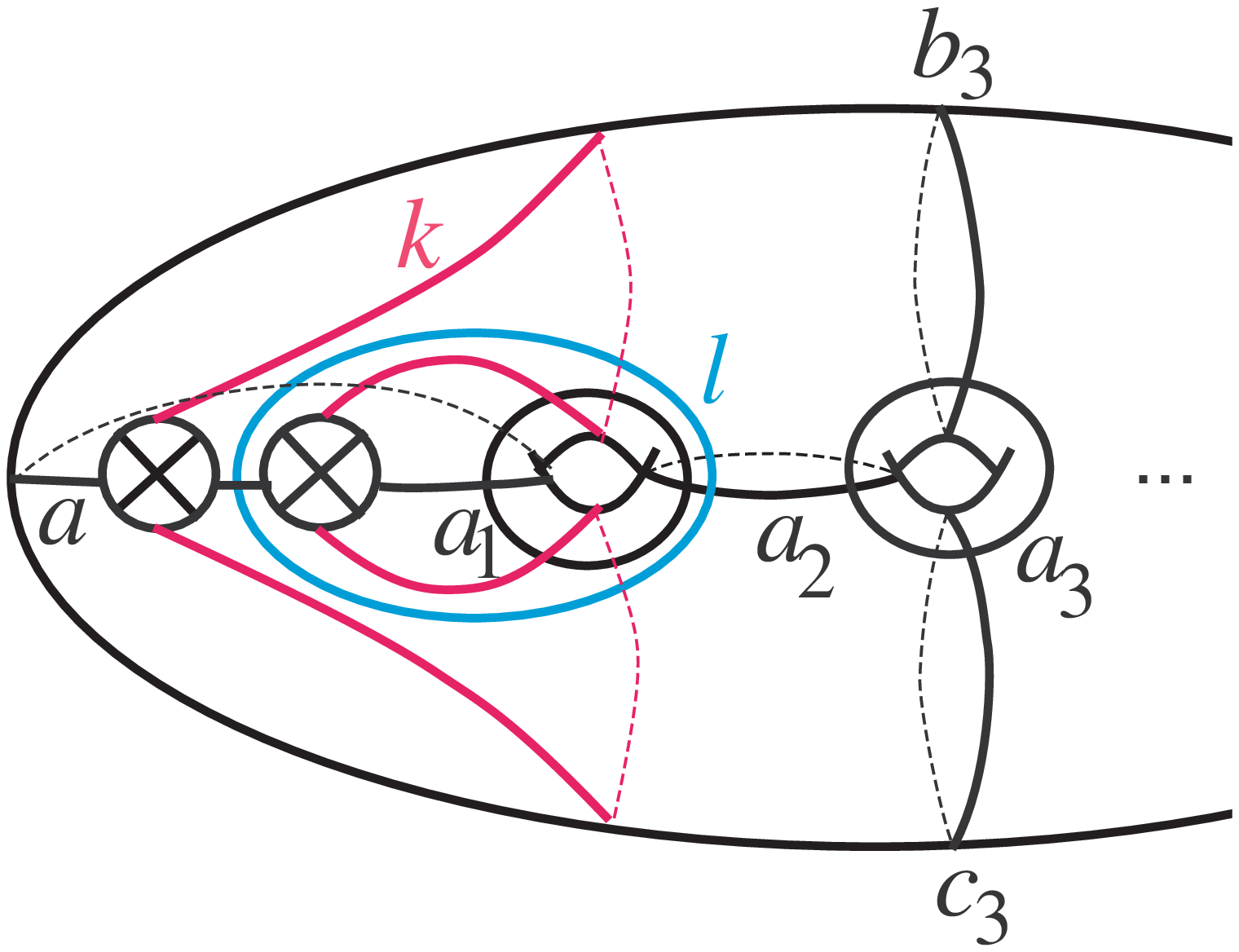}

\hspace{-0.9cm} (iii) \hspace{6.2cm} (iv)
\caption{$\mathcal{C}$ when $g$ is even} \label{fig1b-new}
\end{center}
\end{figure}

If $f: N \rightarrow N$ is a homeomorphism, then we will use the same notation for $f$ and its isotopy class $[f]$ in the rest of the paper.
Let $\mathcal{C} = \{a, a_1, \cdots, a_{g-2}, b_1, b_3, \cdots, b_{g-3}, c_1, c_3,$ $ \cdots, c_{g-3}, d_1, \cdots, d_{n-1}, e,$
$ k, l, p, v \}$ be as shown in Figure \ref{fig1b-new} (i) - (iv) when $g \geq 6$ and $g$ is even.

\begin{lemma}
\label{curves} Suppose $g \geq 6$ and $g$ is even. There exists a homeomorphism $h: N \rightarrow N$ such
that $h([x]) = \lambda([x])$ $\forall \ x \in \mathcal{C}$.
\end{lemma}

\begin{proof} We will give the proof when $N$ has boundary. The proof for the closed case will be similar. We will consider
all the curves in  $\mathcal{C}$ as defined above together with two new curves $r, c$ as shown in Figure \ref{fig1b-n}.
Let $\mathcal{B} = \{a_1, a_2, \cdots, a_{g-3}, b_1, b_3, \cdots, b_{g-3}, c_1, c_3, \cdots,$ $ c_{g-3}, d_1, d_2, \cdots, d_{n-1}, r\}$.
Let $a'_1  \in \lambda([a_1]), a'_2  \in \lambda([a_2]), \cdots, a_{g-3}' \in \lambda([a_{g-3}]), b'_1  \in \lambda([b_1]),$
$ b'_3  \in \lambda([b_3]),$ $ \cdots, b_{g-3}' \in \lambda([b_{g-3}]), c'_1  \in \lambda([c_1]), c'_3  \in \lambda([c_3]),$
$ \cdots, c_{g-3}' \in \lambda([c_{g-3}]),$ $ d'_1  \in \lambda([d_1]), d'_2  \in \lambda([d_2]), \cdots, $
$d_{n-1}' \in \lambda([d_{n-1}]), r' \in \lambda([r])$ be minimally intersecting representatives. By Lemma \ref{intone}
geometric intersection one is preserved. So, a regular neighborhood of union of all the elements in
$\mathcal{B'}= \{a'_1, a'_2, \cdots, a_{g-3}', b'_1, b'_3, \cdots,$ $ b'_{g-3}, c'_1, c'_3, \cdots, c'_{g-3},$
$ d'_1, d'_2, \cdots, d'_{n-1}, r'\}$ is an orientable surface of genus $\frac{g-2}{2}$ with several boundary components.

\begin{figure}[htb]
\begin{center}

\epsfxsize=3.25in \epsfbox{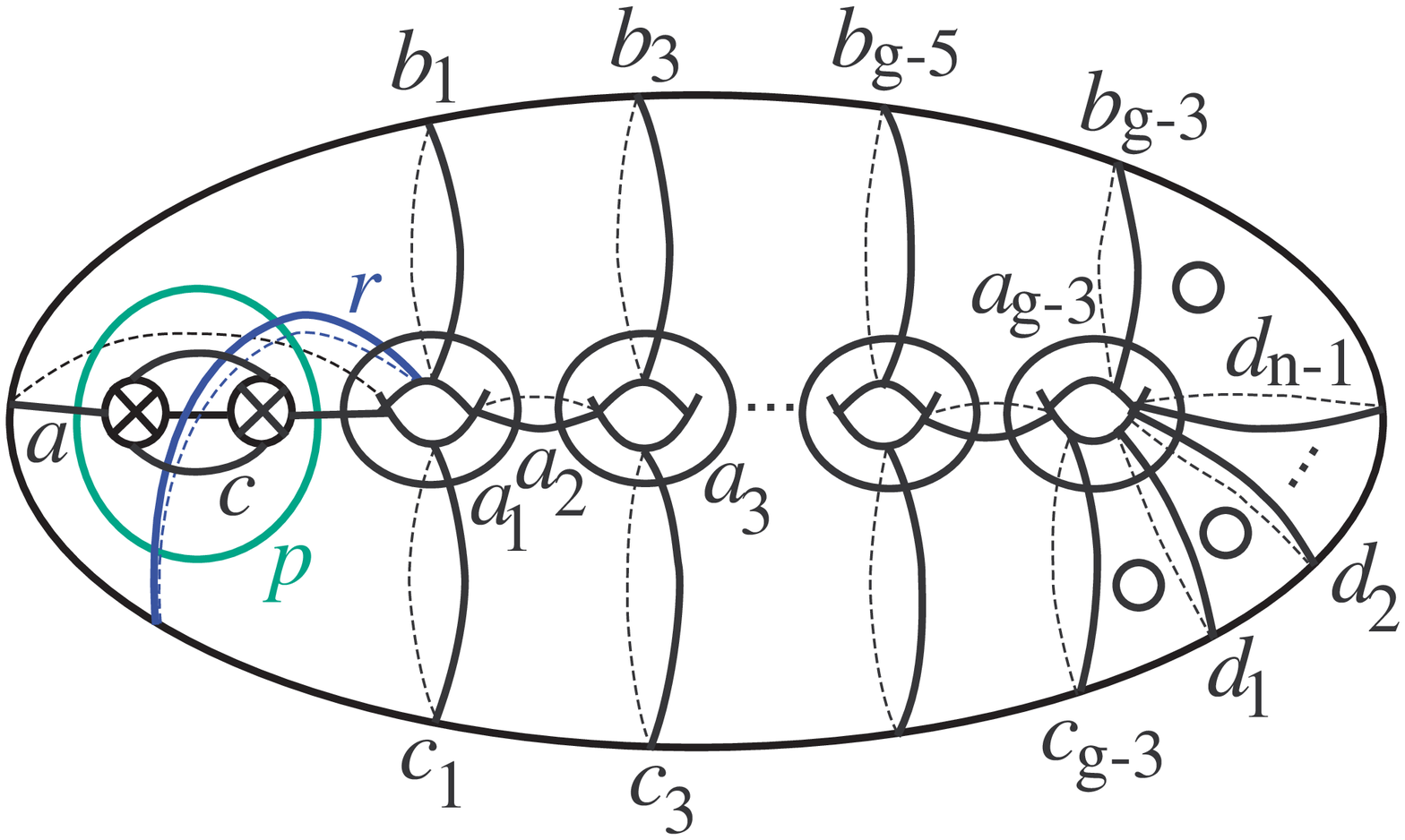} \hspace{-1.cm} \epsfxsize=2.9in \epsfbox{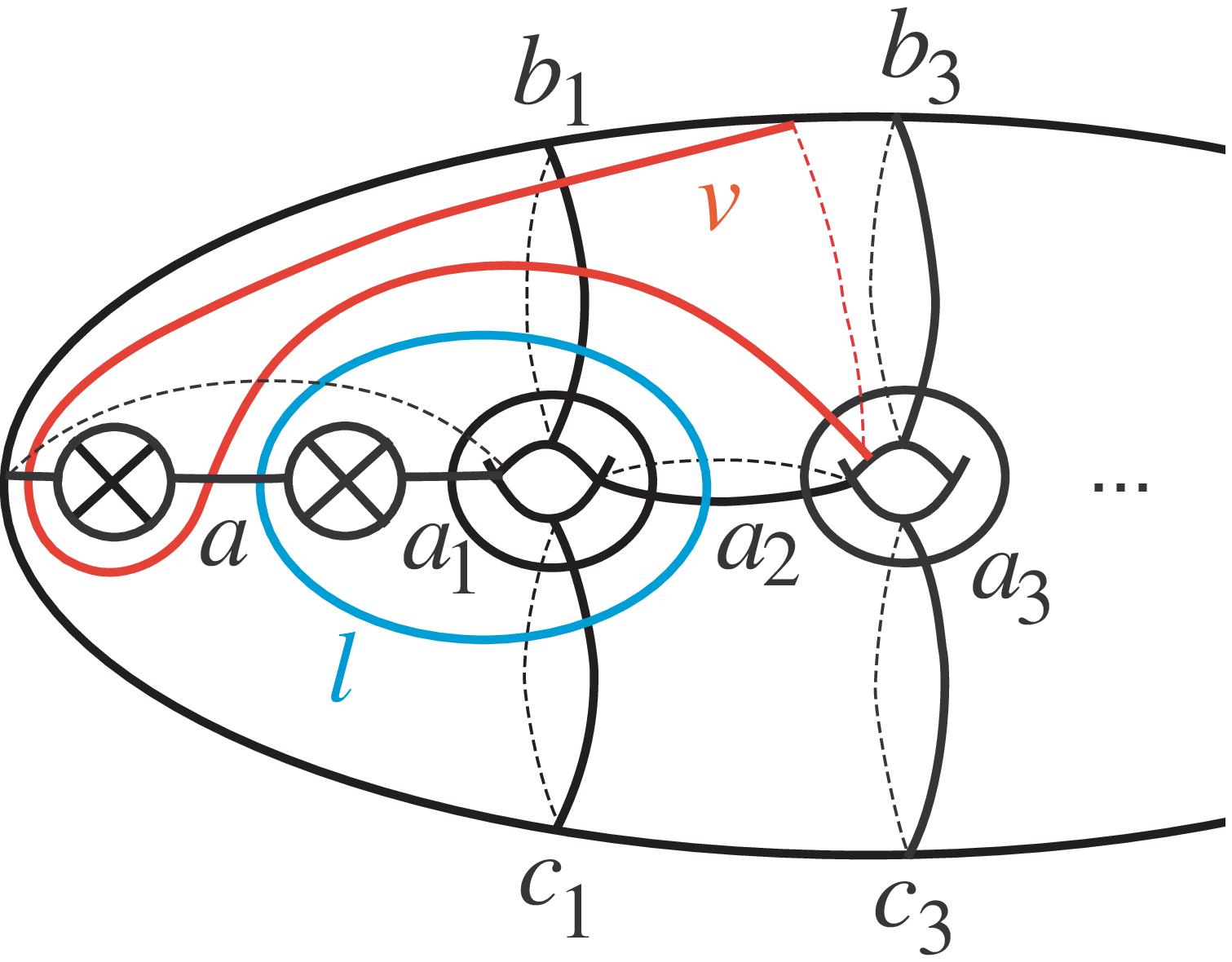}

\hspace{-1cm} (i) \hspace{6.5cm} (ii)
\caption{Curve configuration V} \label{fig1b-n}
\end{center}
\end{figure}

\begin{figure}[htb]
\begin{center}
\hspace{-0.7cm} \epsfxsize=3.1in \epsfbox{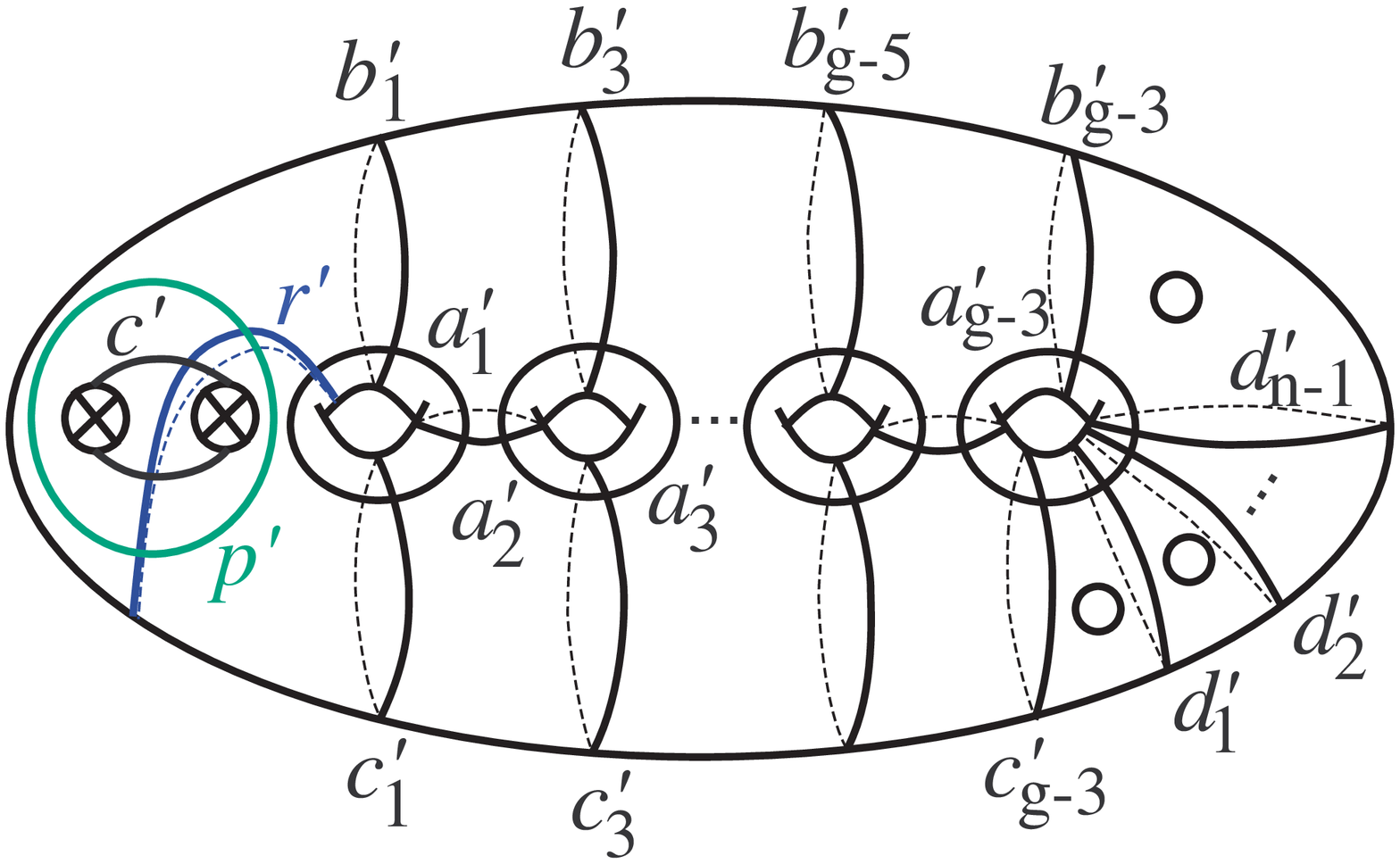}  \hspace{-0.8cm} \epsfxsize=2.9in \epsfbox{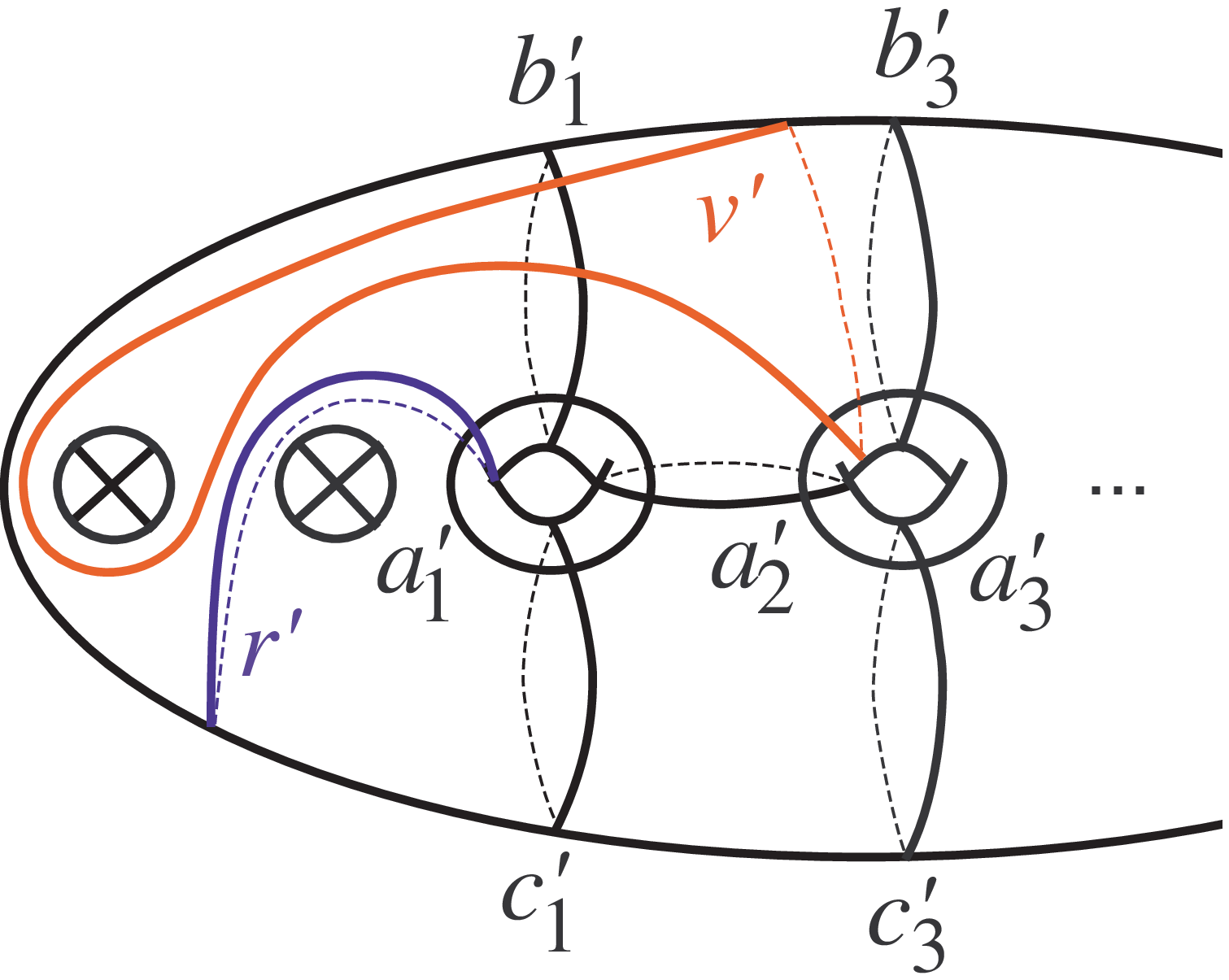}

\hspace{-1.3cm} (i) \hspace{6.7cm} (ii)

\epsfxsize=3.1in \epsfbox{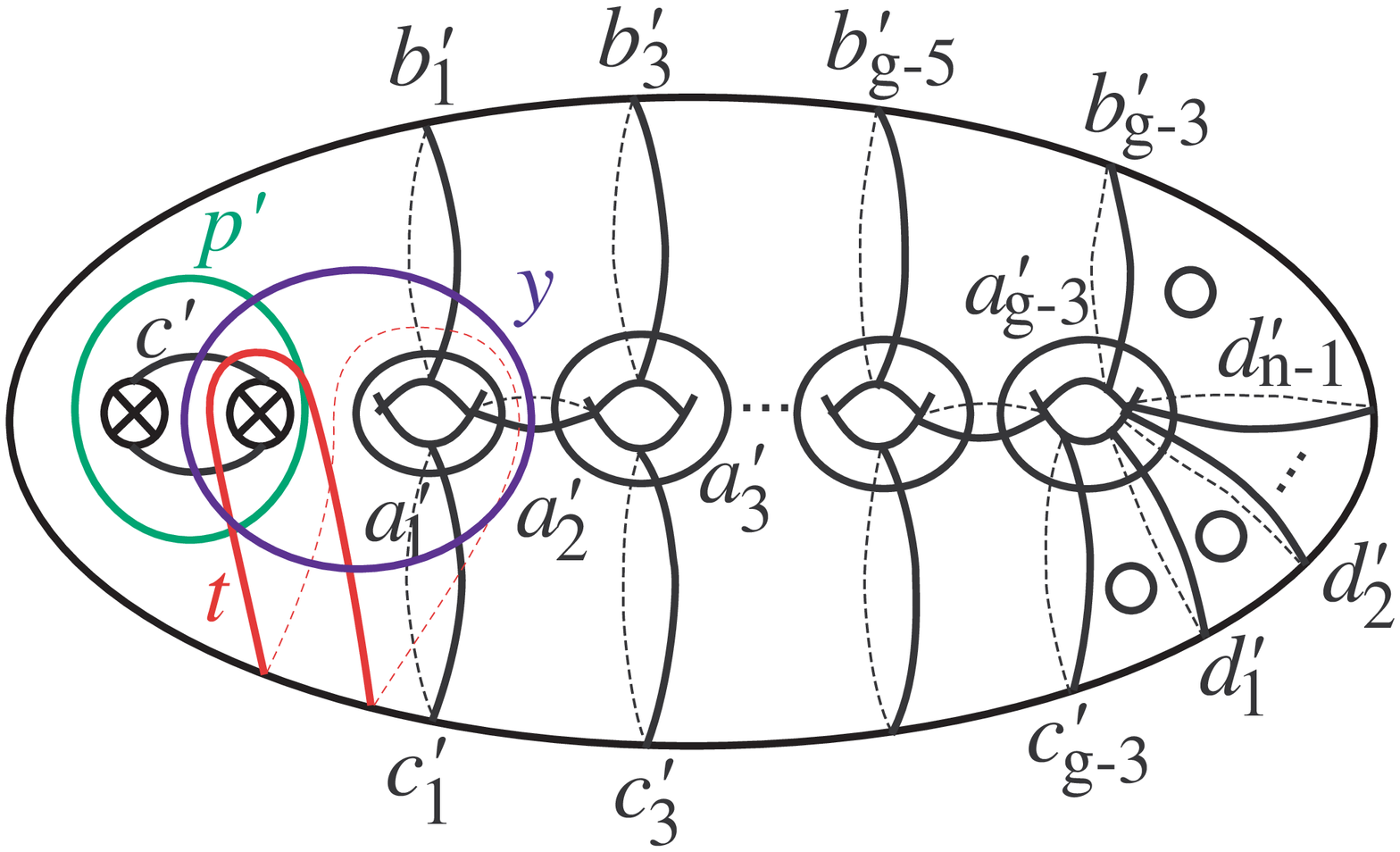}  \hspace{-1cm}  \epsfxsize=3.1in \epsfbox{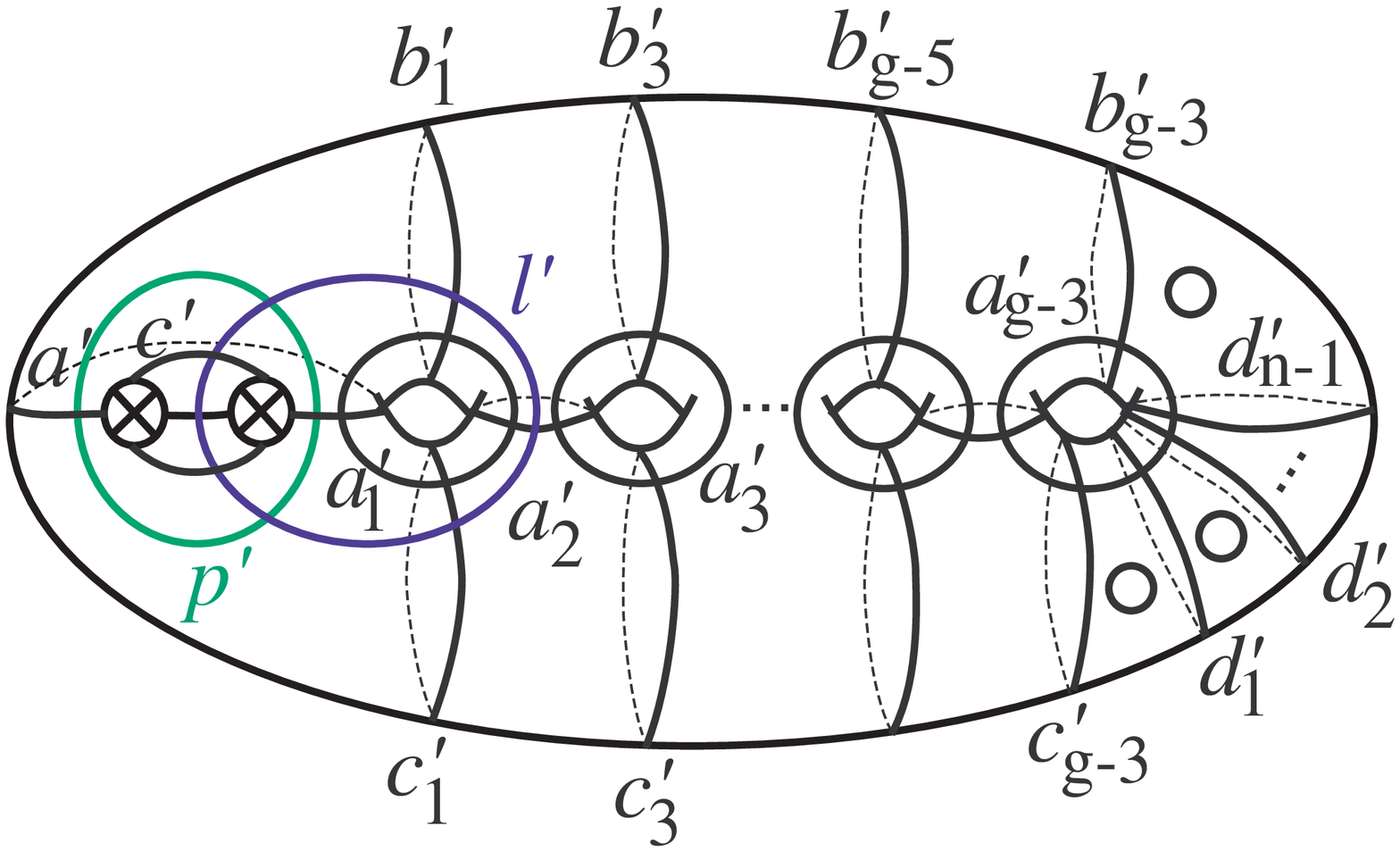}

\hspace{-1.3cm} (iii) \hspace{6.5cm} (iv)

\hspace{-4.3cm} \epsfxsize=2.95in \epsfbox{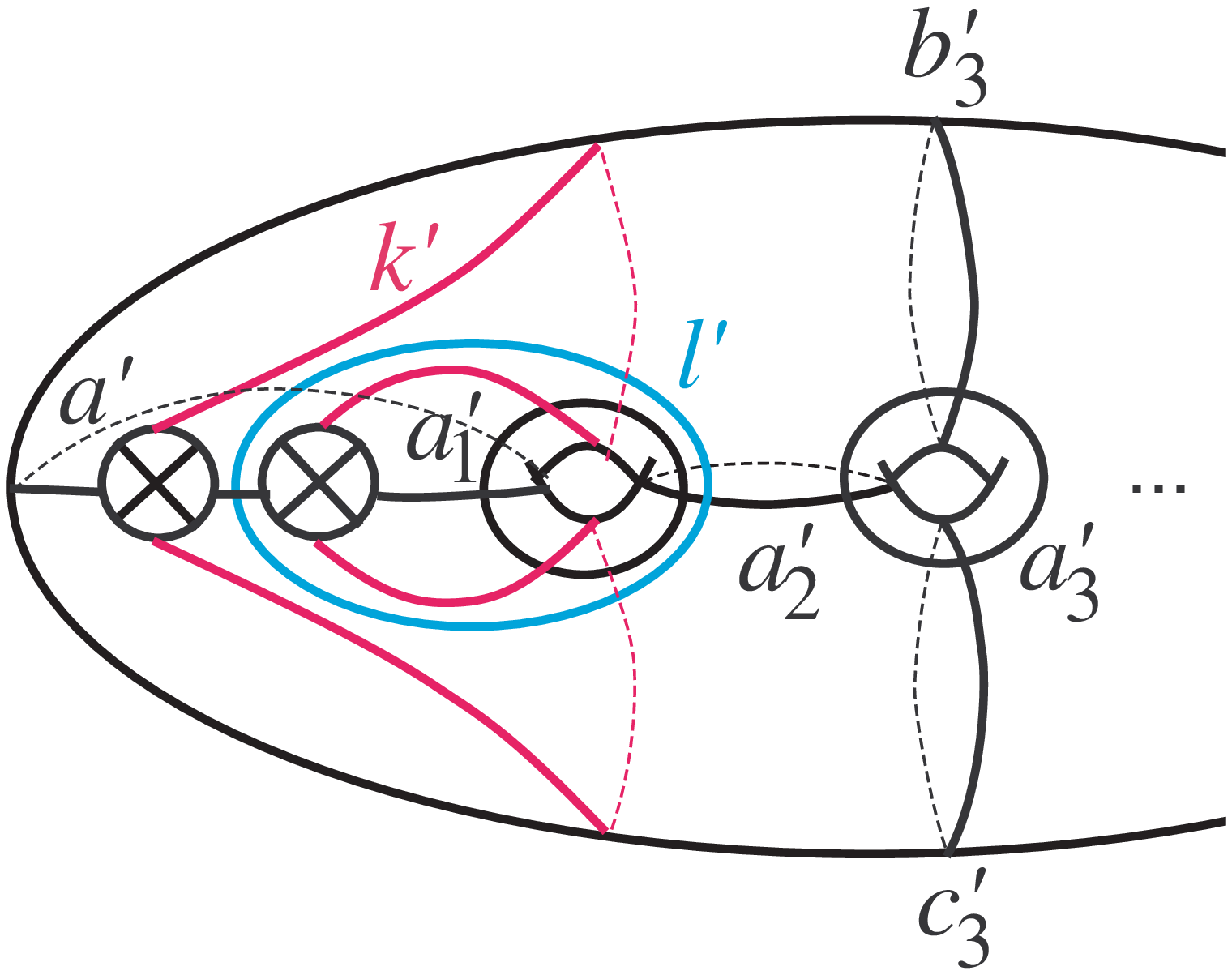}
\hspace{-4.3cm} (v)
\caption{Curve configuration VI} \label{fig1b-V}
\end{center}
\end{figure}

\begin{figure}[htb]
\begin{center}
\epsfxsize=2.1in \epsfbox{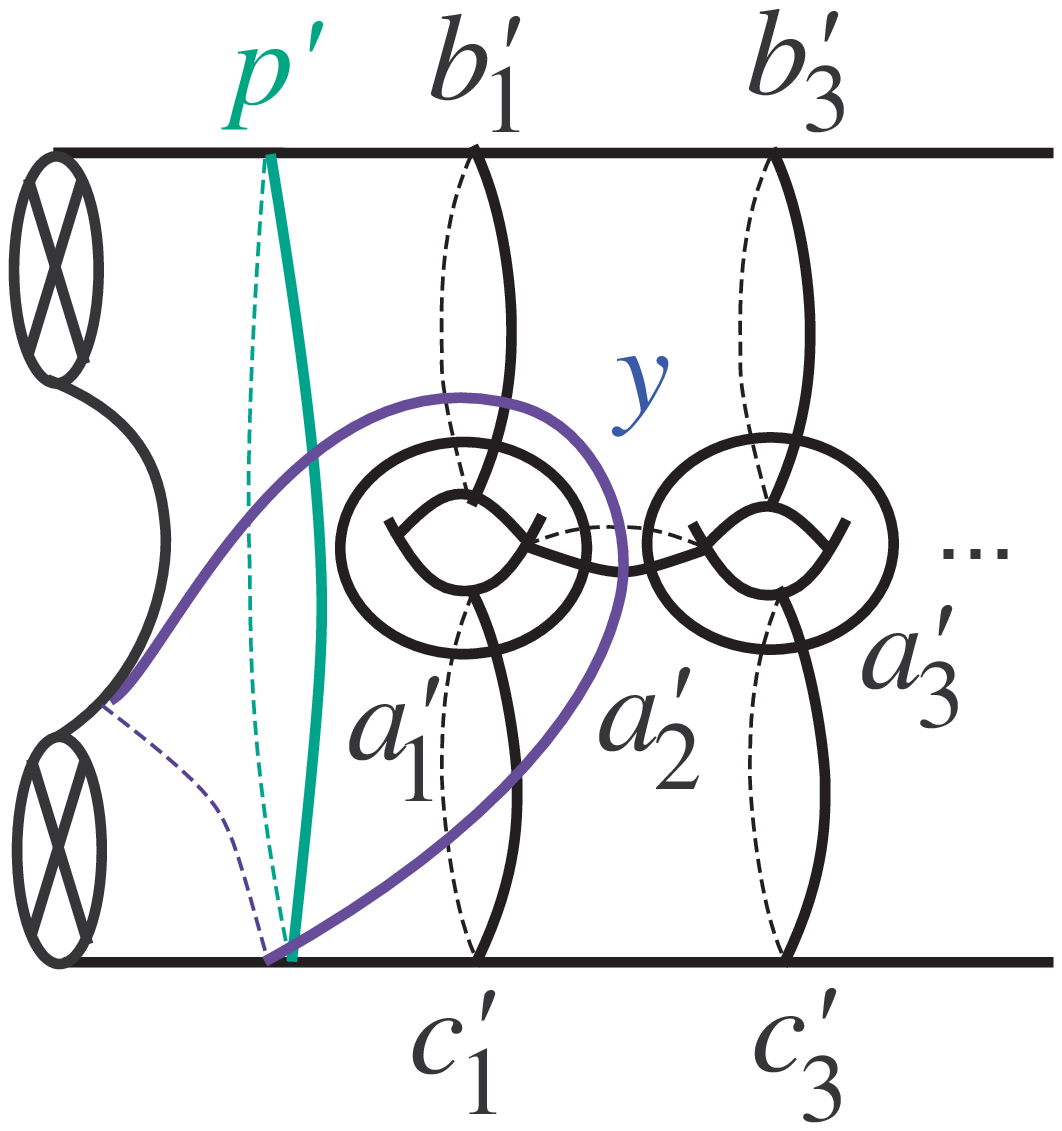}  \hspace{0.5cm} \epsfxsize=2.1in \epsfbox{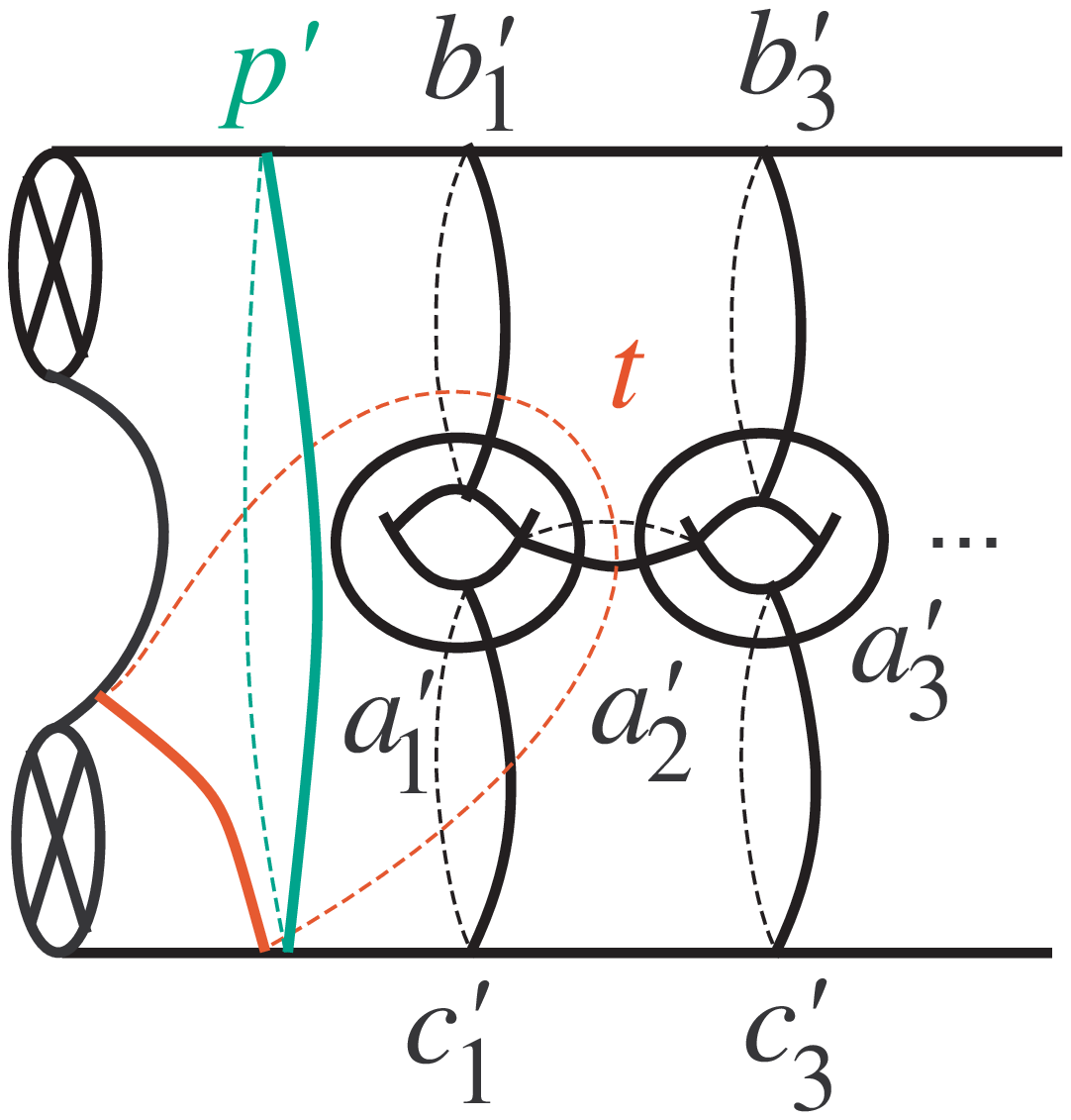}

\hspace{-1.1cm} (i) \hspace{6cm} (ii)

\caption{Curve configuration VII} \label{fig1b-VI}
\end{center}
\end{figure}

By Lemma \ref{piece2-a}, if two curves in $\mathcal{B}$ separate a twice holed projective plane on $N$, then the
corresponding curves in $\mathcal{B'}$ separate a twice holed projective plane on $N$. By Lemma \ref{piece2-aa},
if two curves in $\mathcal{B}$ are the boundary components of a pair of pants where the third boundary component of
the pair of pants is a boundary component of $N$, then the corresponding curves in $\mathcal{B'}$ are the boundary
components of a pair of pants where the third boundary component of the pair of pants is a boundary component of $N$.
By Lemma \ref{piece1}, if three nonseparating curves in $\mathcal{B}$ bound a pair of pants on $N$, then the corresponding
curves in $\mathcal{B'}$ bound a pair of pants on $N$. These imply that the curves in $\mathcal{B'}$ are as shown in
Figure \ref{fig1b-V} (i). Let $P = \{a_2, a_4, a_6, \cdots, a_{g-4}, b_1, b_3, \cdots, b_{g-3}, c_1, c_3, \cdots, c_{g-3},$
$ d_1, d_2, \cdots,$ $ d_{n-1}, c, p\}$. We see that $P$ corresponds to a top dimensional maximal simplex in $\mathcal{T}(N)$.
Let $p', c'$ be disjoint representatives of $\lambda([p]), \lambda([c])$ respectively such that $p'$ and $c'$ have minimal
intersection with the elements of $\mathcal{B'}$. Since $p$ is adjacent to $b_1$ and $c_1$ w.r.t $P$ and $p$ is disjoint
from all the curves in $\mathcal{B} \setminus \{r\}$, it is easy to see that $p', b'_1$ and $c'_1$ bound a pair of pants on
$N$, and $p'$ is as shown in Figure \ref{fig1b-V} (i). Then since $c$ is adjacent to $p$ w.r.t. $P$ and $c$ is disjoint
from all the curves in $\mathcal{B} \setminus \{r\}$, $c'$ has to be the unique 2-sided curve up to isotopy in the subsurface
bounded by $p'$, so $c'$ is also as shown in the figure. We also see that $i([p'], [r'])= 2$.

Now we consider the curve $v$ as shown in Figure \ref{fig1b-n} (ii), and control its image. Let $v' \in \lambda([v])$
such that $v'$ intersects minimally with the elements of $\mathcal{B'}$. Since $v$ is disjoint from each of
$r, a_1, a_2, b_3$, we know $v'$ is disjoint from each of $r', a'_1, a'_2, b'_3$. Since $v$ intersects $b_1$
nontrivially, $v'$ intersects $b'_1$ nontrivially. All this information about $v'$, and injectivity of $\lambda$
implies that $v'$ is as shown in Figure \ref{fig1b-V} (ii).

Let $l' \in \lambda([l])$ such that $l'$ intersects minimally with the elements of $\mathcal{B'} \cup \{v', c', p'\}$. Since $l$ intersects
$b_1, a_2, c_1$ only once, $l'$ intersects  $b'_1, a'_2, c'_1$ only once by Lemma \ref{intone}. Since $l$ is disjoint from
$a_3 \cup b_3 \cup c_3 \cup a_1 \cup v$, $l'$ is disjoint from $a'_3 \cup b'_3 \cup c'_3 \cup a'_1 \cup v'$. There exists a
homeomorphism sending the pair $(p, r)$ to $(p, l)$. Since we showed that $i([p'], [r'])= 2$ as shown in Figure \ref{fig1b-V} (i),
we can see that $p', l'$ intersect twice in the same way, i.e. if $K$ is the nonorientable genus two surface bounded by $p'$, then
the arc of $l'$ in $K$ will separate $K$ into two M\"{o}bius bands. Since $a_1$ and $l$ are separating twice holed projective plane,
we know by Lemma \ref{piece2-a} that $a'_1$ and $l'$ are separating twice holed projective plane. Using all this information
about $l'$, we see that $l'$ is isotopic to either $t$ or $y$ as shown in Figure \ref{fig1b-V} (iii). There exists a
homeomorphism $\phi$ of order two sending each curve in $\mathcal{B'} \cup \{v', c', p'\}$ to itself and switching $t$ and $y$
(the reflection through the plane of the paper, see Figure \ref{fig1b-VI}). Let $\phi_{*}$ be the induced map on $\mathcal{T(N)}$.
By replacing $\lambda$ with $\lambda \circ \phi_{*}$ if necessary, we can assume that $l'$ is isotopic to $y$. We note that to get
the proof of the lemma, it is enough to prove the result for
this $\lambda$.

From now on we assume that the curves in $\mathcal{B'} \cup \{v', c', p', l'\}$ are as shown in Figure \ref{fig1b-V} (i), (ii), (iv).
Let $a' \in \lambda([a])$ such that $a'$ has minimal intersection with the elements of $\mathcal{B'} \cup \{c', p', l'\}$.
By Lemma \ref{intone} geometric intersection number one is preserved. Since $a$ intersects $a_1$ only once, and $a$ is disjoint from $b_1$
and $c_1$, $a'$ intersects $a'_1$ only once, and $a'$ is disjoint from $b'_1$ and $c'_1$. So, there exists a unique arc up to isotopy
of $a'$ in the pair of pants, say $Q$, bounded by $p', b_1', c_1'$, starting and ending on $p'$ and intersecting the arc of $a_1'$ in $Q$
only once. Since $a$ is disjoint from $c$, $a'$ is disjoint from $c'$. So, there exists a unique arc up to isotopy of $a'$ in the
Klein bottle with one hole bounded by $p'$. Since $a$ intersects $l$ only once, $a'$ intersects $l'$ only once. Using all this
information about $a'$, we see that $a'$ is as shown in Figure \ref{fig1b-V} (iv). Hence, $i([p'], [a'])=2$.

Let $k' \in \lambda([k])$ such that $k'$ intersects minimally with the elements of $\mathcal{B'} \cup \{a', c', p', l'\}$. Since $k$
is disjoint from each of $b_1, c_1, l$, we see that $k'$ is disjoint from each of $b'_1, c'_1, l'$. A regular neighborhood of
$b'_1 \cup c'_1 \cup l'$ is an orientable subsurface of genus one with two boundary components such that one of the boundary
components bounds a Klein bottle with one hole, say $K$, on $N$. It is easy to see that $k'$ must be in $K$. Since up to
isotopy there exists a unique nonseparating 2-sided curve in $K$, we see that $k'$ is as shown in Figure \ref{fig1b-V} (v).

\begin{figure}[htb]
\begin{center}

\epsfxsize=2.9in \epsfbox{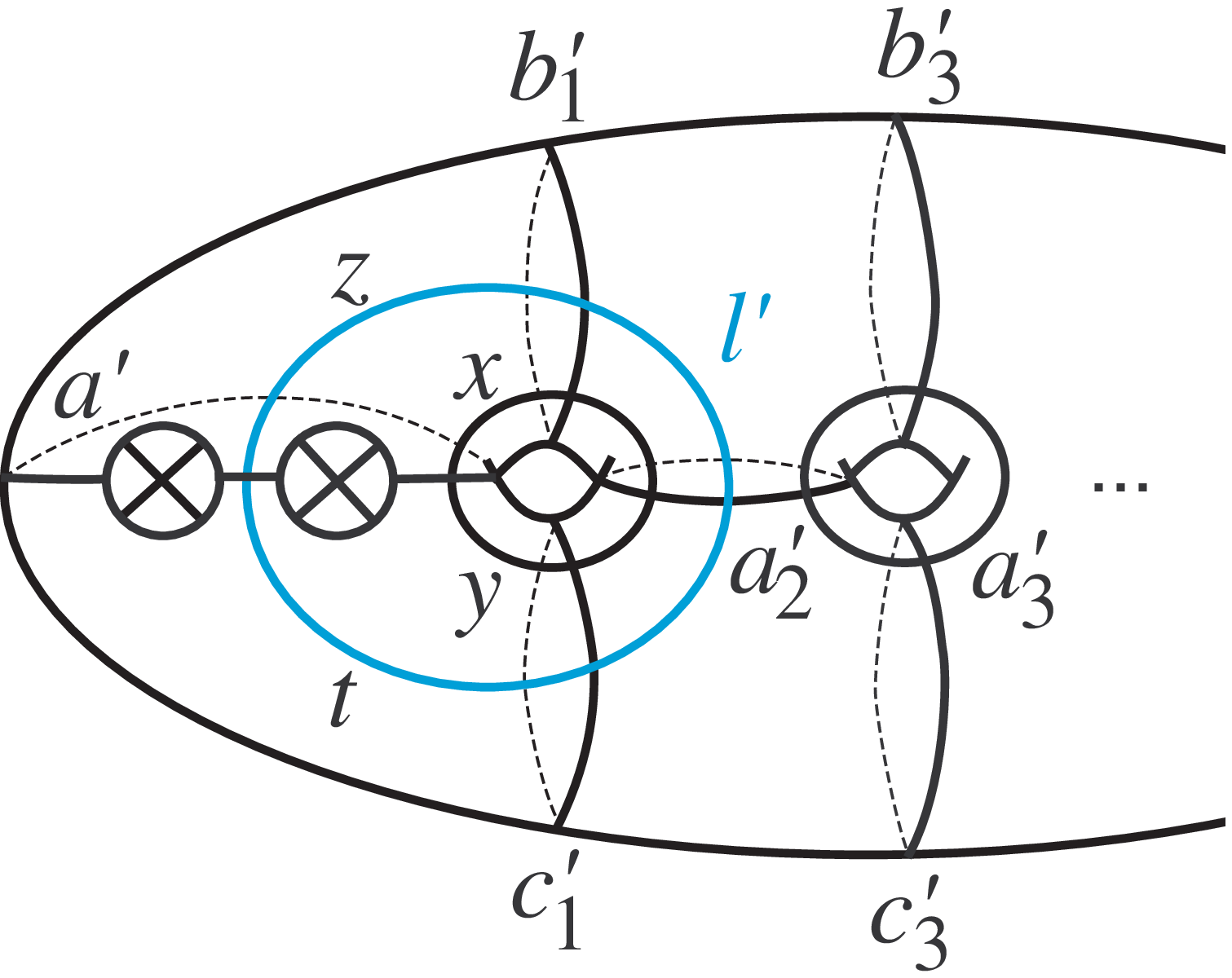} \epsfxsize=2.15in \epsfbox{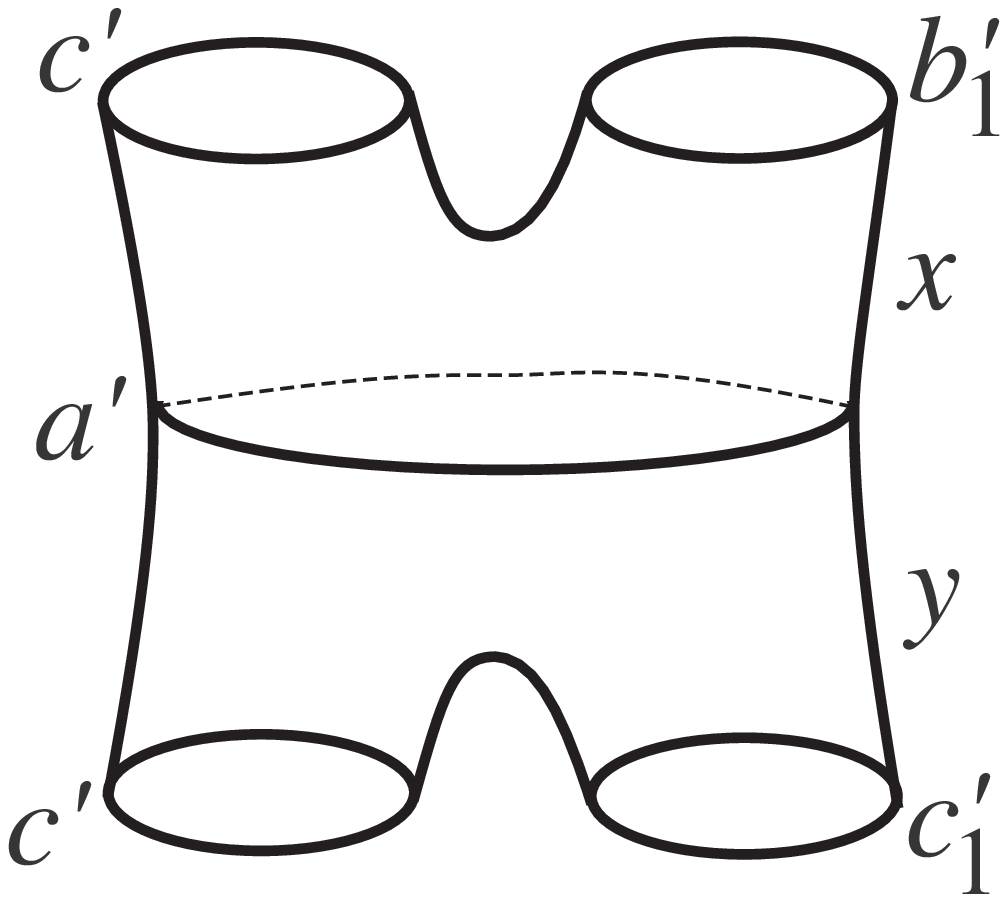}

\hspace{-0.5cm} (i) \hspace{5.65cm}   (ii)

\hspace{1.cm} \epsfxsize=2.95in \epsfbox{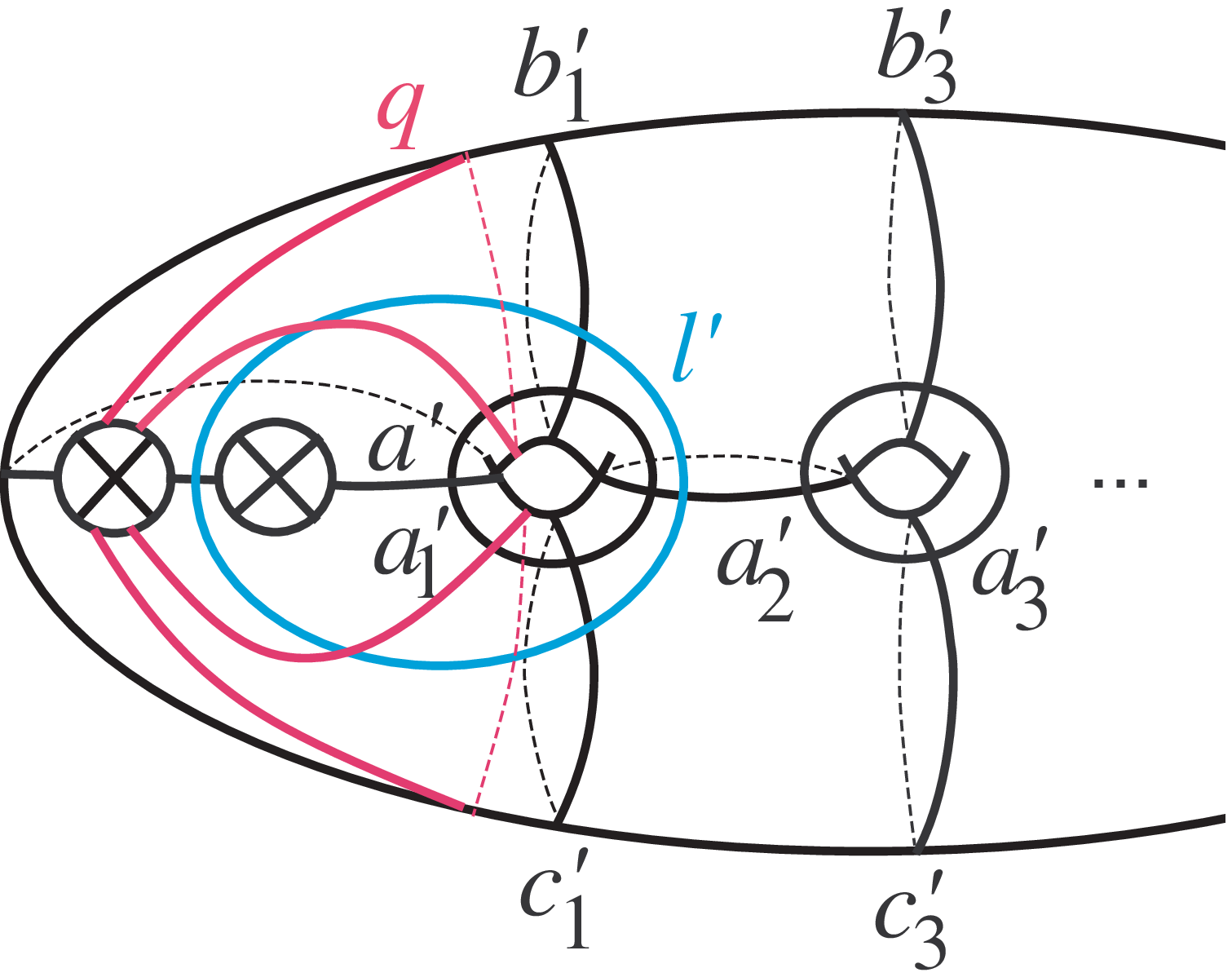}  \hspace{-1.5cm}  \epsfxsize=2.95in \epsfbox{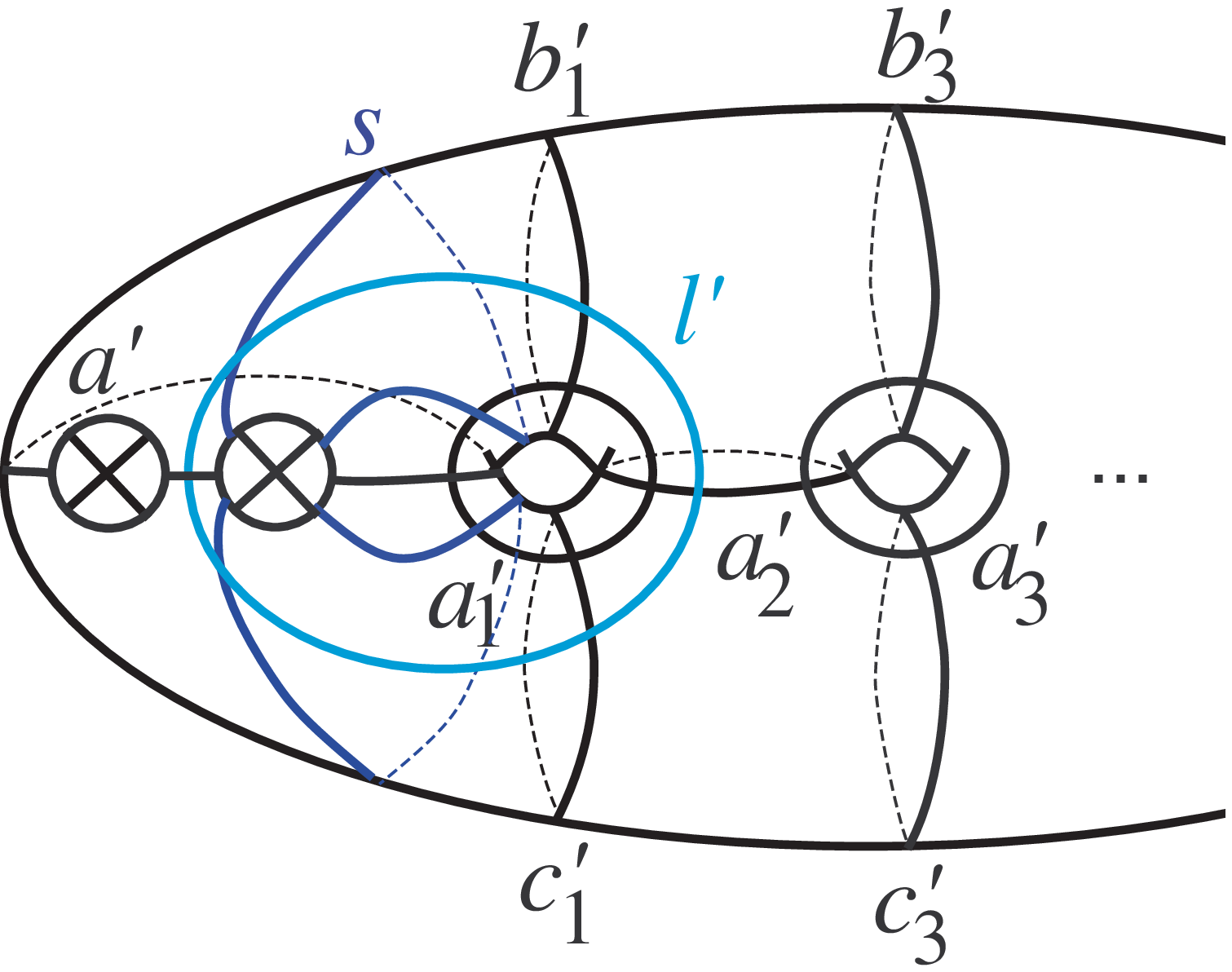}

\hspace{-0.4cm} (iii) \hspace{5.5cm}   (iv)

\caption{Curve configuration VIII}
\label{fig1bbb}
\end{center}
\end{figure}

Let $e' \in \lambda([e])$ such that $e'$ intersects minimally with the elements of $\mathcal{B'} \cup \{a', c', p', l', $ $k', v'\}$.
The curve $e$ bounds a Klein bottle with one hole whose complement is nonorientable on $N$. By using Lemma \ref{int-twi}, we see
that $i([e'], [v'])=2$ as there exists a homeomorphism sending $(e, v)$ to $(p, r)$ where $p$ and $r$ are as given in
Lemma \ref{int-twi}. By using Lemma \ref{int-twi-2}, we see that $i([e'], [a'_1])=2$ as there exists a homeomorphism sending $(e, a_1)$ to $(p, x)$ 
where $p$ and $x$ are as given in Lemma \ref{int-twi-2}. Similarly, by using Lemma \ref{int-twi-2}, we see
that $i([e'], [l'])=2$ as there exists a homeomorphism sending $(e, l)$ to $(p, x)$ where $p$ and $x$ are as given in
Lemma \ref{int-twi-2}. So, we have $i([e'], [a_1'])=2$ and $i([e'], [l'])=2$ in our curve configuration.

Since $e$ is disjoint from $b_1$ and $c_1$, $e'$ is disjoint from $b'_1$ and $c'_1$. There is a nonorientable genus 2 surface $K$ with two boundary 
components $b'_1$ and $c'_1$ as shown in Figure \ref{fig1bbb} (i). The two intersection points of $e'$ and $a'_1$ are on $K$ as $e'$ is on $K$.
Similarly, the two intersection points of $e'$ and $l'$ are on $K$. Let $x$ be the arc of $a'_1$ connecting $b'_1$ to $a'$ in $K$, and
let $y$ be the arc of $a'_1$ connecting $c'_1$ to $a'$ in $K$ as shown in Figure \ref{fig1bbb} (i). Let $z$ be the arc of $l'$ connecting
$b'_1$ to $a'$ in $K$, and let $t$ be the arc of $l'$ connecting $c'_1$ to $a'$ in $K$ as shown in the figure.

Claim: $e'$ intersects each of $x$ and $y$ transversely once.

Proof of the claim: We know that $i([e'], [a_1'])=2$ and $e' \cap a'_1 \subseteq x \cup y$. Suppose both of the intersection points of
$e'$ and $a'_1$ are on $x$. Then $e'$ does not intersect $y$. When we cut $N$ along $c', b'_1$ and $c'_1$ we get a four holed sphere as
shown in Figure \ref{fig1bbb} (ii). Since $e'$ is disjoint from $y, c'_1, a'$ it would have to be disjoint from $c'$, see the figure.
This gives a contradiction, since $e$ and $c$ intersect nontrivially, $e'$ and $c'$ should intersect nontrivially. So, both of the
intersection points of $e'$ and $a'_1$ cannot be on $x$. Similarly, both of the intersection points of $e'$ and $a'_1$ cannot be on $y$. Hence, $e'$ intersects each of $x$ and $y$ transversely once.

With similar arguments by cutting $N$ along $k', b'_1$ and $c'_1$, we see that $e'$ intersects each of $z$ and $t$ transversely once. Since $e$
doesn't intersect any of $a, b_1$ and $c_1$, we know that  $e'$ doesn't intersect any of $a', b'_1$ and $c'_1$. We also know that $e'$
intersects each of $x, y, z, t$ exactly once transversely. All this information implies that $e'$ has to be isotopic to either $q$ or $s$
where $q$ and $s$ are as shown in Figure \ref{fig1bbb} (iii) and (iv). Since $i([e'], [v'])=2$, $e'$ cannot be isotopic to $q$. Hence,
$e'$ is isotopic to $s$. This shows that $e'$ is as we wanted. Hence, there is a homeomorphism $h: N \rightarrow N$ such that
$h([x]) = \lambda([x])$ for all $x \in \mathcal{C}$.\end{proof}

\begin{lemma}
\label{int-twice-2} Suppose that $g \geq 5$ and $g$ is odd. Let $a, p, r$ be curves as shown in Figure \ref{fig1b-n9} (i).
There exist $a' \in \lambda([a]), p' \in \lambda([p]), r' \in \lambda([r])$ such that $i([p'], [a'])=2$ and $i([p'], [r'])=2$. \end{lemma}

\begin{proof} The proof is similar to the proof of Lemma \ref{int-twi}. We will give the proof when $g \geq 7$ and $n \geq 2$.
The proof for the remaining cases will be similar.

Let $a, p, r, c$ be as shown in Figure \ref{fig1b-n9} (i). We consider the curve configuration
$\mathcal{B} = \{a_1, a_2, \cdots, a_{g-4}, b_1, b_3, \cdots, b_{g-4}, c_0, c_1, c_3, c_5, \cdots,$ $ c_{g-4}, d_1, d_2, \cdots, d_{n-1}, r\}$
as shown in Figure \ref{fig1b-n9} (i). Let $a'_1  \in \lambda([a_1]), a'_2  \in \lambda([a_2]),$
$\cdots, a_{g-4}' \in \lambda([a_{g-4}]), b'_1  \in \lambda([b_1]), b'_3  \in \lambda([b_3]),$
$ \cdots, b_{g-4}' \in \lambda([b_{g-4}]), c'_0  \in \lambda([c_0]), c'_1  \in \lambda([c_1]), \cdots, c_{g-4}' \in \lambda([c_{g-4}]),$
$ d'_1  \in \lambda([d_1]), d'_2  \in \lambda([d_2]), \cdots, $ $d_{n-1}' \in \lambda([d_{n-1}]), r' \in \lambda([r])$ be minimally
intersecting representatives. By Lemma \ref{intone} geometric intersection one is preserved. So, a regular neighborhood of union of
all the elements in $\mathcal{B'}= \{a'_1, a'_2, \cdots, a_{g-4}', b'_1, b'_3, \cdots,$
$ b'_{g-4}, c'_0, c'_1, c'_3, \cdots, c'_{g-4}, d'_1, d'_2, \cdots,$ $d'_{n-1}, r'\}$ is an orientable surface of genus $\frac{g-3}{2}$
with several boundary components.

\begin{figure}[htb]
\begin{center}

\hspace{0.1cm}  \epsfxsize=3.17in \epsfbox{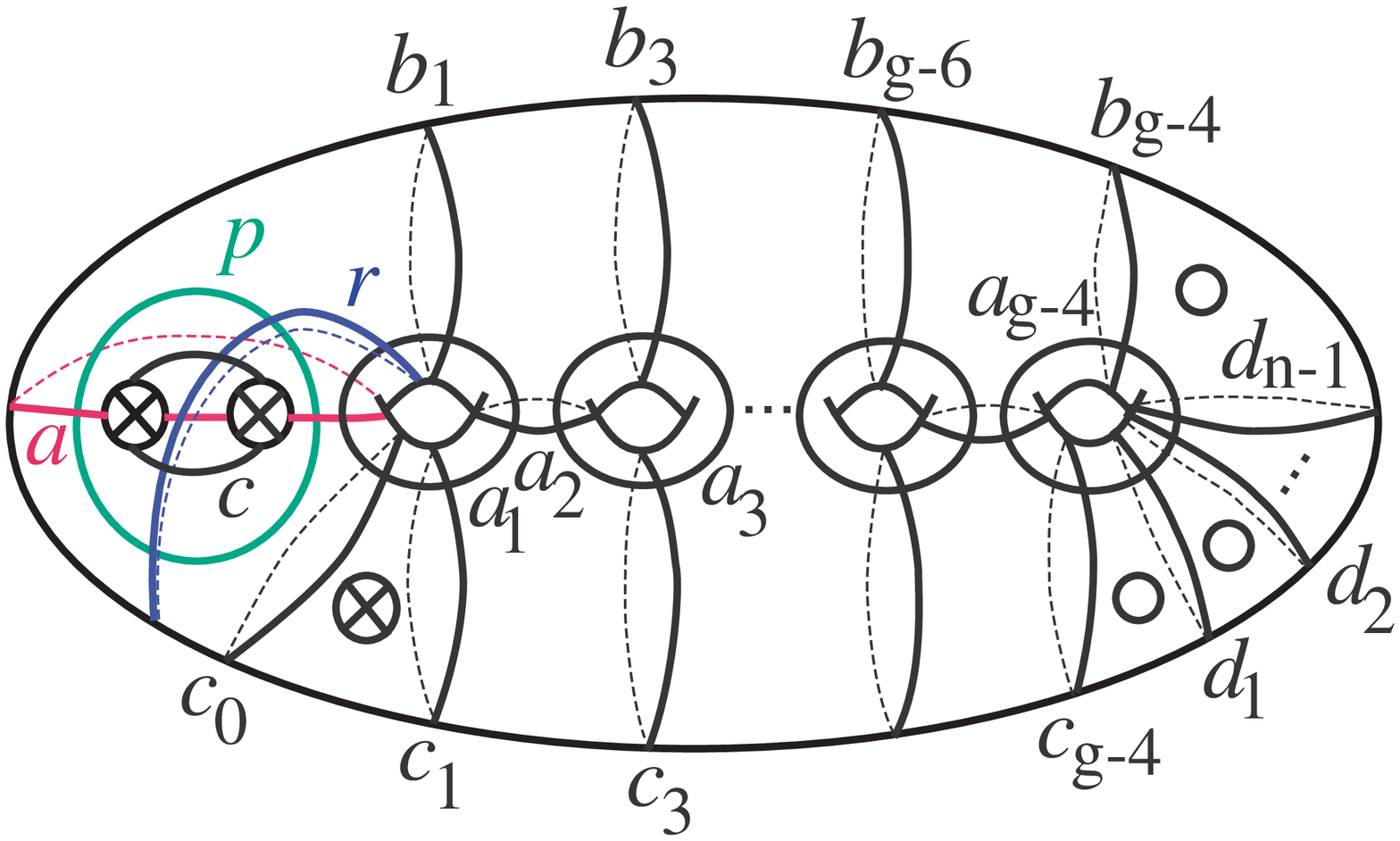}

\hspace{-1.5cm}(i)

\epsfxsize=3.17in \epsfbox{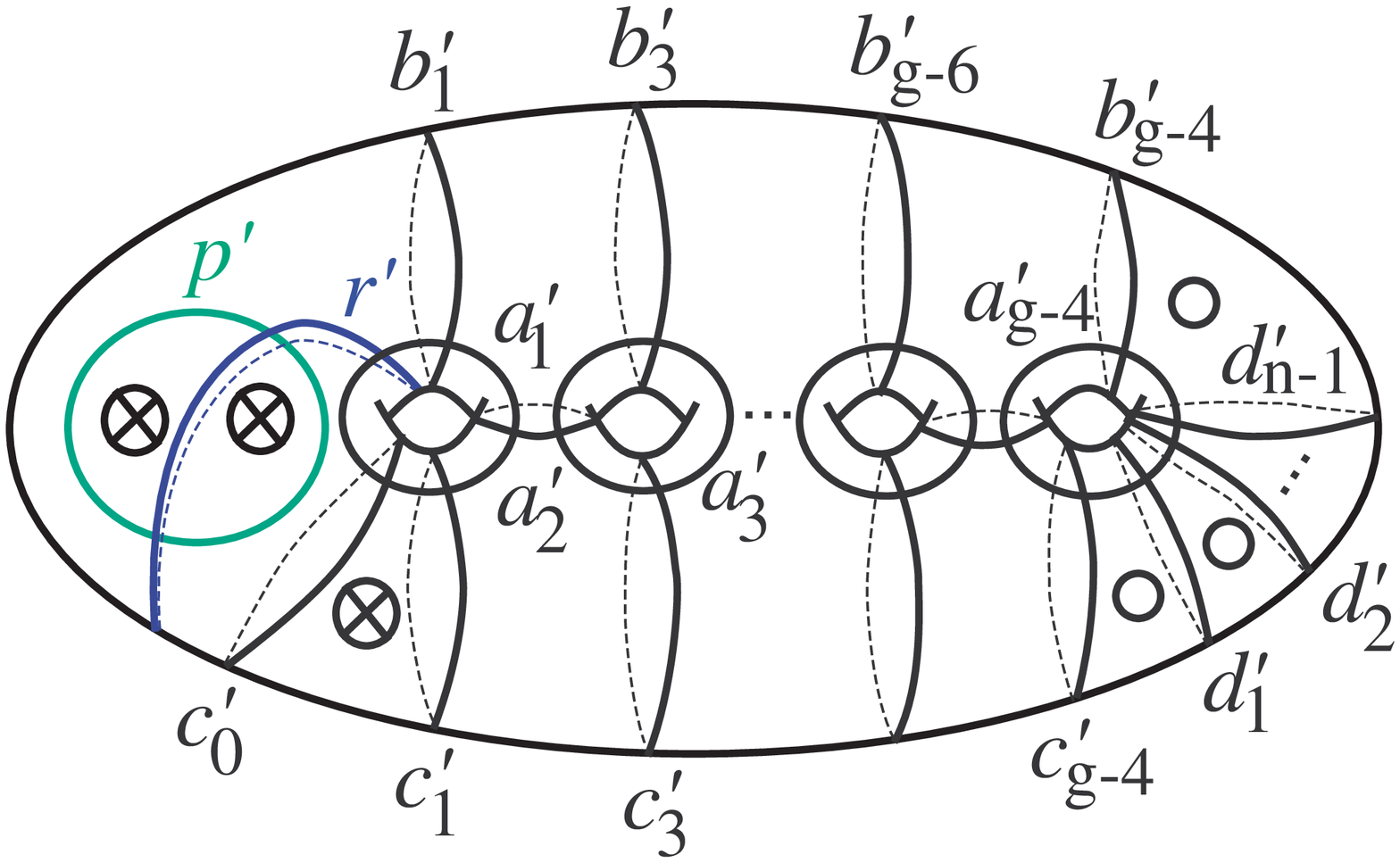} \hspace{-1.5cm} \epsfxsize=3.17in \epsfbox{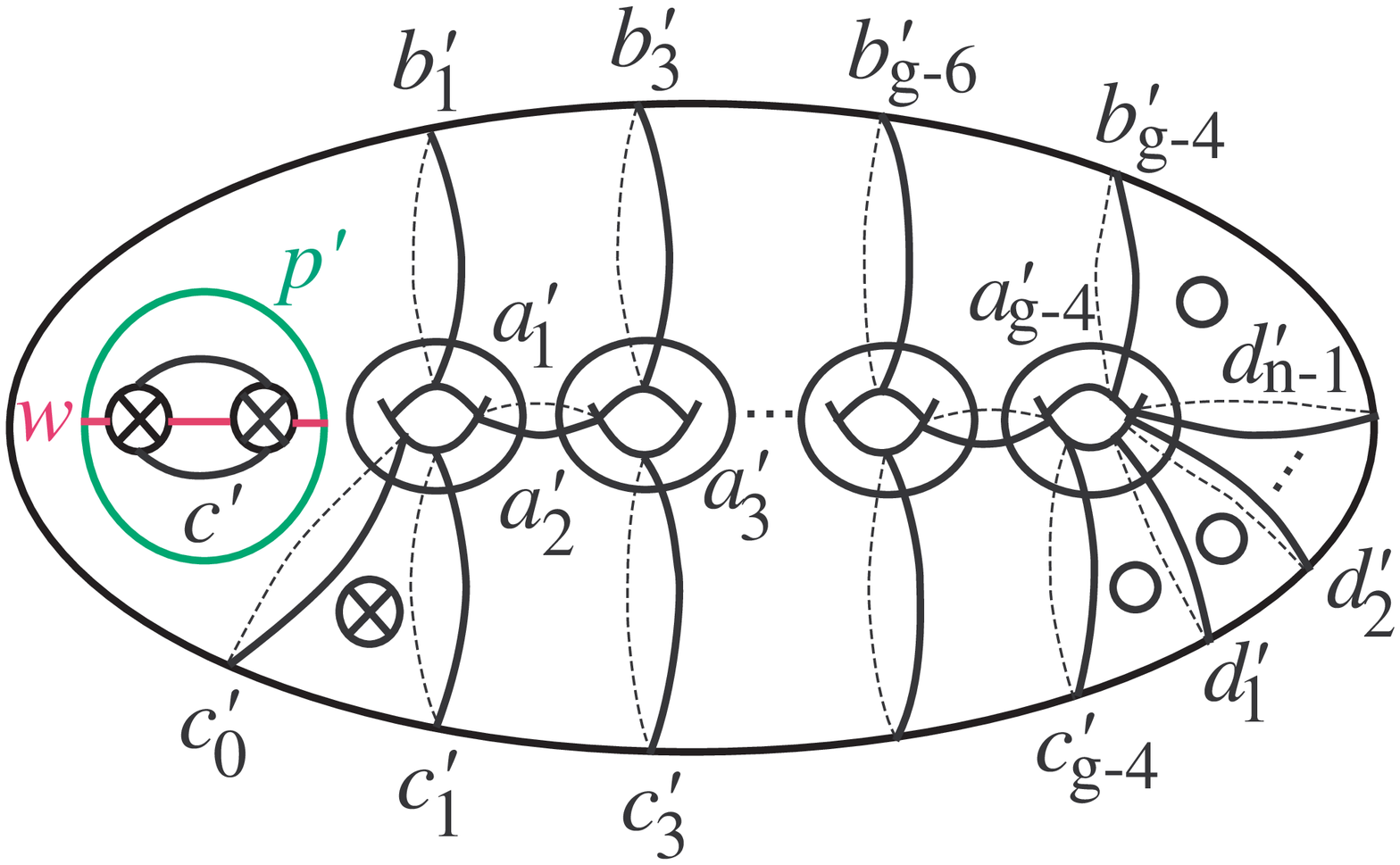}

\hspace{-1cm} (ii) \hspace{6.5cm} (iii)

\caption{Curve configuration IX} \label{fig1b-n9}
\end{center}
\end{figure}

By Lemma \ref{piece2-a}, if two curves in $\mathcal{B}$ separate a twice holed projective plane on $N$, then the corresponding
curves in $\mathcal{B'}$ separate a twice holed projective plane on $N$. By Lemma \ref{piece2-aa}, if two curves in $\mathcal{B}$
are boundary components of a pair of pants where the third boundary component of the pair of pants is a boundary component of $N$,
then the corresponding curves in $\mathcal{B'}$ are boundary components of a pair of pants where the third boundary component of
the pair of pants is a boundary component of $N$. By Lemma \ref{piece1}, if three nonseparating curves in $\mathcal{B}$ bound a pair
of pants on $N$, then the corresponding curves in $\mathcal{B'}$ bound a pair of pants on $N$. These imply that the curves in
$\mathcal{B'}$ are as shown in Figure \ref{fig1b-n9} (ii).

Let $P = \{a_2, a_4, a_6, \cdots, a_{g-5}, b_1, b_3, \cdots, b_{g-4},$ $ c_0, c_1, c_3, \cdots, c_{g-4}, d_1, d_2, \cdots,$
$d_{n-1}, p, c\}$. We see that $P$ corresponds to a top dimensional maximal simplex in $\mathcal{T}(N)$.
Let $p' \in \lambda([p])$, $c' \in \lambda([c])$ such that $p'$ and $c'$ have minimal intersection with the
elements of $\mathcal{B'}$. Let $P' = \{a'_2, a'_4, a'_6, \cdots, a'_{g-5}, b'_1, b'_3, \cdots, b'_{g-4},$
$c'_0, c'_1, c'_3, \cdots, c'_{g-4}, d'_1, d'_2, \cdots,$ $d'_{n-1}, p', c'\}$. By using Lemma \ref{adjacent},
we see that since $p$ is adjacent to $b_1$ and $c_0$ w.r.t $P$ and $p$ is disjoint from all the curves in
$\mathcal{B} \setminus \{r\}$, $p'$ is adjacent to $b'_1$ and $c'_0$ w.r.t $P'$ and $p'$ is disjoint from all
the curves in $\mathcal{B'} \setminus \{r'\}$. This implies that $p', b'_1$ and $c'_0$ bound a pair of pants on $N$,
and $p'$ is as shown in Figure \ref{fig1b-n9} (ii). Hence, $i([p'], [r'])=2$.

Since $c$ is adjacent to $p$ w.r.t. $P$ and $c$ is disjoint from all the curves in $\mathcal{B} \setminus \{r\}$, $c'$
is adjacent to $p'$ w.r.t. $P'$ and $c'$ is disjoint from all the curves in $\mathcal{B'} \setminus \{r'\}$. So, $c'$ has
to be the unique 2-sided curve up to isotopy in the subsurface bounded by $p'$. Hence, $c'$ is also as shown in Figure
\ref{fig1b-n9} (iii). We have $\{a'_1, a'_2, \cdots, a_{g-4}', b'_1, b'_3, \cdots,$
$ b_{g-4}', c'_0, c'_1, c'_3, \cdots, c_{g-4}', d'_1, d'_2, \cdots, d'_{n-1}, p', c'\}$ as shown in Figure \ref{fig1b-n9} (iii).
Let $a' \in \lambda([a])$ such that $a'$ has minimal intersection with the elements of $\mathcal{B'} \cup \{c', p'\}$.
By Lemma \ref{intone} geometric intersection number one is preserved. Since $a$ intersects $a_1$ only once, and $a$ is disjoint
from $b_1$ and $c_0$, $a'$ intersects $a'_1$ only once, and $a'$ is disjoint from $b'_1$ and $c'_0$. So, there exists a unique
arc up to isotopy of $a'$ in the pair of pants, say $Q$, bounded by $p', b_1', c_0'$, starting and ending on $p'$ and intersecting
the arc of $a_1'$ in $Q$ only once. Since $a$ is disjoint from $c$, $a'$ is disjoint from $c'$. These imply that there exists a
unique arc $w$ up to isotopy of $a'$ in the nonorientable genus 2 surface bounded by $p'$. 
%So, the elements in the set $\{a', c', p'\} \cup \mathcal{B'} \setminus \{r'\}$ are as shown in the Figure \ref{fig1b-n9-b}
%(ii) up to an action of a %power of Dehn twist about $p'$.
Hence, $i([p'], [a'])=2$.\end{proof}

\begin{lemma}
\label{int-twice-2-b} Suppose that $g \geq 5$ and $g$ is odd. Let $p, x, y$ be curves as shown in Figure \ref{fig1b-n9-c} (i).
There exist $p' \in \lambda([p]), x' \in \lambda([x]), y' \in \lambda([y])$ such that $i([p'], [x'])=2$ and $i([p'], [y'])=2$. \end{lemma}

\begin{proof} The proof is similar to the proof of Lemma \ref{int-twice-2}. We will give the proof when $g \geq 7$ and $n \geq 2$.
The proof for the remaining cases will be similar.

We consider the curve configuration $\mathcal{B} = \{a_1, a_2, \cdots, a_{g-4}, b_1, b_3, \cdots, b_{g-4}, c_0, c_1, c_3, $
$ \cdots, c_{g-4}, d_1, d_2, \cdots, d_{n-1}, p, r, x, y\}$ as shown in Figure \ref{fig1b-n9-c} (i). Let
$a'_1  \in \lambda([a_1]), a'_2  \in \lambda([a_2]),$
$\cdots, a_{g-4}' \in \lambda([a_{g-4}]), b'_1  \in \lambda([b_1]), b'_3  \in \lambda([b_3]),$
$ \cdots, b_{g-4}' \in \lambda([b_{g-4}]), c'_0  \in \lambda([c_0]), c'_1  \in \lambda([c_1]), \cdots, c_{g-4}' \in \lambda([c_{g-4}]),$
$ d'_1  \in \lambda([d_1]), d'_2  \in \lambda([d_2]), \cdots, $ $d_{n-1}' \in \lambda([d_{n-1}]),$
$ p' \in \lambda([p]), r' \in \lambda([r]), x' \in \lambda([x]), y' \in \lambda([y])$ be minimally intersecting representatives.

\begin{figure}[htb]
\begin{center}

\hspace{0.1cm}  \epsfxsize=3.17in \epsfbox{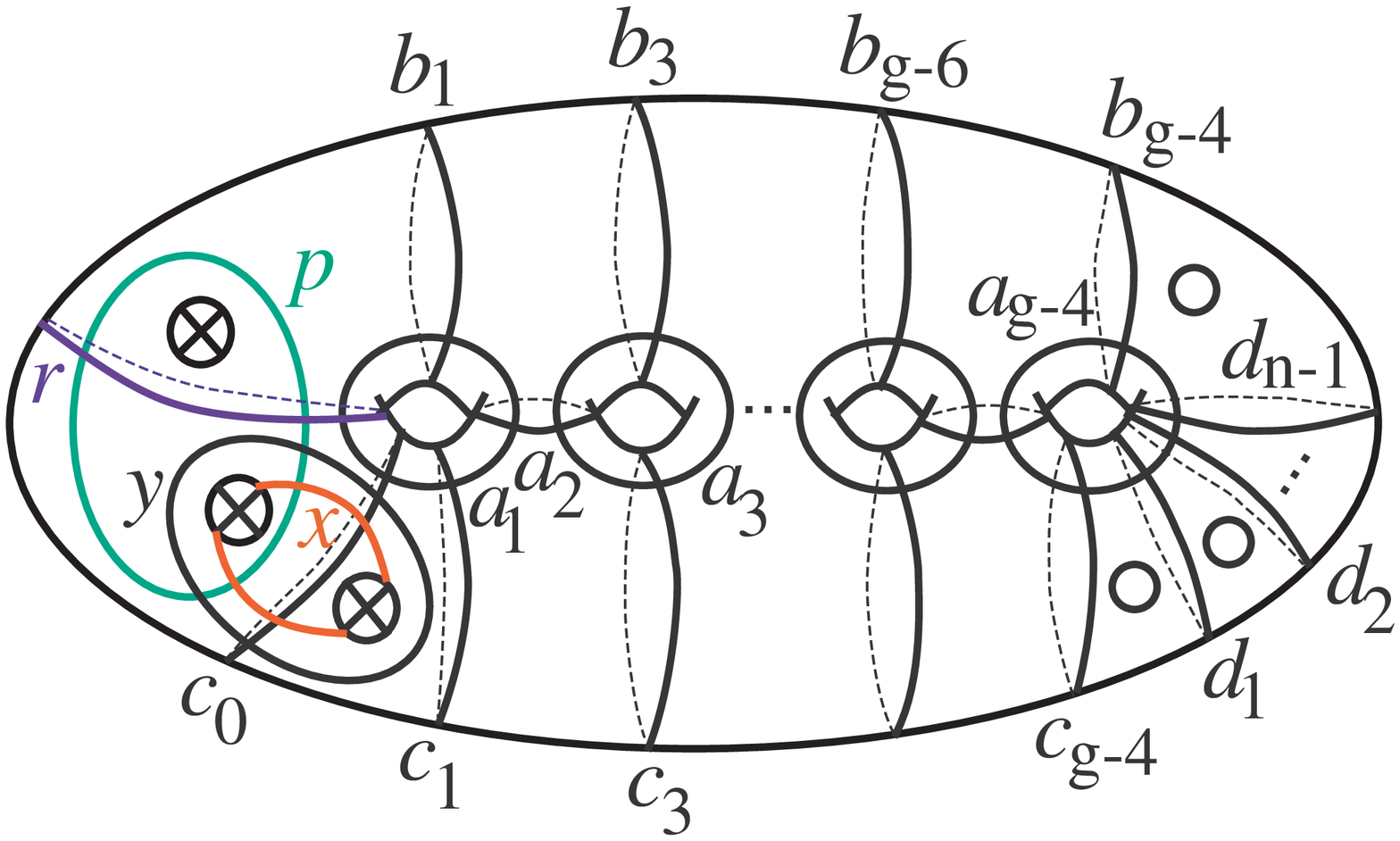}

\hspace{-1.5cm} (i)

\epsfxsize=3.17in \epsfbox{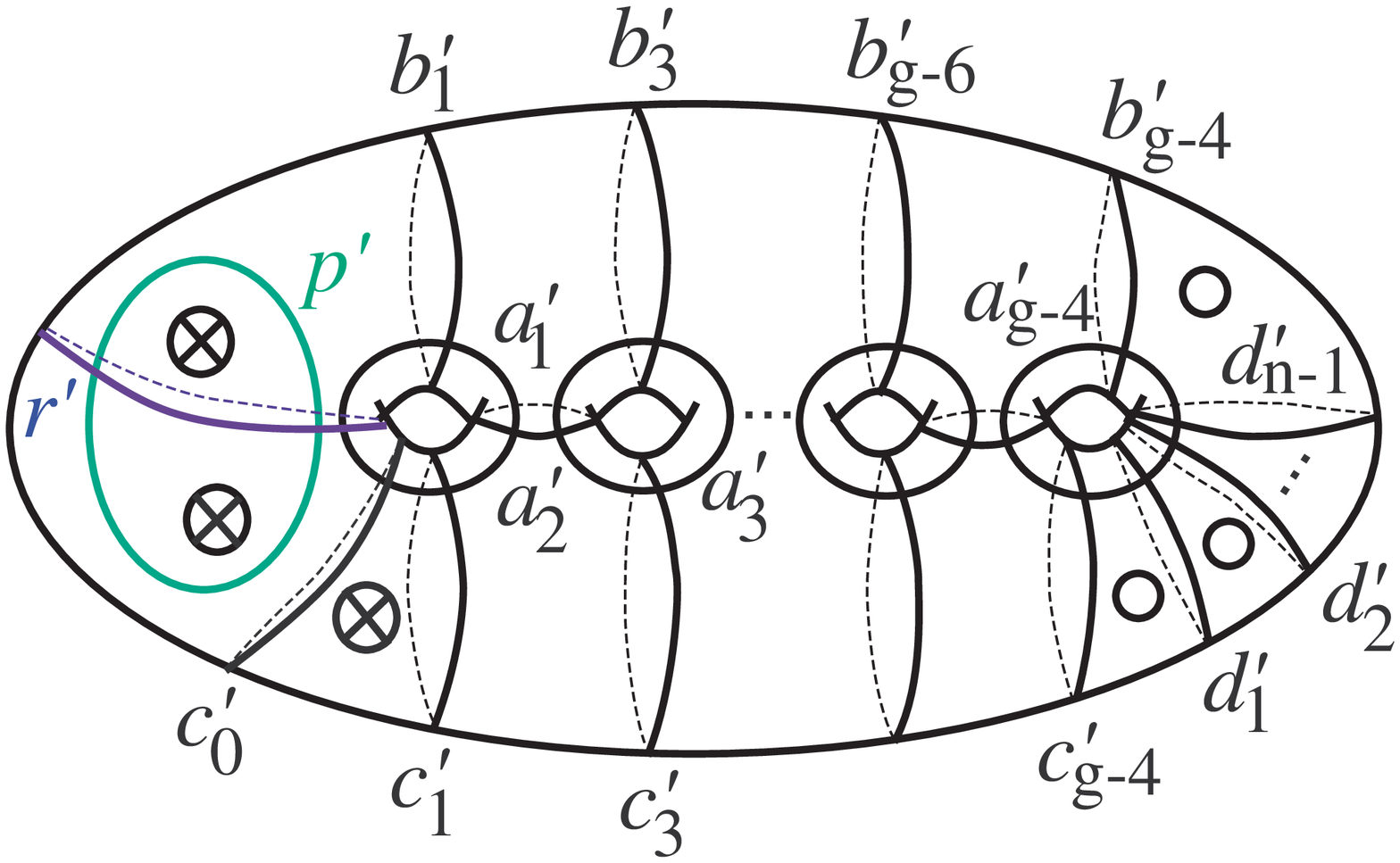} \hspace{-1.5cm} \epsfxsize=3.17in \epsfbox{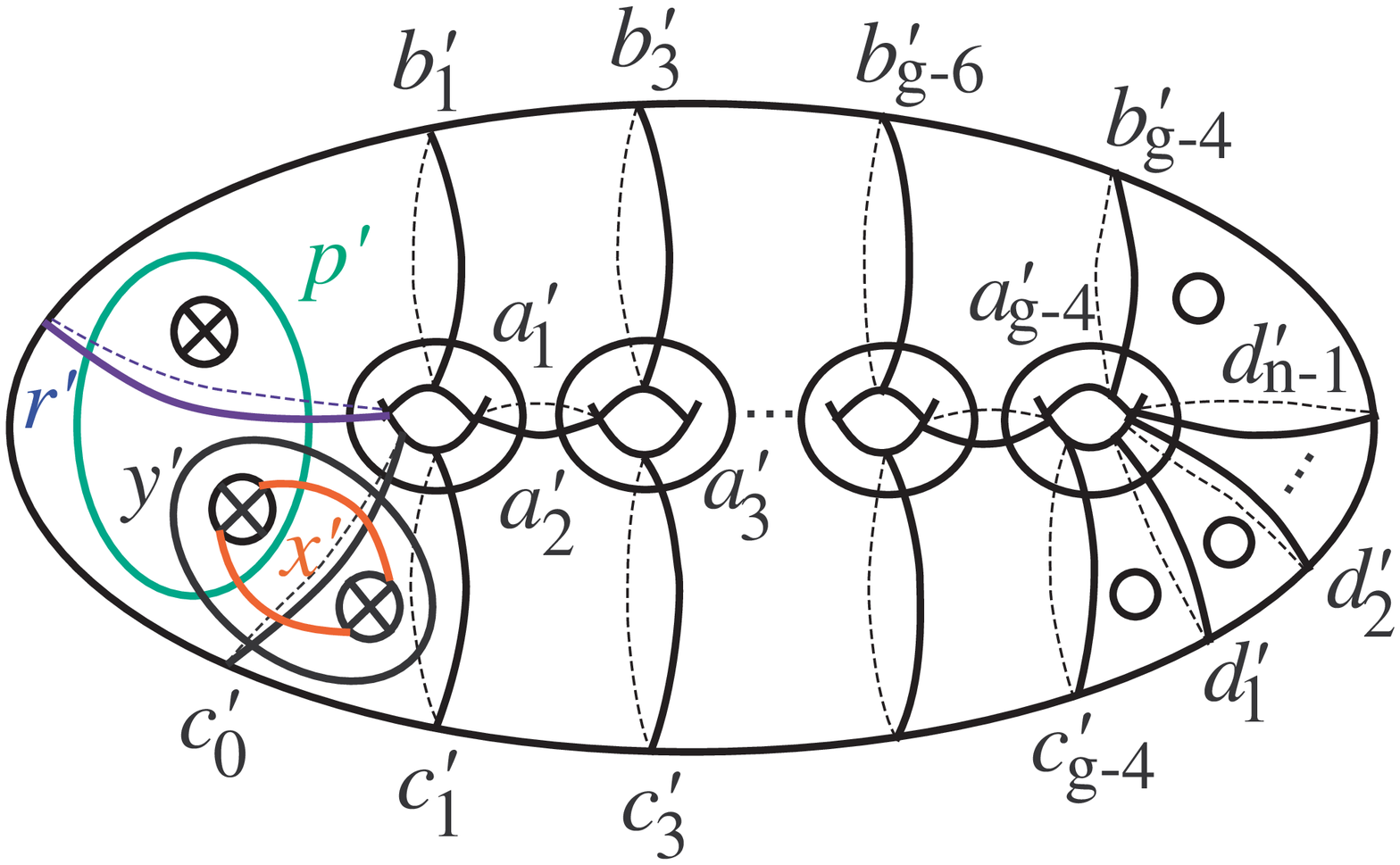}

\hspace{-1cm} (ii) \hspace{6.5cm} (iii)

\caption{Curve configuration X} \label{fig1b-n9-c}
\end{center}
\end{figure}

As in the proof of Lemma \ref{int-twice-2}, the curves in $\{a'_1, a'_2, \cdots, a_{g-4}', b'_1, b'_3, \cdots,$ $ b_{g-4}', c'_0, c'_1,$
$\cdots, c_{g-4}', d'_1, d'_2, \cdots, d'_{n-1}, p', r'\}$ will be as shown in Figure \ref{fig1b-n9-c} (ii). Since there exists a
homeomorphism sending $p$ to $y$ and $p'$ bounds a Klein bottle with one hole, we know that $y'$ bounds a Klein bottle with one hole. The curve
$y$ is disjoint from each of $r, a_1, c_1$ and it has nontrivial intersection with $c_0$. So, $y'$ is disjoint from each of $r', a'_1, c'_1$
and it has nontrivial intersection with $c'_0$. All this information about $y'$ implies that $y'$ is as shown in Figure \ref{fig1b-n9-c}
(iii). Since $x$ is disjoint from each of $r, a_1, c_1$ and it has nontrivial intersection with $c_0$ and $\lambda$ is injective, $x'$ has to
be in the Klein bottle bounded by $y'$ and should be as shown in the figure. So, we get $i([p'], [x'])=2$ and $i([p'], [y'])=2$.\end{proof}\\

\begin{figure}[htb]
\begin{center}
\hspace{-0.5cm}  \epsfxsize=3.0in \epsfbox{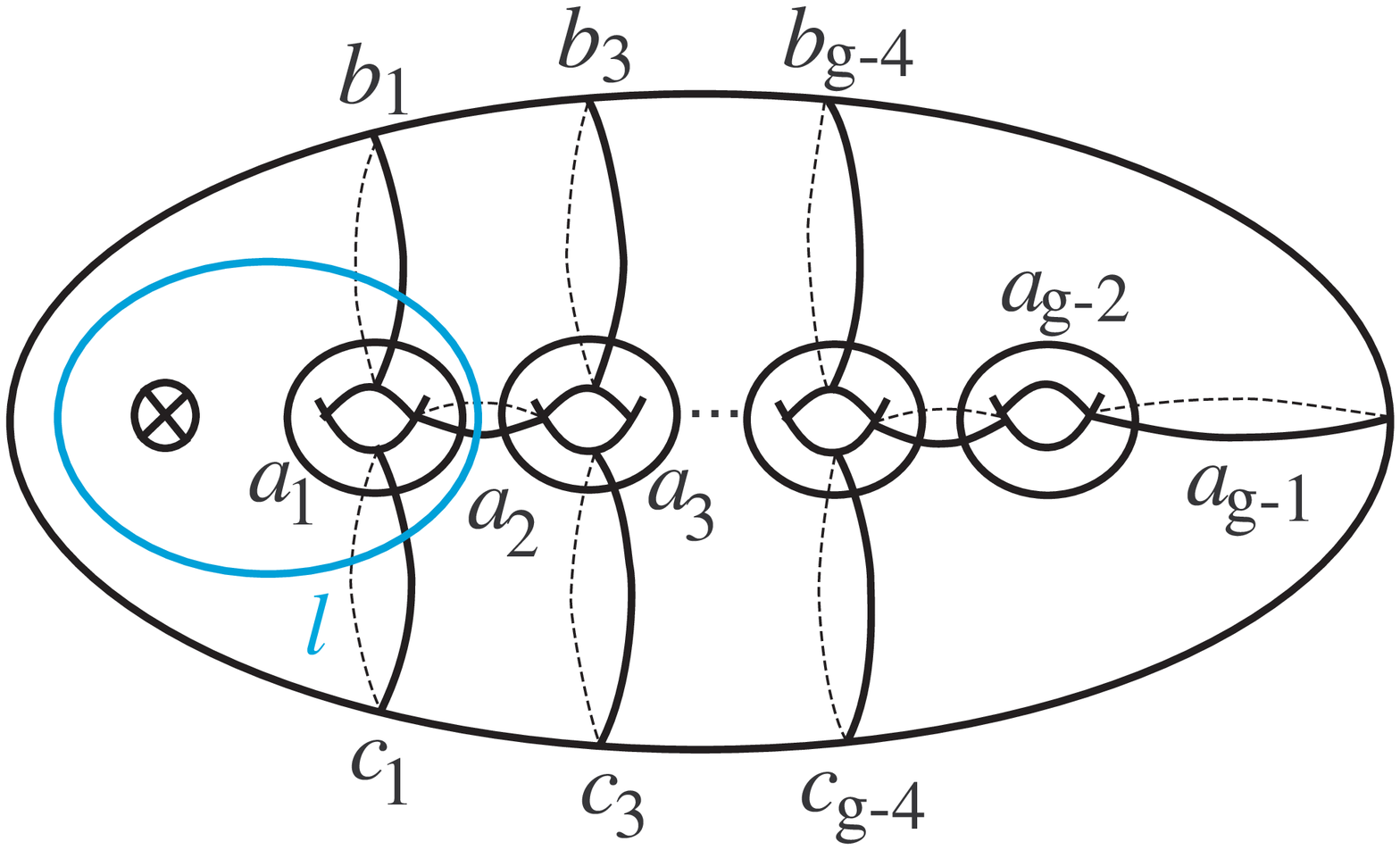}  \hspace{-0.5cm}  \epsfxsize=3.2in \epsfbox{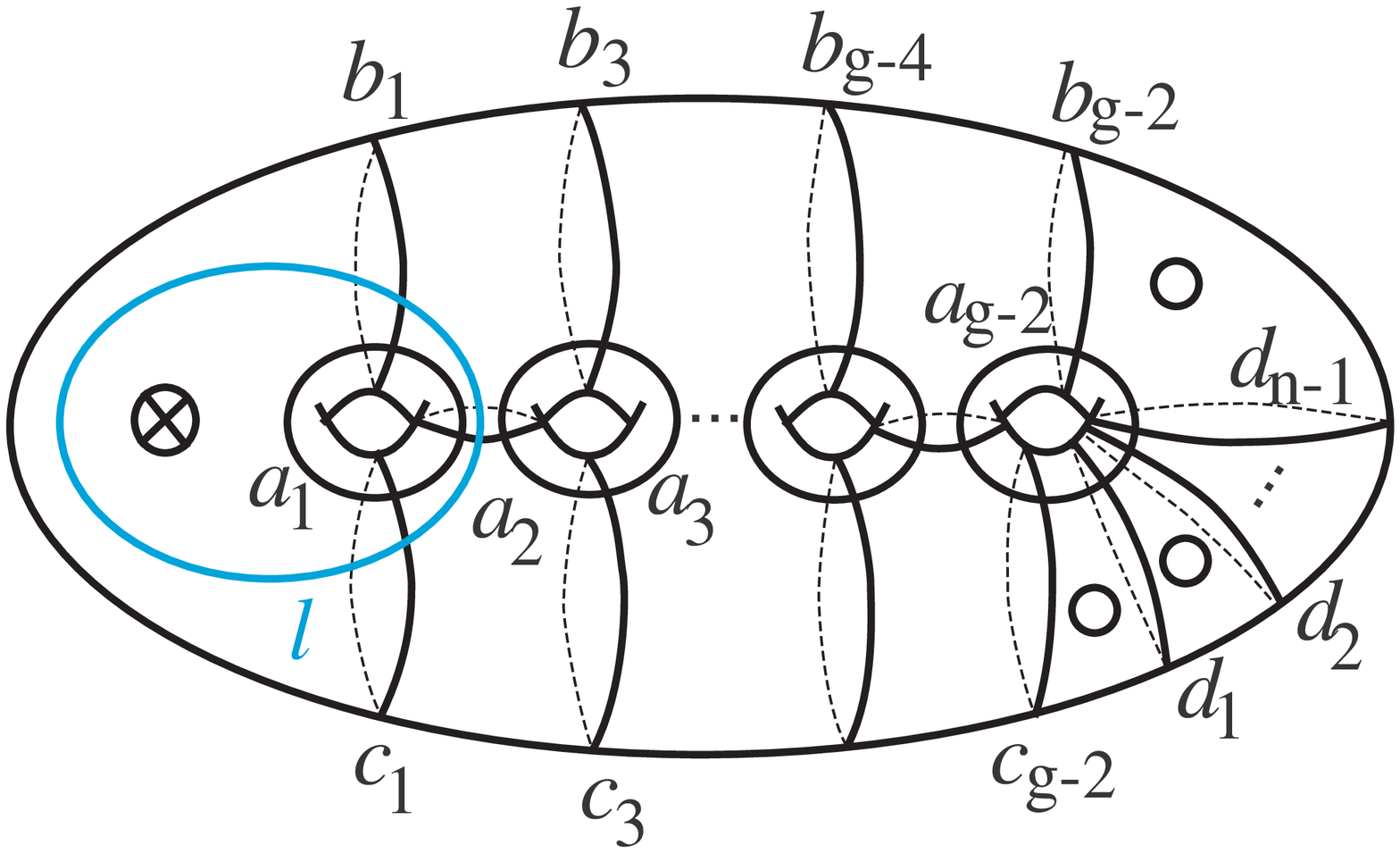}

\hspace{-1.3cm} (i) \hspace{5.9cm} (ii)

\hspace{-1.5cm} \epsfxsize=2.95in \epsfbox{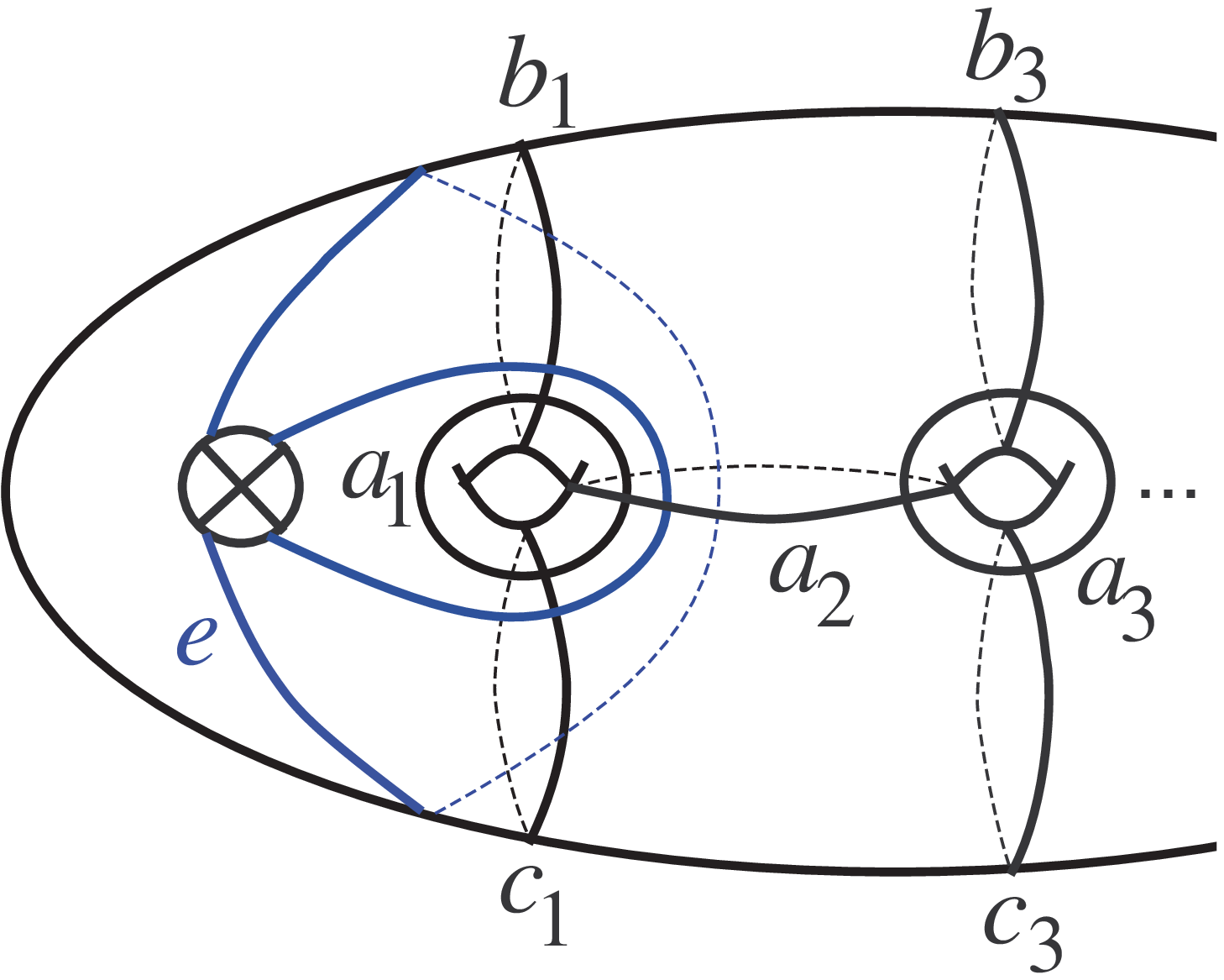} \hspace{-1cm}  \epsfxsize=2.95in \epsfbox{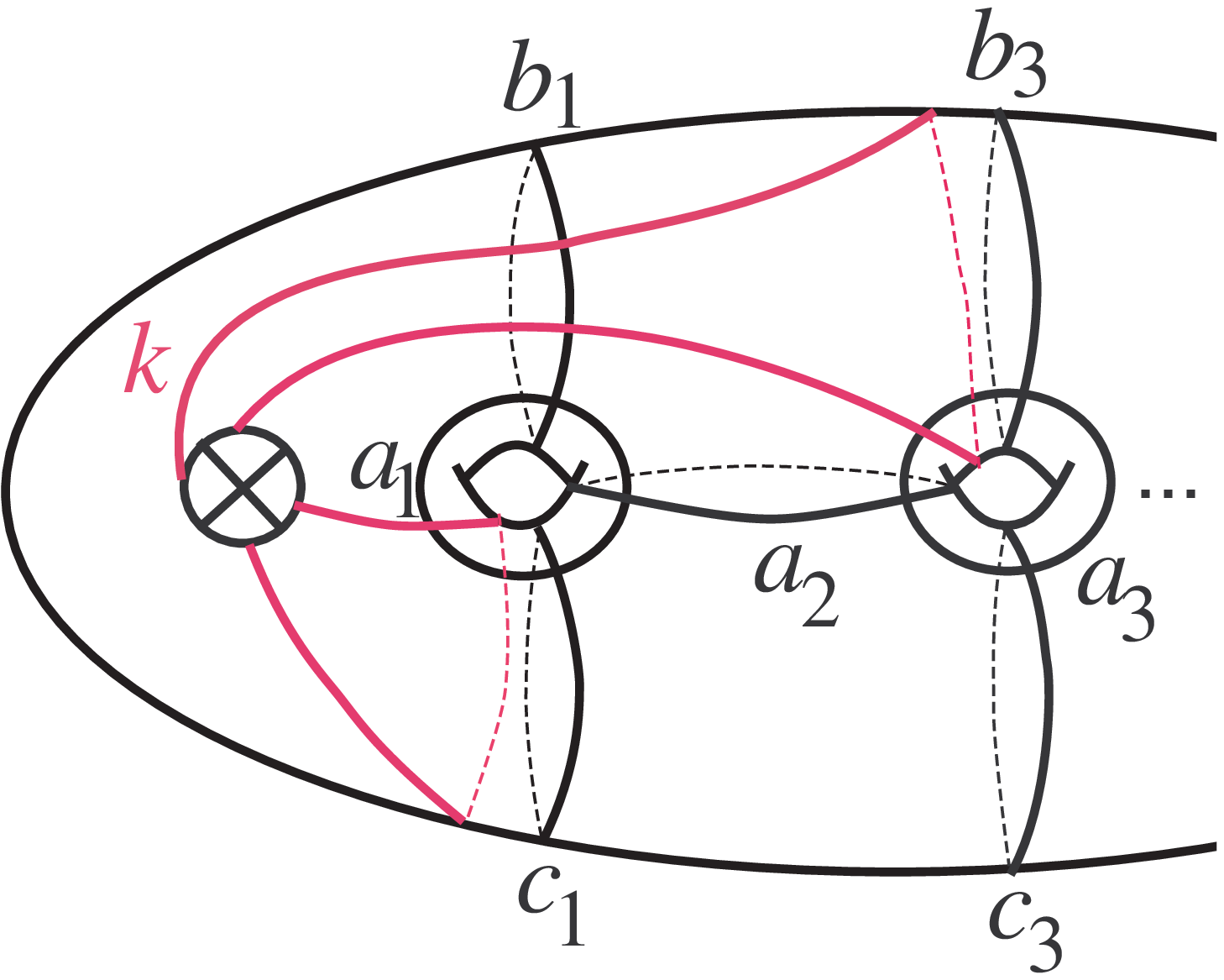}  \hspace{-1.2cm}

\hspace{-1.3cm} (iii) \hspace{5.5cm} (iv)
\caption{$\mathcal{C}$ when $g$ is odd}
\label{fig7-new}
\end{center}
\end{figure}

Let $\mathcal{C} = \{a_1, \cdots, a_{g-1}, b_1, b_3, \cdots,$ $b_{g-2}, c_1,$ $ c_3, \cdots, c_{g-2}, d_1, \cdots, d_{n-1}, e, k, l \}$
be as shown in Figure \ref{fig7-new} (i) - (iv) when $g \geq 5$ and $g$ is odd.

\begin{figure}[htb]
\begin{center}
\epsfxsize=3.2in \epsfbox{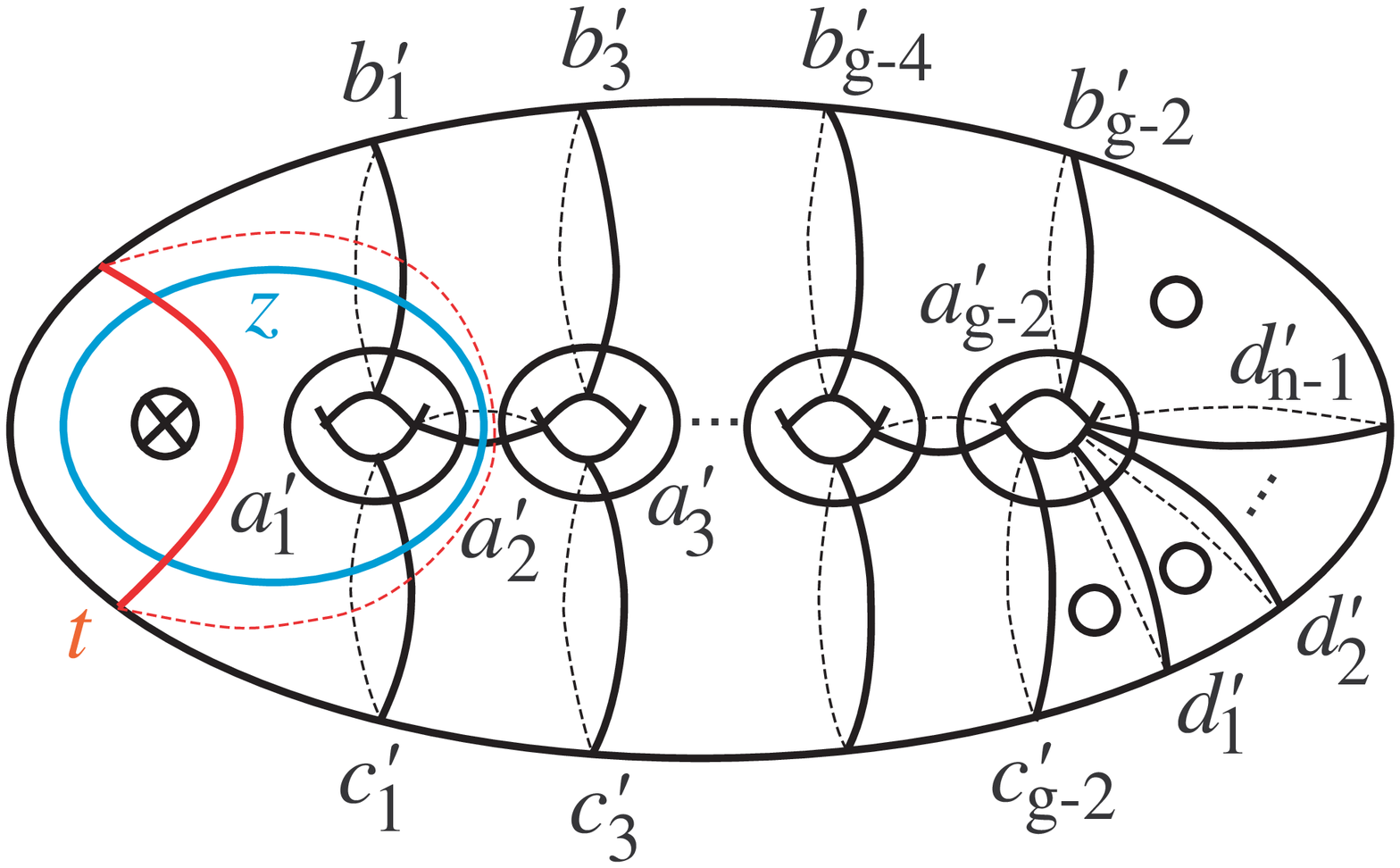} \hspace{-1cm} \epsfxsize=2.95in \epsfbox{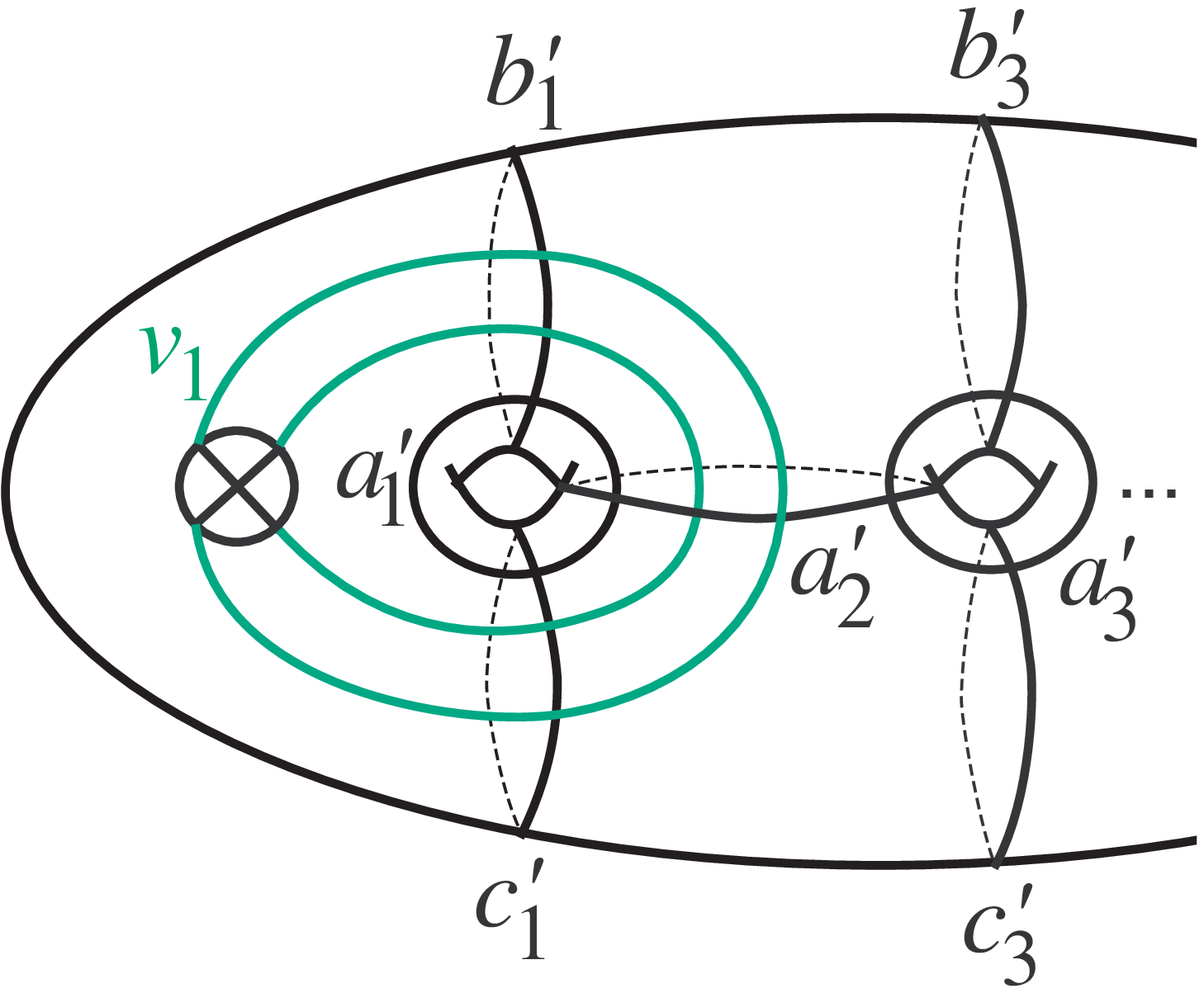}

\hspace{-1.3cm} (i) \hspace{5.9cm} (ii)

\hspace{0.1cm} \epsfxsize=2.95in \epsfbox{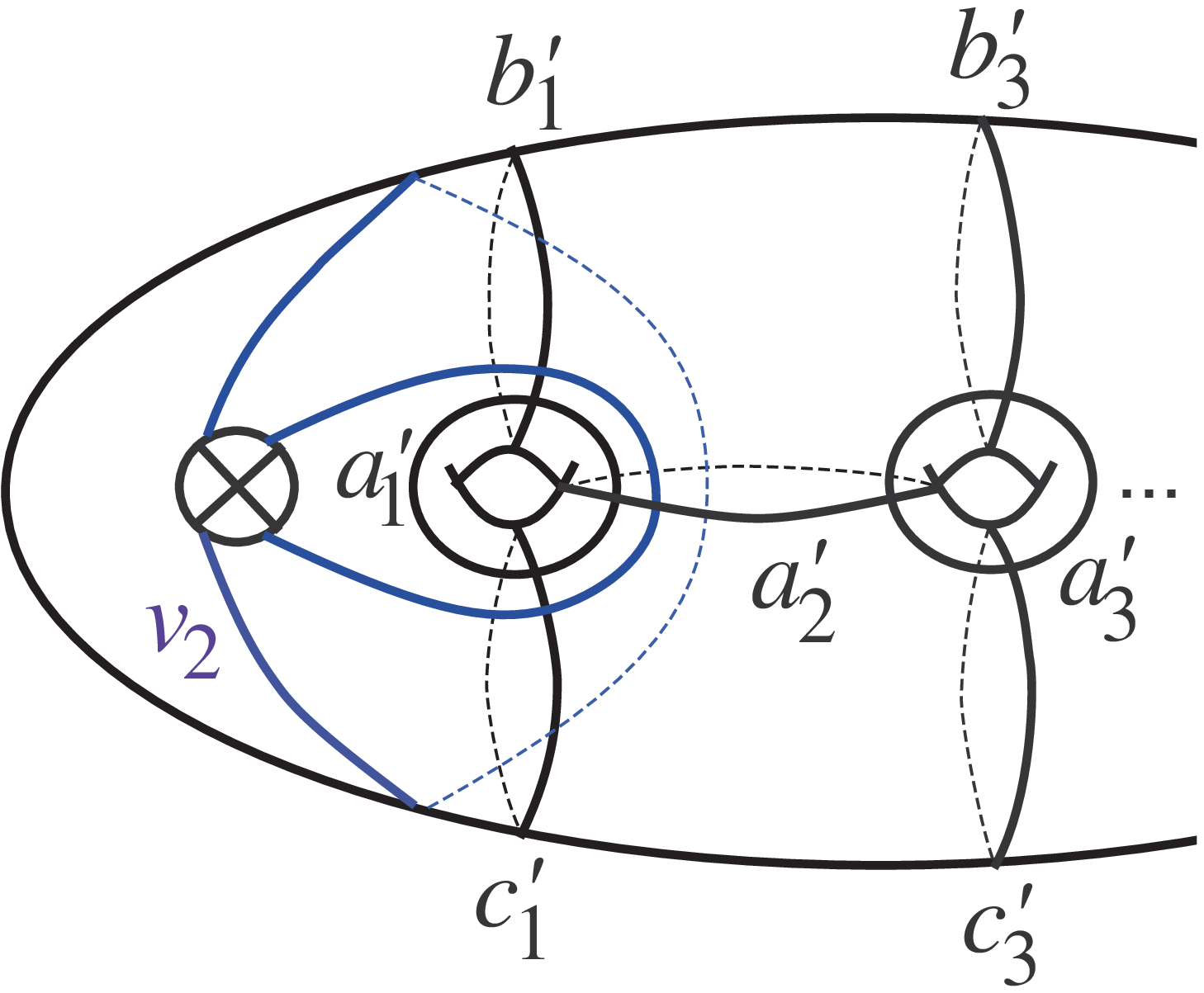} \hspace{-0.6cm} \epsfxsize=2.95in \epsfbox{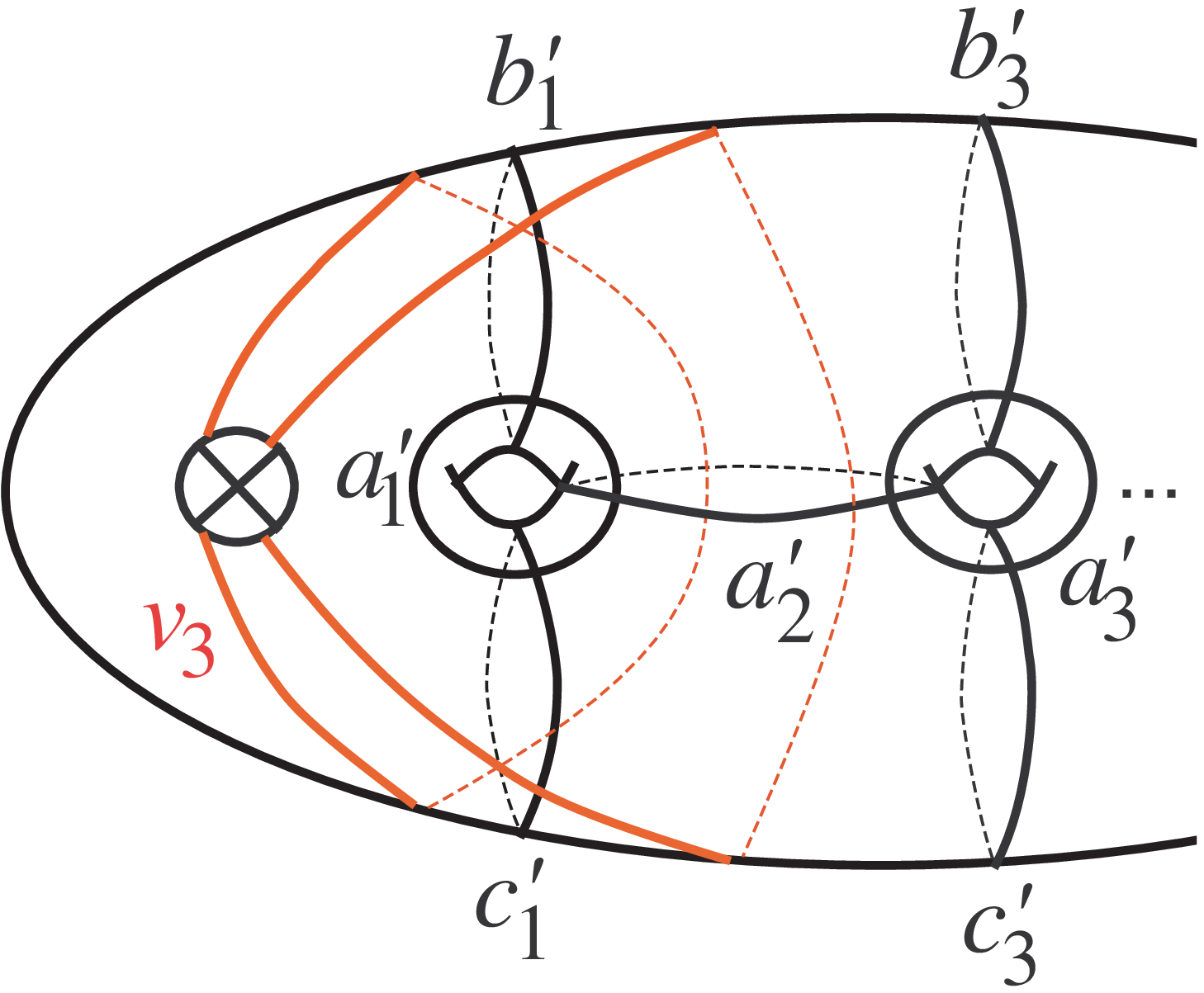}

\hspace{-1.3cm} (iii) \hspace{5.7cm} (iv)

\caption{Curve configuration XI}
\label{fig7-new-a}
\end{center}
\end{figure}

\begin{lemma}
\label{curves-2} Suppose $g \geq 5$ and $g$ is odd. There exists a homeomorphism $h: N \rightarrow N$ 
such that $h([x]) = \lambda([x])$ $\forall \ x \in \mathcal{C}$.
\end{lemma}

\begin{proof} We will give the proof when $N$ has boundary. The proof for the closed case will be similar. We will consider all the
curves in  $\mathcal{C}$ as shown in Figure \ref{fig7-new} (ii). Let
$\mathcal{B} = \{a_1, a_2, \cdots, a_{g-2}, b_1, b_3, \cdots, b_{g-2}, c_1, c_3, \cdots,$ $ c_{g-2}, d_1, d_2, \cdots, d_{n-1}\}$.
Let $a'_1  \in \lambda([a_1]), a'_2  \in \lambda([a_2]), \cdots,$ $ a_{g-2}' \in \lambda([a_{g-2}]), b'_1  \in \lambda([b_1]),$
$b'_3  \in \lambda([b_3]),$
$ \cdots, b_{g-2}' \in \lambda([b_{g-2}]), c'_1  \in \lambda([c_1]), c'_3 \in \lambda([c_3]), \cdots, c_{g-2}' \in \lambda([c_{g-2}]),$
$d'_1 \in \lambda([d_1]), d'_2  \in \lambda([d_2]), \cdots, $ $d_{n-1}' \in \lambda([d_{n-1}])$ be minimally intersecting representatives.
By Lemma \ref{intone} geometric intersection one is preserved. So, a regular neighborhood of union of all the elements in
$\mathcal{B'}= \{a'_1, a'_2, \cdots,$ $ a_{g-2}', b'_1, b'_3, \cdots,$ $ b_{g-2}', c'_1, c'_3, \cdots, c_{g-2}', d'_1, d'_2, \cdots,$
$d'_{n-1}\}$ is an orientable surface of genus $\frac{g-1}{2}$ with several boundary components.

By Lemma \ref{piece1}, if three nonseparating curves in $\mathcal{B}$ bound a pair of pants on $N$, then the corresponding curves in
$\mathcal{B'}$ bound a pair of pants on $N$. By Lemma \ref{piece2-bb}, if two curves in $\mathcal{B}$ are boundary components of
a pair of pants where the third boundary component of the pair of pants is a boundary component of $N$, then the corresponding curves
in $\mathcal{B'}$ are boundary components of a pair of pants where the third boundary component of the pair of pants is a boundary
component of $N$. These imply that the curves in $\mathcal{B'}$ are as shown in Figure \ref{fig7-new-a} (i).

Since $l$ intersects each of $b_1, c_1, a_2$ only once nontrivially and $l$ is disjoint from each of $a_1, b_3, c_3, a_3$, we know
that $l'$ intersects each of $b'_1, c'_1, a'_2$ only once nontrivially and $l'$ is disjoint from each of $a'_1, b'_3, c'_3, a'_3$.
Since $a_1$ and $l$ are nonisotopic and $\lambda$ is injective, we know that $a'_1$ and $l'$ are nonisotopic. 
Using all this information about $l'$, it is easy to see that $l'$ is isotopic to $t$ or $z$ where the
curves $t, z$ are as shown in Figure \ref{fig7-new-a} (i). There exists a homeomorphism $\phi$ of order two sending each curve in $\mathcal{B'}$
to itself and switching $t$ and $z$. Let $\phi_{*}$ be the induced map on $\mathcal{T(N)}$. By replacing $\lambda$ with
$\lambda \circ \phi_{*}$ if necessary, we can assume that $l'$ is isotopic to $z$. We note that to get the proof of the lemma, it is
enough to prove the result for this $\lambda$.

Let $e' \in \lambda([e])$ such that $e'$ intersects minimally with the elements of $\mathcal{B'}$. The curve $e$ bounds a Klein bottle with
one hole whose complement is nonorientable on $N$. There exists a homeomorphism sending the pair $(e, b_1)$ to $(p, x)$ where $p$ and $x$ are as
shown in Figure \ref{fig1b-n9-c}. Since the geometric intersection number of $[p']$ and $[x']$ is two by Lemma \ref{int-twice-2-b}, we see
that the geometric intersection number of $[e']$ and $[b'_1]$ is two. There exists also a homeomorphism sending the pair $(e, c_1)$ to $(p, x)$ 
where $p$ and $x$ are as shown in Figure \ref{fig1b-n9-c}. With similar reasoning we can see that the geometric intersection number of $[e']$ 
with $[c'_1]$ is two. The curve $e$ is disjoint from each of $a_1, b_3, a_3, c_3$, so $e'$ is disjoint from each of $a'_1, b'_3, a'_3, c'_3$. Since $e'$
intersects each of $b'_1$ and $c'_1$ twice essentially and $e'$ is disjoint from each of $a'_1, b'_3, a'_3, c'_3$, we see that $e'$
is isotopic to $v_1, v_2$ or $v_3$ as shown in Figure \ref{fig7-new-a} (ii), (iii) and (iv). Since $v_1$ and $v_3$ both bound M\"{o}bius bands,
$e'$ is not isotopic to either of them. So, $e'$ is isotopic to $v_2$.

\begin{figure}[htb]
\begin{center}\hspace{-.3cm}
\epsfxsize=2.95in \epsfbox{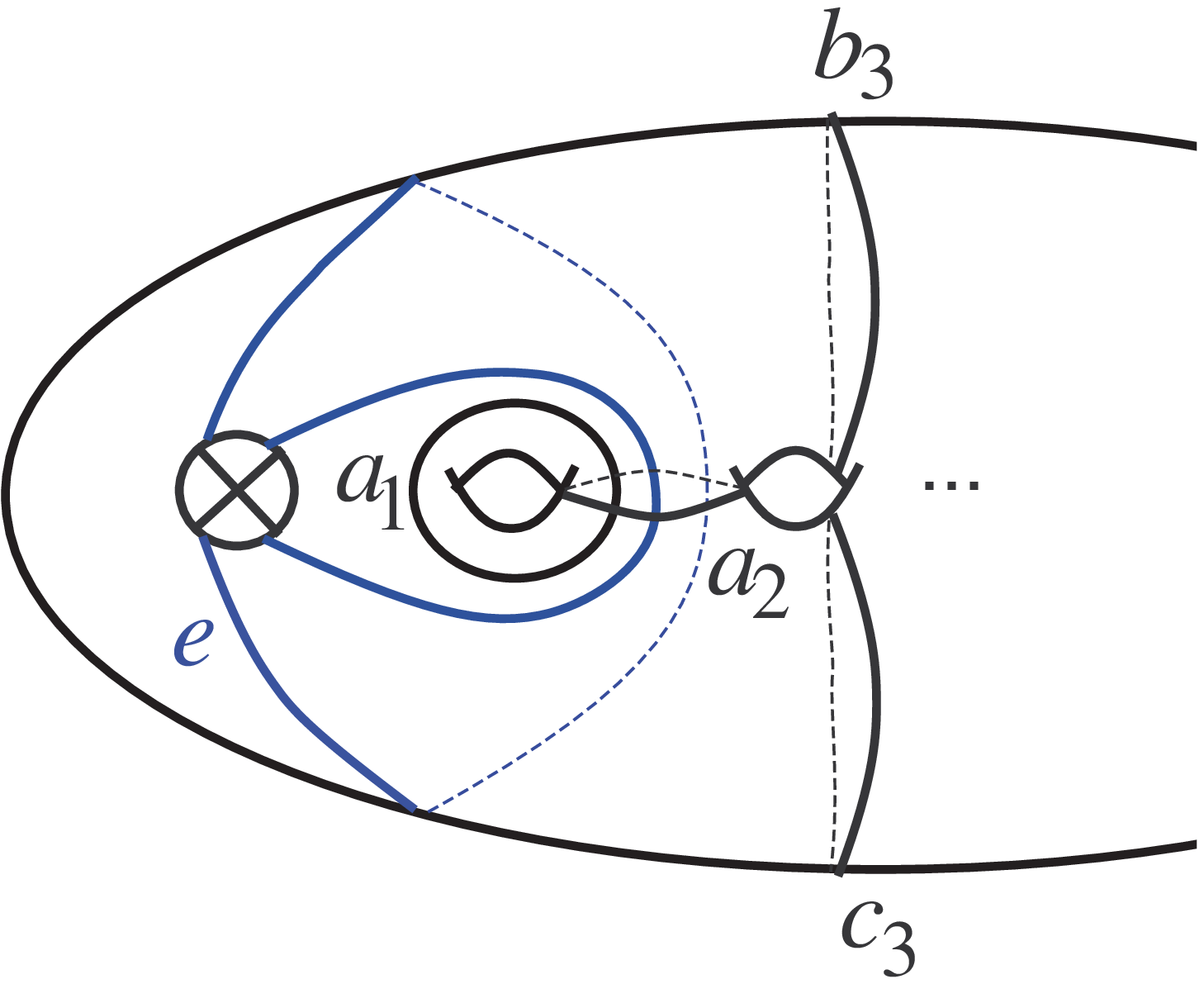}  \hspace{-1cm} \epsfxsize=2.95in \epsfbox{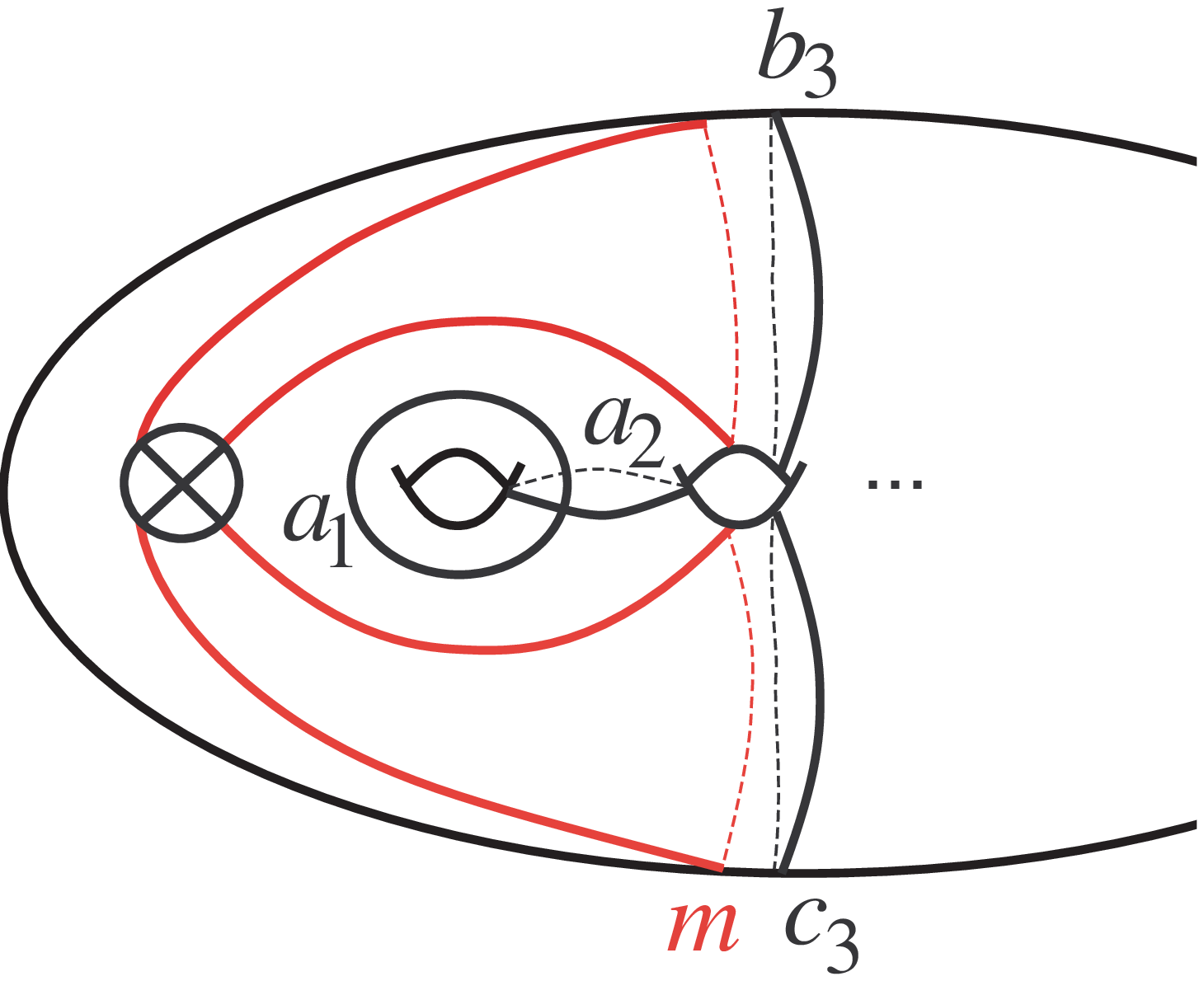}

\hspace{-1.3cm} (i) \hspace{6.5cm} (ii)

\epsfxsize=2.95in \epsfbox{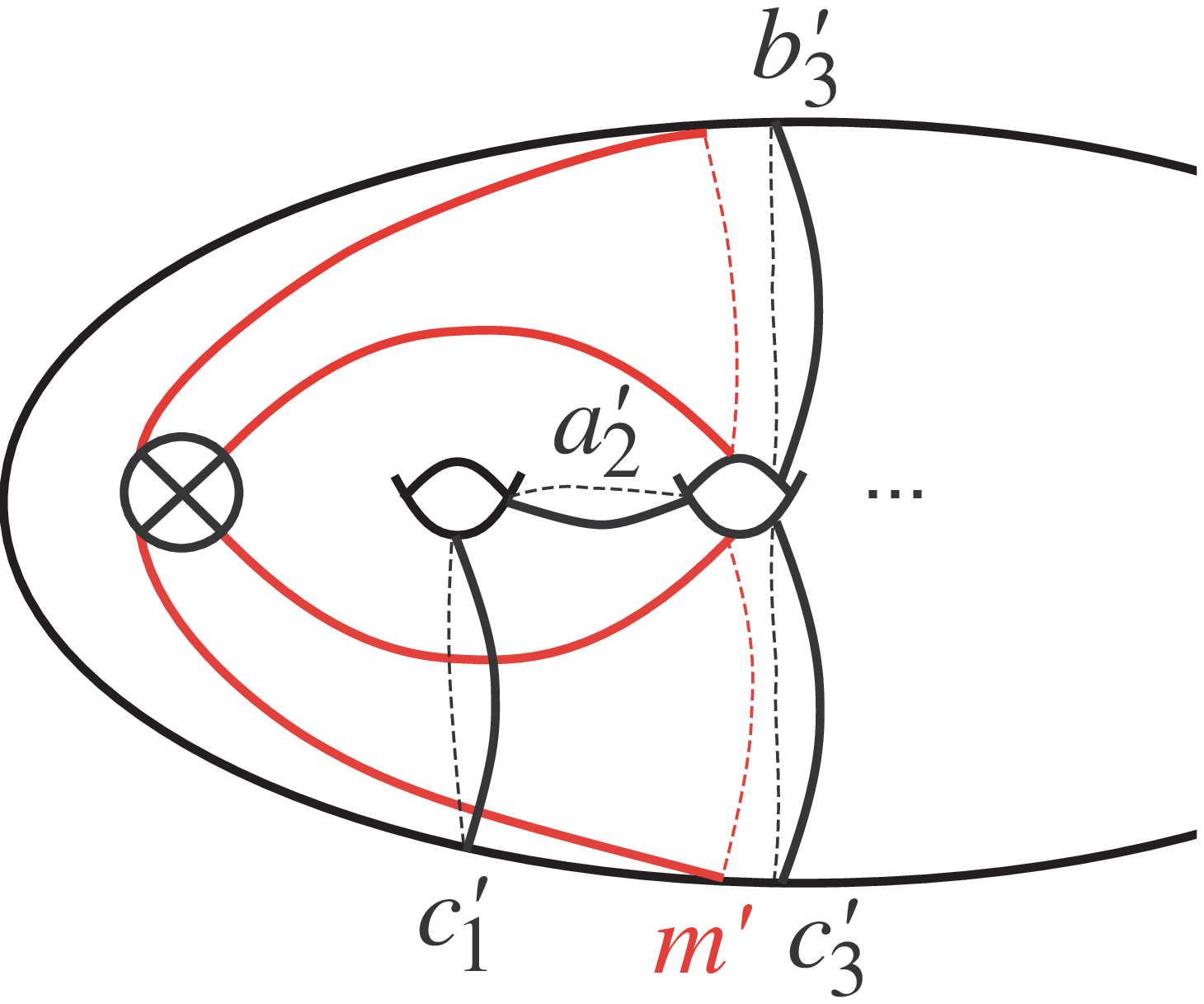}  \hspace{-1cm} \epsfxsize=2.95in \epsfbox{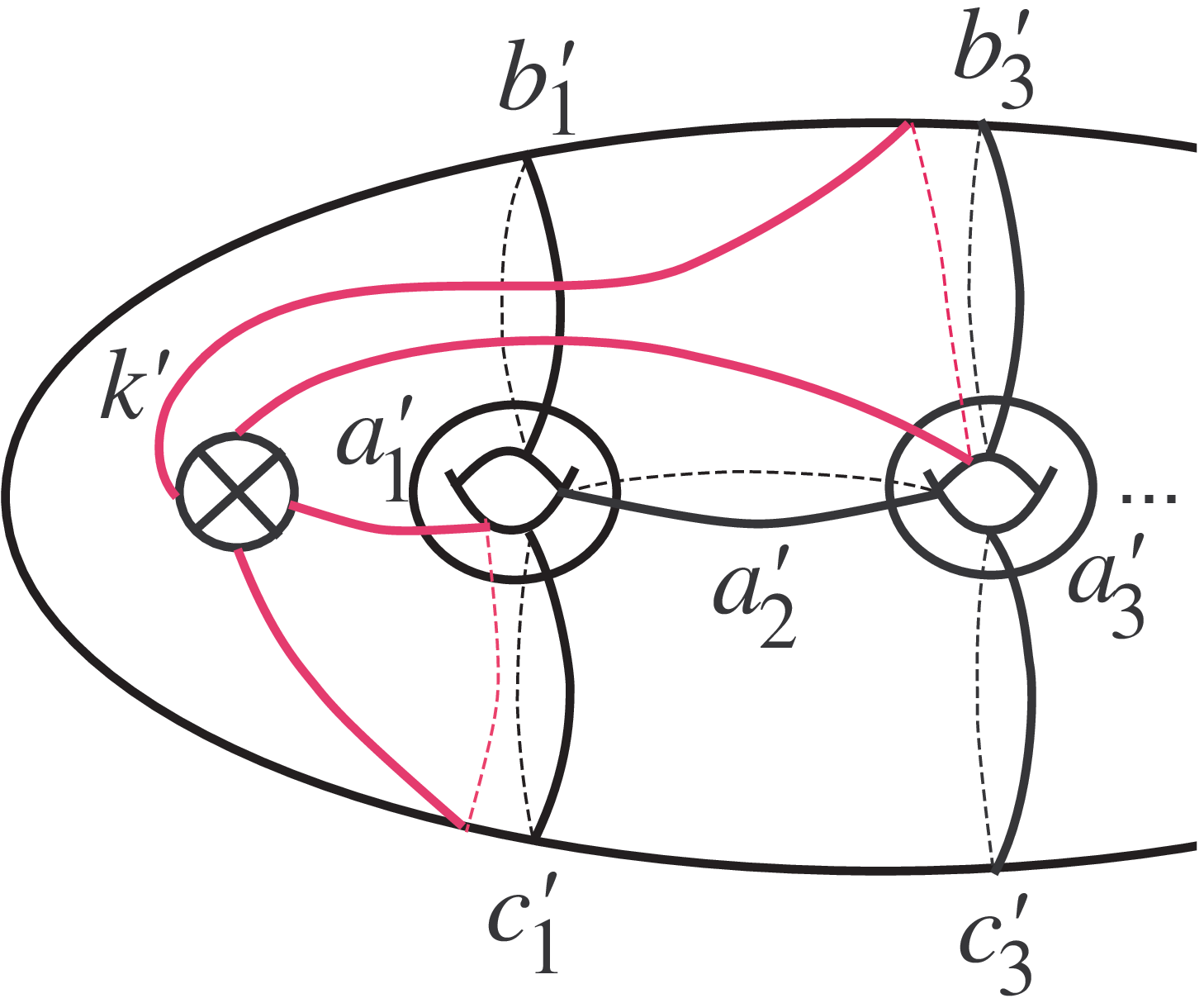}

\hspace{-1.2cm} (iii) \hspace{6.2cm} (iv)
\caption{Curve configuration XII} \label{fig1c-pr}
\end{center}
\end{figure}

Let $k' \in \lambda([k])$ such that $k'$ intersects minimally with the elements of $\mathcal{B'} \cup \{e'\}$. To see
that $k'$ is as shown in Figure \ref{fig1c-pr} (iv), we will first consider the curve $m$ given in Figure \ref{fig1c-pr} (ii).
The curve $m$ is the unique nontrivial curve up to isotopy that is disjoint from all the curves in $\{a_1, a_2, b_3, c_3, e\}$
and intersects $b_1$ nontrivially. Let $m' \in \lambda([m])$ such that $m'$ intersects minimally with the elements of $\mathcal{B'} \cup \{e', k'\}$.
Since $m'$ is disjoint from all the curves in $\{a'_1, a'_2, b'_3, c'_3, e'\}$ and intersects $b'_1$ nontrivially, and there is a unique such
curve up to isotopy we see that $m'$ is as shown in Figure \ref{fig1c-pr} (iii). The curve $k$ is the unique nontrivial curve up to isotopy that is
disjoint from all the curves in $\{c_1, a_2, b_3, c_3, m\}$ and intersects $b_1$ nontrivially. Since $k'$ is
disjoint from all the curves in $\{c'_1, a'_2, b'_3, c'_3, m'\}$ and intersects $b'_1$ nontrivially, we see that
$k'$ is as shown in Figure \ref{fig1c-pr} (iv). Hence, there is a homeomorphism $h: N \rightarrow N$ such that $h([x]) = \lambda([x])$
for all $x \in \mathcal{C}$.\end{proof}

{\section{Enlarging $\mathcal{C}$}

From now on, we let $h$ be a homeomorphism $h: N \rightarrow N$ which comes from both Lemma \ref{curves} and Lemma \ref{curves-2} such that $h([x]) = \lambda([x])$ for each $x \in \mathcal{C}$. Our aim is to show that $\lambda$ is induced by
$h$. We will enlarge the set $\mathcal{C}$ to other sets to get our main result. The idea is similar to what
is given by Aramayona and Leininger in \cite{AL2}.

\begin{figure}[ht]
\begin{center}
\hspace{-0.3cm} \epsfxsize=3.2in \epsfbox{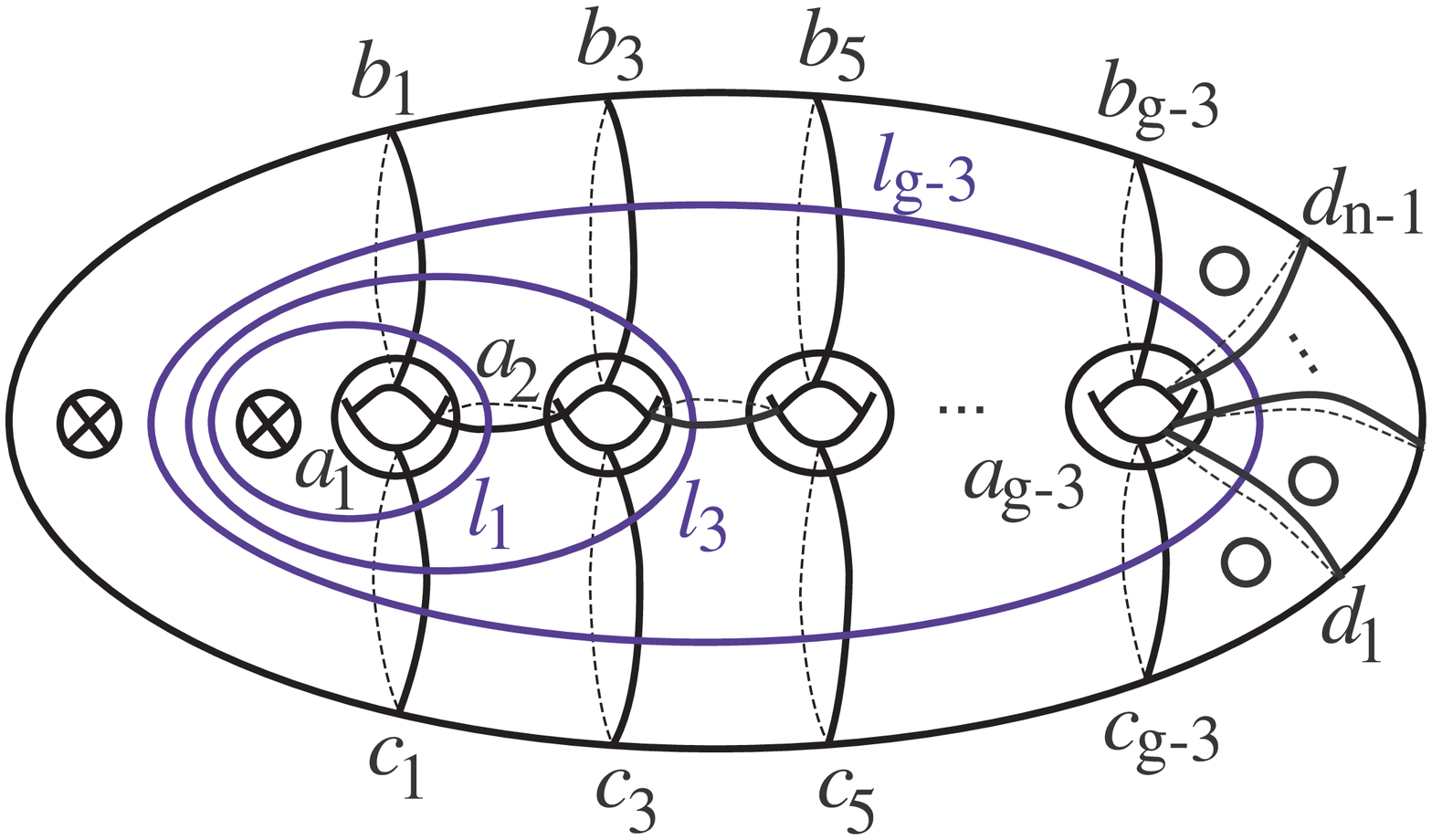}  \hspace{-1cm} \epsfxsize=3.2in \epsfbox{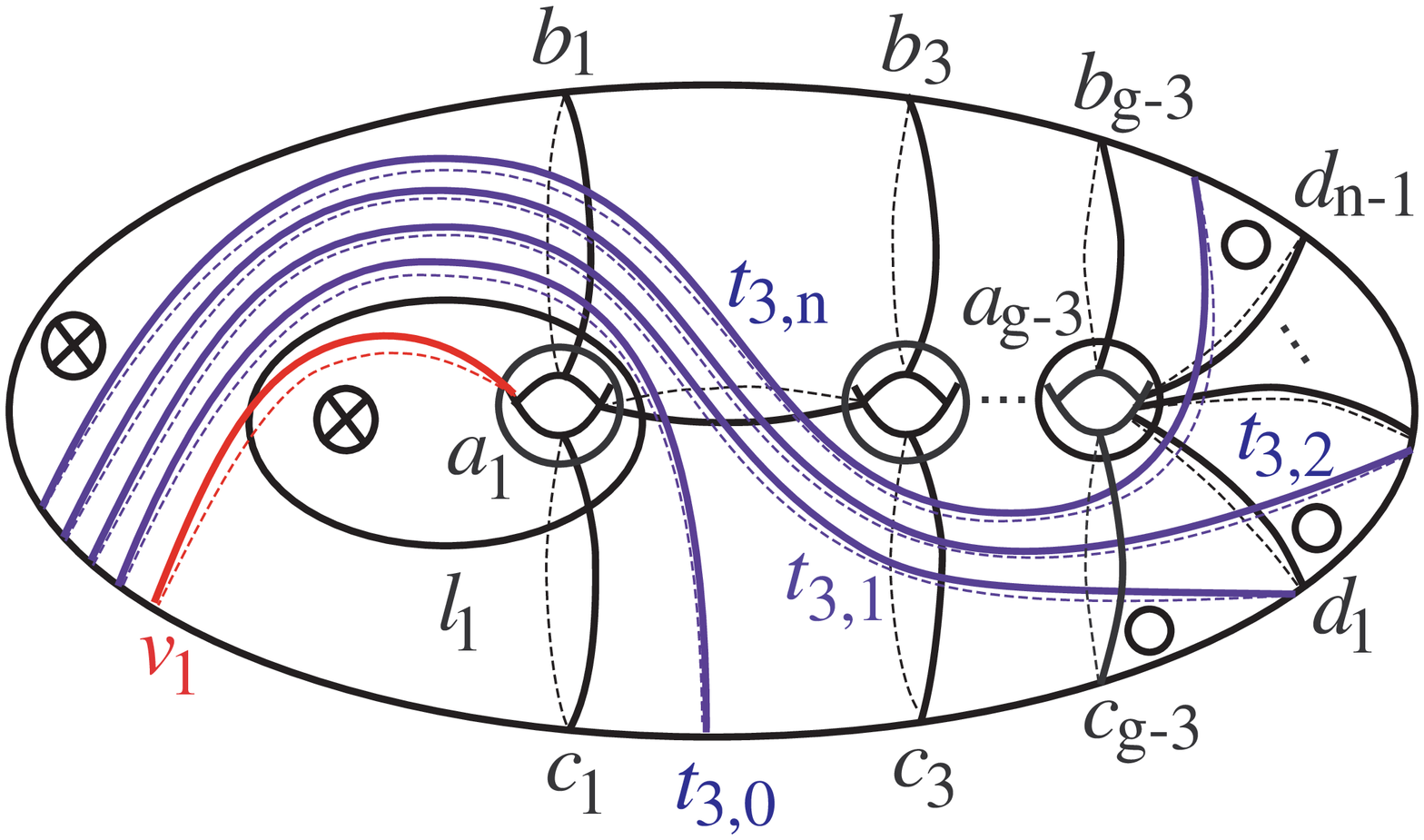}

\hspace{-1.3cm} (i) \hspace{6.5cm} (ii)

\hspace{-0.3cm} \epsfxsize=3.2in \epsfbox{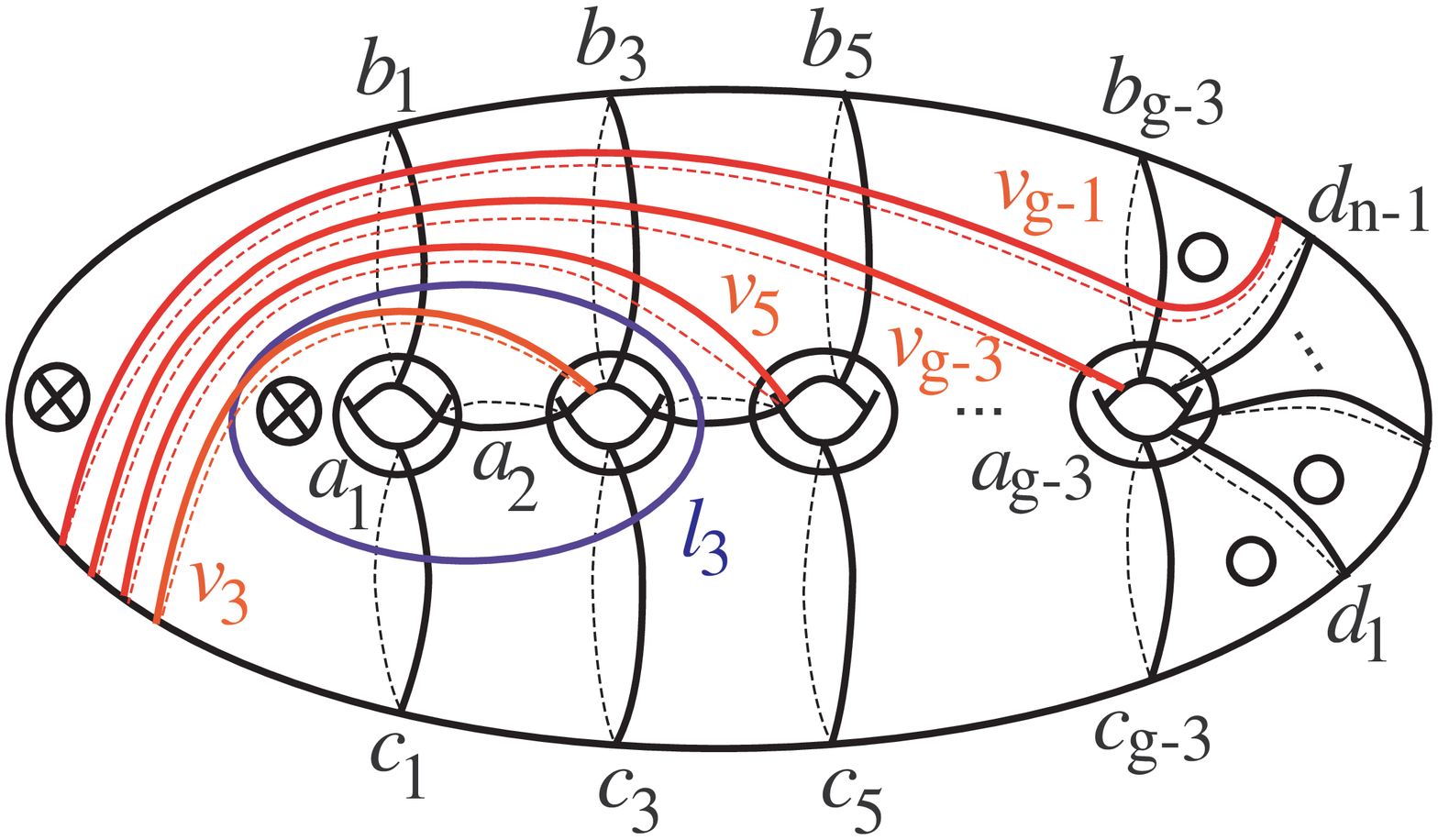}  \hspace{-1cm} \epsfxsize=3.2in \epsfbox{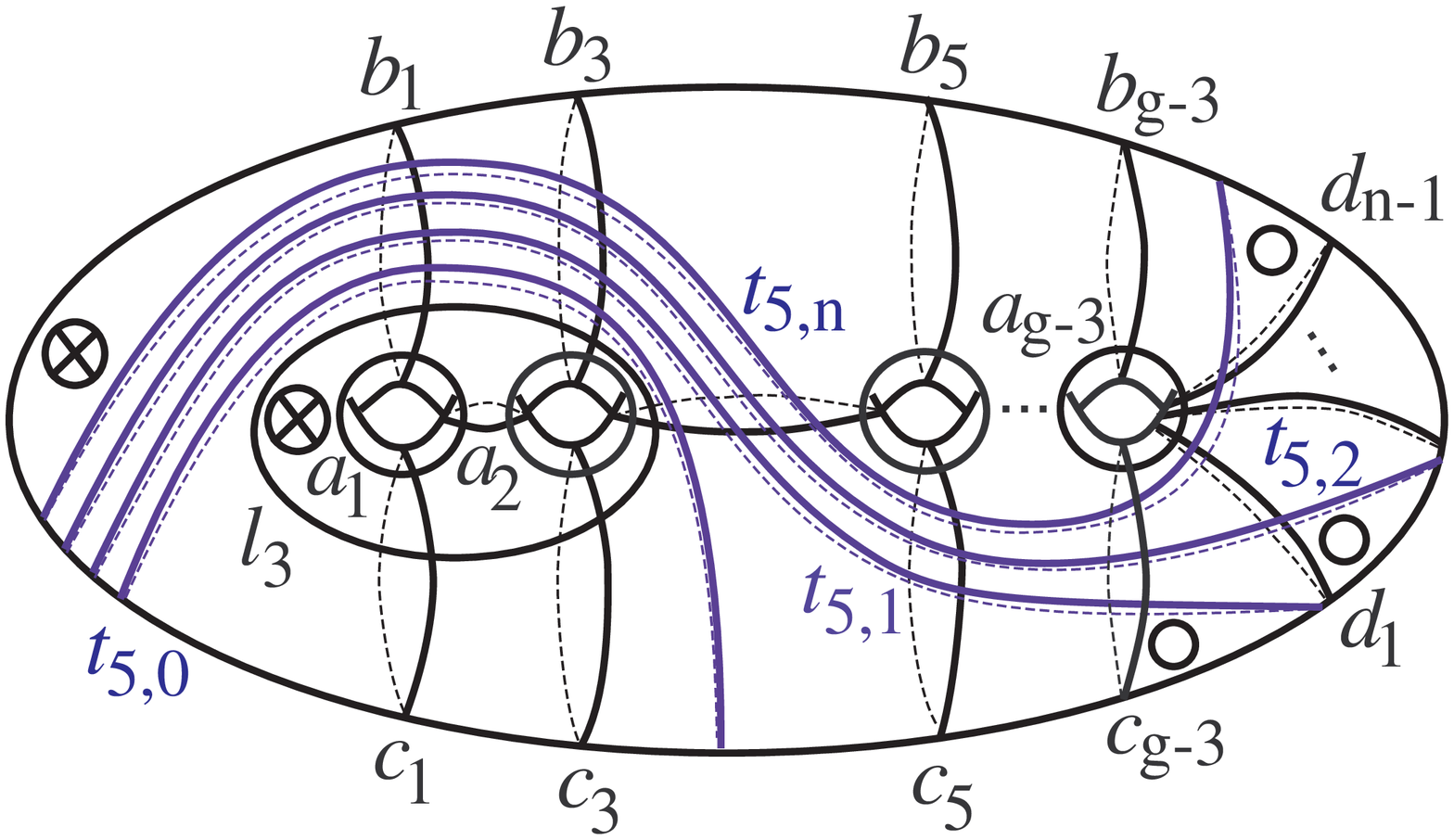}

\hspace{-1.2cm} (iii) \hspace{6.2cm} (iv)

\caption{$\mathcal{B}_0$, (even genus) has separating curves that have odd genus nonorientable surfaces on both sides.}
\label{fig25}
\end{center}
\end{figure}

For $g \geq 6$ and $g$ even, we let $\mathcal{B}_0 = \{l_1, l_3, \cdots, l_{g-3}, t_{3,0}, t_{3,1}, \cdots, t_{3,n}, t_{5,0}, t_{5,1}, \cdots,$
$t_{5,n}, \cdots, $ $ t_{g-1,0}, t_{g-1,1},$ $ \cdots, t_{g-1,n-1}\}$ where the curves are as shown in Figure \ref{fig25}.
We note that each $t_{i,j}$ is a separating curve that has odd genus nonorientable surfaces on both sides, and $\mathcal{B}_0$ has curves
of every topological type that satisfy this condition.

\begin{lemma}
\label{B_0-even} If $g \geq 6$ and $g$ is even, then $h([x]) = \lambda([x])$ $\forall \ x \in \mathcal{C} \cup \mathcal{B}_0$.\end{lemma}

\begin{proof} We will give the proof when there is boundary. The closed case is similar. By Lemma \ref{curves} we know that
$h([x]) = \lambda([x])$ $\forall \ x \in \mathcal{C}$. Consider the curves in Figure \ref{fig25}. We know the result for $l_1$
since $l_1=l$ and $l \in \mathcal{C}$. The curve $l_3$ is the unique nontrivial curve up to isotopy which intersects each curve in
$\{a_4, b_1, b_3, c_1, c_3\}$ only once, disjoint from $a_5, b_5, c_5$  and bounds a pair of pants with $a_3$ and $l_1$.
Since all these properties are preserved by $\lambda$ by Lemma \ref{intone} and Lemma \ref{piece1}, and $h([x]) = \lambda([x])$
for all these curves, we have $h([l_3]) = \lambda([l_3])$. With similar arguments we see that $h([l_i]) = \lambda([l_i])$ for all
$i= 1, 3, 5, \cdots, g-3$.

By following the proof of Lemma \ref{curves} it is easy to see that we have the result for $v_1$ and $v_3$ as $v_1 =r$ and $v_3 =v$
where $r, v$ are as shown in Figure \ref{fig1b-n}. Showing the result for each $v_i$, $i = 1, 3, 5, \cdots g-1$ is similar.
The curve $t_{3,0}$ is the unique nontrivial curve up to isotopy that is disjoint from all the curves in $\{a_1, c_1, l_1, v_1, v_3, a_3, c_3\}$ and
intersects $a_2$ nontrivially. Since we know that $h([x]) = \lambda([x])$ for all these curves, by using that $\lambda$ is
superinjective and injective, we get $h([t_{3,0}]) = \lambda([t_{3,0}])$. When we cut $N$ by all the curves in $\{t_{3,0}, v_3, a_3, a_4, a_5, 
\cdots, a_{g-3}, d_1\}$, we get a pair of pants $P$ and the curve 
$t_{3,1}$ is the unique nontrivial curve nonisotopic to $t_{3,0}$ in $P$. We also note that $t_{3,1}$ intersects $c_{g-3}$ nontrivially. Since we know that $h([x]) = \lambda([x])$ for all these curves, by using that $\lambda$ is superinjective and injective, we get $h([t_{3,1}]) = \lambda([t_{3,1}])$.
Similar arguments show $h([t_{3,i}]) = \lambda([t_{3,i}])$ for all $i= 2, 3, \cdots, n$.

The curve $t_{5,0}$ is the unique nontrivial curve up to isotopy that is disjoint from all the curves in $\{a_1, a_2, a_3, c_1, c_3, l_3, v_5, a_5, c_5\}$
and intersects $a_4$ nontrivially. Since we know that $h([x]) = \lambda([x])$ for all these curves, by using that $\lambda$ is
superinjective and injective, we get $h([t_{5,0}]) = \lambda([t_{5,0}])$.

When we cut $N$ by all the curves in $\{t_{5,0}, v_5, a_5, a_6, \cdots, a_{g-3}, d_1\}$, 
we get a pair of pants $Q$ and the curve 
$t_{5,1}$ is the unique nontrivial curve nonisotopic to $t_{5,0}$ in $Q$. We also see that $t_{5,1}$ intersects $c_{g-3}$ nontrivially. Since we know that $h([x]) = \lambda([x])$ for all these curves, by using that $\lambda$ is superinjective and injective, we get $h([t_{5,1}]) = \lambda([t_{5,1}])$. Similar arguments show $h([t_{5,i}]) = \lambda([t_{5,i}])$ for all $i= 2, 3, \cdots, n$.
We can see that $h([t_{i,j}]) = \lambda([t_{i,j}])$  for each $t_{i,j} \in \mathcal{B}_0$ with similar arguments.\end{proof}\\

\begin{figure}[ht]
\begin{center}
\hspace{-1.6cm} \epsfxsize=3.2in \epsfbox{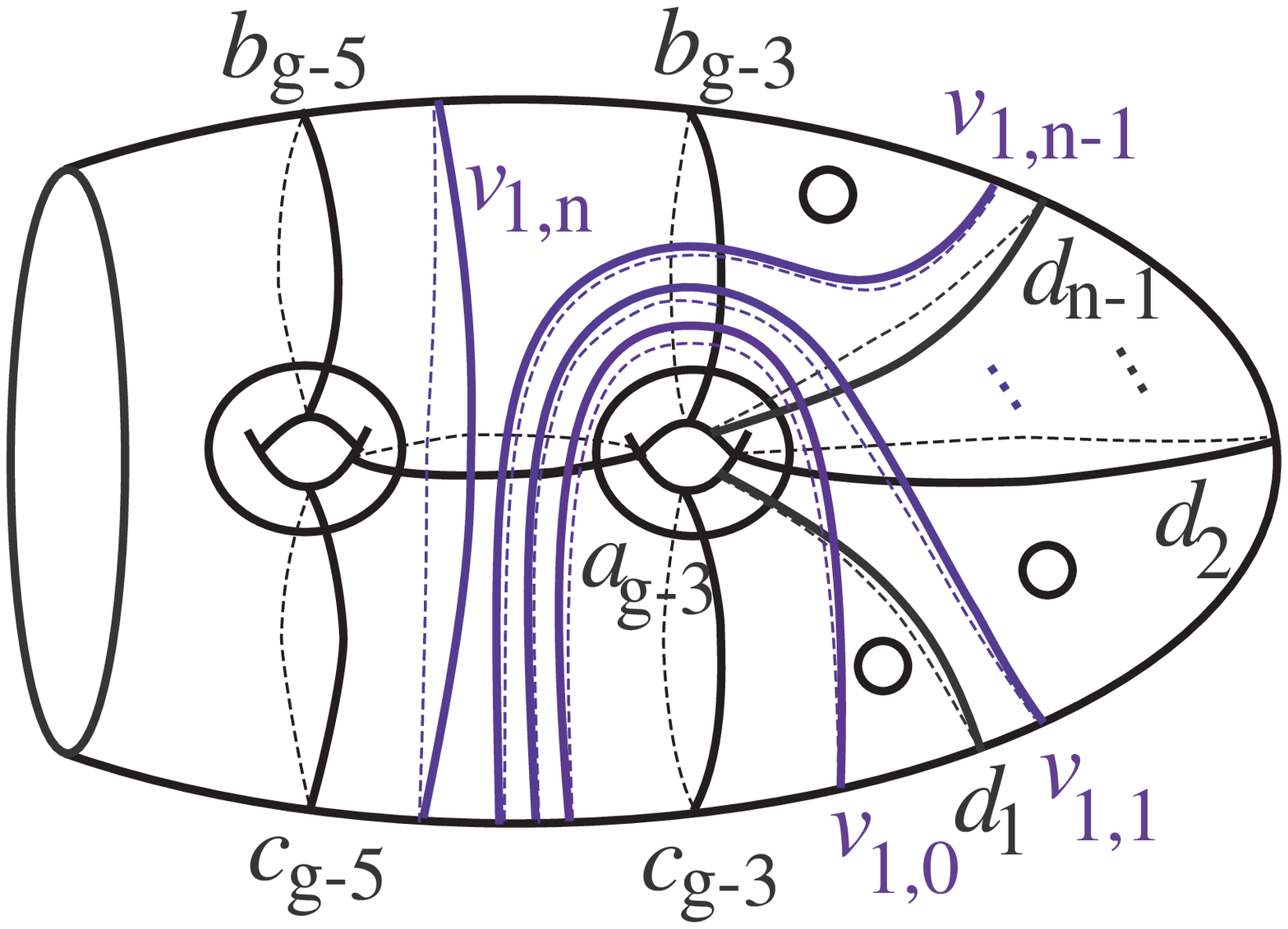}\hspace{-1cm} \epsfxsize=3.2in \epsfbox{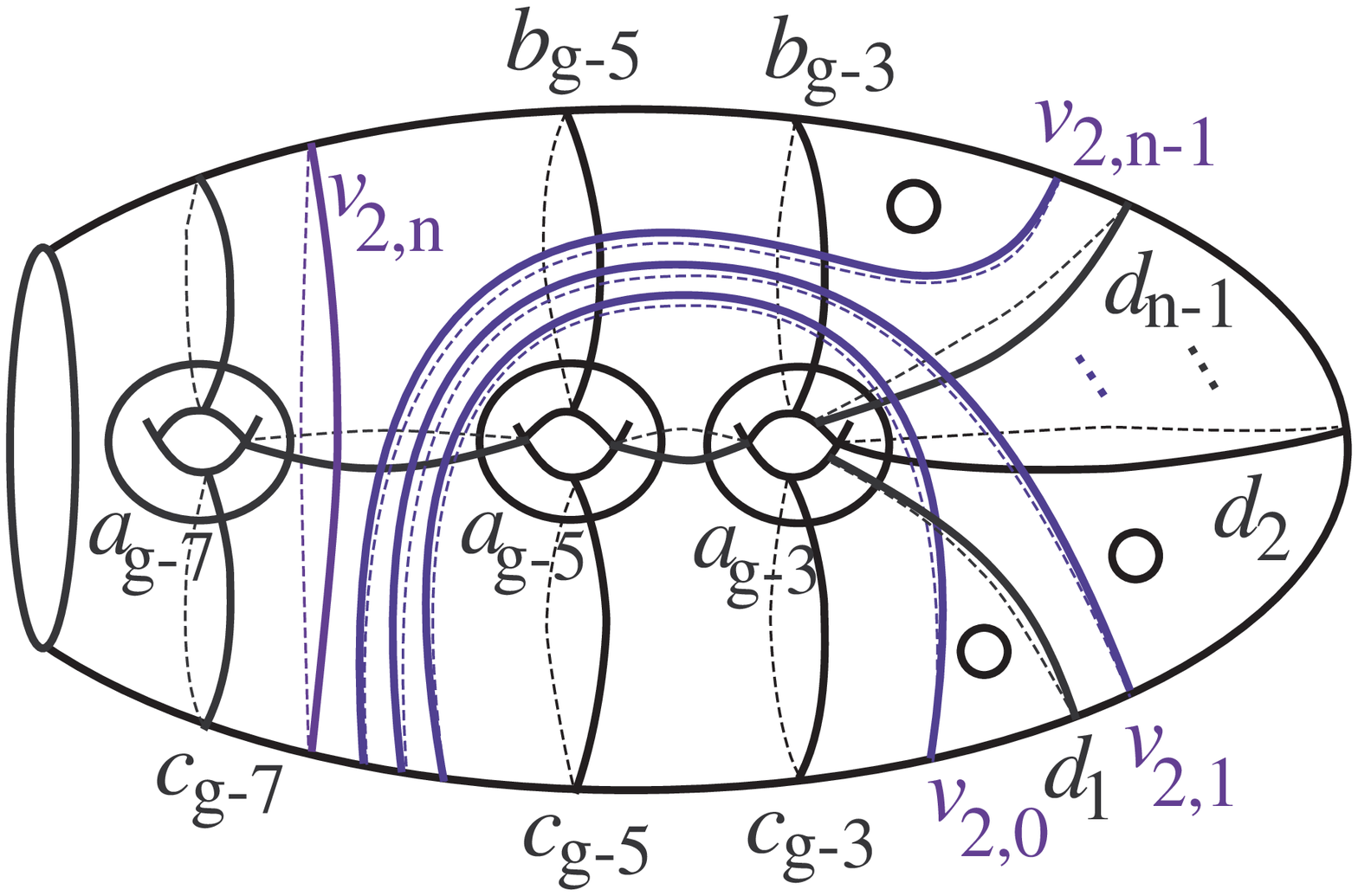}

\hspace{-1cm} (i) \hspace{6.7cm} (ii)

\epsfxsize=3.1in \epsfbox{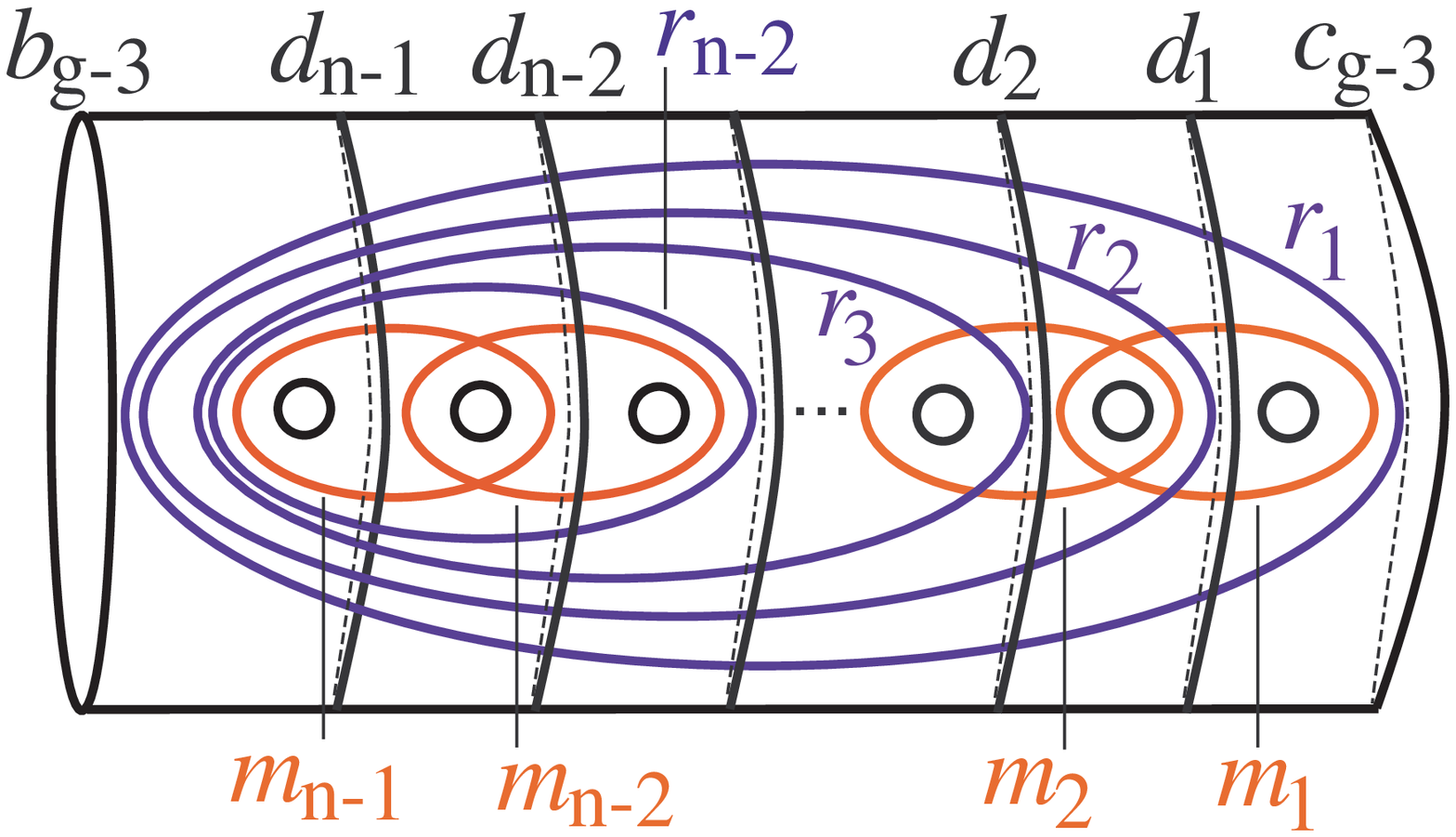} \hspace{-1cm} \epsfxsize=3.2in \epsfbox{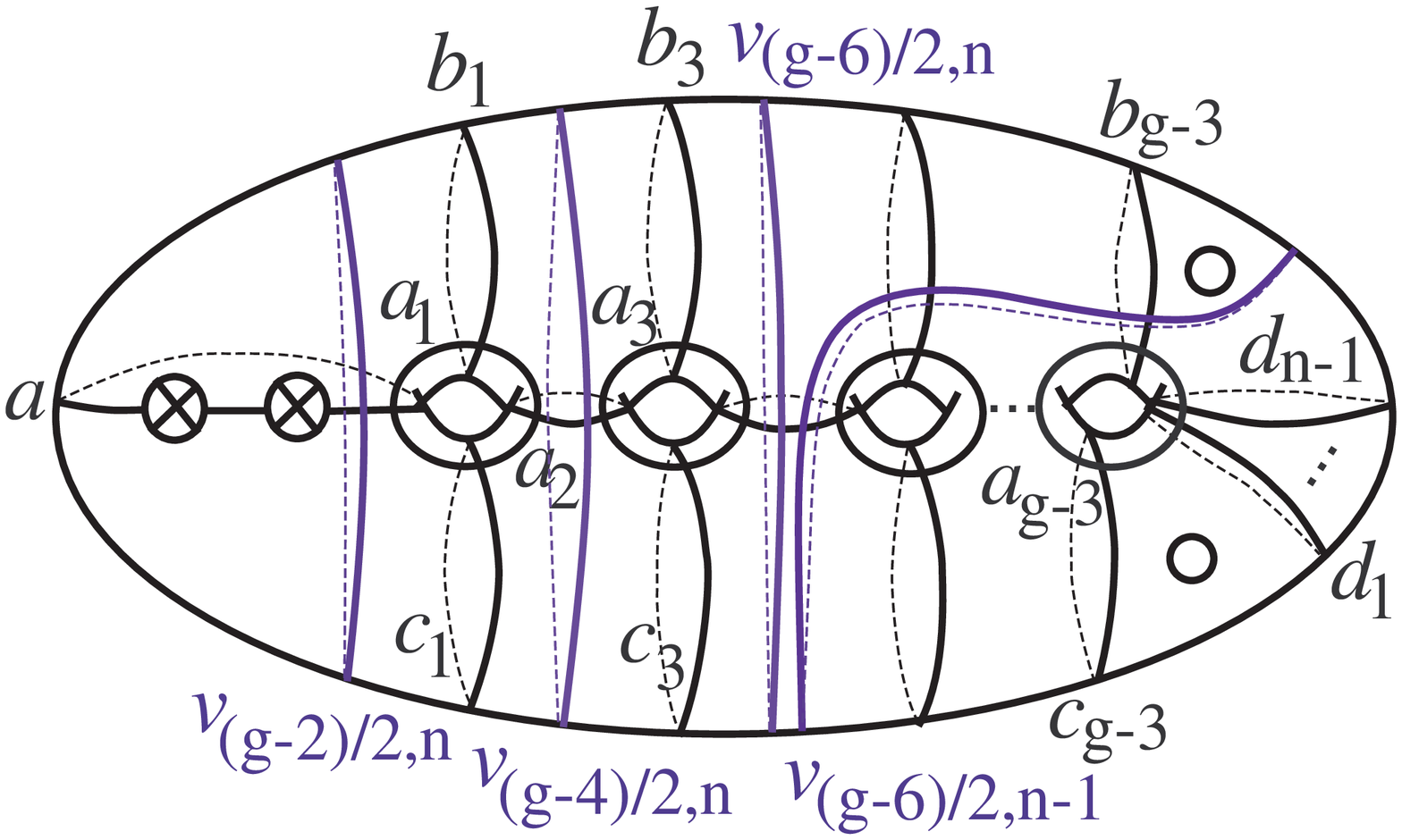}

\hspace{-1cm} (iii) \hspace{6.5cm} (iv)

\caption{$\mathcal{B}_1$ (even genus) } \label{fig23}
\end{center}
\end{figure}

\begin{figure}[ht]
\begin{center}
\hspace{-1cm} \epsfxsize=3.2in \epsfbox{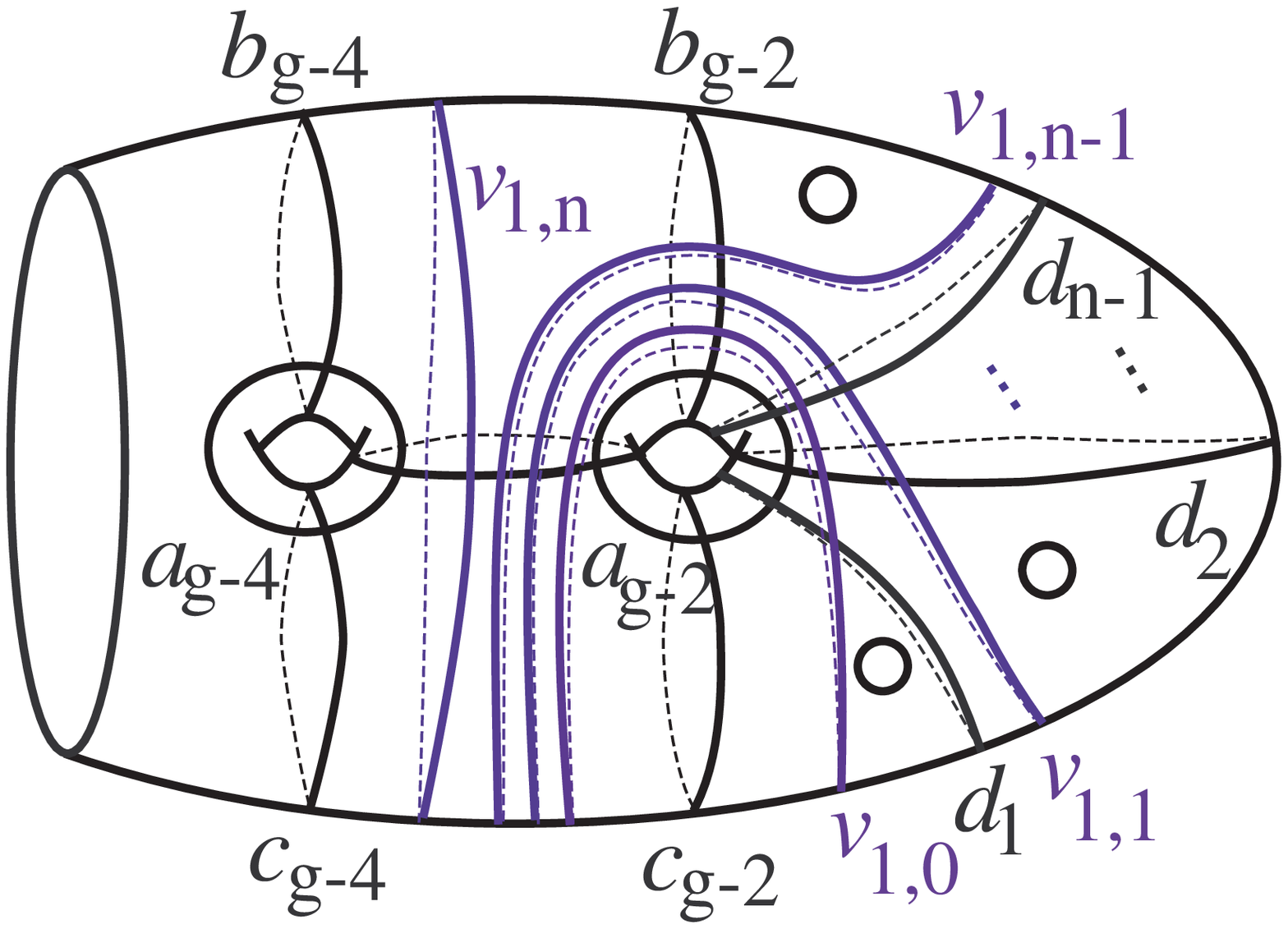} \hspace{-1cm} \epsfxsize=3.2in  \epsfbox{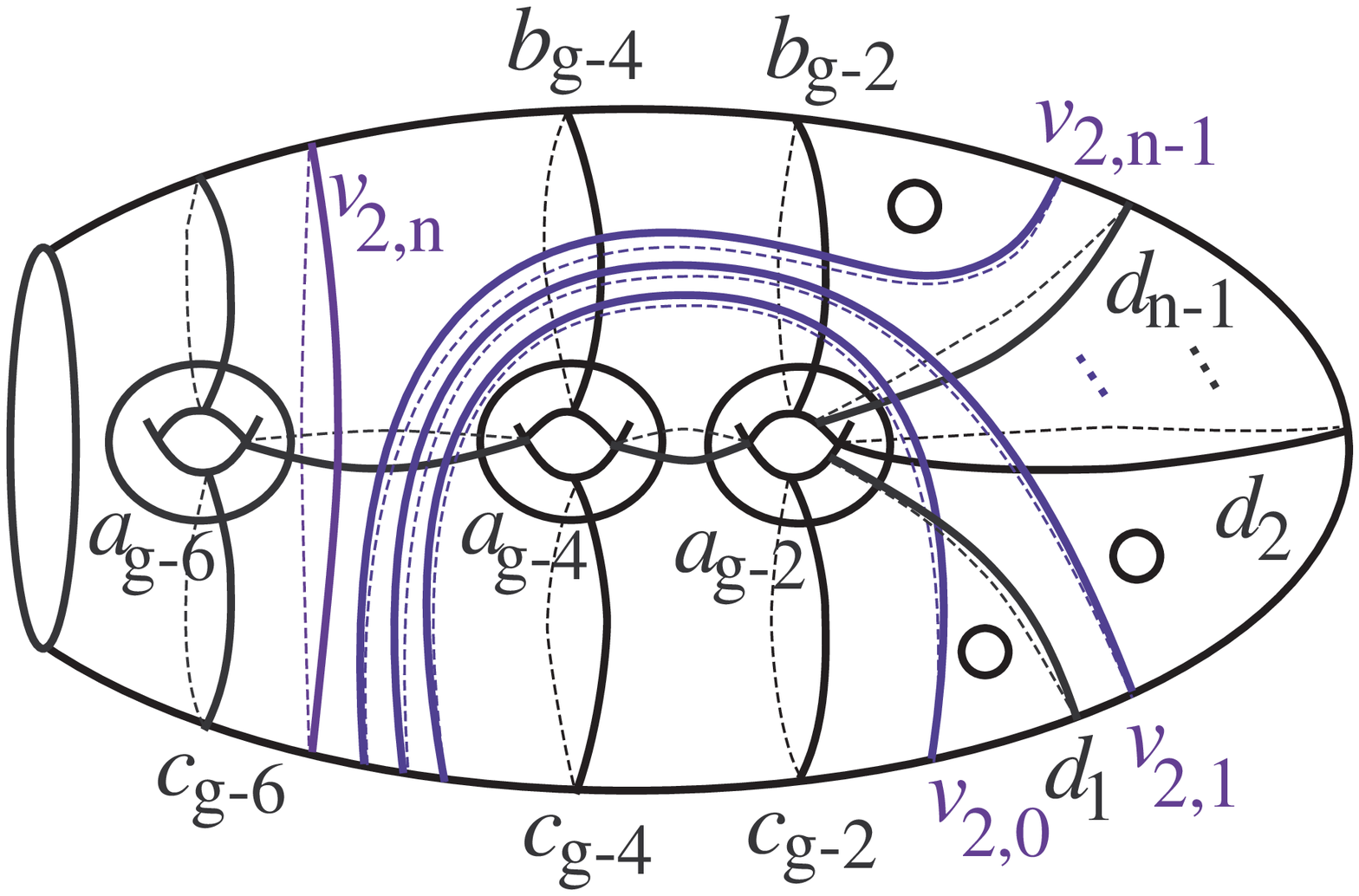}

\hspace{-1cm} (i) \hspace{6.7cm} (ii)

\epsfxsize=3.1in \epsfbox{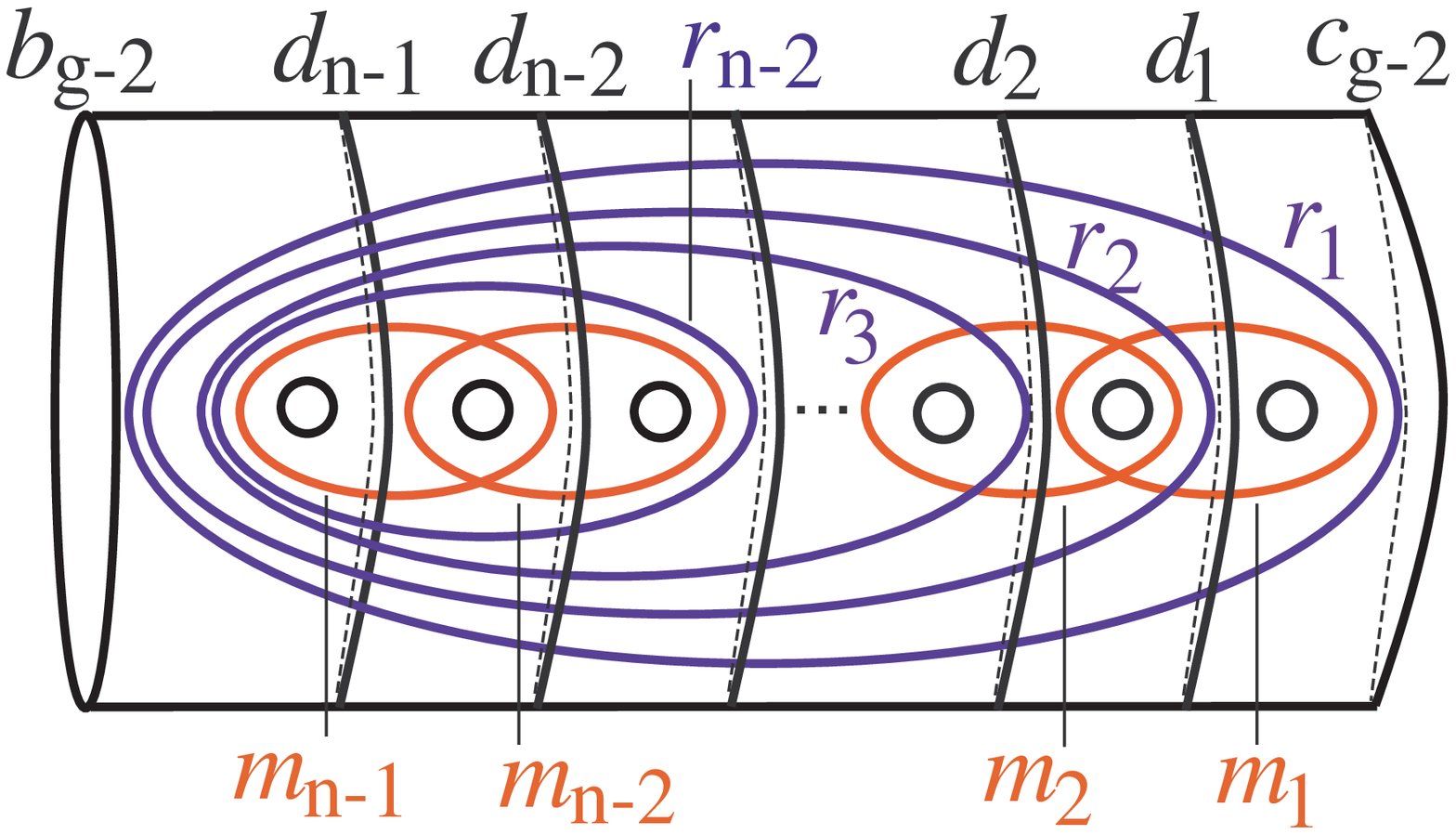} \hspace{-1cm} \epsfxsize=3.2in \epsfbox{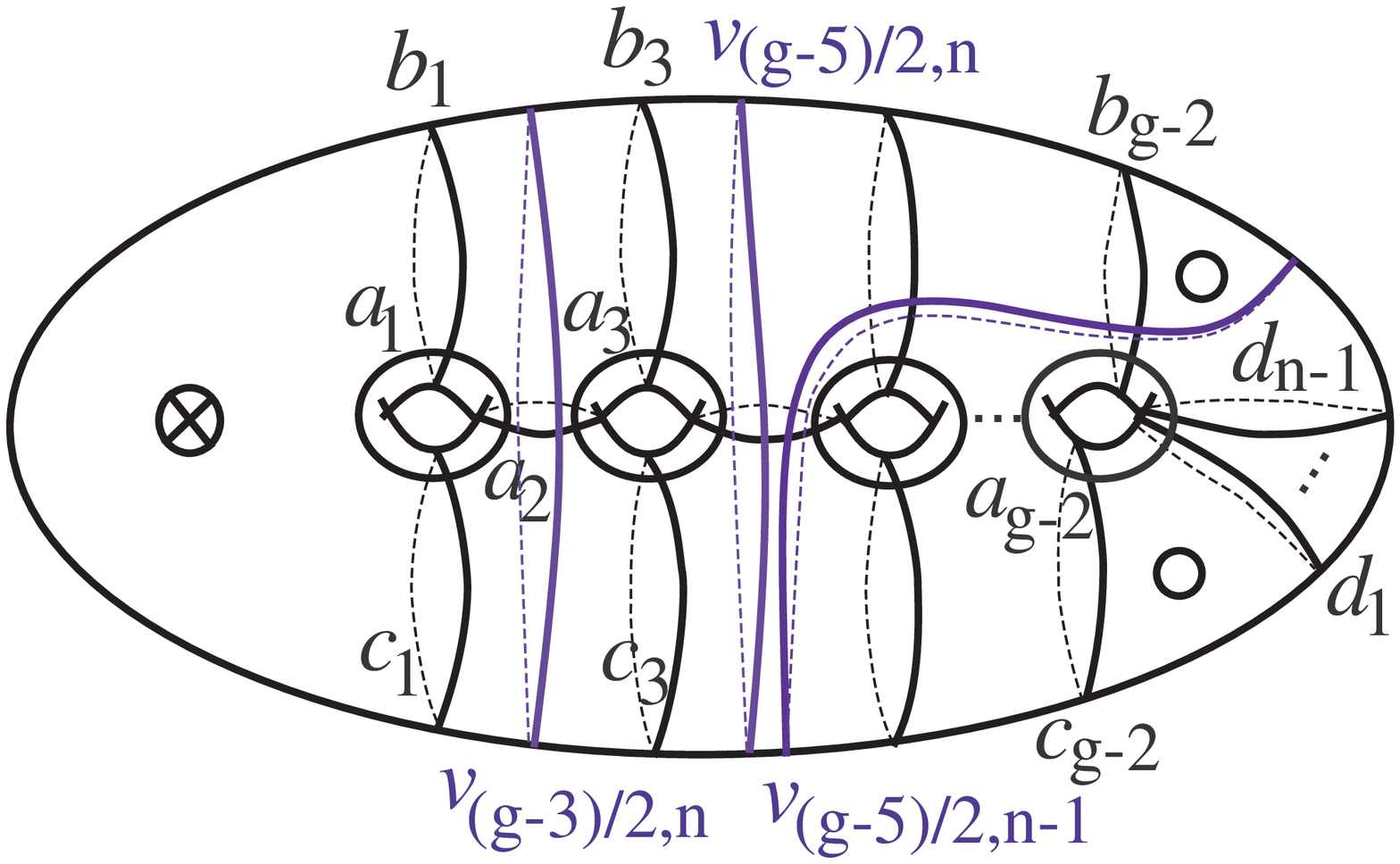}

\hspace{-1cm} (iii) \hspace{6.5cm} (iv)

\caption{ $\mathcal{B}_1$ (odd genus) } \label{fig24}
\end{center}
\end{figure}

For $g \geq 6$ and $g$ even, we let $\mathcal{B}_1 = \{r_1, r_2, \cdots, r_{n-2}, m_1, m_2, \cdots, m_{n-1}, v_{1,0}, v_{1,1}, \cdots, $
$ v_{1,n}, v_{2,0}, v_{2,1}, \cdots, v_{2,n}, \cdots, $ $ v_{(g-2)/2,0}, v_{(g-2)/2,1}, \cdots, v_{(g-2)/2,n}\}$ where the curves are as
shown in Figure \ref{fig23}. All the curves in $\mathcal{B}_1$ are separating curves such that one of the connected components is orientable,
and $\mathcal{B}_1$ has curves of every topological type that satisfies this condition.

\begin{lemma}
\label{B_1-even} If $g \geq 6$ and $g$ is even, then
$h([x]) = \lambda([x])$ $\forall \ x \in \mathcal{C} \cup \mathcal{B}_0 \cup \mathcal{B}_1$.\end{lemma}

\begin{proof} We will give the proof when there is boundary. The closed case is similar. By Lemma \ref{B_0-even} we know that
$h([x]) = \lambda([x])$ $\forall \ x \in \mathcal{C} \cup \mathcal{B}_0$. Consider the curves in Figure \ref{fig23}. The curve $m_1$ is
the unique nontrivial curve up to isotopy that intersects $d_1$ nontrivially and is disjoint from all the other curves in $\mathcal{C}$ given in
Figure \ref{fig1b-new}. This gives us  $h([m_1]) = \lambda([m_1])$ since $\lambda$ is superinjective.
Similarly $h([m_i]) = \lambda([m_i])$ for all $i= 1, 2, \cdots, n-1$. The curve $r_1$ is the unique nontrivial curve up to isotopy that intersects
each $d_i$ for $i= 1, 2, \cdots, n-1$ and it is disjoint from all the other curves in $\mathcal{C} \cup \{m_1, m_2, \cdots, m_{n-1}\}$.
Since $h([x]) = \lambda([x])$ for all these curves and $\lambda$ is superinjective, we get $h([r_1]) = \lambda([r_1])$.
Similar arguments show $h([r_i]) = \lambda([r_i])$
for all $i= 1, 2, \cdots, n-2$. The curve $v_{1,0}$ is the unique nontrivial curve up to isotopy that intersects $a_{g-4}$, $b_{g-3}$ and all the
$d_i$, $i= 1, 2, \cdots, n-1$ nontrivially and it is disjoint from all the other curves in $\mathcal{C} \cup \{r_1\}$. Since
$h([x]) = \lambda([x])$ for all these curves, by using that $\lambda$ is superinjective we see that
$h([v_{1,0}]) = \lambda([v_{1,0}])$. The curve $v_{1,1}$ is the unique nontrivial curve up to isotopy that intersects each of $a_{g-4}$, $b_{g-3}$,
$d_i$, $i= 2, 3, \cdots, n-1$ nontrivially and it is disjoint from all the other curves in $\mathcal{C} \cup \{d_1, r_2, v_{1,0}\}$.
By using that $h([x]) = \lambda([x])$ for all these curves and $\lambda$ is superinjective, we see that
$h([v_{1,1}]) = \lambda([v_{1,1}])$. With similar arguments we can see that $h([v_{i,j}]) = \lambda([v_{i,j}])$.\end{proof}\\

For $g \geq 5$ and $g$ odd, we let $\mathcal{B}_1 = \{r_1, r_2, \cdots, r_{n-2}, m_1, m_2, \cdots, $ $m_{n-1}, v_{1,0}, v_{1,1}, \cdots,$
$ v_{1,n}, v_{2,0}, v_{2,1}, \cdots, v_{2,n}, \cdots, $ $ v_{(g-3)/2,0}, v_{(g-3)/2,1}, \cdots, v_{(g-3)/2,n}\}$ where the curves are as
shown in Figure \ref{fig24}. All the curves in $\mathcal{B}_1$ are separating curves such that one of the connected 
components is orientable,
and $\mathcal{B}_1$ has curves of every topological type that satisfies this condition.

\begin{lemma}
\label{B_1-odd} If $g \geq 5$ and $g$ is odd, then $h([x]) = \lambda([x])$ $\forall \ x \in \mathcal{C} \cup \mathcal{B}_1$.\end{lemma}

\begin{proof} The proof is similar to the even genus case given in Lemma \ref{B_1-even} (use Lemma \ref{curves-2} and see
Figure \ref{fig24}).\end{proof}\\

\begin{figure}
\begin{center}
\hspace{-1cm} \epsfxsize=3.in \epsfbox{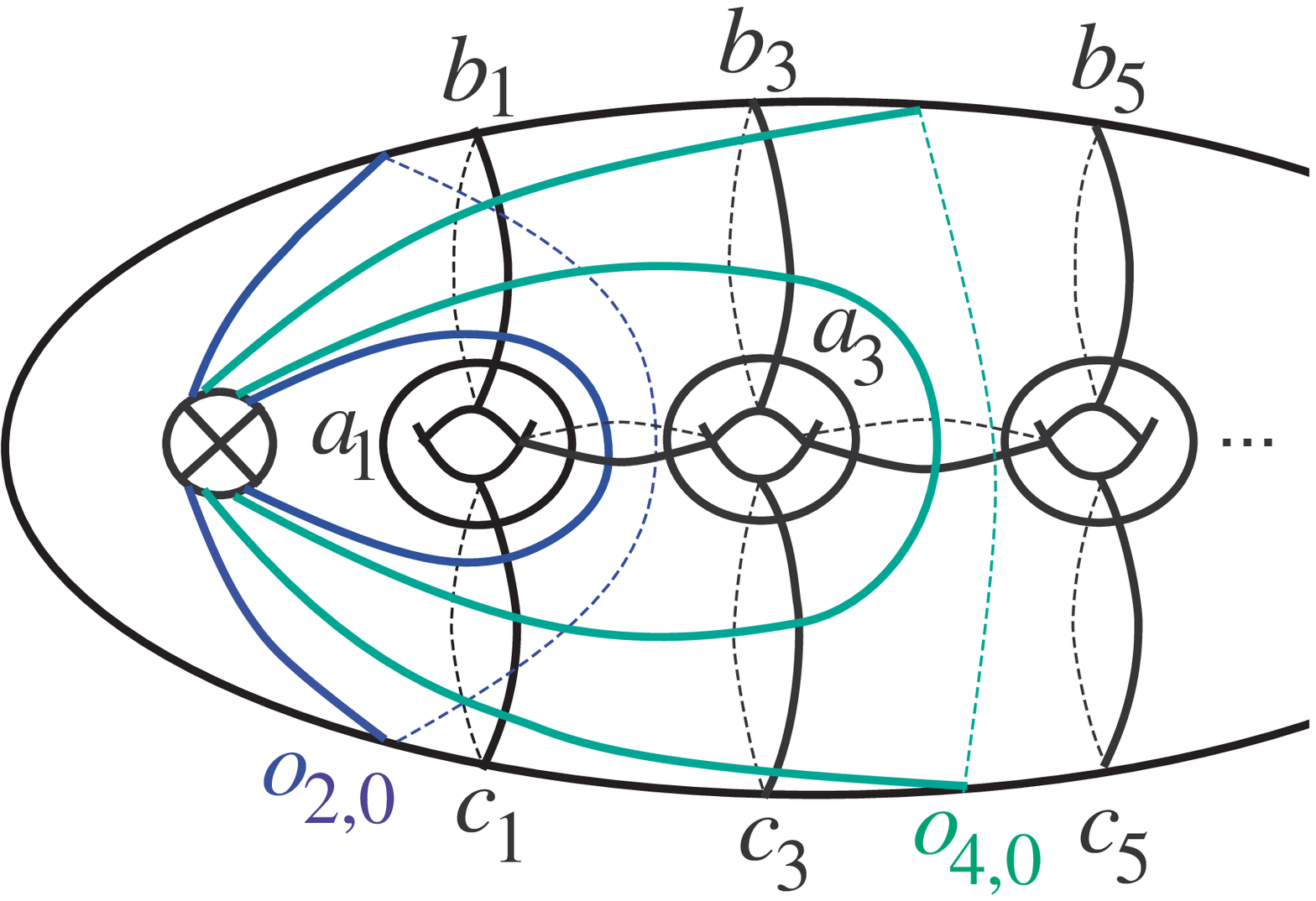}  \hspace{-0.4cm} \epsfxsize=3.in \epsfbox{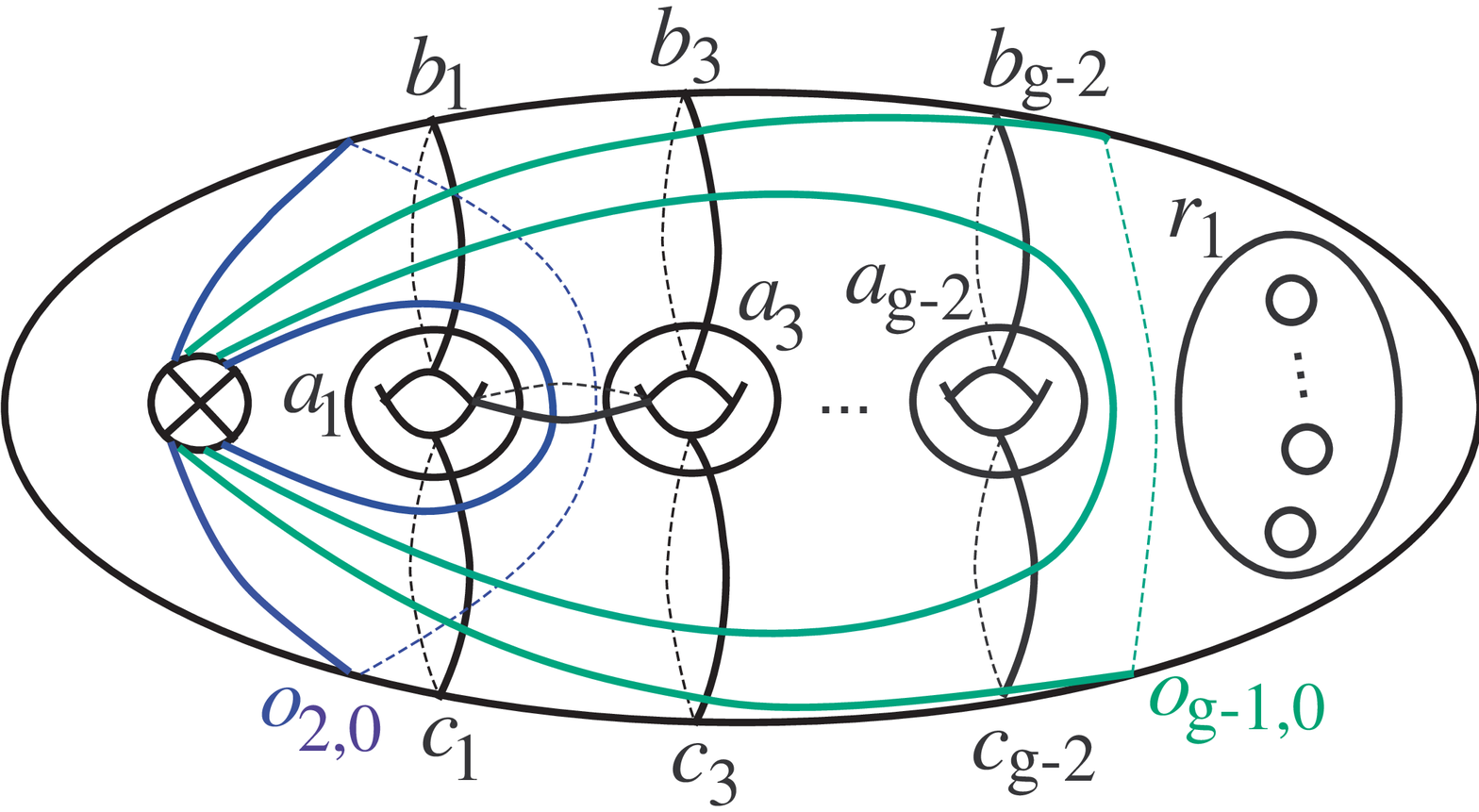}

\hspace{-1cm} (i) \hspace{6.5cm} (ii)

\hspace{-0.5cm} \epsfxsize=3.2in \epsfbox{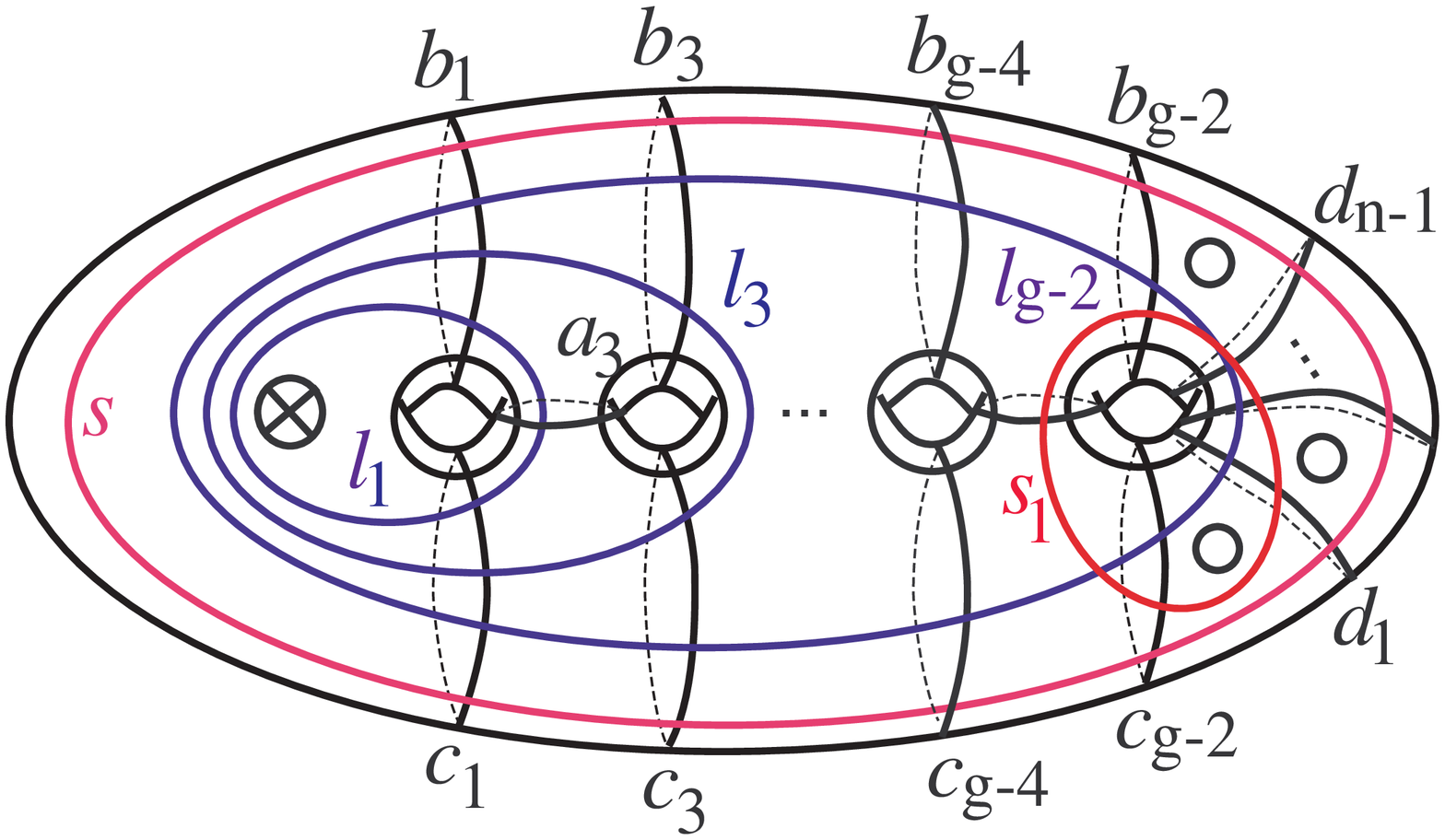}  \hspace{-0.7cm} \epsfxsize=3.2in \epsfbox{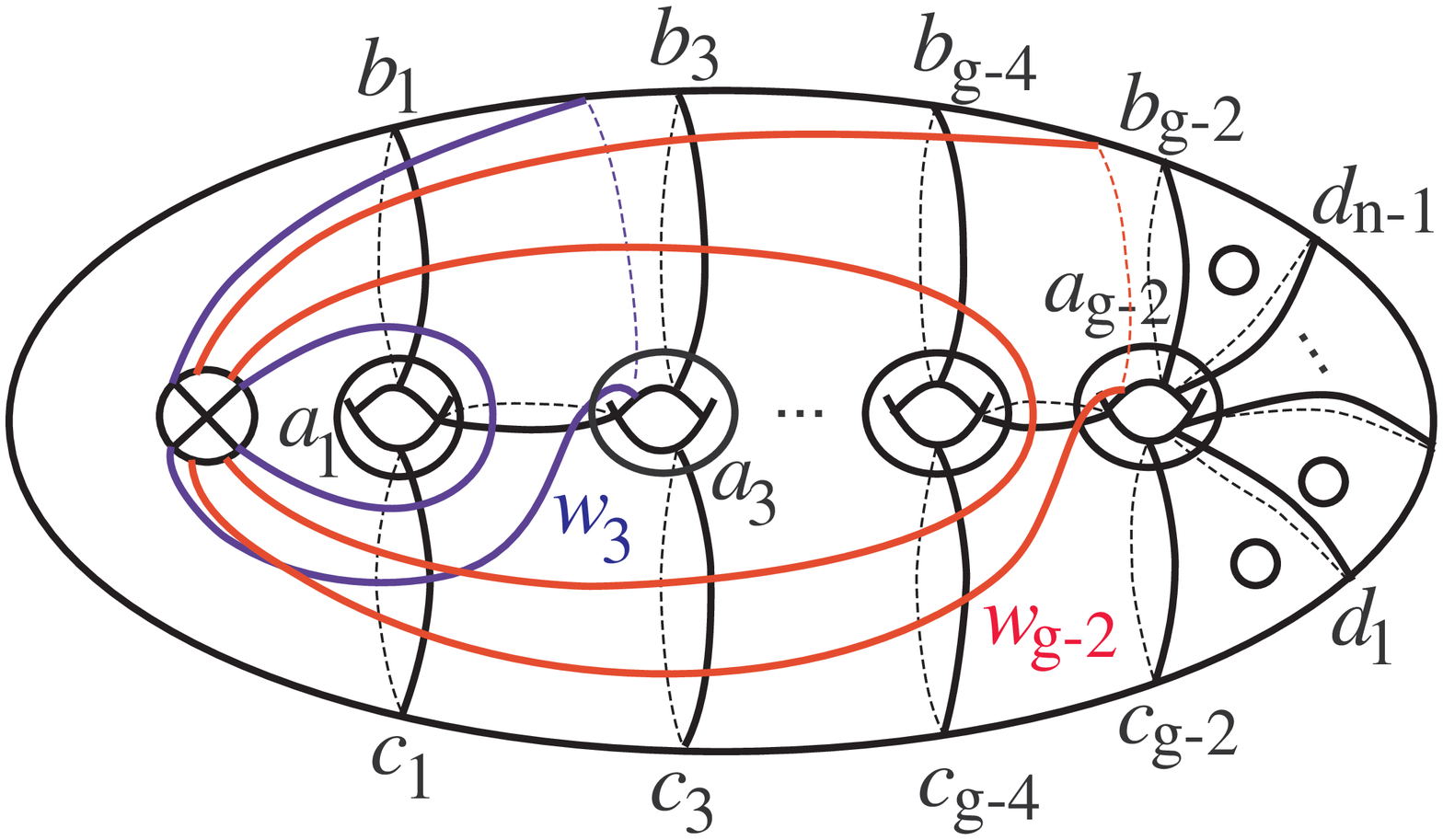}

\hspace{-1cm} (iii) \hspace{6.5cm} (iv)

\hspace{-0.5cm} \epsfxsize=3.2in \epsfbox{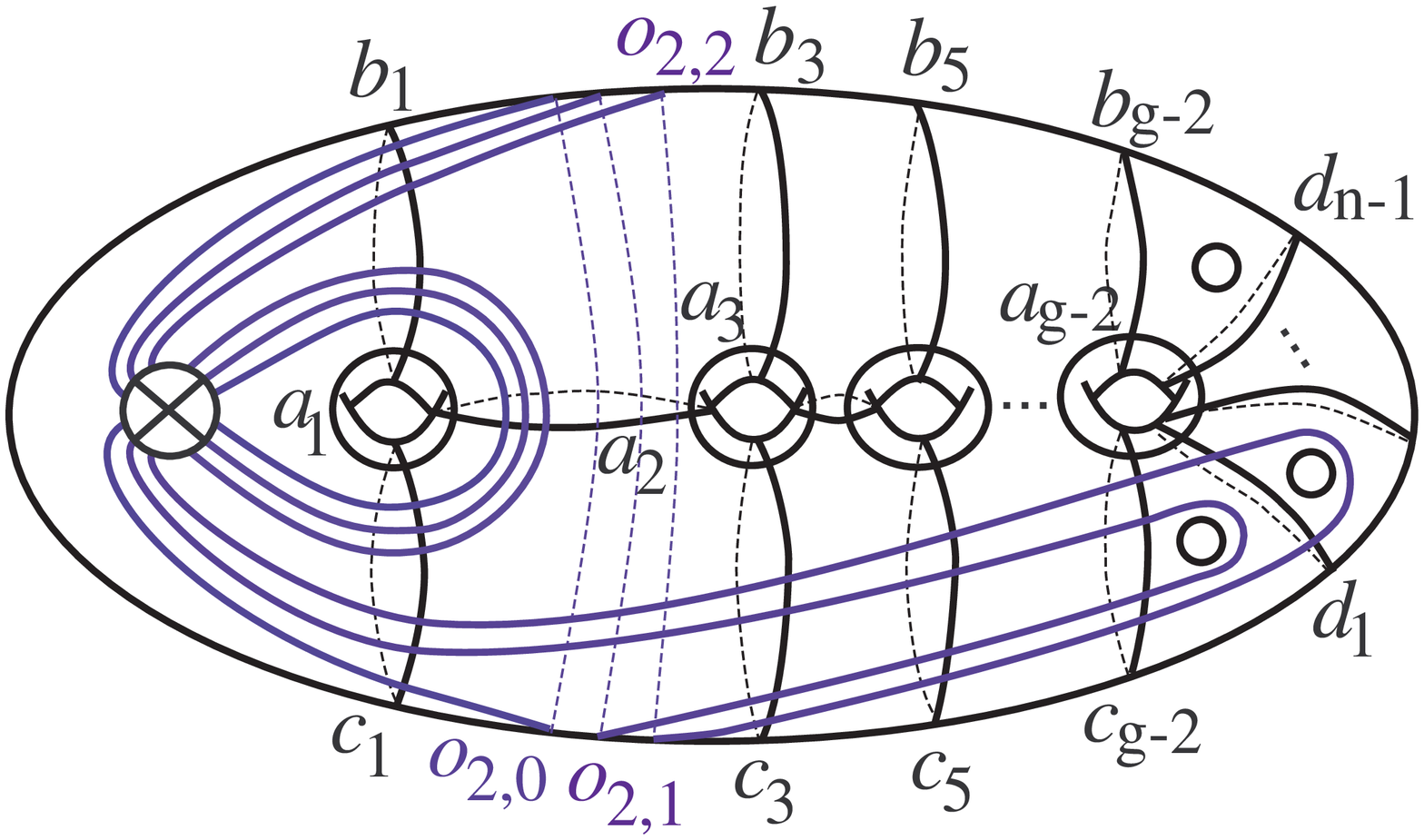}  \hspace{-0.7cm} \epsfxsize=3.2in \epsfbox{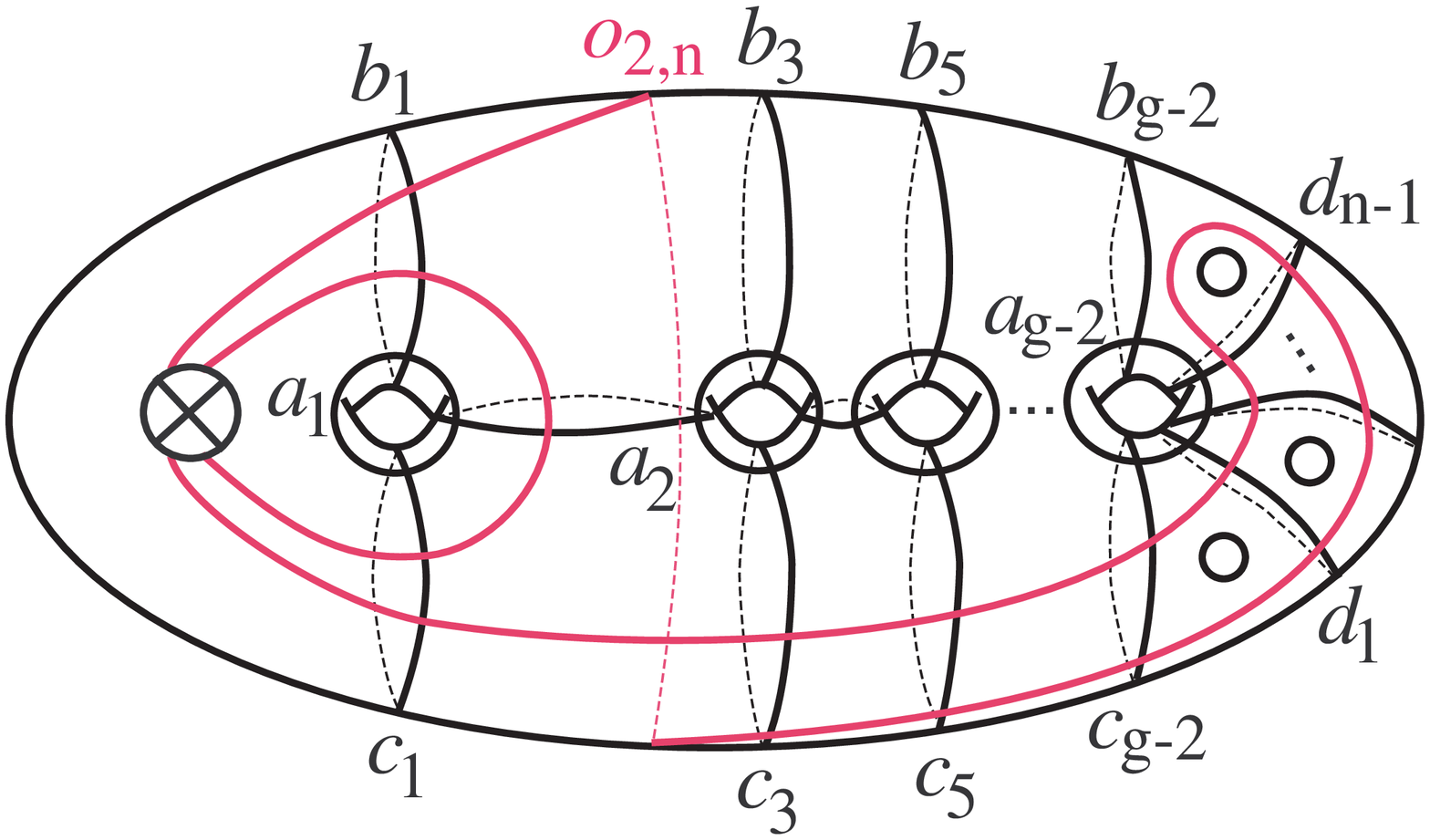}

\hspace{-1cm} (v) \hspace{6.5cm} (vi)

\hspace{-0.5cm} \epsfxsize=3.2in \epsfbox{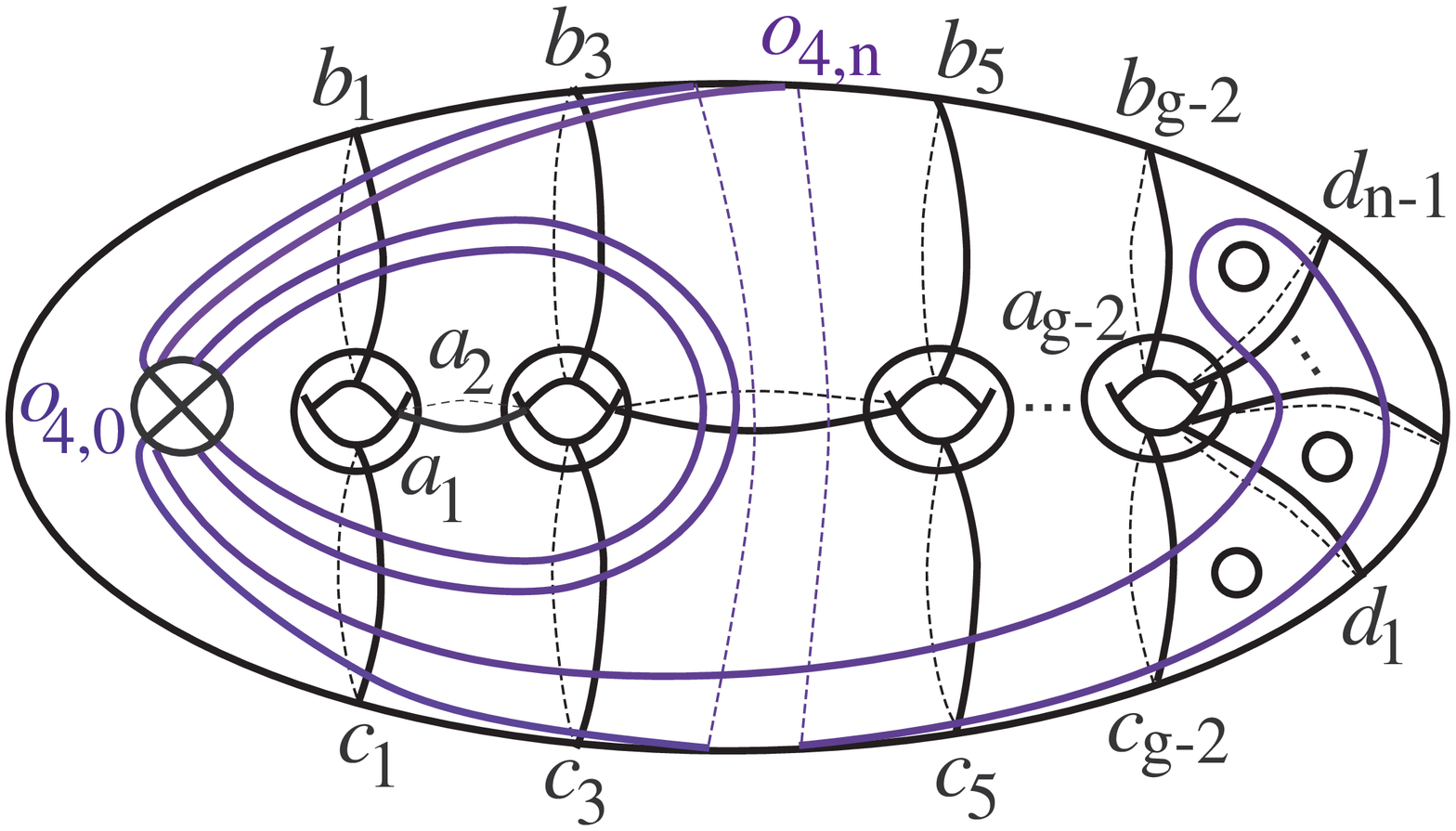} \hspace{-0.7cm} \epsfxsize=3.2in \epsfbox{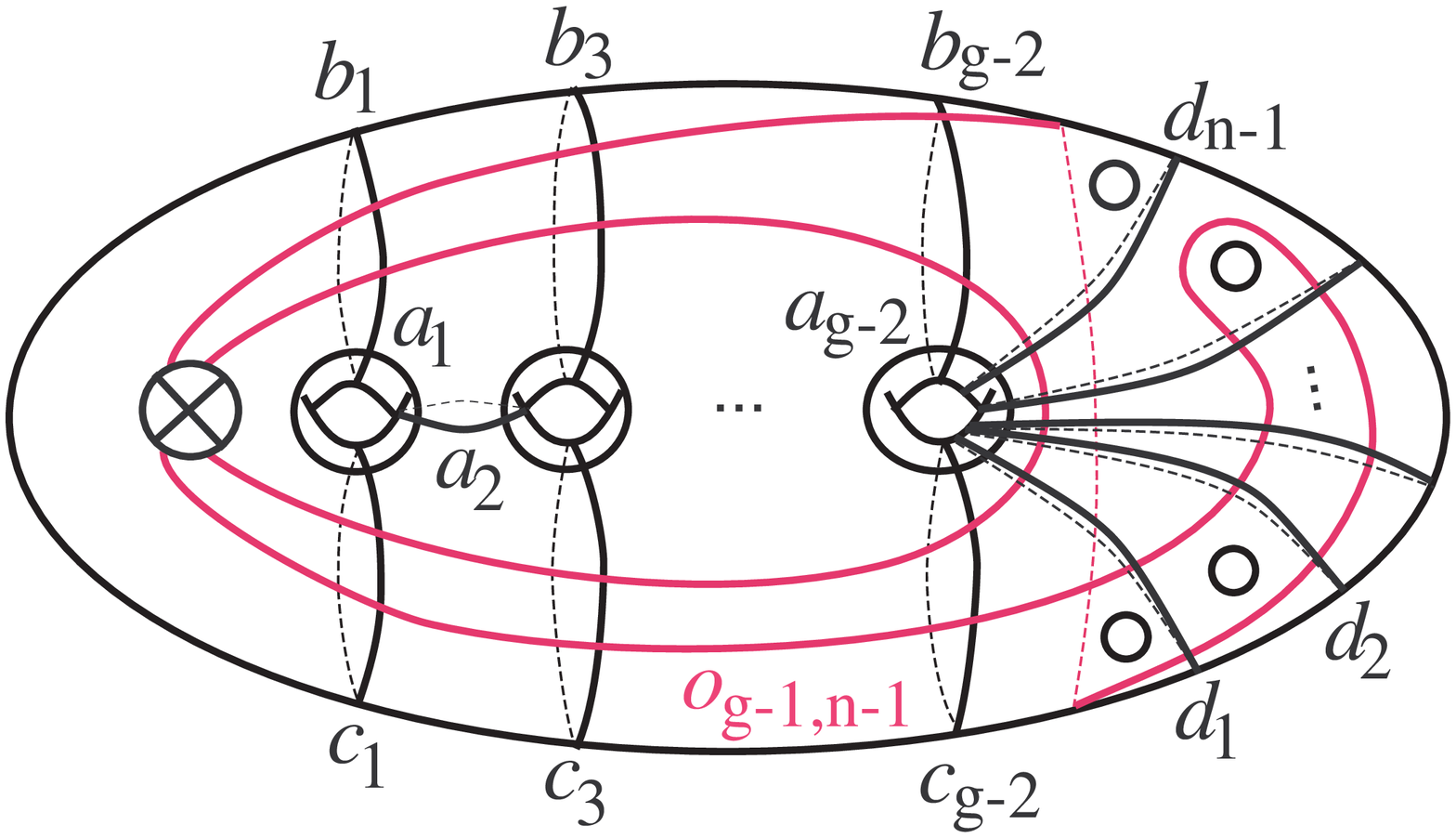}

\hspace{-1cm} (vii) \hspace{6.5cm} (viii)

\caption{$\mathcal{B}_2$, (odd genus) separating curves that have nonorientable surfaces on both sides} \label{fig26}
\end{center}
\end{figure}

\begin{figure}
\begin{center}
\hspace{-1.7cm} \epsfxsize=2.9in \epsfbox{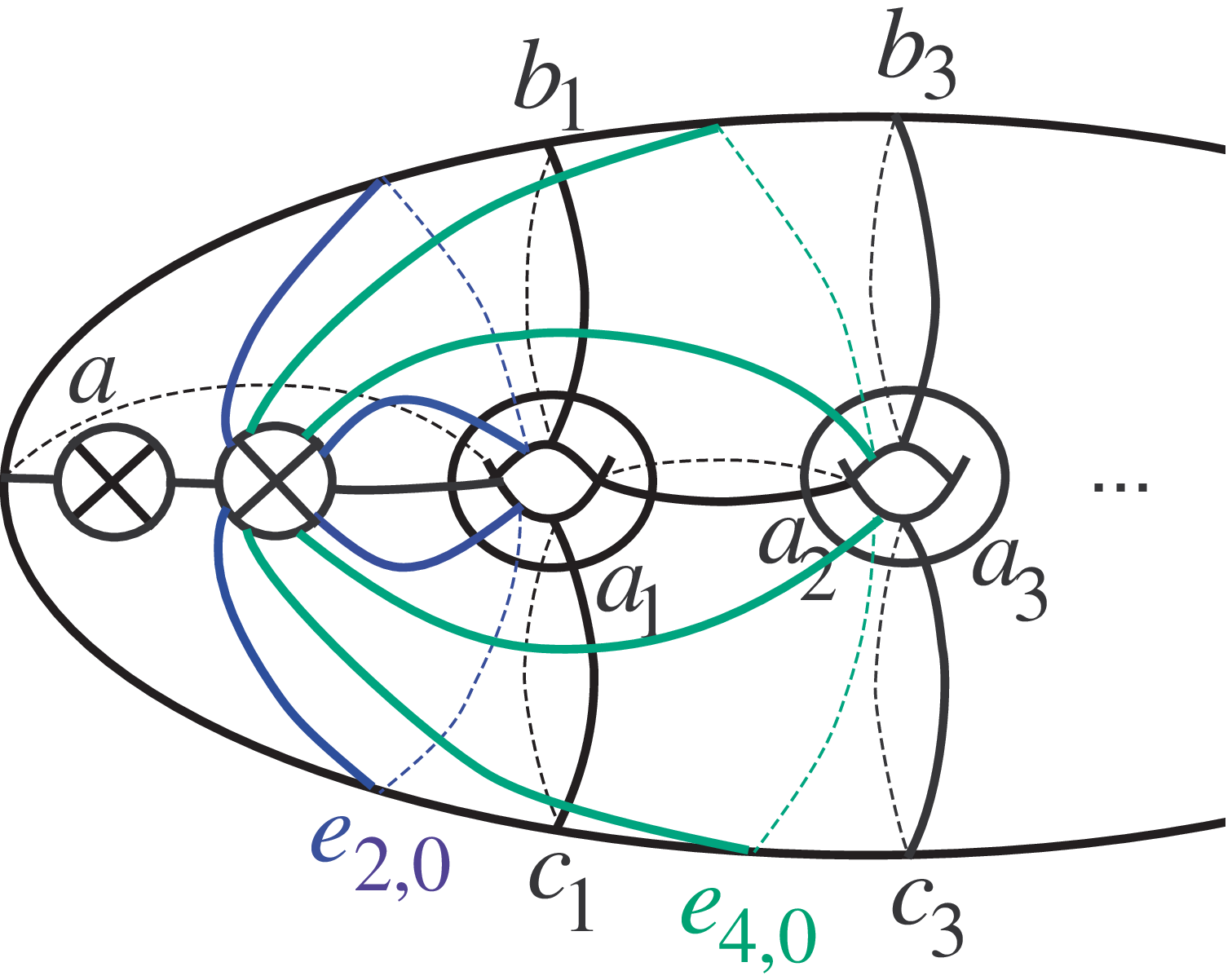}  \hspace{-0.4cm} \epsfxsize=2.8in \epsfbox{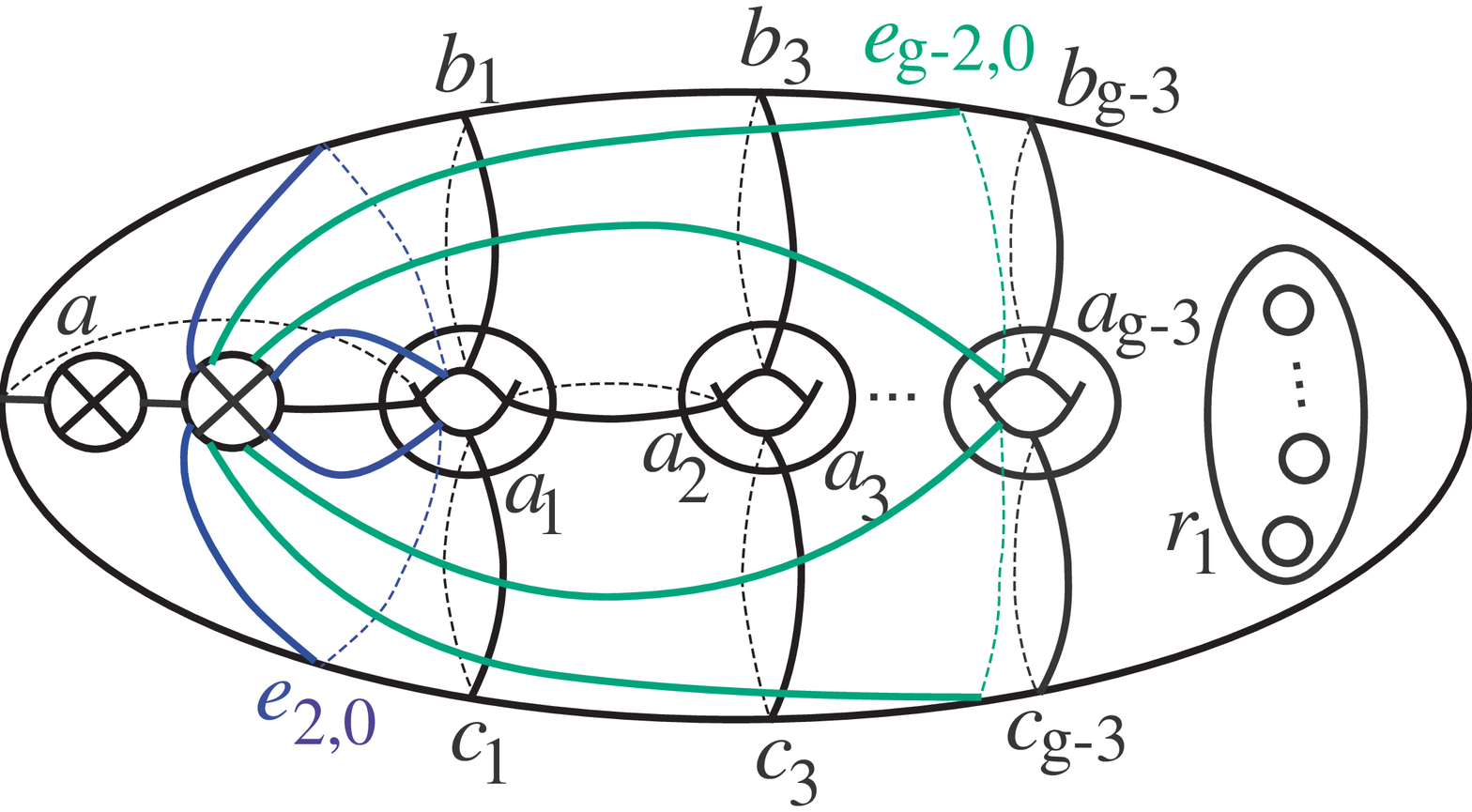}

\hspace{-1cm} (i) \hspace{6.5cm} (ii)

\hspace{-1cm} \epsfxsize=3.2in \epsfbox{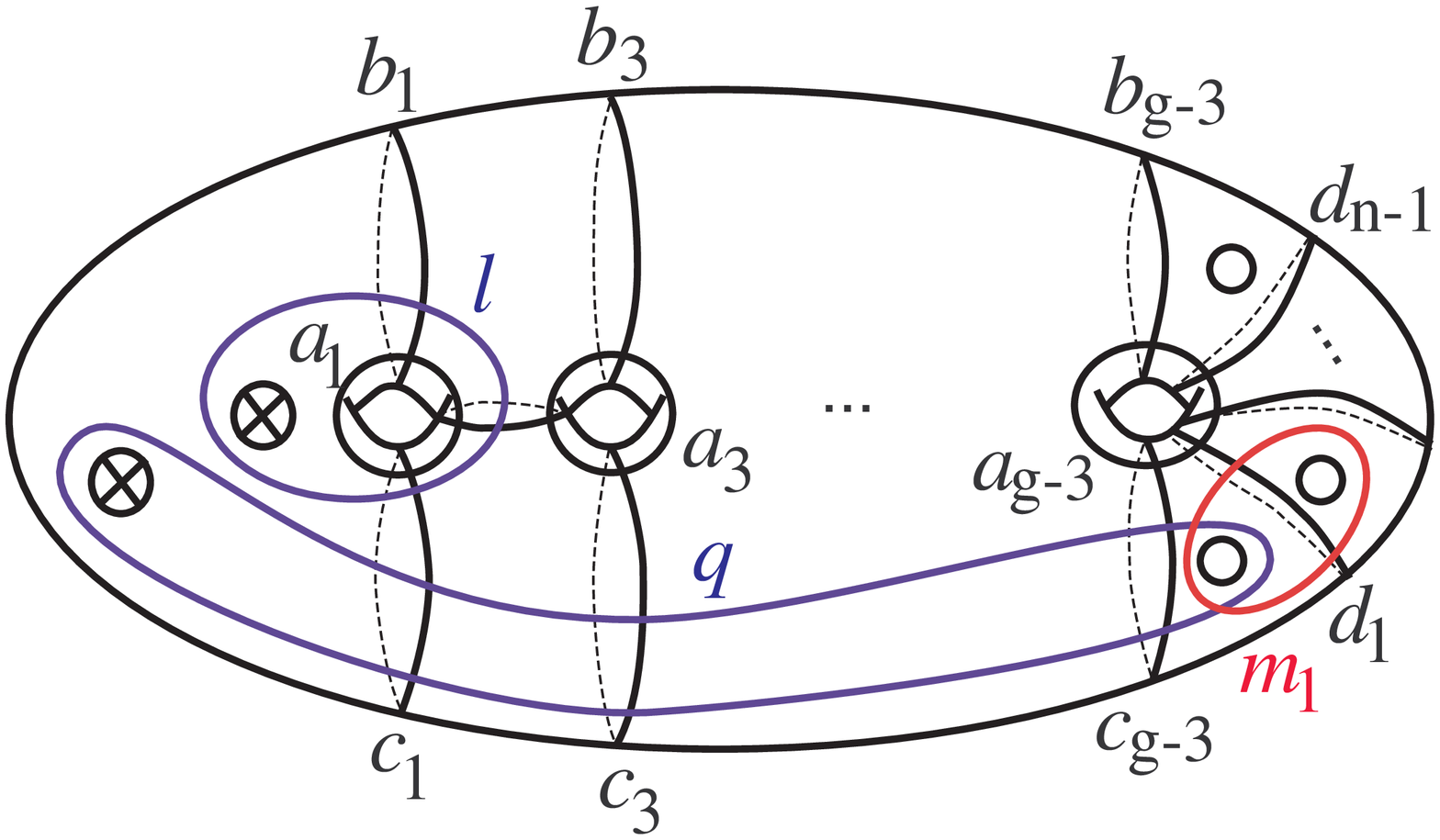}  \hspace{-0.4cm} \epsfxsize=3.2in \epsfbox{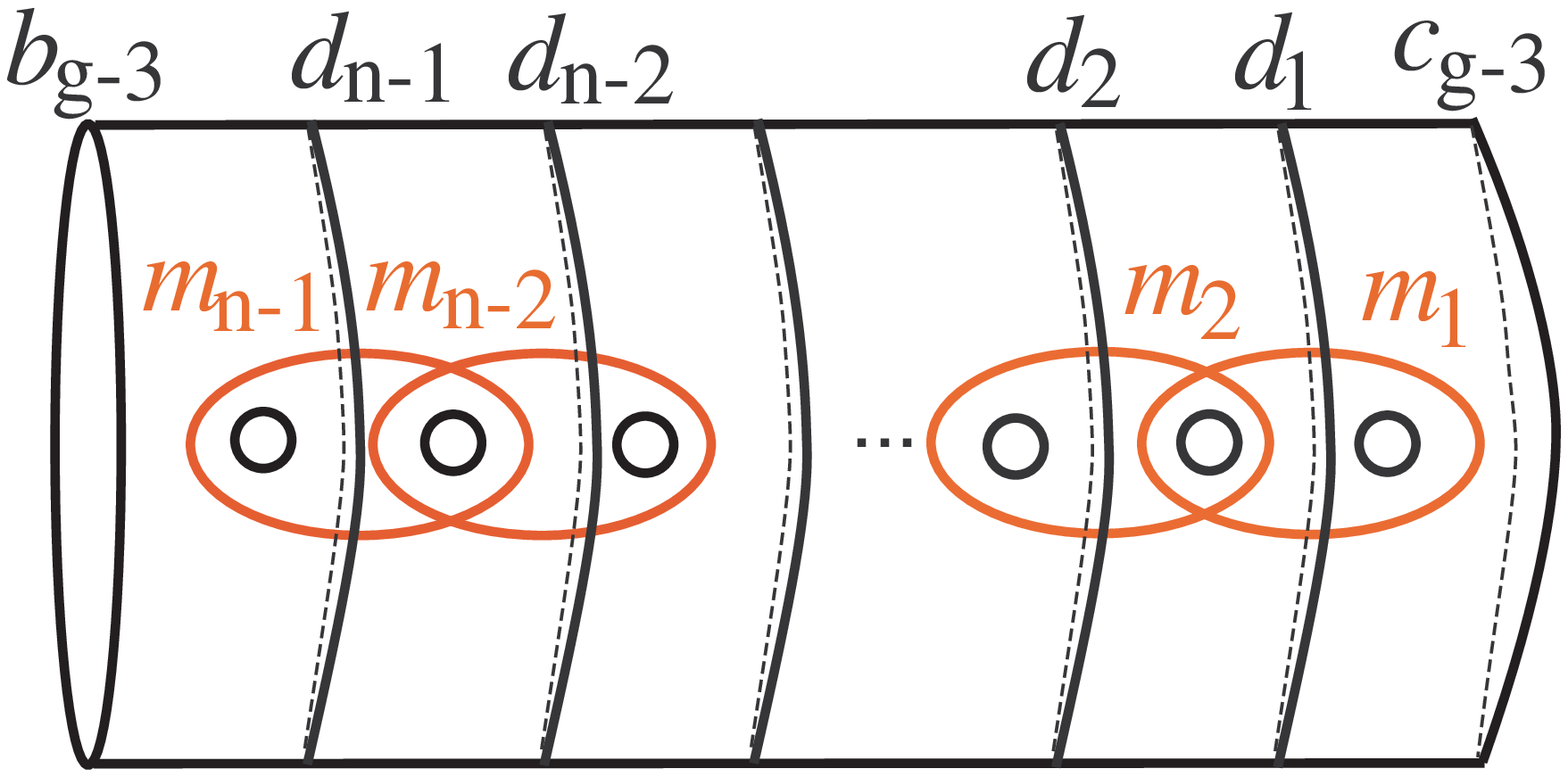}

\hspace{-1cm} (iii) \hspace{6.5cm} (iv)

\hspace{-0.5cm} \epsfxsize=3.2in \epsfbox{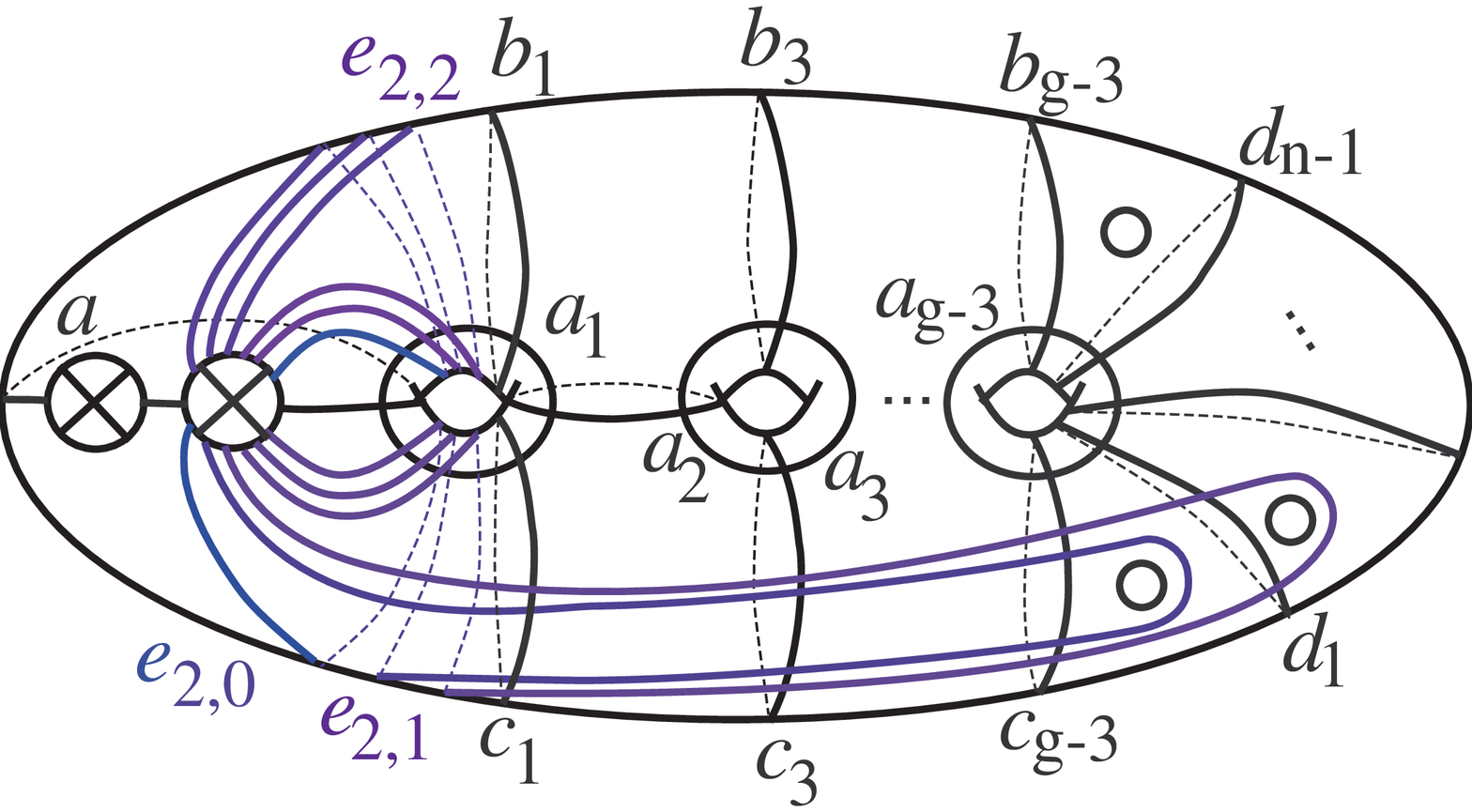}  \hspace{-0.7cm} \epsfxsize=3.2in \epsfbox{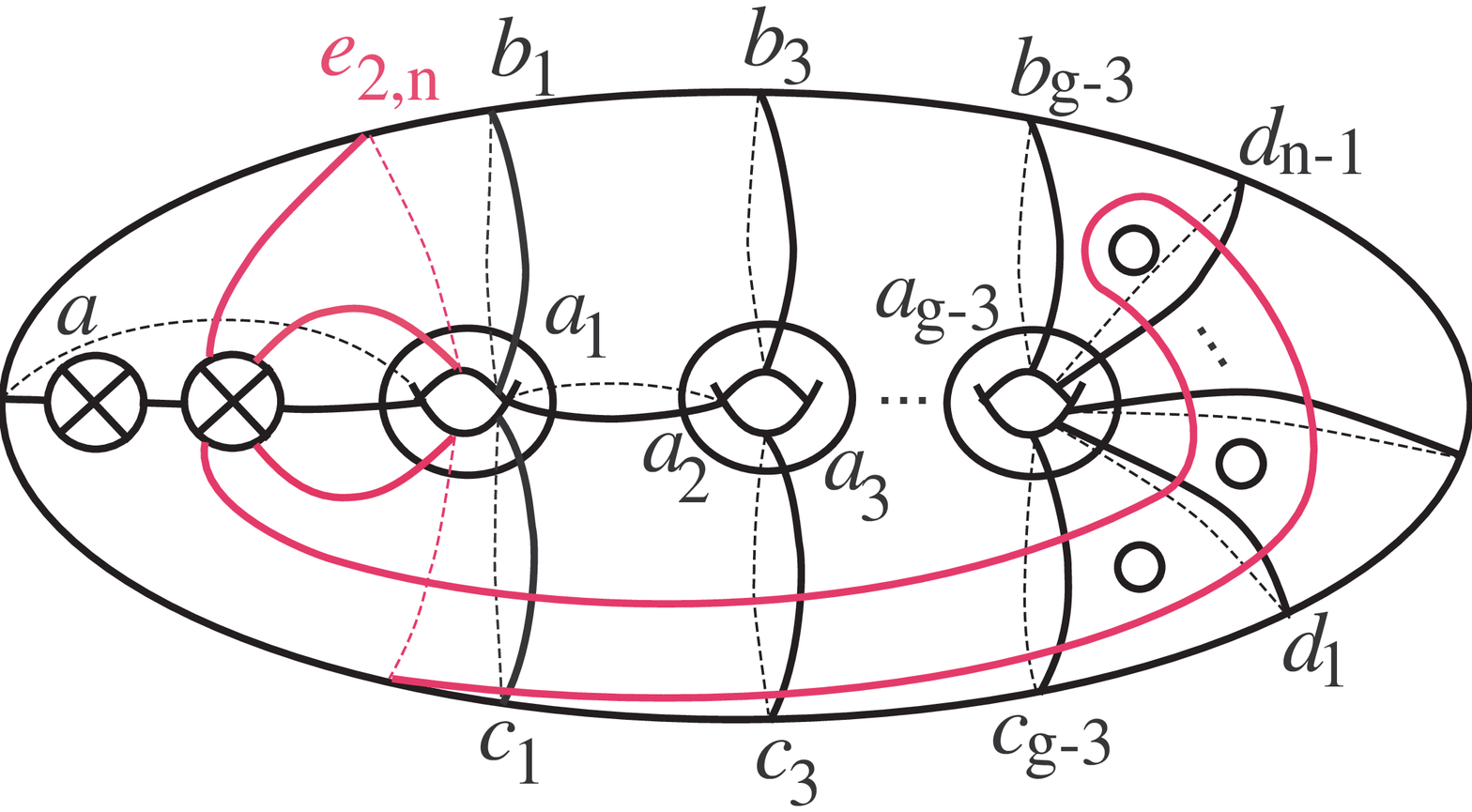}

\hspace{-1cm} (v) \hspace{6.5cm} (vi)

\hspace{-0.5cm} \epsfxsize=3.2in \epsfbox{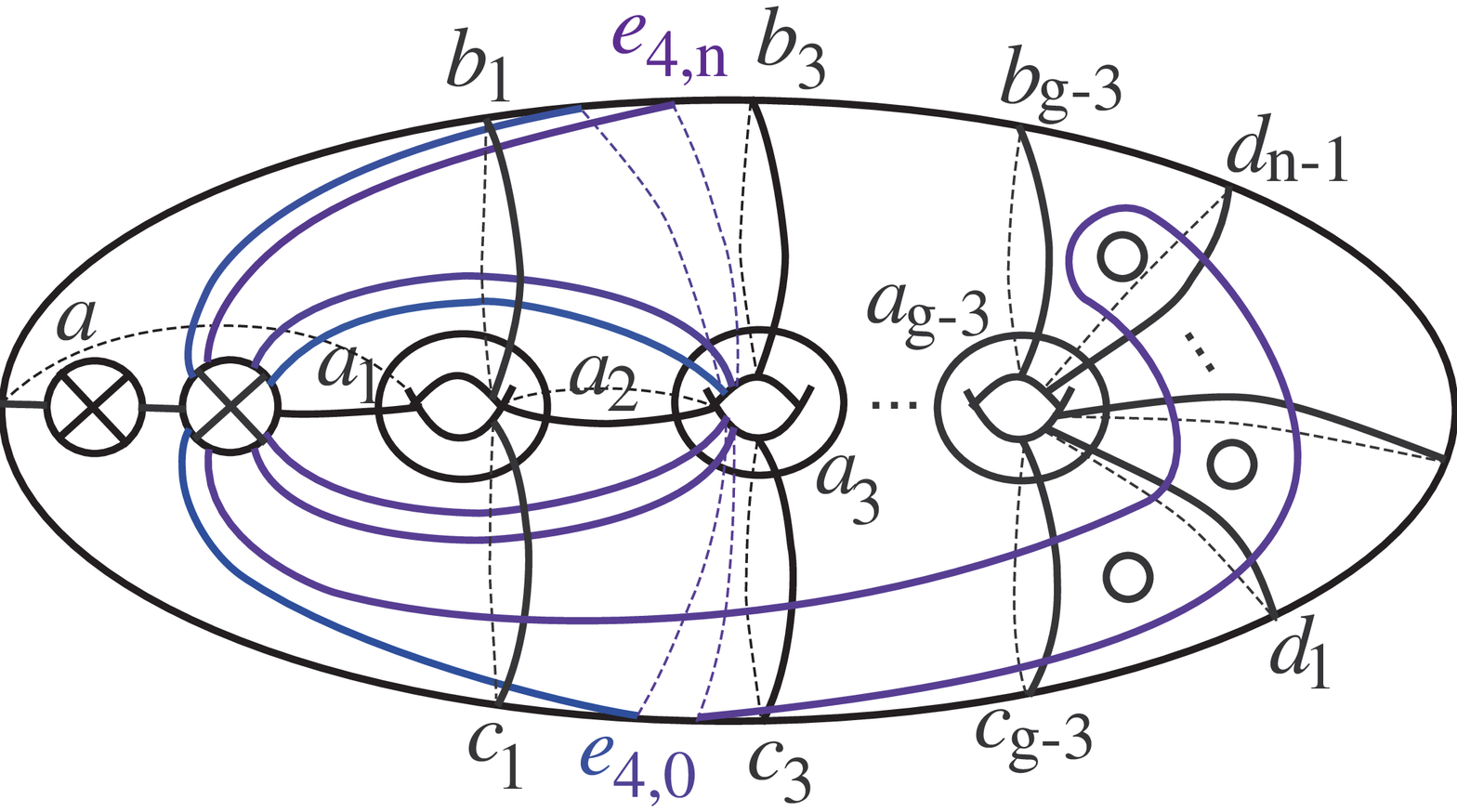} \hspace{-0.7cm} \epsfxsize=3.2in \epsfbox{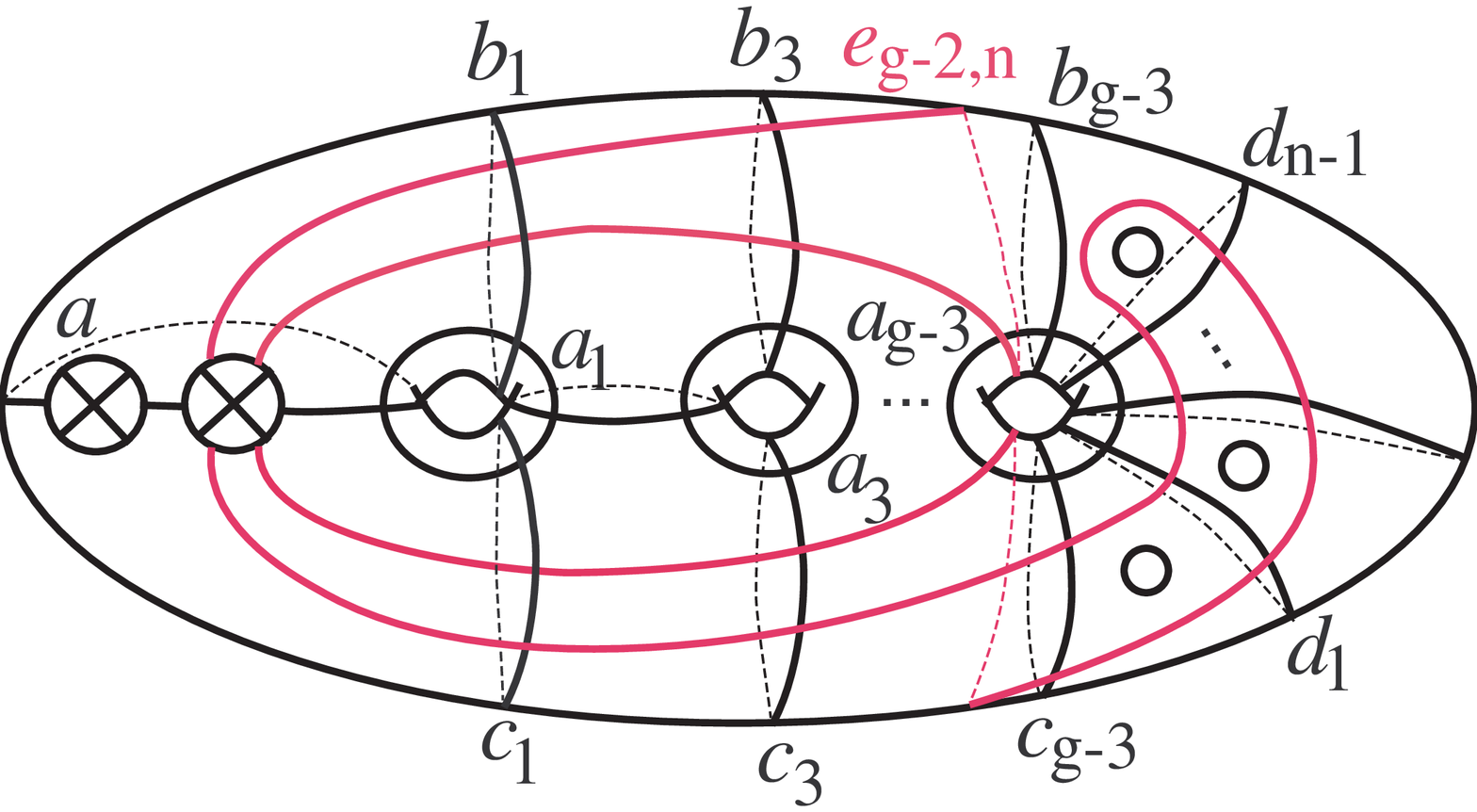}

\hspace{-0.7cm} (vii) \hspace{6.4cm} (viii)
\caption{$\mathcal{B}_2$, (even genus) separating curves that have even genus nonorientable surfaces on both sides} \label{fig27}
\end{center}
\end{figure}

For $g \geq 5$ and $g$ odd, let $\mathcal{B}_2 = \{o_{2,0}, o_{2,1}, \cdots, o_{2,n}, o_{4,0}, o_{4,1}, \cdots, o_{4,n},$
$ \cdots, o_{g-1,0}, o_{g-1,1},$ $ \cdots, o_{g-1,n-1}\}$ where the curves are as shown in Figure \ref{fig26}. All the curves
in $\mathcal{B}_2$ are separating curves such that both of the connected components are nonorientable and one side has even genus.
$\mathcal{B}_2$ has curves of every topological type that satisfies this condition.

\begin{lemma}
\label{A-odd} If $g \geq 5$ and $g$ is odd, then $h([x]) = \lambda([x])$ $\forall \ x \in \mathcal{C} \cup \mathcal{B}_1 \cup \mathcal{B}_2$.\end{lemma}

\begin{proof} We will give the proof when there is boundary by using the properties of $\lambda$ that we proved in Section 2.
The closed case is similar. By Lemma \ref{B_1-odd} we have
$h([x]) = \lambda([x])$ $\forall \ x \in \mathcal{C} \cup \mathcal{B}_1$. Consider the curves in Figure \ref{fig26}. We know the
result for $o_{2,0}$ since $o_{2,0}=e$ and $e \in \mathcal{C}$. The curve $o_{4,0}$ is the unique nontrivial curve up to isotopy that is disjoint
from all the curves in $\{a_1, a_2, a_3, a_5, b_5, c_5, o_{2,0}\}$ and intersects $b_3$ nontrivially. Since we know that $h([x]) = \lambda([x])$
for all these curves, by using that $\lambda$ is superinjective we get $h([o_{4,0}]) = \lambda([o_{4,0}])$.
Similar arguments show  $h([o_{i,0}]) = \lambda([o_{i,0}])$ for all $i = 2, 4, \cdots, g-3$. The curve $o_{g-1,0}$ appears in the configuration
when there is at least one boundary component. We consider two cases to control its image. If there is only one boundary component, then $o_{g-1,0}$
is the unique nontrivial curve up to isotopy that is disjoint from all the curves in $\{a_1, a_2, a_3, \cdots,
a_{g-2}, o_{2,0}\}$ and intersects $b_3$ nontrivially. Since we know that $h([x]) = \lambda([x])$ for all these curves, by using that $\lambda$ is
superinjective we get $h([o_{g-1,0}]) = \lambda([o_{g-1,0}])$. If there is more than one boundary
component, then $o_{g-1,0}$ is the unique nontrivial curve up to isotopy that is disjoint from all the curves in $\{a_1, a_2, a_3, \cdots,
a_{g-2}, o_{2,0}, r_1\}$ and intersects $b_3$ nontrivially. The curve $r_1 \in \mathcal{B}_1$. Since we know that $h([x]) = \lambda([x])$ for
all these curves, by using superinjectivity again we get $h([o_{g-1,0}]) = \lambda([o_{g-1,0}])$.

Since $l_1=l \in \mathcal{C}$ we have $h([l_1]) = \lambda([l_1])$. The curve $l_3$ is the unique nontrivial curve up to isotopy that is disjoint
from all the curves in $\{a_1, a_2, a_3, l_1, a_5, b_5, c_5\}$ and intersects $a_4$ once nontrivially and it bounds a pair of pants together
with $a_3$ and $l_1$. Since we know that $h([x]) = \lambda([x])$ for all these curves, by using that $\lambda$ preserves all these properties
listed as shown in Section 2, we get $h([l_3]) = \lambda([l_3])$. Similar arguments show  $h([l_i]) = \lambda([l_i])$ for all $i = 1, 3, 5, \cdots, g-2$.
The curve $w_3$ is the unique nontrivial curve up to isotopy that is disjoint from all the curves in $\{b_3, c_3, o_{2,0}\}$ and intersects $l_3$ and
$a_3$ only once and it bounds a projective plane with two boundary components together with
$b_3$. Since we know that $h([x]) = \lambda([x])$ for all these curves, and $\lambda$ preserves these properties, we get $h([w_3]) = \lambda([w_3])$.
With similar ideas it is easy to get $h([w_i]) = \lambda([w_i])$ for all $i = 3, 5, 7, \cdots, g-2$.

The curve $o_{2,1}$ is the unique nontrivial curve up to isotopy that is disjoint from all the curves in $\{a_3, a_4, a_5, \cdots, a_{g-2}, d_1, w_3, o_{2,0}\}$
and intersects $c_3$ nontrivially. Since $h([x]) = \lambda([x])$ for all these curves, and $\lambda$ preserves these properties,
we get $h([o_{2,1}]) = \lambda([o_{2,1}])$. With similar ideas and using that $o_{2,2}$ is also disjoint from $m_1$ (see Figure \ref{fig24}),
we get $h([o_{2,2}]) = \lambda([o_{2,2}])$. Similarly, by using also the curves $m_i \in \mathcal{B}_1$ in
Figure \ref{fig24}, we get $h([o_{2,i}]) = \lambda([o_{2,i}])$ for all $i = 0, 1, 2, \cdots, n$. Getting $h([o_{i,j}]) = \lambda([o_{i,j}])$
for all $o_{i,j} \in \mathcal{B}_2$ with $i \leq g-3$ is similar. The curve $s_1$ is the unique nontrivial curve up to isotopy that is disjoint from each of $a_{g-2}$ and $v_{1,1}$ and intersects each of $b_{g-2}, c_{g-2}, a_{g-3}, d_i$
for all $i$ once, bounds a pair of pants together with $a_{g-2}$ and a boundary component of $N$, and intersects $l_{g-2}$ nontrivially (see Figure \ref{fig24} for $v_{1,1}$). Since we
know that $h([x]) = \lambda([x])$ for all these curves, and $\lambda$ preserves these properties we have $h([s_1]) = \lambda([s_1])$.
The curve $s$ is the unique nontrivial curve up to isotopy that is disjoint from all the curves in $\{a_i, l_i, m_i, s_1, r_1\}$ and intersects
each $b_i, c_i, d_i$ once. Since we know that $h([x]) = \lambda([x])$ for all these curves and $\lambda$ preserves these properties,
we have $h([s]) = \lambda([s])$. The curve $o_{g-1,n-1}$ is the unique nontrivial curve up to isotopy that is disjoint from all the curves in
$\{a_1, a_2, \cdots, a_{g-2}, s, o_{g-1,0}, o_{2,1}, m_1, m_2, \cdots, m_{n-2}\}$ and intersects $b_1$ and $m_{n-1}$ nontrivially. Since we know
that $h([x]) = \lambda([x])$ for all these curves and $\lambda$ is superinjective, we have
$h([o_{g-1,n-1}]) = \lambda([o_{g-1,n-1}])$. With similar arguments, we get $h([o_{g-1,j}]) = \lambda([o_{g-1,j}])$ for all $j$.\end{proof}\\

For $g \geq 6$ and $g$ is even, let $\mathcal{B}_2 = \{e_{2,0}, e_{2,1}, \cdots, e_{2,n}, e_{4,0}, e_{4,1}, \cdots, e_{4,n}, \cdots, e_{g-2,0},$
$e_{g-2,1},$ $ \cdots, e_{g-2,n-1}\}$ where the curves are as shown in Figure \ref{fig27}. All the curves in $\mathcal{B}_2$ are separating
curves such that both of the connected components are nonorientable and one side has even genus. $\mathcal{B}_2$ has curves of every topological type that satisfies this condition.

\begin{lemma}
\label{A-even} If $g \geq 6$ and $g$ is even, then $h([x]) = \lambda([x])$
$\forall \ x \in \mathcal{C} \cup \mathcal{B}_0 \cup \mathcal{B}_1 \cup \mathcal{B}_2$.\end{lemma}

\begin{proof} We will give the proof when there is boundary by using the properties of $\lambda$ that we proved in Section 2.
The closed case is similar. By Lemma \ref{B_1-even} we have
$h([x]) = \lambda([x])$ $\forall \ x \in \mathcal{C} \cup \mathcal{B}_0 \cup \mathcal{B}_1$. Consider the curves in Figure \ref{fig27}. We know the result
for $e_{2,0}$ since $e_{2,0}=e$ and $e \in \mathcal{C}$. The curve $e_{4,0}$ is the unique nontrivial curve up to isotopy that is disjoint from all
the curves in $\{a_1, a_2, b_3, c_3, e_{2,0}\}$ and intersects $b_1$ nontrivially. Since we know that $h([x]) = \lambda([x])$ for all
these curves, by using that $\lambda$ is superinjective we get $h([e_{4,0}]) = \lambda([e_{4,0}])$. The curve $o_{6,0}$
is the unique nontrivial curve up to isotopy that is disjoint from all the curves in $\{a_3, a_4, b_5, c_5, e_{4,0}\}$ and intersects $b_3$ nontrivially.
Since we know that $h([x]) = \lambda([x])$ for all these curves, by using that $\lambda$ is superinjective we get
$h([e_{6,0}]) = \lambda([e_{6,0}])$. Similar arguments show  $h([e_{i,0}]) = \lambda([e_{i,0}])$ for all $i = 2, 4, 6, \cdots, g-2$.

The curve $q$ is the unique nontrivial curve up to isotopy that is disjoint from all the curves in $\{l, b_1, a_1, a_2, \cdots, a_{g-3}, d_1\}$
and intersects $c_1$ nontrivially. Since we know that $h([x]) = \lambda([x])$ for all these curves, by using that $\lambda$ is superinjective
we get $h([q]) = \lambda([q])$. The curve $e_{2,1}$ is the unique nontrivial curve up to isotopy that is disjoint from all
the curves in $\{q, e_{2,0}, b_1, a_2, a_3, \cdots, a_{g-3}, d_1\}$ and intersects $c_1$ nontrivially. Since we know that $h([x]) = \lambda([x])$
for all these curves, by using that $\lambda$ is superinjective we get $h([e_{2,1}]) = \lambda([e_{2,1}])$. The curve
$e_{2,2}$ is the unique nontrivial curve up to isotopy that is disjoint from all the curves in $\{q, m_1, e_{2,1}, b_1, a_2, a_3, \cdots, a_{g-3}, d_2\}$
and intersects $c_1$ nontrivially. Since we know that $h([x]) = \lambda([x])$ for all these curves, by using that $\lambda$ is superinjective
we get $h([e_{2,2}]) = \lambda([e_{2,2}])$. We note that $m_1 \in \mathcal{B}_1$. By using similar arguments and the other
curves $m_i \in \mathcal{B}_1$ we get $h([e_{2,i}]) = \lambda([e_{2,i}])$ for all $i= 1, 2, \cdots, n$.
The curve $e_{4,1}$ is the unique nontrivial curve up to isotopy that is disjoint from all the curves in
$\{q, e_{4,0}, b_3, a_1, a_2, a_4, a_5, \cdots, a_{g-3}, d_1\}$ and intersects $c_3$ nontrivially. Since we know that $h([x]) = \lambda([x])$
for all these curves, by using that $\lambda$ is superinjective we get $h([e_{4,1}]) = \lambda([e_{4,1}])$. The curve
$e_{4,2}$ is the unique nontrivial curve up to isotopy that is disjoint from all the curves in $\{q, m_1, e_{4,1}, b_3, a_4, a_5, \cdots, a_{g-3}, d_2\}$
and intersects $c_3$ nontrivially. Since we know that $h([x]) = \lambda([x])$ for all these curves, by using that $\lambda$ is superinjective
we get $h([e_{4,2}]) = \lambda([e_{4,2}])$. By using curves $m_i \in \mathcal{B}_1$ we get $h([e_{4,i}]) = \lambda([e_{4,i}])$
for all $i= 1, 2, \cdots, n$. Similarly, we get $h([e_{i,j}]) = \lambda([e_{i,j}])$ for all $e_{i,j} \in \mathcal{B}_2$.\end{proof}\\

\begin{figure}
\begin{center}
\hspace{.3cm} \epsfxsize=3.1in \epsfbox{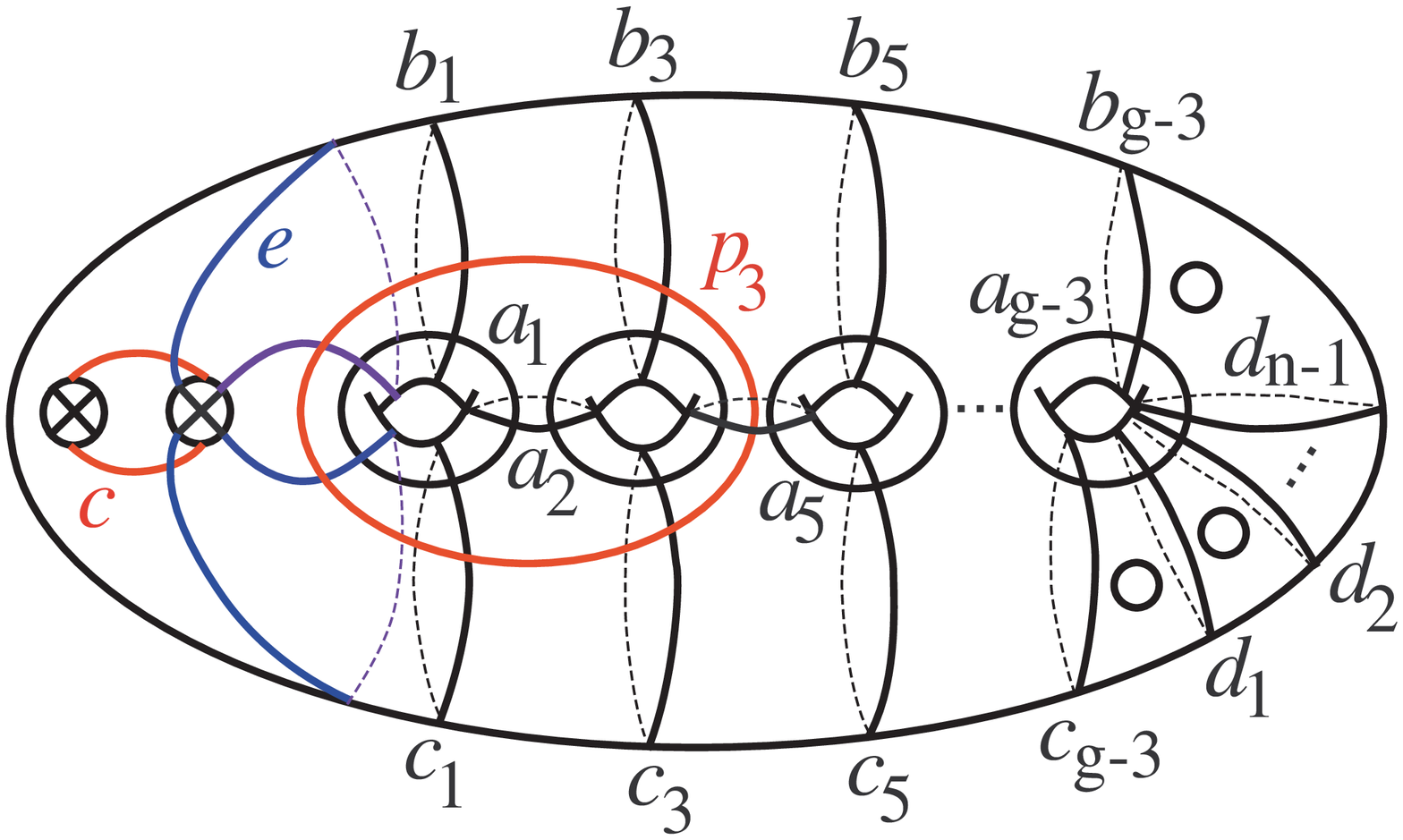}  \hspace{-1cm} \epsfxsize=3.1in \epsfbox{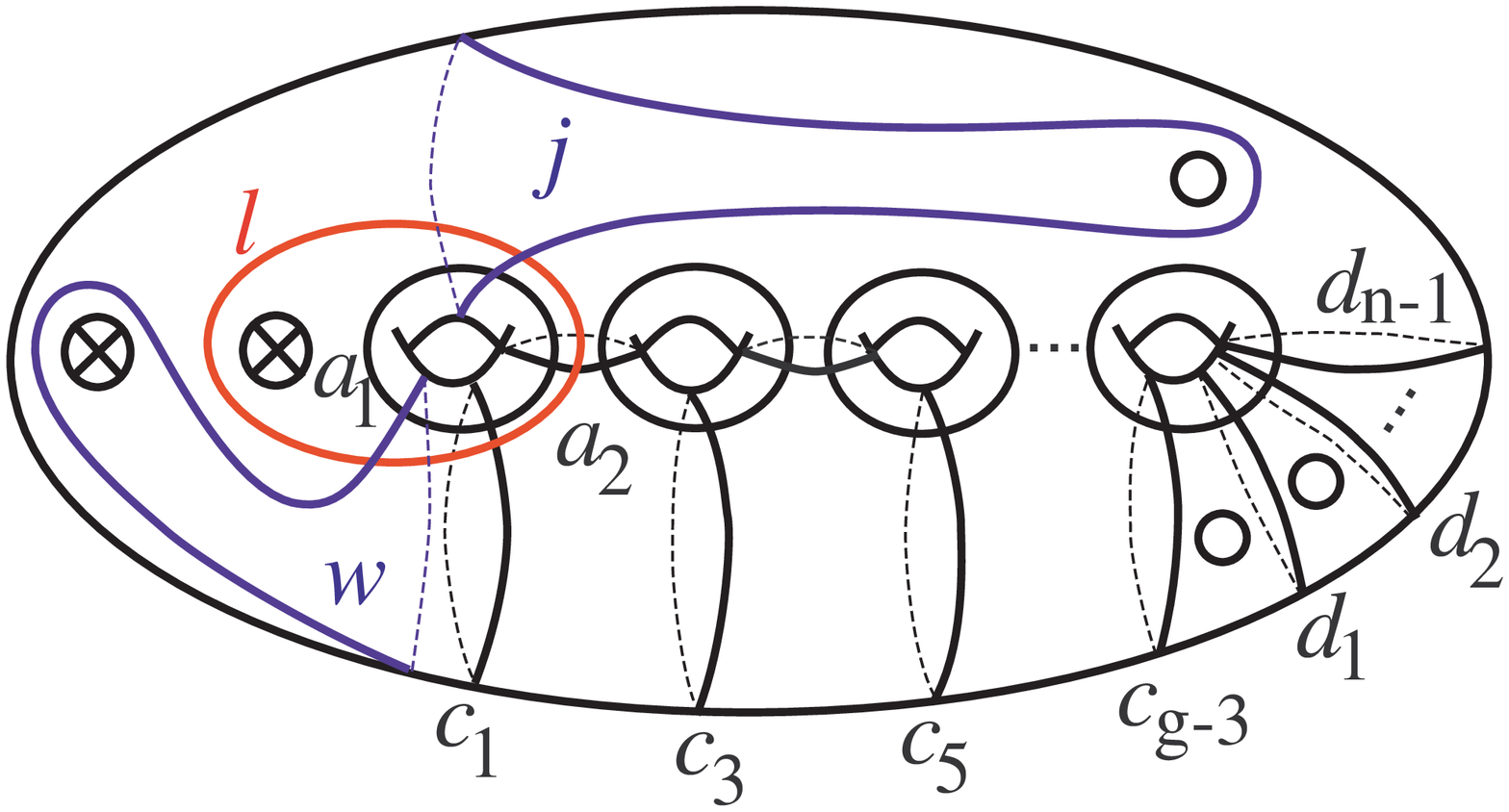}

\hspace{-1cm} (i) \hspace{6.5cm} (ii)
 
\caption{Curves } \label{fig-l1a}
\end{center}
\end{figure}

We will say that a subset $A \subset \mathcal{T}(N)$ has trivial stabilizer if it satisfies the following condition:
If $f \in Mod_N$ is such that $f([x]) = [x]$ for every vertex $x \in A$, then $f$ is identity.
When $g$ is even, $g \ge 6$, and $n \ge 0$ we consider the curves represented in Figure \ref{fig-l1a} and we set

\begin{gather*}
\CC_1 = \{ a_1, a_2, \cdots, a_{g-3}, b_1, b_3, \cdots, b_{g-3}, c_1, c_3, \cdots, c_{g-3}, d_1, d_2, \cdots, d_{n-1}, p_3, e\}\,,\\
\CC_2 = \{ a_1, a_2, \cdots, a_{g-3}, b_1, b_3 ,\cdots, b_{g-3}, c_1, c_3, \cdots, c_{g-3}, d_1, d_2, \cdots, d_{n-1}, p_3, l\}\,,\\
\CC_3 = \{ a_1, a_2, \cdots, a_{g-3}, c_1, c_3, \cdots, c_{g-3}, d_1, d_2, \cdots, d_{n-1}, p_3, j, l, c\}\,,\\
\CC_4 = \{ a_1, a_2, \cdots, a_{g-3}, c_1, c_3, \cdots, c_{g-3}, d_1, d_2, \cdots, d_{n-1}, p_3, j, w, c\}\,.
\end{gather*}

\begin{lemma} \label{abc} Suppose that $g$ is even, $g \ge 6$ and $n \ge 0$. The configurations  $\CC_1, \CC_2, \CC_3, \CC_4$ defined above have trivial stabilizers.
\end{lemma}

\begin{proof}
{\it Proof for $\CC_1$:}
The following assertion is known to experts for orientable surfaces (see Castel \cite[Proposition 2.1.3]{Caste1}, for instance).
It can be easily proved using Epstein \cite{Epste1} for non-orientable surfaces in the same way as for orientable surfaces.

\bigskip\noindent
{\it Assertion 1.}
Let $\{ x_1, \cdots, x_l \}$, $\{ y_1, \cdots, y_l \}$ be two collections of curves such that

(i) $x_1, \cdots, x_l$ (resp. $y_1, \cdots, y_l$) are pairwise nonisotopic;

(ii) $x_i$ is isotopic to $y_i$ for all $i \in \{1, \cdots, l \}$;

(iii) there are no three distinct indices $i, j, k \in \{ 1, \cdots, l \}$ such that $i ([x_i], [x_j]) \neq 0$, $i ([x_i], [x_k]) \neq 0$ and $i ([x_j], [x_k]) \neq 0$.
 
Then there exists a homeomorphism $h : N \to N$ isotopic to the identity such that $h(x_i) = y_i$ for all $i \in \{1, \cdots, l \}$.

\begin{figure}[htb]
\begin{center}
\hspace{1.5cm} \epsfxsize=3.3in \epsfbox{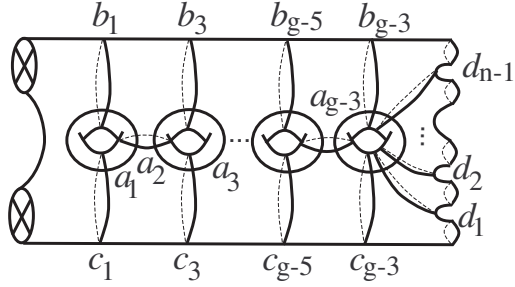}
\caption{The homeomorphism $\sigma$.} \label{fig1-l-b}
\end{center}
\end{figure} 
 
We denote by $\sigma$ the homeomorphism of $N$ of order $2$ defined by the symmetry with respect to the plane of the paper in the representation of $N$ given in Figure \ref{fig1-l-b}. 
We have $\sigma(x) = x$ for all $x \in \{ a_1, a_2, \cdots, a_{g-3}, b_1, b_3, \cdots, b_{g-3}, c_1, c_3, \cdots, c_{g-3}, d_1, d_2, \cdots, d_{n-1}, j \}$, and $\sigma$ pointwise fixes the curves $a_1, a_3, \cdots, a_{g-3}$. 
 
Let $h$ be a homeomorphism of $N$ such that $[h(x)] = [x]$ for every $x \in \CC_1$.
By Assertion 1, we can assume that $h(x) = x$ for all $x \in \{a_1, a_2, \cdots, a_{g-3}, b_1, b_3, \cdots,$ $ b_{g-3}, c_1, c_3,
\allowbreak
 \cdots, c_{g-3}, d_1, d_2, \cdots,$ $ d_{n-1} \}$.
Set $\Gamma = a_1 \cup \cdots \cup a_{g-3} \cup b_1 \cup \cdots \cup b_{g-3} \cup c_1 \cup \cdots \cup c_{g-3} \cup d_1 \cup \cdots \cup d_{n-1}$.
Then $h(\Gamma) = \Gamma$.
Cutting $N$ along $\Gamma$ we get a one holed Klein bottle, $K$, $g-4$ disks $D_1, \cdots, D_{g-4}$, and $n$ annuli $A_1, \cdots, A_n$.
For every $i \in \{1, \cdots, n\}$ one of the boundary components of $A_i$, denoted by $z_i$, is a boundary component of $N$.
The homeomorphism $h$ should send each piece $X \in \{ D_1, \cdots, D_{g-4}, A_1, \cdots, A_n \}$ onto itself. 
Moreover, either the restriction of $h$ to $X$ preserves the orientation for all $X \in \{D_1, \cdots, D_{g-4}, A_1, \cdots, A_n\}$, or the restriction of $h$ to $X$ reverses the orientation for all $X \in \{D_1, \cdots, D_{g-4}, A_1, \cdots, A_n \}$.
Suppose that the restriction of $h$ to $X$ preserves the orientation for all $X$.
Then $h$ should also preserve the orientation of each $x \in \{a_1, \cdots, a_{g-3}, b_1, \cdots, b_{g-3}, c_1, \cdots, c_{g-3}, d_1, \cdots, d_{n-1}\}$, hence we may assume that the restriction of $h$ to $\Gamma$ is the identity. 
Then the restriction of $h$ to $X$ is isotopic to the identity with an isotopy which pointwise fixes the boundary, if $X=D_i$ is a disk, and pointwise fixes the boundary component of $X$ different from $z_i$, if $X=A_i$ is an annulus.
So, in this case, we can assume that $h$ is the identity on $X$ for all $X \in \{D_1, \cdots, D_{g-4}, A_1, \cdots, A_n \}$.
Suppose that $h$ reverses the orientation of $X$ for all $X \in \{D_1, \cdots, D_{g-4}, A_1, \cdots, A_n\}$.
Then $h \sigma$ preserves the orientation of each $X$, and therefore, as above, we can assume that the restriction of $h \sigma$ to $X$ is the identity for all $X \in \{ D_1, \cdots, D_{g-4}, A_1, \cdots, A_n \}$.

\begin{figure}
\begin{center}
\hspace{-.3cm} \epsfxsize=2.7in \epsfbox{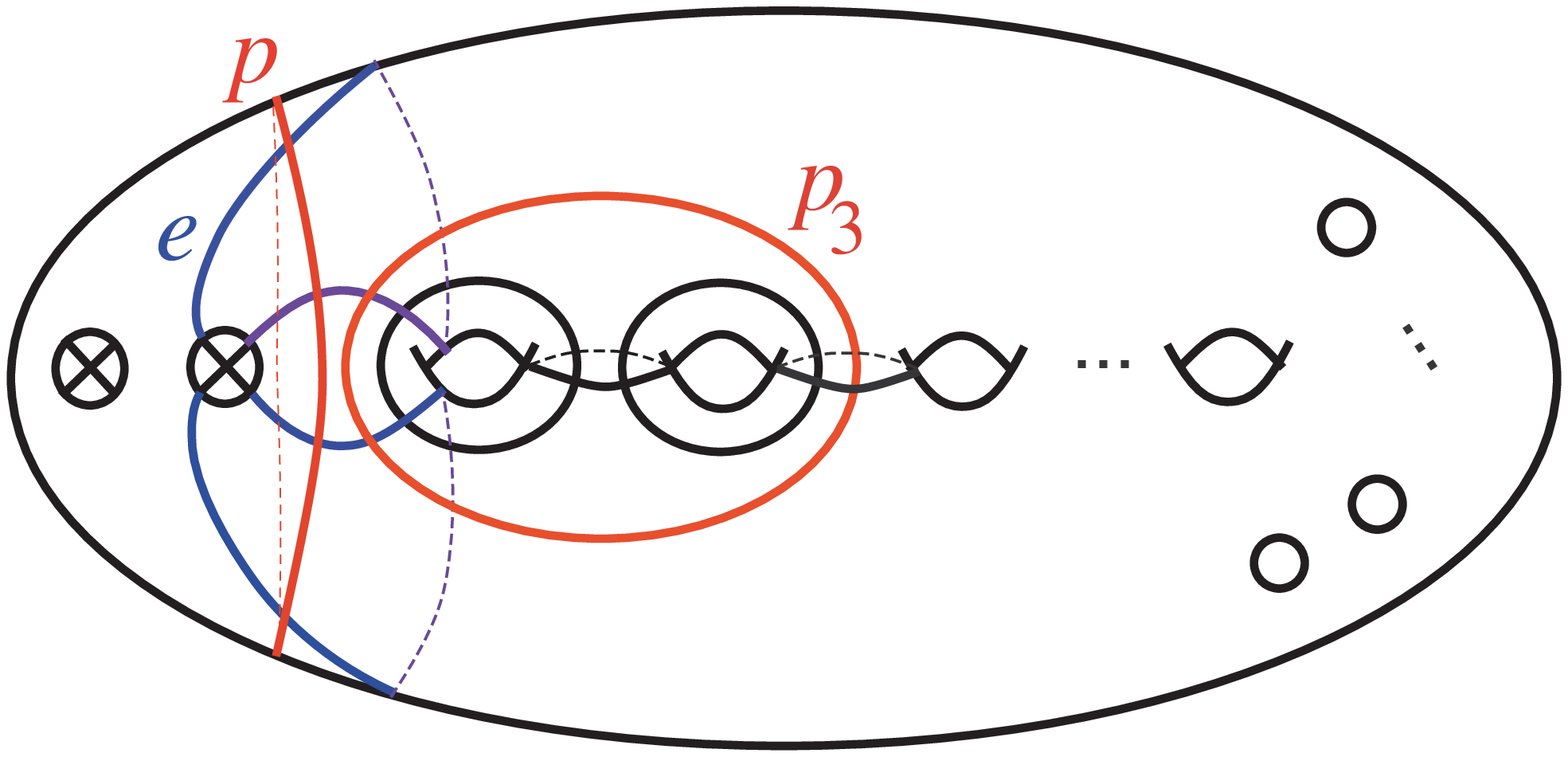}  \hspace{0cm}  \epsfxsize=2.7in \epsfbox{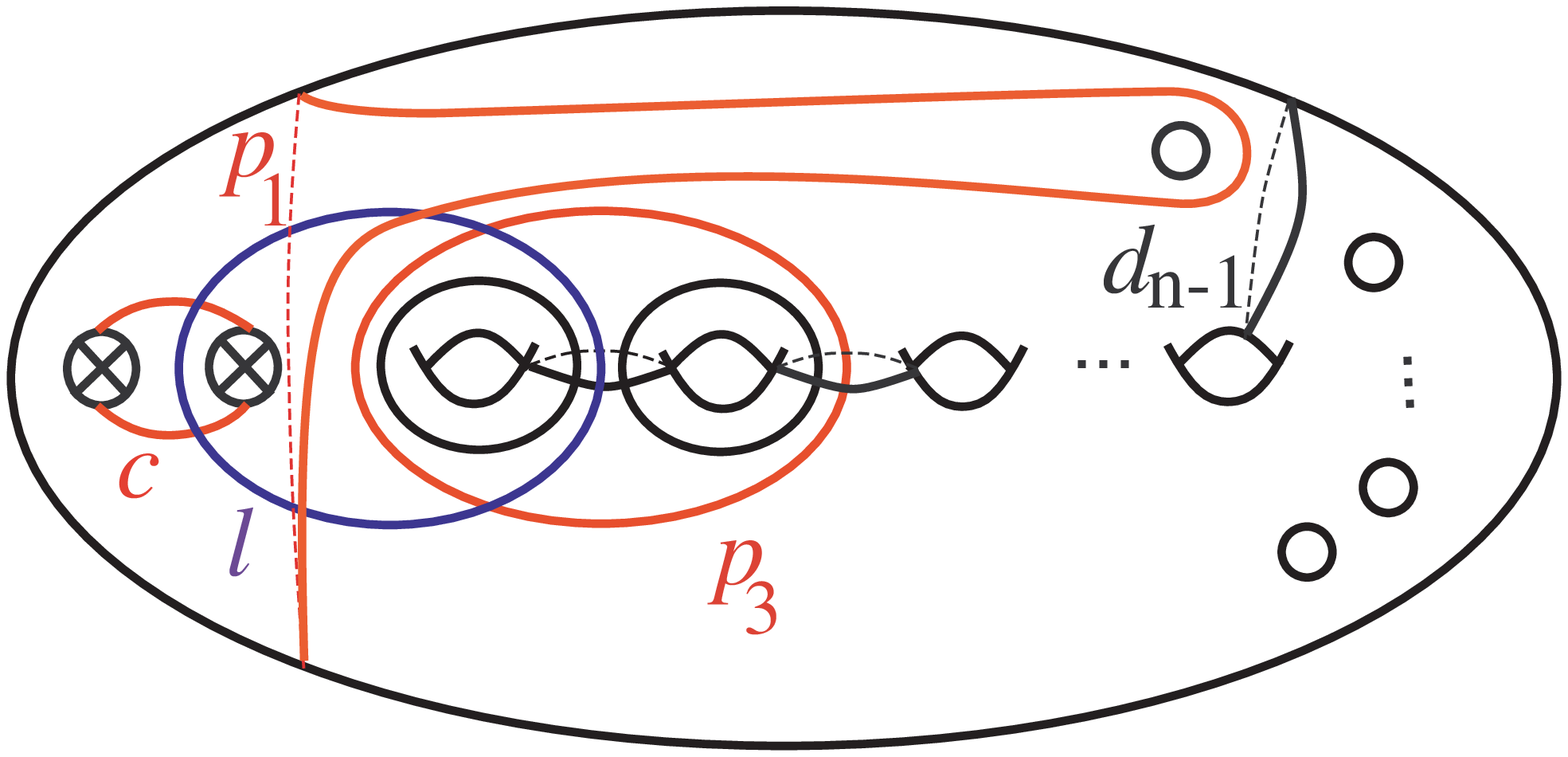}  

\hspace{-1cm} (i) \hspace{6.5cm} (ii)
\caption{Curves } \label{fig-l1c}
\end{center}
\end{figure}
 
Consider the curve $p$ drawn in Figure \ref{fig-l1c}. 
The curve $p$ is the only boundary component of a regular neighborhood of $\Gamma$ that bounds a one holed Klein bottle, hence we can assume that $h(p) = p$.
Cutting $N$ along $p$ we get a one holed Klein bottle, $K_0$, and an orientable surface of genus $\frac{g-2}{2}$ and $n+1$ boundary components, $M$.
We have $h(K_0) = K_0$ and $h(M)=M$.
Moreover, by the above, we can assume that either $h|_M = \Id_M$, or $h|_M = \sigma|_M$.
But $[\sigma (p_3)] \neq [p_3]$, hence $h|_M \neq \sigma_M$, and therefore $h|_M = \Id_M$.
Then we can also assume that $h$ is the identity on $p$.

Now, by Epstein \cite{Epste1}, we can assume that $h(e) = e$. The intersection $p \cap e$ has $4$ points that are fixed by $h$ since they lie in $p$. The curve $e$ has two arcs contained in $K_0$.
The extremities of each of these arcs are fixed and distinct, and the extremities of one arc are different from the extremities of the other one. 
Hence, $h$ preserves each arc and the orientation of each arc, so we can assume that $h$ is the identity on $e$. 
Cutting $K_0$ along $e$ we obtain a disk and a M\"obius band. The homeomorphism 
$h$ sends each of these pieces onto itself, and the restriction of $h$ to each piece is isotopic to the identity with an isotopy that pointwise fixes the boundary. 
So, we can assume that $h$ is the identity on $K_0$, and therefore that $h$ is the identity on the whole $N$.
 
{\it Proof for $\CC_2$:}
The proof for $\CC_2$ is similar to the proof for $\CC_1$.
 
{\it Proof for $\CC_3$:}  Let $h$ be a homeomorphism of $N$ such that $[h(x)] = [x]$ for every $x \in \CC_3$.
By Assertion 1 we may assume that $h(x) = x$ for all $x \in \{a_1, a_2, \cdots,$ $ a_{g-3},$ $ c_1, c_3, \cdots, \allowbreak c_{g-3}, d_1, d_2, \cdots, d_{n-1},j \}$.
Set $\Gamma = a_1 \cup \cdots \cup a_{g-3} \cup c_1 \cup \cdots \cup c_{g-3} \cup d_1 \cup \cdots \cup d_{n-1} \cup j$.
Then $h(\Gamma) = \Gamma$.
Cutting $N$ along $\Gamma$ we obtain a two holed Klein bottle, $K$, $\frac{g-2}{2}$ disks $D_1, \cdots, D_{\frac{g-2}{2}}$, and $n-1$ annuli, $A_1, \cdots, A_{n-1}$.
For every $i \in \{1, \cdots, n-1\}$ one of the boundary components of $A_i$, denoted by $z_i$, is a boundary component of $N$.
It is easily seen that the homeomorphism $h$ should send each piece $X \in \{ D_1, \cdots, D_{\frac{g-2}{2}}, A_1, \cdots, A_{n-1} \}$ onto itself.
Moreover, either the restriction of $h$ to $X$ preserves the orientation for all $X \in \{D_1, \cdots, D_{\frac{g-2}{2}}, A_1, \cdots, A_{n-1} \}$, or the restriction of $h$ to $X$ reverses the orientation for all $X \in \{D_1, \cdots, D_{\frac{g-2}{2}}, A_1, \cdots, A_{n-1} \}$.
Suppose that the restriction of $h$ to $X$ preserves the orientation for all $X$.
Then $h$ also preserves the orientation of each $x \in \{a_1, \cdots, a_{g-3}, c_1, \cdots, c_{g-3}, d_1, \cdots, d_{n-1}, j\}$, hence we can suppose that the restriction of $h$ to $\Gamma$ is the identity. 
Then the restriction of $h$ to $X$ is isotopic to the identity with an isotopy that pointwise fixes the boundary, if $X=D_i$ is a disk, and pointwise fixes the boundary component of $X$ different from $z_i$, if $X=A_i$ is an annulus.
So, in this case, we may assume that $h$ is the identity on $X$ for all $X \in \{D_1, \cdots, D_{\frac{g-2}{2}}, A_1, \cdots, A_{n-1} \}$.
Suppose that the restriction of $h$ to $X$ reverses the orientation for all $X \in \{D_1, \cdots, D_{\frac{g-2}{2}}, A_1, \cdots, A_{n-1}\}$.
Then the restriction of $h \sigma$ to $X$ preserves the orientation for all $X$, hence, again, we can assume that $h \sigma$ is the identity on $X$ for all $X \in \{ D_1, \cdots, D_{\frac{g-2}{2}}, A_1, \cdots, A_{n-1} \}$.
 
Consider the curve $p_1$ drawn in Figure \ref{fig-l1c}.
The curve $p_1$ is the only boundary component of a regular neighborhood of $\Gamma$ that is a boundary curve of a two holed Klein bottle whose second boundary curve is a boundary component of $N$.
Hence, we may assume that $h(p_1) = p_1$.
Cutting $N$ along $p_1$ we obtain a two holed Klein bottle, $K_0$, and an orientable surface of genus $\frac{g-2}{2}$ and $n$ boundary components, $M$.
We have $h(K_0) = K_0$ and $h(M)=M$.
Moreover, by the above, we can assume that either $h|_M = \Id_M$, or $h|_M = \sigma|_M$.
Since $[\sigma (l)] \neq [l]$, we have $h|_M \neq \sigma|_M$, hence $h|_M = \Id_M$.
Then we can also suppose that $h$ is the identity on $p_1$.
 
The intersection $p_1 \cap l$ has two points. The homeomorphism 
$h$ fixes these two points since they are included in $p_1$. The curve $l$ has a single arc included in $K_0$ whose extremities are distinct and fixed under $h$, hence $h$ preserves the arc and the orientation of this arc, and therefore we can suppose that $h$ is the identity on $l$. The intersection 
$l \cap c$ has two points, fixed by $h$ since they are included in $l$.
Cutting $c$ along $l$ we get two arcs, $x_1,x_2$.
Cutting $K_0$ along $l$ we get two pieces whose topologies are different, hence $h$ sends each of these pieces into itself.
Since each arc $x_i$ is contained in a different piece, we can suppose that $h$ sends $x_1$ onto $x_1$ and $x_2$ onto $x_2$.
In this situation $h$ should also preserve the orientation of each $x_i$, hence we can assume that $h$ is the identity on $c$. 
 
Cutting $K_0$ along $l \cup c$ we obtain a disk, $K_0$, and an annulus, $A_0$.
One of the boundary components of $A_0$, denoted by $z_0$, is a boundary component of $N$. The homeomorphism $h$ sends each of these pieces onto itself, the restriction of $h$ to $D_0$ is isotopic to the identity with an isotopy that pointwise fixes the boundary, and the restriction of $h$ to $A_0$ is isotopic to the identity with an isotopy that pointwise fixes the boundary component different from $z_0$.
So, we can suppose that $h$ is the identity on $K_0$, that is, $h$ is the identity on the whole $N$.
 
{\it Proof for $\CC_4$:}
The proof for $\CC_4$ is similar to the proof for $\CC_3$.
\end{proof}

\begin{figure}[htb]
\begin{center}
\hspace{1.1cm} \epsfxsize=3.2in \epsfbox{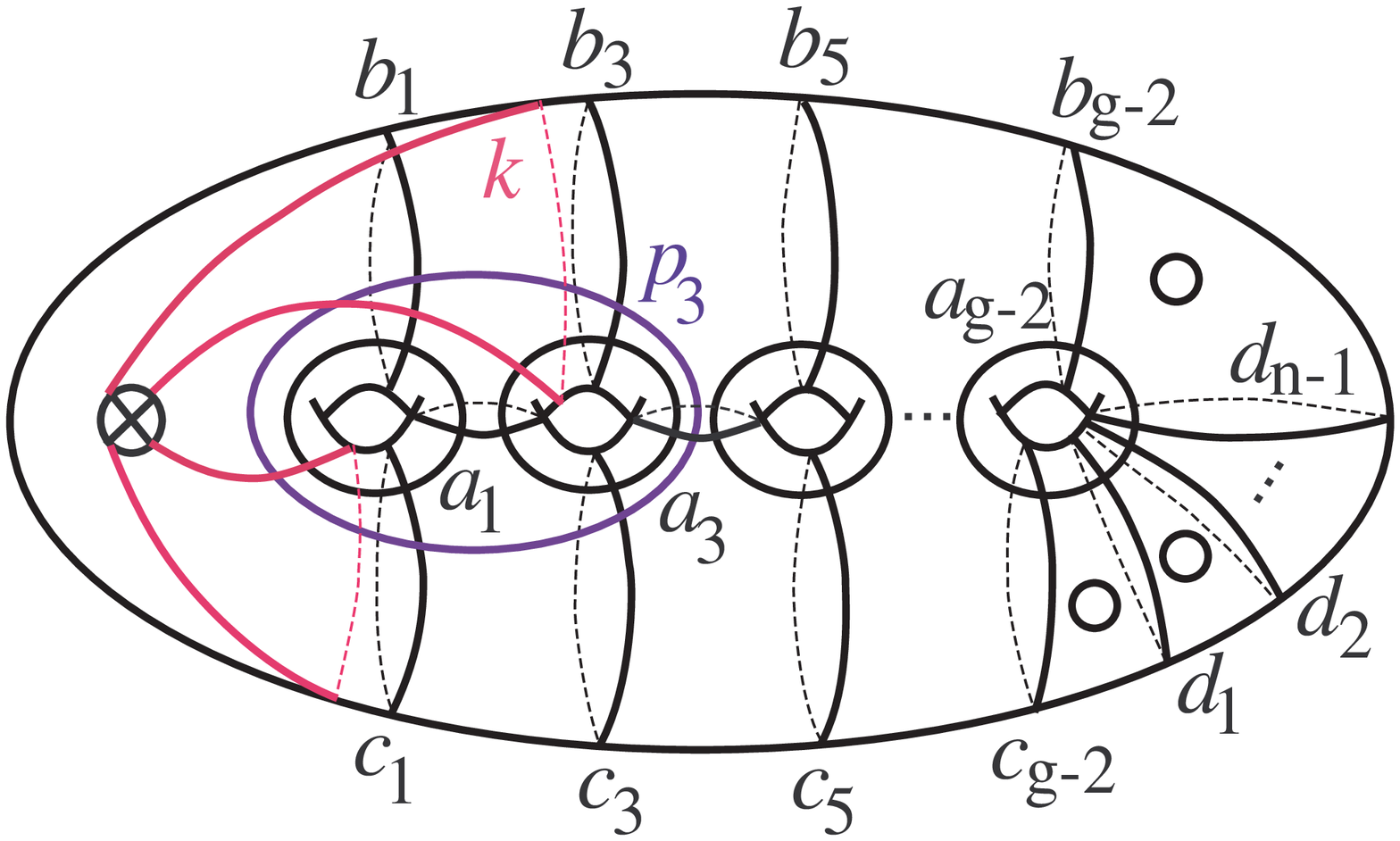}

\hspace{-0.2cm} (i) 

\hspace{.2cm} \epsfxsize=2.7in \epsfbox{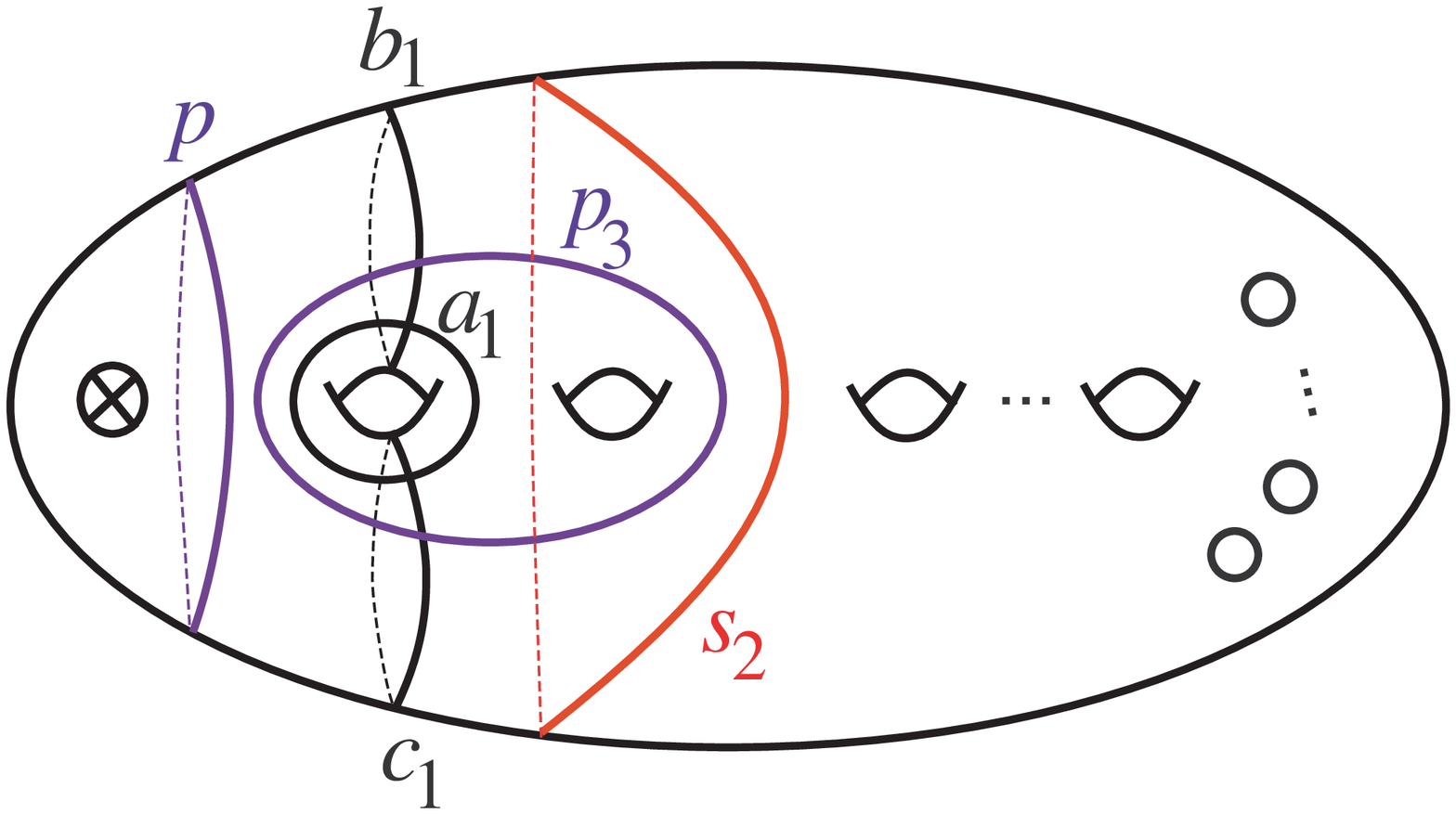}   \epsfxsize=2.8in \epsfbox{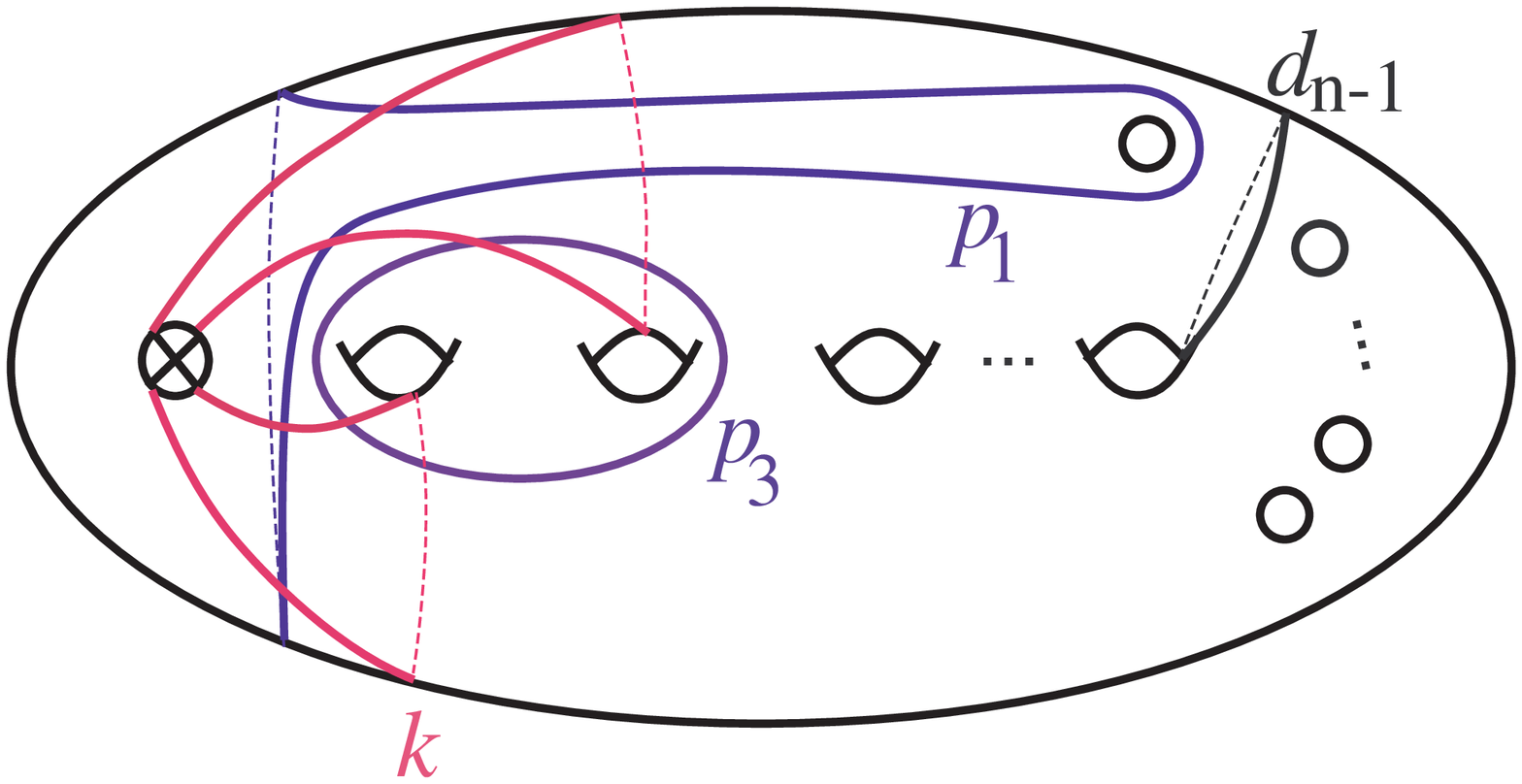}

\hspace{0.2cm} (ii) \hspace{6.cm} (iii)
\caption{Curves on $N$} \label{fig-l5}
\end{center}
\end{figure} 

Now, we assume that $g$ is odd, $g \ge 5$ and $n \ge 0$.
We consider the curves drawn in Figure \ref{fig-l5} (i) and we set
\begin{gather*}
\CC_1 = \{ a_1, a_2, \cdots, a_{g-2}, b_1, b_3, \cdots, b_{g-2}, c_1, c_3, \cdots, c_{g-2}, d_1, d_2, \cdots, d_{n-1}, p_3\}\,,\\
\CC_2 = \{ a_1, a_2, \cdots, a_{g-2}, b_1, b_3, \cdots, b_{g-2}, c_1, c_3, \cdots, c_{g-2}, d_1, d_2, \cdots, d_{n-1}, s_2\}\,,\\
\CC_3 = \{ a_1, a_2, \cdots, a_{g-2}, c_1, c_3, \cdots, c_{g-2}, d_1, d_2, \cdots, d_{n-1}, p_3, k\}\,.
\end{gather*}

\begin{figure}[htb]
\begin{center}
\hspace{1.5cm} \epsfxsize=3.3in \epsfbox{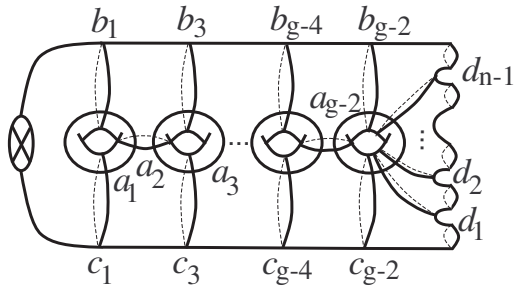}
\caption{The homeomorphism $\sigma$.} \label{fig1-l-c}
\end{center}
\end{figure} 

\begin{lemma} \label{abcd} Suppose that $g$ is odd, $g \geq 5$ and $n \geq 0$. The above defined collections $\CC_1, \CC_2, \CC_3$ have trivial stabilizers.
\end{lemma}
 
\begin{proof} {\it Proof for $\CC_1$.}
We denote by $\sigma$ the homeomorphism of $N$ of order two defined by the symmetry with respect to the plane of the paper in the representation of $N$ given in Figure \ref{fig1-l-c}.
We have $\sigma(x) = x$ for all $x \in \{ a_1, a_2, \cdots, a_{g-2}, b_1, b_3, \cdots, b_{g-2}, c_1, c_3, \cdots,$ $ c_{g-2}, d_1, d_2, \cdots, d_{n-1} \}$, and $\sigma$ pointwise fixes the curves $a_1, a_3, \cdots, a_{g-2}$.
 
Let $h$ be a homeomorphism of $N$ such that $[h(x)] = [x]$ for every $x \in \CC_1$.
By Assertion 1 in the proof of Lemma \ref{abc}, we can assume that $h(x) = x$ for all $x \in \{a_1, a_2, \cdots, a_{g-2}, b_1, b_3, \cdots, b_{g-2}, c_1,$ $ c_3, \cdots, c_{g-2}, d_1, d_2, \cdots, d_{n-1} \}$.
We set $\Gamma = a_1 \cup \cdots \cup a_{g-2} \cup b_1 \cup \cdots \cup b_{g-2} \cup c_1 \cup \cdots \cup c_{g-2} \cup d_1 \cup \cdots \cup d_{n-1}$.
Then $h(\Gamma) = \Gamma$.
Cutting $N$ along $\Gamma$ we get a M\"obius band, $K$, $g-3$ disks, $D_1, \cdots, D_{g-3}$, and $n$ annuli, $A_1, \cdots, A_n$.
For every $i \in \{1, \cdots, n\}$ one of the boundary components of $A_i$, denoted by $z_i$, is a boundary component of $N$.
The homeomorphism $h$ should send each piece $X \in \{ D_1, \cdots, D_{g-3}, A_1, \cdots, A_n \}$ onto itself.
Moreover, either the restriction of $h$ to $X$ preserves the orientation for all $X \in \{D_1, \cdots, D_{g-3}, A_1, \cdots, A_n\}$, or the restriction of $h$ to $X$ reverses the orientation for all $X \in \{D_1, \cdots, D_{g-3}, A_1, \cdots, A_n \}$.
Suppose that the restriction of $h$ to $X$ preserves the orientation for all $X$.
Then $h$ should also preserve the orientation of each $x \in \{a_1, \cdots, a_{g-2}, b_1, \cdots, b_{g-2}, c_1, \cdots, c_{g-2}, d_1, \cdots, d_{n-1}\}$, hence we may assume that the restriction of $h$ to $\Gamma$ is the identity.
Then the restriction of $h$ to $X$ is isotopic to the identity with an isotopy that pointwise fixes the boundary, if $X=D_i$ is a disk, and pointwise fixes the boundary of $X$ different from $z_i$, if $X=A_i$ is an annulus.
So, in this case, we can assume that $h$ is the identity on $X$ for all $X \in \{D_1, \cdots, D_{g-3}, A_1, \cdots, A_n \}$.
Suppose that $h$ reverses the orientation of $X$ for all $X \in \{D_1, \cdots, D_{g-3}, A_1, \cdots, A_n\}$.
Then $h \sigma$ preserves the orientation of each $X$, and therefore, as above, we can assume that the restriction of $h \sigma$ to  $X$ is the identity for all $X \in \{ D_1, \cdots, D_{g-3}, A_1, \cdots, A_n \}$.
 
Consider the curve $p$ drawn in Figure \ref{fig-l5} (ii).
The curve $p$ is the only boundary component of a regular neighborhood of $\Gamma$ which bounds a M\"obius band, hence we can assume that $h(p) = p$.
Cutting $N$ along $p$, we get a M\"obius band, $K_0$, and an orientable surface of genus $\frac{g-1}{2}$ and $n+1$ boundary components, $M$.
We have $h(K_0) = K_0$ and $h(M)=M$.
Moreover, by the above, we can suppose that either $h|_M = \Id_M$, or $h|_M = \sigma|_M$.
But $[\sigma (p_3)] \neq [p_3]$, hence $h|_M \neq \sigma|_M$, thus $h|_M = \Id_M$.
Then we can also assume that $h$ is the identity on $p$.
Now, the restriction of $h$ to $K_0$ is isotopic to the identity with an isotopy that pointwise fixes the boundary, hence we can assume that $h|_{K_0} = \Id_{K_0}$, that is, $h$ is the identity. 

{\it Proof for $\CC_2$.} The proof for $\CC_2$ is similar to the proof for $\CC_1$.

{\it Proof for $\CC_3$.} Let $h$ be a homeomorphism of $N$ such that $[h(x)] = [x]$ for every $x \in \CC_3$. By Assertion 1 in the proof of Lemma \ref{abc}, we can assume that $h(x) = x$ for all $x \in \{a_1, a_2, \cdots, a_{g-2}, c_1, c_3, \cdots, c_{g-2}, d_1, d_2, \cdots, d_{n-1} \}$.
We set $\Gamma = a_1 \cup \cdots \cup a_{g-2} \cup c_1 \cup \cdots \cup c_{g-2} \cup d_1 \cup \cdots \cup d_{n-1}$.
Then $h(\Gamma) = \Gamma$.
Cutting $N$ along $\Gamma$ we obtain a non-orientable surface of genus $1$ with two boundary components, $K$, $\frac{g-3}{2}$ disks $D_1, \cdots, D_{\frac{g-3}{2}}$, and $n-1$ annuli $A_1, \cdots, A_{n-1}$.
For all $i \in \{1, \cdots, n-1\}$ one of the boundary components of $A_i$, denoted by $z_i$, is a boundary component of $N$.
It is easily seen that $h$ should send each piece $X \in \{ D_1, \cdots, D_{\frac{g-3}{2}}, A_1, \cdots, A_{n-1} \}$ into itself.
Moreover, either the restriction of $h$ to $X$ preserves the orientation for all $X \in \{D_1, \cdots, D_{\frac{g-3}{2}}, A_1, \cdots, A_{n-1} \}$, or the restriction of $h$ to $X$ reverses the orientation for all $X \in \{D_1, \cdots, D_{\frac{g-3}{2}}, A_1, \cdots, A_{n-1} \}$.
Suppose that the restriction of $h$ to $X$ preserves the orientation for all $X$.
Then $h$ preserves the orientation of each $x \in \{a_1, \cdots, a_{g-2}, c_1, \cdots, c_{g-2}, d_1, \cdots, d_{n-1}\}$, and therefore we can assume that $h$ is the identity on $\Gamma$. 
Then the restriction of $h$ to $X$ is isotopic to the identity with an isotopy that pointwise fixes the boundary, if $X=D_i$ is a disk, and pointwise fixes the boundary component of $X$ different from $z_i$, if $X=A_i$ is an annulus.
So, in this case, we can assume that $h$ is the identity on $X$ for all $X \in \{D_1, \cdots, D_{\frac{g-3}{2}}, A_1, \cdots, A_{n-1} \}$.
Suppose that $h$ reverses the orientation of $X$ for all $X$.
Then the restriction of $h \sigma$ to $X$ preserves the orientation for all $X \in \{ D_1, \cdots, D_{\frac{g-3}{2}}, A_1, \cdots, A_{n-1} \}$, hence, as above, we can assume that $h \sigma$ is the identity on $X$ for all $X$.

Consider the curve $p_1$ drawn in Figure \ref{fig-l5} (iii).
The curve $p_1$ is the only boundary component of a regular neighborhood of $\Gamma$ which is a boundary component of a genus $1$ two holed non-orientable subsurface whose second boundary component is a boundary component of $N$, hence we may assume that $h(p_1) = p_1$.
Cutting $N$ along $p_1$ we obtain a two holed non-orientable surface of genus $1$, $K_0$, and an orientable surface of genus $\frac{g-1}{2}$ and $n$ boundary components, $M$.
We have $h(K_0) = K_0$ and $h(M)=M$.
Moreover, by the above, we can suppose that either $h|_M = \Id_M$, or $h|_M = \sigma|_M$.
Since $[\sigma (p_3)] \neq [p_3]$, $h|_M \neq \sigma|_M$, hence $h|_M = \Id_M$.
Then we can also assume that $h$ is the identity on $p_1$.

The intersection $p_1 \cap k$ has $4$ points.
They are fixed by $h$, since they are included in $p_1$. The curve 
$k$ has two arcs included in $K_0$.
It is easily shown using Epstein \cite{Epste1} that the image of each arc is isotopic to itself with respect to the extremities, hence we may assume that $h$ sends each arc onto itself.  
Each arc has two different extremities, and the extremities of one arc are different from the extremities of the other, hence $h$ preserves the orientation of each arc, therefore we can assume that $h$ is the identity on $k$.
Cutting $K_0$ along $k$ we get a disk, $D_0$, and an annulus, $A_0$.
One of the boundary components of $A_0$, denoted by $z_0$, is a boundary component of $N$.
It is clear that $h$ must send each of these pieces to itself.
Moreover, the restriction of $h$ to $D_0$ is isotopic to the identity with an isotopy that pointwise fixes the boundary, and the restriction of $h$ to $A_0$ is isotopic to the identity with an isotopy that pointwise fixes the boundary component of $A_0$ different from $z_0$. So, we can assume that $h$ is the identity on $K_0$, hence $h$ is the identity on the 
whole $N$.\end{proof}\\
  
\begin{figure}
\begin{center}
\epsfxsize=3.3in \epsfbox{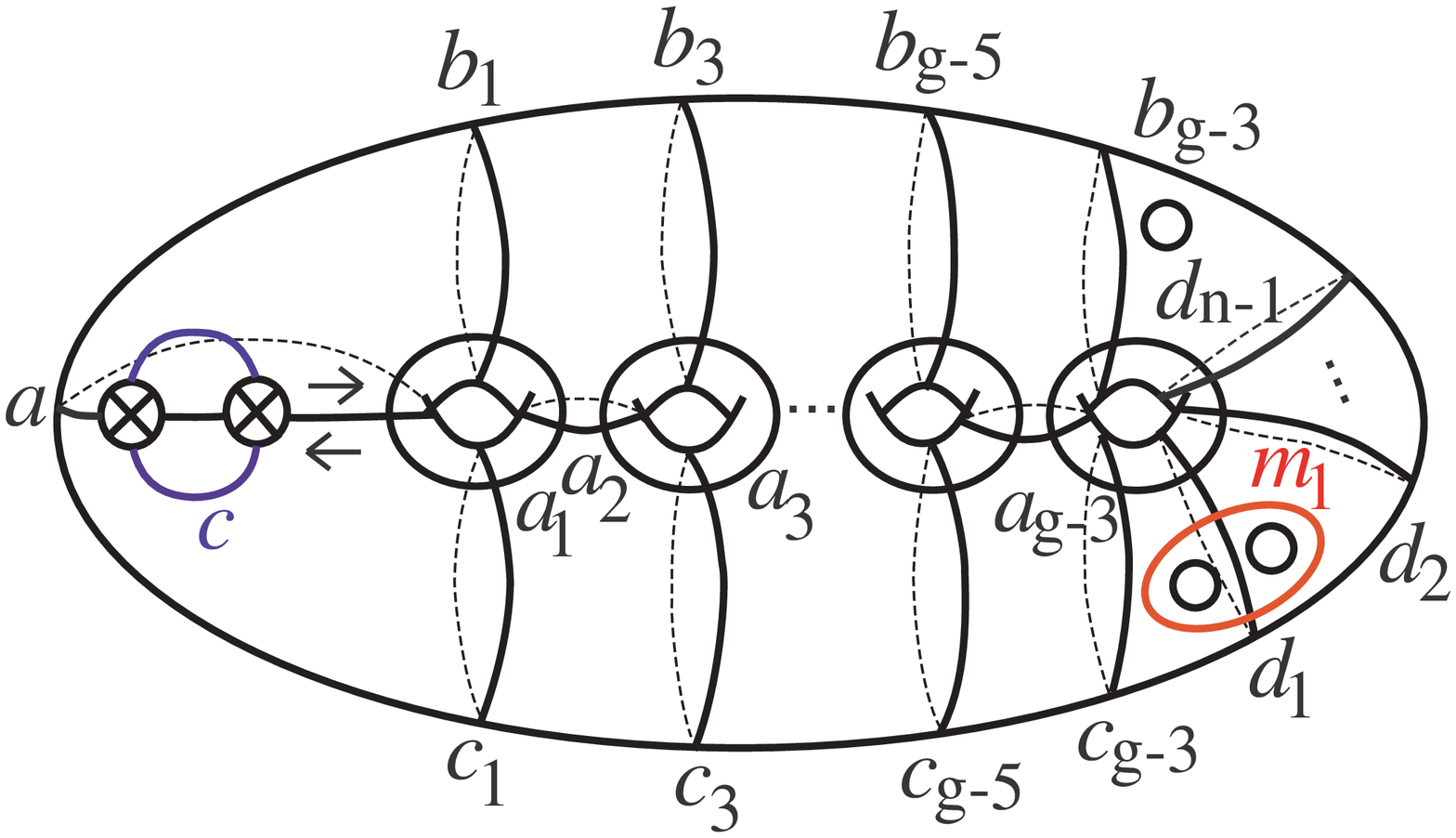}\hspace{-1cm} \epsfxsize=3.1in \epsfbox{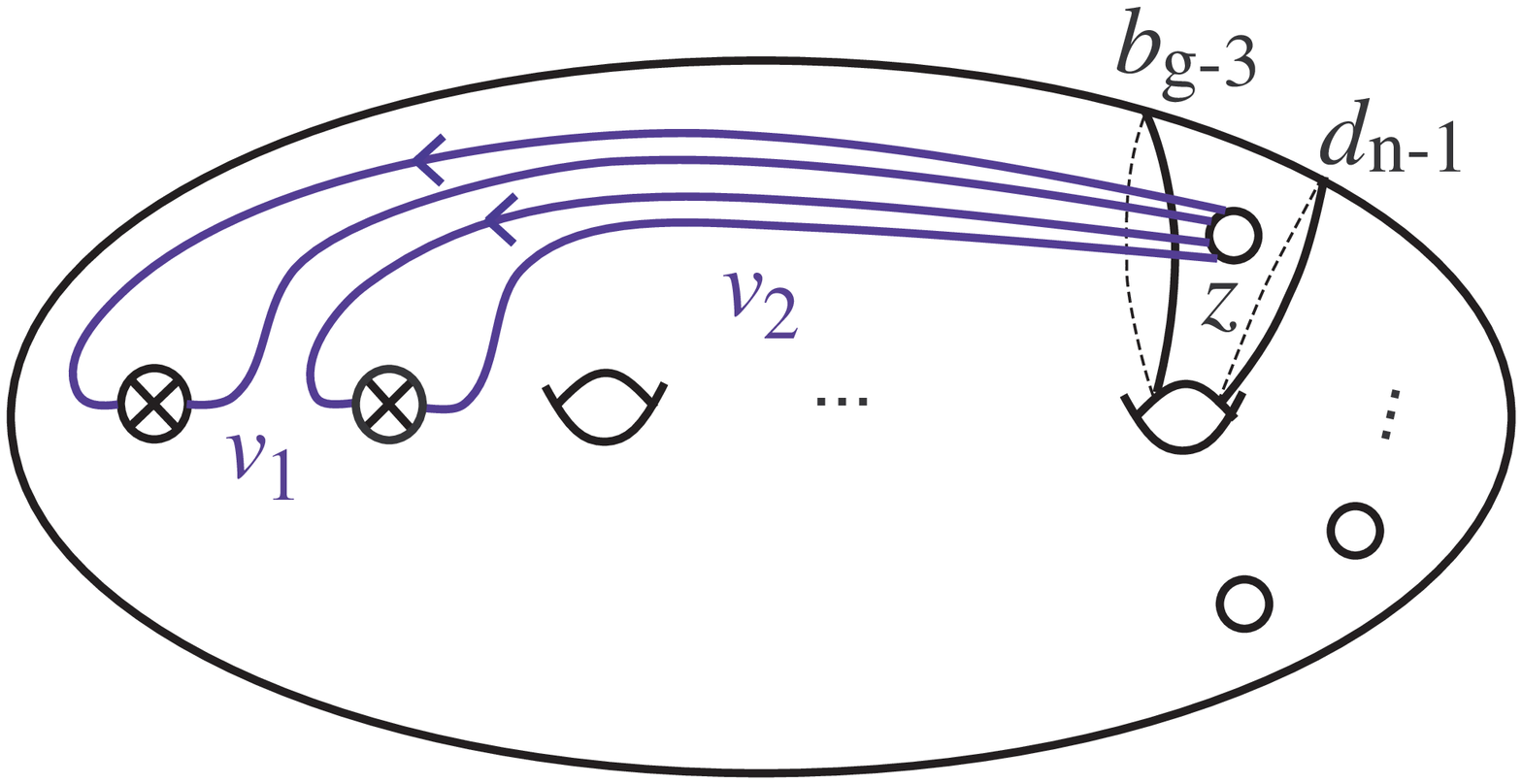}

\hspace{-1cm} (i) \hspace{6.5cm} (ii)

\hspace{0.1cm} \epsfxsize=3.1in \epsfbox{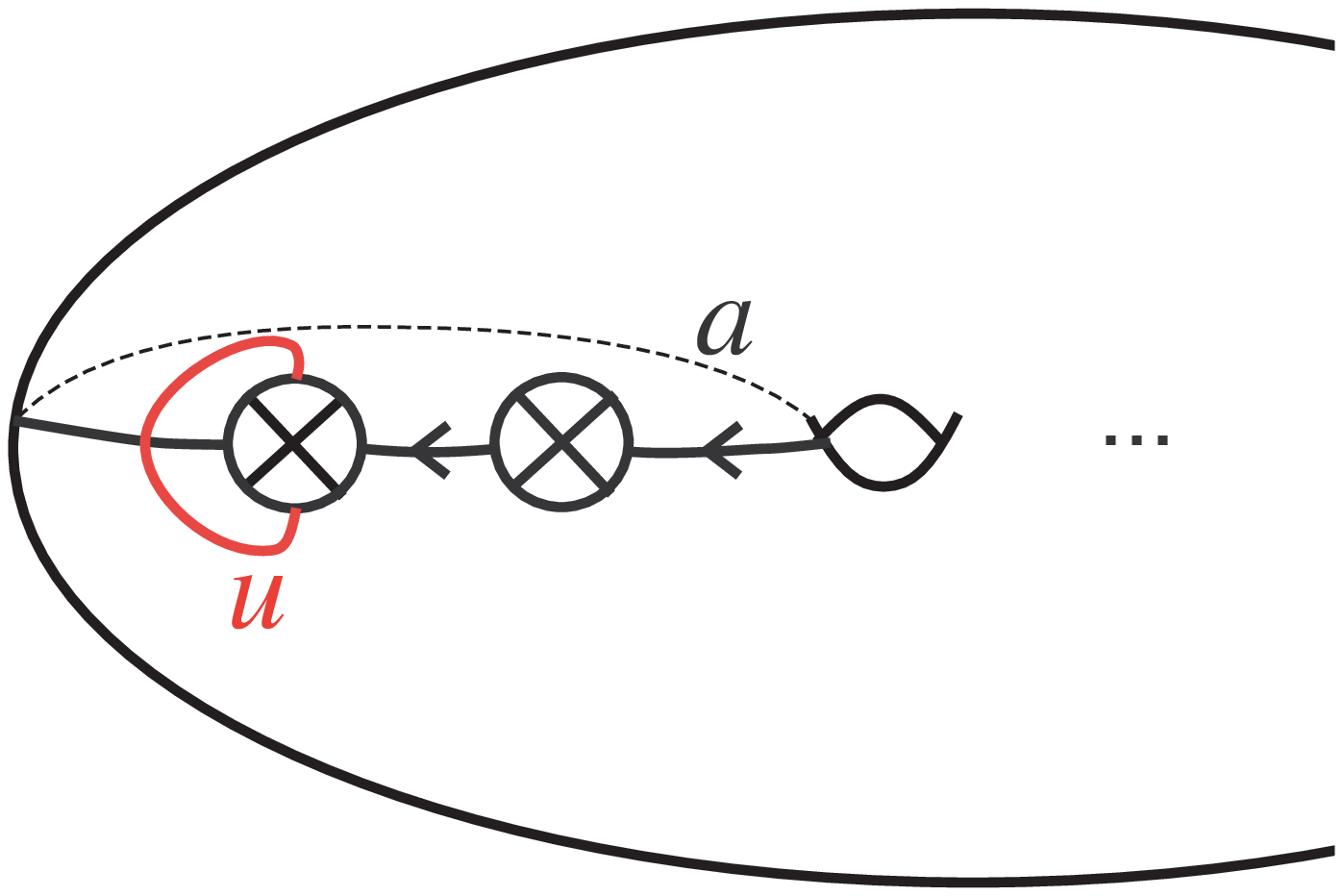} \hspace{-1cm} \epsfxsize=3.2in \epsfbox{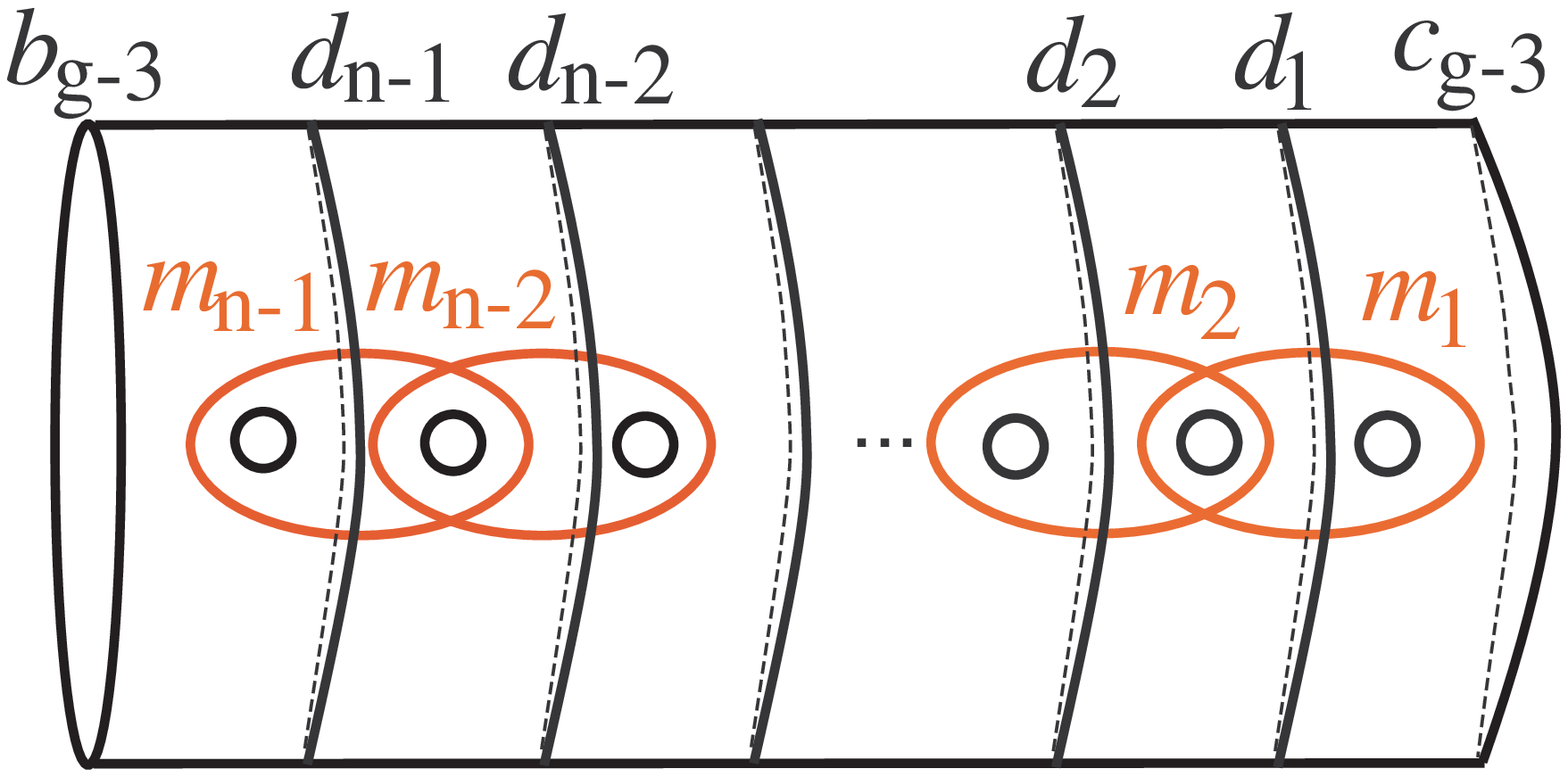}

\hspace{-1cm} (iii) \hspace{6.5cm} (iv)
 
\caption{Curve configuration XIII} \label{fig35}
\end{center}
\end{figure}

In the following lemma we consider the curves given in Figure \ref{fig35}. Let $t_x$ be the
Dehn twist about $x$. Let $\sigma_i$ be the half twist along $m_i$ (i.e. the elementary braid supported in a regular neighborhood of an arc connecting the two boundary components of $N$ that $m_i$ separates so that $\sigma_i ^2$ is a right Dehn twist along $m_i$.) Let $y= y_{u, a}$ be the crosscap slide of $u$ along $a$. We recall that the crosscap slide is 
defined as follows: Consider a M\"{o}bius band $A$ with one hole. By attaching another M\"{o}bius band $B$ to $A$ 
along one of the boundary components of $A$ we get a Klein bottle $K$ with one hole. By sliding $B$ once along the 
core of $A$, we get a homeomorphism of $K$ which is identity on the boundary of $K$. This homeomorphism can be 
extended by identity to a homeomorphism of the whole surface if $K$ is embedded in a surface. The isotopy class of this homeomorphism is called a crosscap slide. Let $\xi_1$ be the boundary slide of $z$ along $v_1$, and $\xi_2$ be the boundary slide of $z$ along $v_2$. We recall that  boundary slide $\xi_1$ is defined as follows: Let $A$ be a regular neighborhood of $z \cup v_1$ on $N$. Then $A$ is a M\"{o}bius band with one hole $z$ (a nonorientable surface of genus one with two boundary components). Let $k$ be the other boundary component of $A$. By sliding $z$ along the core of $A$ (in the direction of $v_1$) we get a homeomorphism which is identity on $k$. This homeomorphism can be extended by identity to a homeomorphism of $N$. The isotopy class of this homeomorphism is called boundary slide of $z$ along $v_1$.
  
\begin{lemma}
\label{prop1} If $g \geq 6$ and $g$ is even, then $Mod_N$ is generated by $\{t_x: x \in \{a, c, a_1, a_2,$ $\cdots, a_{g-3}, b_1, b_3, \cdots, b_{g-3}, c_1, c_3, \cdots, c_{g-3},$
$ d_1, d_2, \cdots, d_{n-1}\} \} \cup \{\sigma_1, \sigma_2, \cdots, \sigma_{n-1}, y, \xi_1, \xi_2\}$.\end{lemma}

\begin{proof} The proof follows from the proof of Theorem 4.14 given by Korkmaz in \cite{K2} by observing that his boundary slides and our boundary slides are conjugate to each other where the conjugating elements are products of half twists that we consider.\end{proof}\\

Let $G = \{t_x: x \in \{a, c, a_1, a_2,$ $ \cdots, a_{g-3}, b_1, b_3, \cdots, b_{g-3}, c_1, c_3, \cdots, c_{g-3},
d_1, d_2, \cdots,$ $ d_{n-1}\} \} \cup \{\sigma_1, \sigma_2, \cdots, \sigma_{n-1}, y, \xi_1, \xi_2\}$, where the curves are as shown in Figure \ref{fig35}.

\begin{figure}
\begin{center}
\hspace{-0.4cm} \epsfxsize=2.7in \epsfbox{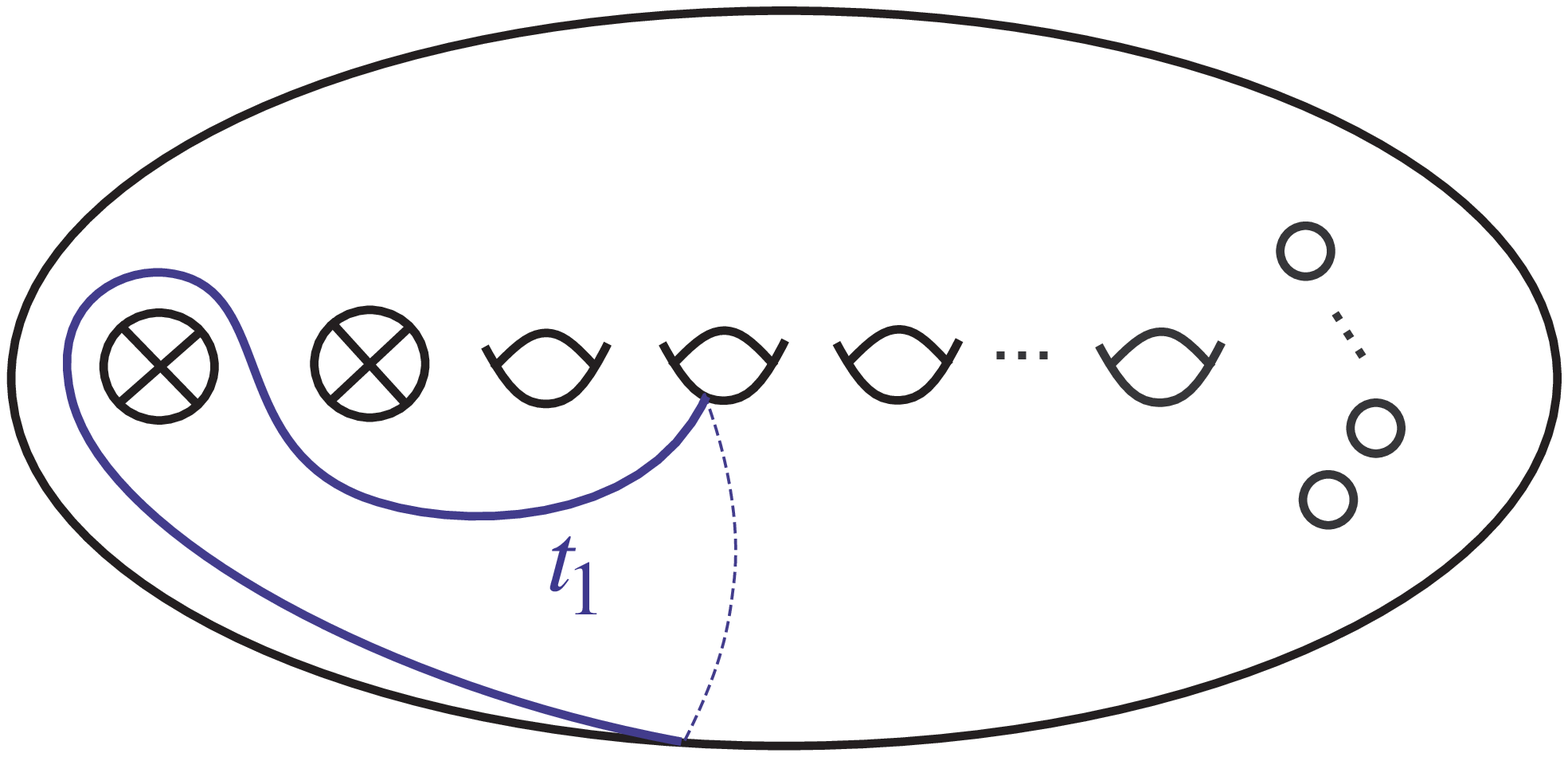} \hspace{0.2cm} \epsfxsize=2.7in \epsfbox{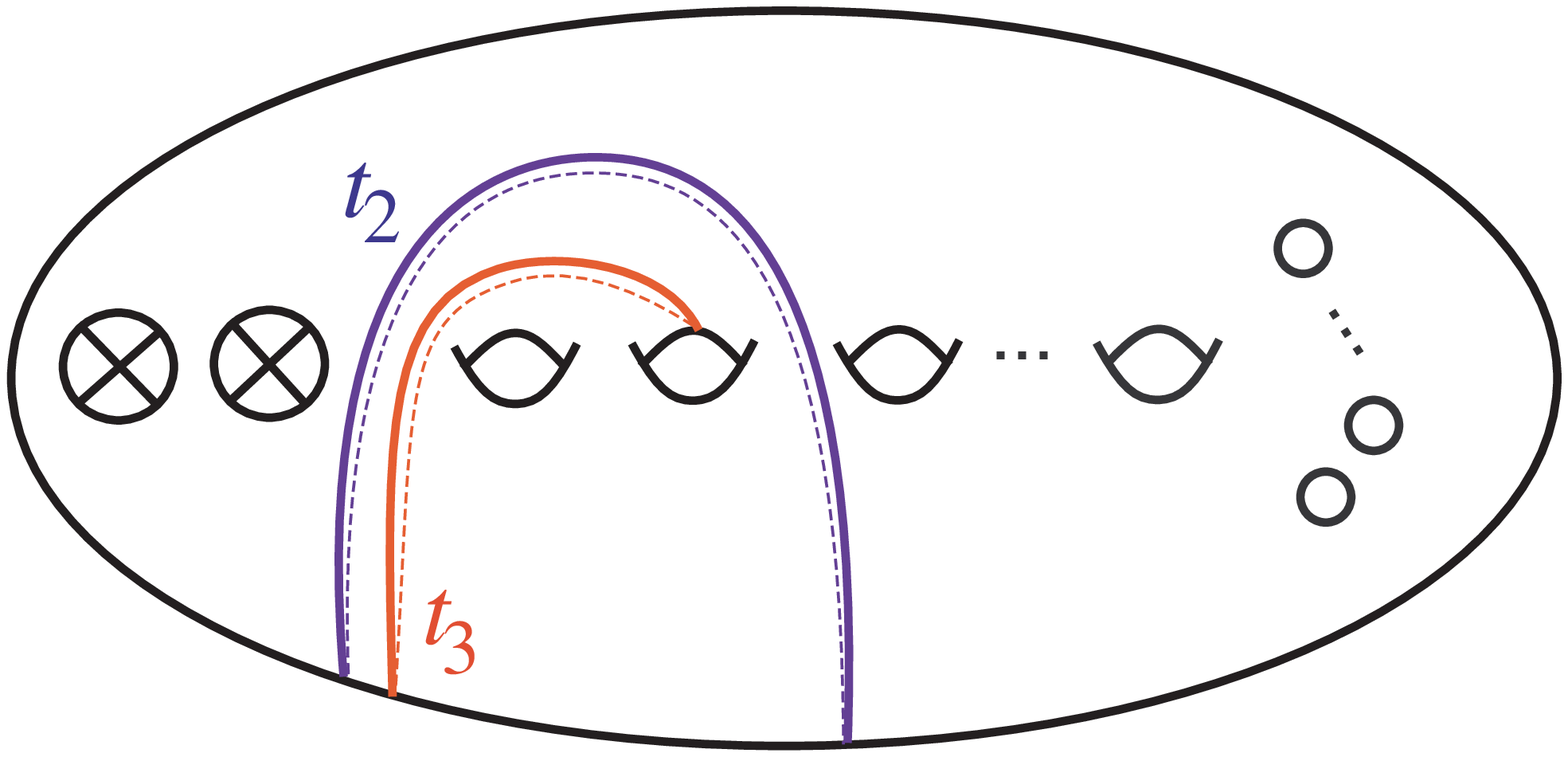} 

\hspace{-1cm} (i) \hspace{6.5cm} (ii)

\hspace{-0.4cm} \epsfxsize=2.7in \epsfbox{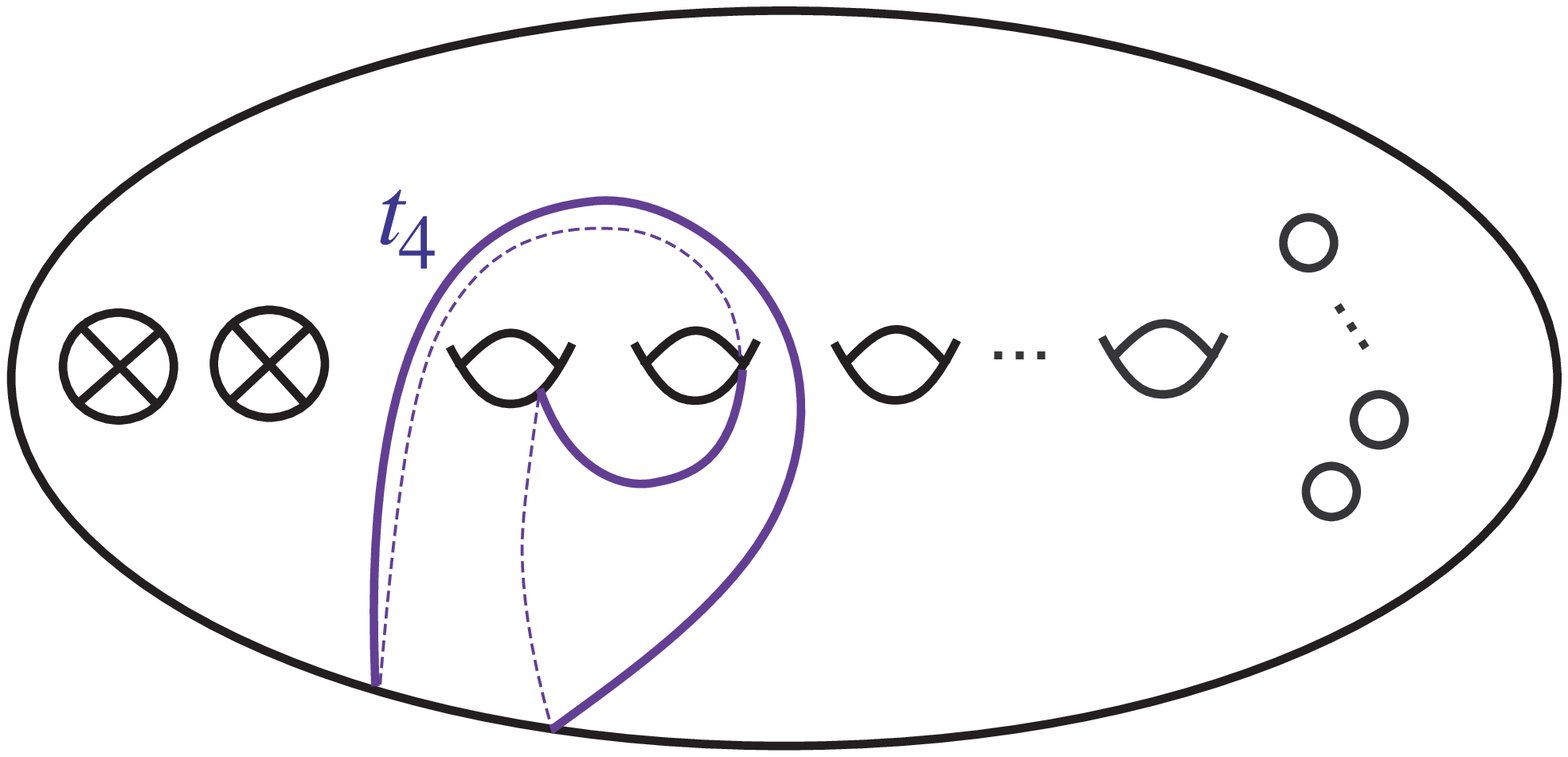} \hspace{0.2cm}  \epsfxsize=2.7in \epsfbox{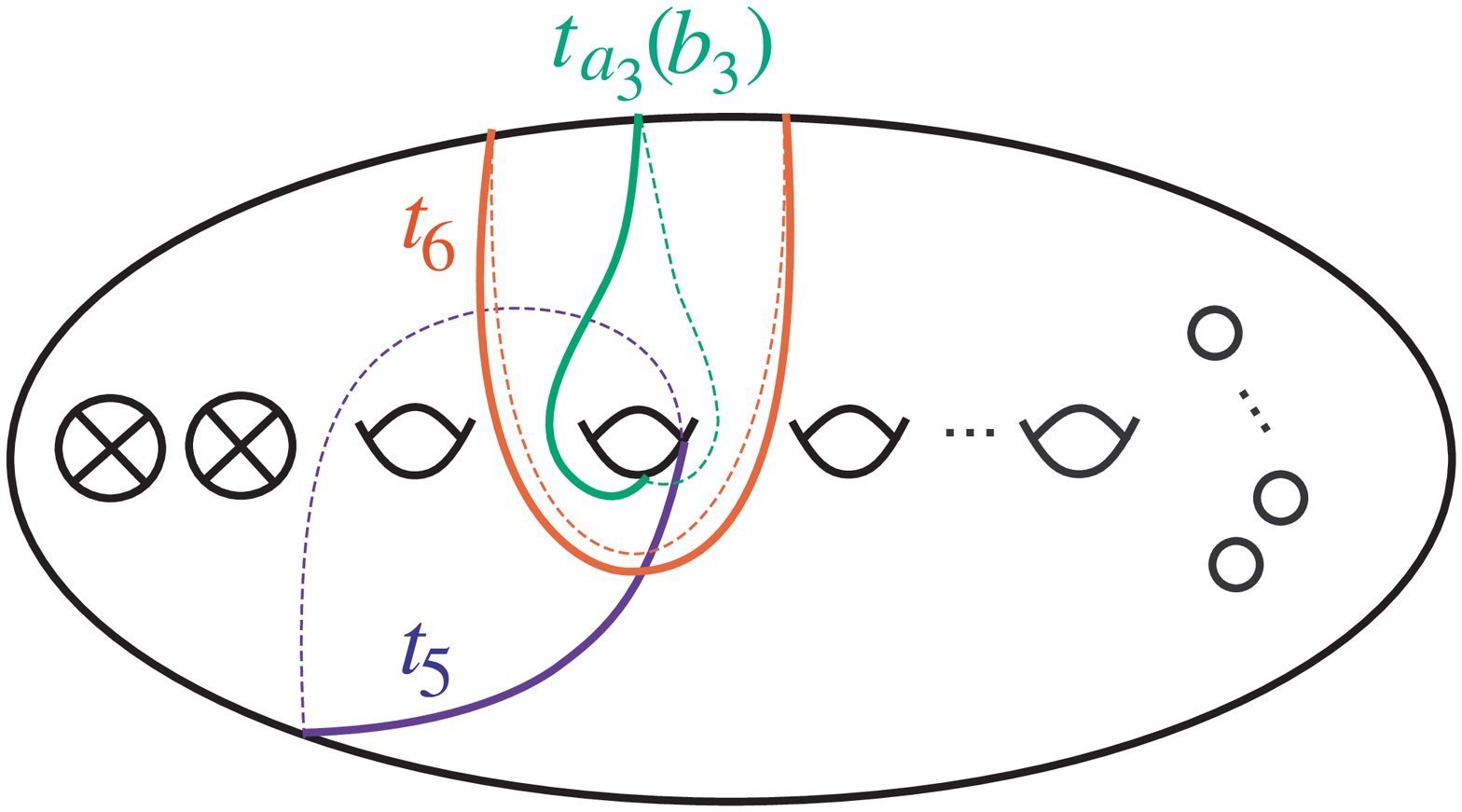}  

\hspace{-1cm} (iii)   \hspace{6.5cm} (iv)

\hspace{-0.4cm} \epsfxsize=2.7in \epsfbox{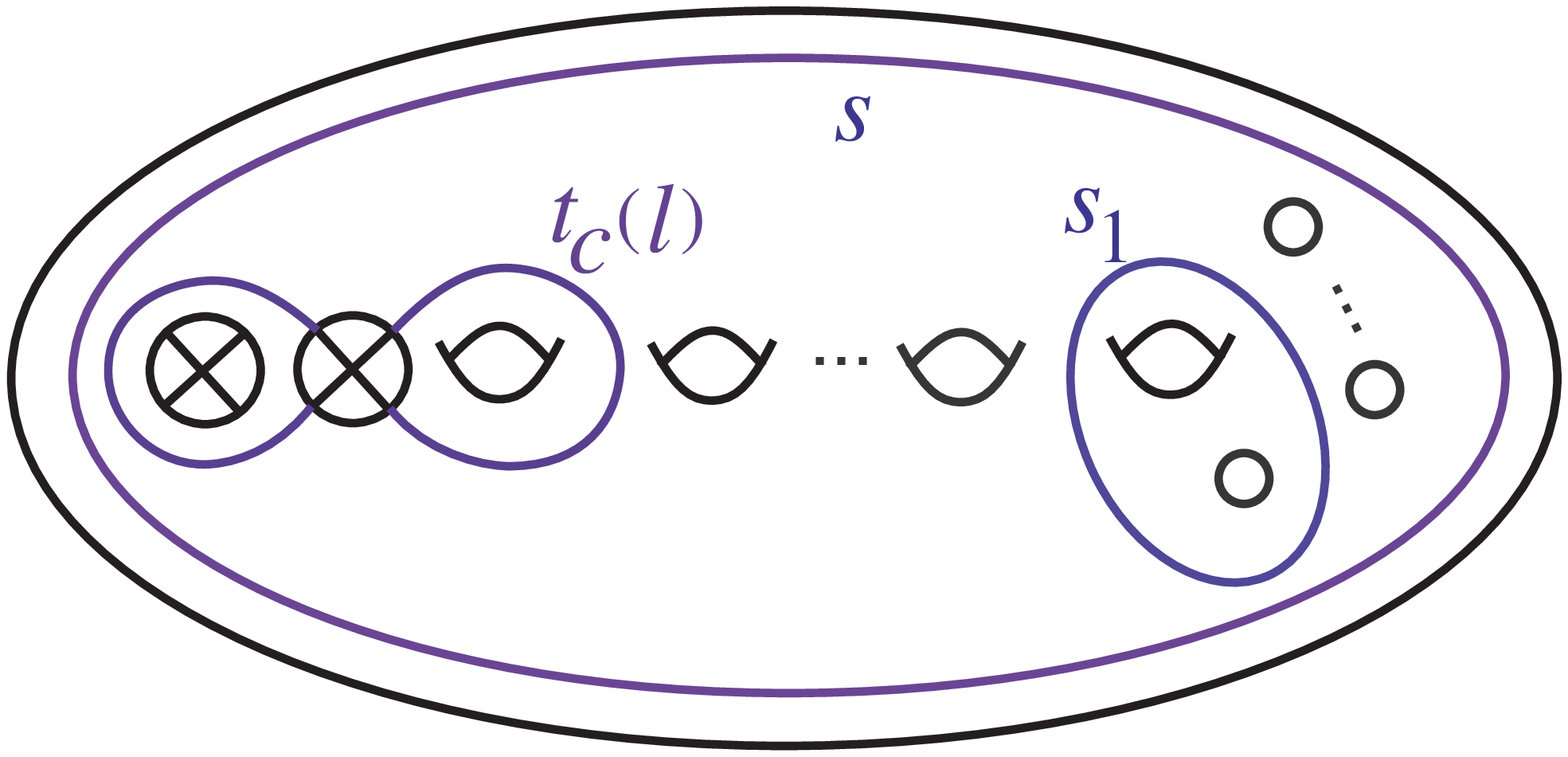} \hspace{0.2cm} \epsfxsize=2.7in \epsfbox{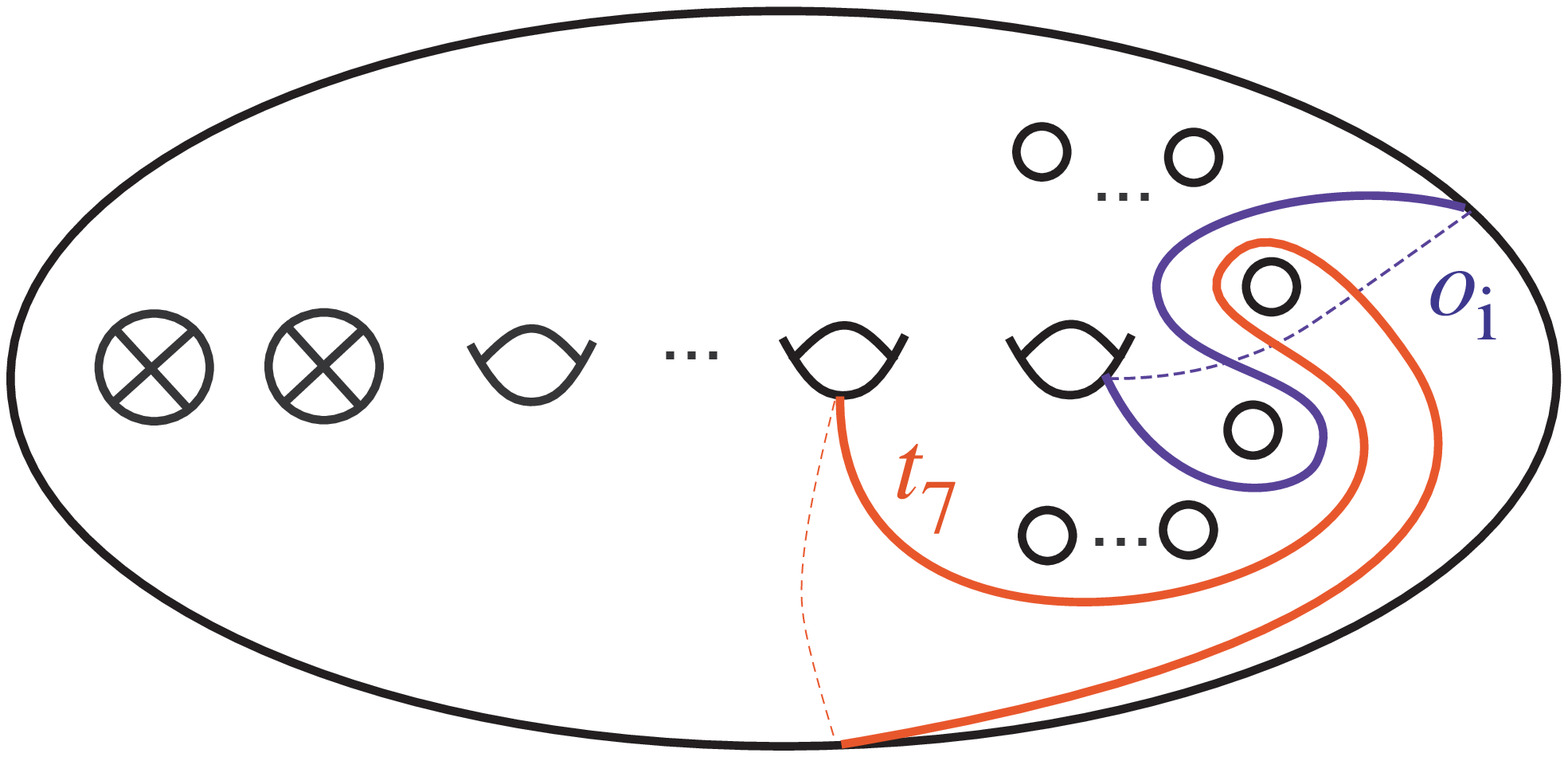}

\hspace{-1cm} (v) \hspace{6.5cm} (vi)

 \hspace{-0.4cm} \epsfxsize=2.7in \epsfbox{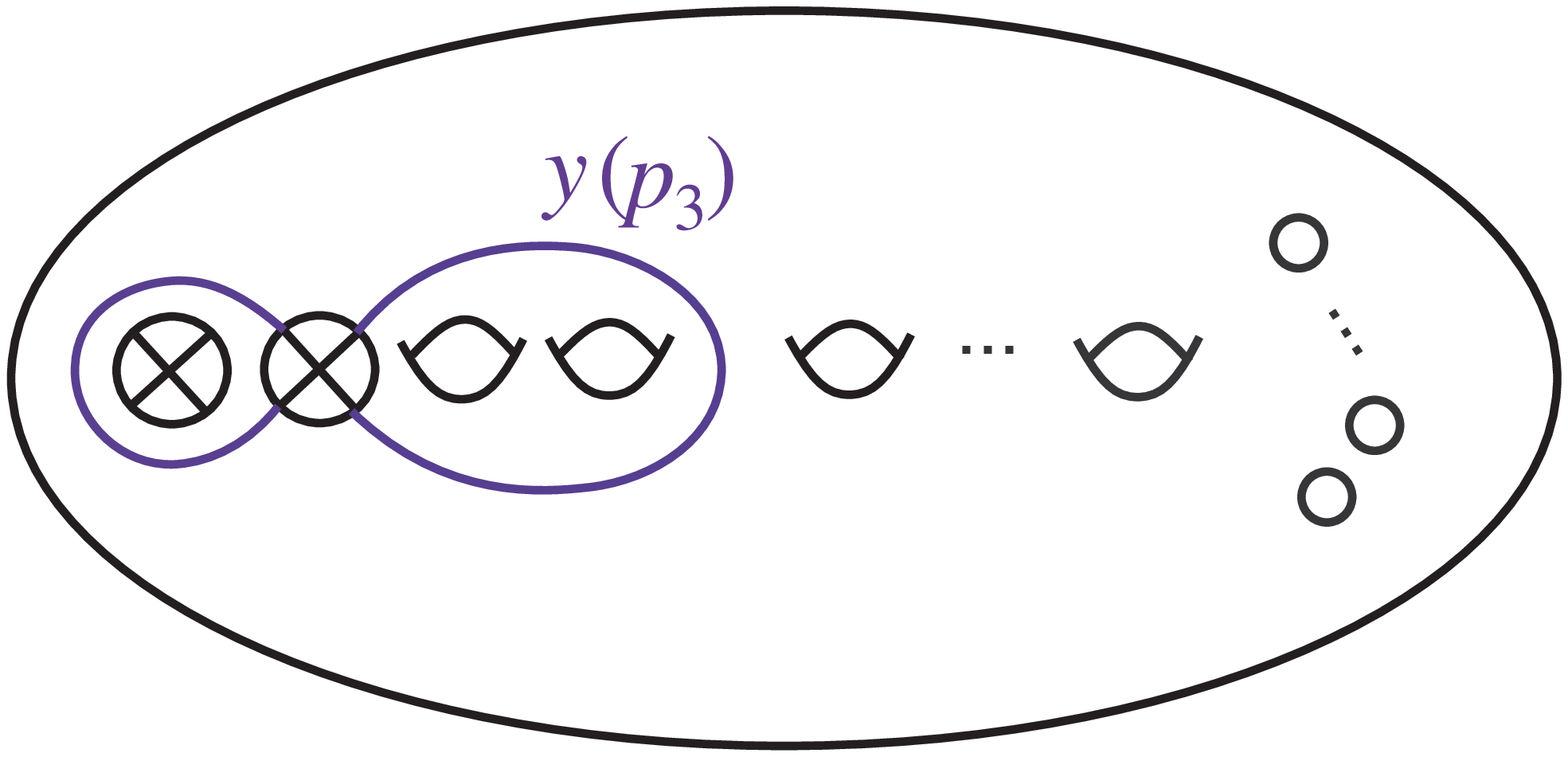} \hspace{0.2cm} \epsfxsize=2.7in \epsfbox{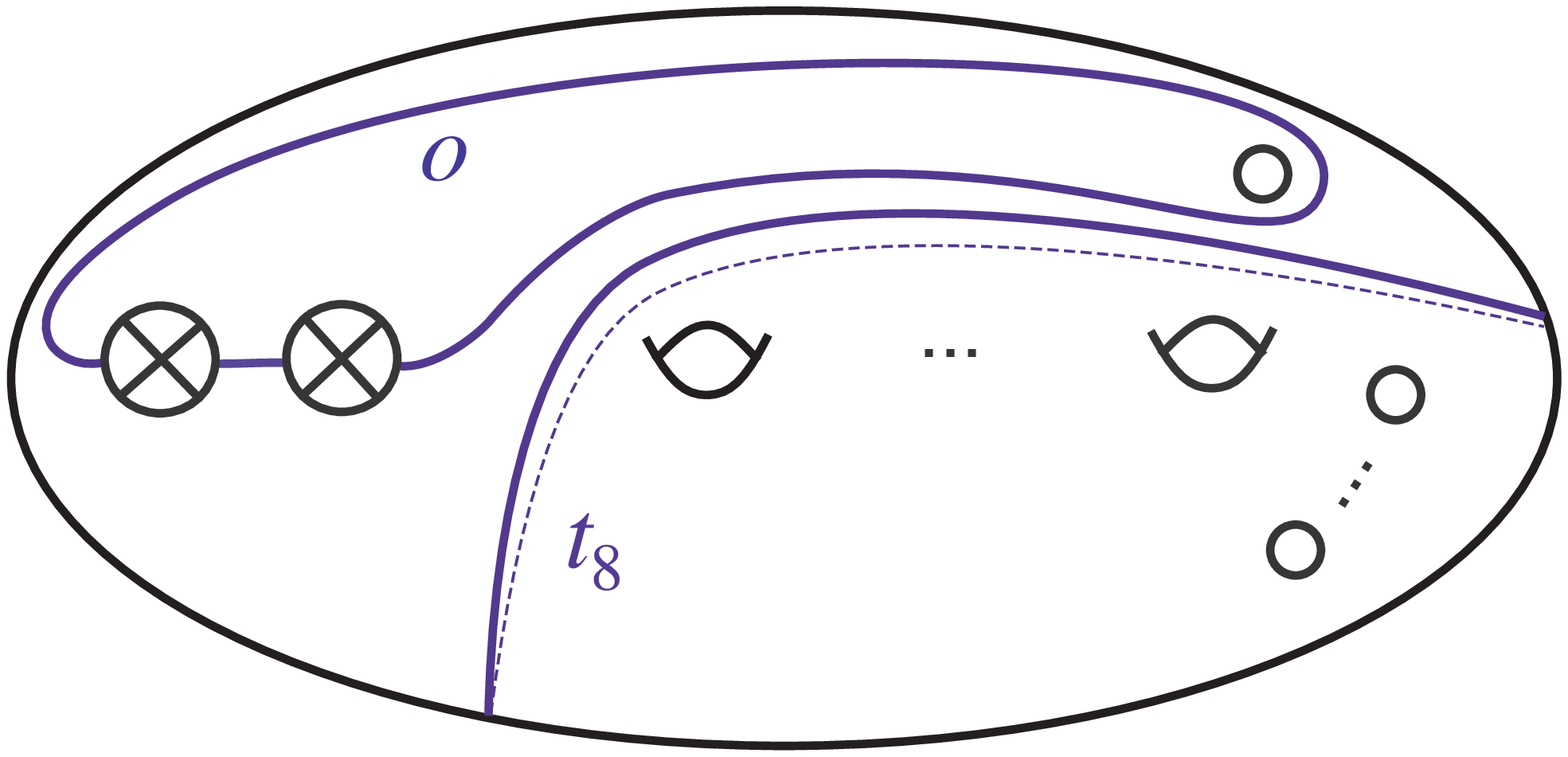}

\hspace{-1cm} (vii) \hspace{6.5cm} (viii)

\caption{Curve configuration XIV} \label{fig-new}
\end{center}
\end{figure}

\begin{lemma} \label{prop1} Suppose $g \geq 6$ and $g$ is even. $\forall \ f \in G$, $\exists$ a set $L_f \subset \mathcal{T}(N)$
such that $L_f$ has trivial stabilizer and $\lambda([x])= h([x])$ $\ \forall \ x \in L_f \cup f(L_f)$.\end{lemma}

\begin{proof} We assume that $g \ge 8$.
The case $g=6$ can be proven in a similar way.
We consider the collections $\CC_1, \CC_2, \CC_3, \CC_4$ of curves of Lemma 3.6.
We already know that they have trivial stabilizers.
We first prove that $h([x])= \lambda ([x])$ for all $x \in \CC_1 \cup \CC_2 \cup \CC_3 \cup \CC_4 = \{ a_1, \dots, a_{g-3}, b_1, \dots, b_{g-3}, c_1, \dots, c_{g-3}, d_1, \dots, d_{n-1}, p_3, e,l,j,c,w\}$.
By Lemma 2.14, we know that this is true for $x \in \{ a_1, \dots, a_{g-3}, b_1, \dots, b_{g-3}, c_1, \dots, c_{g-3}, d_1, \dots, 
\allowbreak
d_{n-1}, e,l\}$.
The curve $p_3$ is the unique curve up to isotopy disjoint from $a_1, a_2, a_3, a_5, b_5, c_5, p$, which intersects each of $a_4, b_1,c_1$ once, which intersects $l$ nontrivially, and which bounds a pair of pants in $N$ together with $a_1$ and $a_3$.
Since $h([x]) = \lambda ([x])$ for all these curves and $\lambda$ preserves these properties, we have $h([p_3])= \lambda ([p_3])$.
The curve $c$ is the unique curve up to isotopy disjoint from $a_1,a_2,a_3,b_1,b_3,c_1,c_3,a$ which intersects $l$ nontrivially. 
Since $h([x]) = \lambda([x])$ for all these curves and $\lambda$ preserves these properties, we have $h([c]) = \lambda ([c])$.
The curve $j$ is the unique curve up to isotopy disjoint from $a_2, a_3, \dots, a_{g-3}, b_1, d_{n-1}$ which intersects $b_3$ nontrivially. 
Since $h([x]) = \lambda([x])$ for all these curves and $\lambda$ preserves these properties, we have $h([j]) = \lambda ([j])$.
To study the curve $w$ we use the curve $t_1$ depicted in Figure \ref{fig-new} (i).
The curve $t_1$ is the unique curve up to isotopy disjoint from $a_1,a_2,b_1,c_3,l$ which intersects $c_1$ nontrivially.
Since $h([x]) = \lambda([x])$ for all these curves and $\lambda$ preserves these properties, we have $h([t_1]) = \lambda ([t_1])$.
The curve $w$ is the unique curve up to isotopy disjoint from $b_1,c_1, t_1$ which intersects $a$ nontrivially.
Since $h([x]) = \lambda([x])$ for all these curves and $\lambda$ preserves these properties, we have $h([w]) = \lambda ([w])$.
  
Now, we prove that $h([t_{a_3} (b_3)]) = \lambda ([t_{a_3} (b_3)])$.
To do this we use the curves $t_2, t_3, t_4, t_5,$ $ t_6$ depicted in Figure \ref{fig-new} (ii)-(iv).
The curve $t_2$ is the unique curve up to isotopy disjoint from $a_1, a_2, a_3, a_5, b_5, 
\allowbreak
c_1, c_3, c_5, p$ which intersects $b_1$ nontrivially. 
Since $h([x]) = \lambda([x])$ for all these curves and $\lambda$ preserves these properties, we have $h([t_2]) = \lambda ([t_2])$.
The curve $t_3$ is the unique curve up to isotopy disjoint from $a_1, a_2, b_3, c_1, p$ which intersects $b_1$ nontrivially. 
Since $h([x]) = \lambda([x])$ for all these curves and $\lambda$ preserves these properties, we have $h([t_3]) = \lambda ([t_3])$.
The curve $t_4$ is the unique curve up to isotopy disjoint from $a_2, c_1, p_3, t_2, t_3$ which intersects $c_3$ nontrivially. 
Since $h([x]) = \lambda([x])$ for all these curves and $\lambda$ preserves these properties, we have $h([t_4]) = \lambda ([t_4])$.
The curve $t_5$ is the unique curve up to isotopy disjoint from $a_1, a_2, l, t_2, t_4$ which intersects $c_1$ once. 
Since $h([x]) = \lambda([x])$ for all these curves and $\lambda$ preserves these properties, we have $h([t_5]) = \lambda ([t_5])$.
We have $|a_3 \cap b_3|=1$ and $t_6$ is the boundary curve of a regular neighborhood of $a_3 \cup b_3$.
We have $h([a_3])= \lambda ([a_3])$ and $h([b_3]) = \lambda ([b_3])$, hence, by Lemma 2.6 and Lemma 2.7, $h([t_6]) = \lambda ([t_6])$.
Finally, $t_{a_3} (b_3)$ is the unique curve up to isotopy disjoint from $t_5,t_6$ which intersects $b_3$ nontrivially.
Since $h([x]) = \lambda([x])$ for all these curves and $\lambda$ preserves these properties, we have $h([t_{a_3} (b_3)]) = \lambda ([t_{a_3} (b_3)])$.

Now, we prove that $h ([t_x(y)]) = \lambda ([t_x(y)])$ for all $x,y \in \{ a_1, \dots, a_{g-3}, b_1, \dots, b_{g-3}, c_1,$ $ \dots, c_{g-3}, 
\allowbreak
d_1, \dots, d_{n-1}, a, p_3 \}$.
Let $x,y$ be two curves such that $| x \cap y|=1$, and let $z$ be the boundary curve of a regular neighborhood of $x \cup y$.
Suppose that $h ([x]) = \lambda( [x])$ and $h([y])=\lambda ([y])$.
We have $h([z])= \lambda ([z])$ by Lemma 2.6 and Lemma 2.7.
Up to isotopy, there are exactly two curves that intersect $x$ and $y$ once and are disjoint from $z$: $t_x(y)$ and $t_x^{-1} (y)$.
Since $\lambda$ preserves these properties, either $\lambda ([t_x(y)]) = h([t_x (y)])$, or $\lambda ([ t_x (y)]) = h ([t_x^{-1} (y)])$.
Note that, if $\lambda ([t_x(y)]) = h([t_x(y)])$, then $\lambda ([ t_x^{-1} (y)]) = h ([t_x^{-1} (y)])$, hence $\lambda ([t_y(x)]) = h ([t_y(x)])$, since $t_y(x)$ is isotopic to $t_x^{-1} (y)$.
Let $x,y_1,y_2$ be three curves such that $|x \cap y_1| = |x \cap y_2|=1$ and $y_1 \cap y_2 = \emptyset$.
We assume that $\lambda ([x]) = h ([x])$, $\lambda ([y_1]) = h ([y_1])$ and $\lambda ([y_2]) = h ([y_2])$. 
Observe that $i (t_x (y_1), t_x (y_2))=0$ and $i(t_x(y_1), t_x^{-1} (y_2)) \neq 0$.
So, by the above, if $\lambda ([t_x (y_1)]) = h ([t_x (y_1)])$, then $\lambda ([t_x (y_2)]) = h ([t_x (y_2)])$.
Now, since we already know that $h([t_{a_3} (b_3)]) = \lambda ([t_{a_3} (b_3)])$, the above implies that $\lambda ([t_x(y)]) = h([t_x(y)])$ for all  $x,y \in \{ a_1, \dots, a_{g-3}, b_1, \dots, b_{g-3}, c_1, \dots, c_{g-3}, d_1, \dots, d_{n-1}, a, p_3 \}$.
 
For $f=t_x$, where $x \in \{a_1, \dots, a_{g-3}, b_1, \dots, b_{g-3}, c_1, \dots, c_{g-3}, d_1, \dots, d_{n-1} \}$, we set $L_f = \CC_2$.
By the above, we know that $\lambda ([y]) = h ([y])$ for all $y \in L_f \cup f(L_f)$, except for $y=t_x(l)$ where $x \in \{ a_2, b_1, c_1 \}$.
In the latter case, using that $|a_1 \cap x|=|l \cap x|= 1$, $l \cap a_1 = \emptyset$, and $\lambda ([t_x(a_1)]) = h([t_x(a_1)])$, we show as above that $\lambda ([t_x(l)]) = h ([t_x(l)])$.
 
For $f=t_a$ we set $L_f = \CC_1$.
We know that $\lambda ([x]) = h ([x])$ for all $x \in L_f$, we have $f(x)=x$ for all $x \in \CC_1 \setminus \{ a_1, p_3\}$, and we also know by the above that $h ([t_a(x)]) = \lambda ([t_a(x)])$ for $x \in \{ a_1, p_3\}$.
So, $h ([x]) = \lambda ([x])$ for all $x \in L_f \cup f(L_f)$.

For $f=t_c$ we set $L_f = \CC_2$.
We know that $\lambda ([x]) = h ([x])$ for all $x \in L_f$ and we have $t_c(x)=x$ for all $x \in \CC_2 \setminus \{ l \}$.
So, it remains to show that $ h ([t_c(l)]) = \lambda ([t_c(l)])$.
To do this, we use the curves $s_1$ and $s$ depicted in Figure \ref{fig-new} (v).
We know that $h ([t_{c_{g-3}} (a_{g-3}) ]) = \lambda ([t_{c_{g-3}} (a_{g-3}) ])$.
Using again the same argument, we deduce that $h ([(t_{d_1}^{-1}t_{c_{g-3}}) (a_{g-3}) ]) = \lambda ([(t_{d_1}^{-1}t_{c_{g-3}})(a_{g-3}) ])$.
But $(t_{d_1}^{-1}t_{c_{g-3}}) (a_{g-3})=s_1$, hence $h ([s_1]) = \lambda ([s_1])$.
The curve $s$ is the unique curve up to isotopy disjoint from $a_1, \dots, a_{g-3}, m_1, \dots, m_{n-1}, l, p, s_1$.
Since $h([x]) = \lambda([x])$ for all these curves and $\lambda$ preserves these properties, we have $h([s]) = \lambda ([s])$.
Finally, $t_c (l)$ is the unique curve up to isotopy disjoint from $a_1, a_3, b_3, c_3, s, e_{4,0}$ which intersects $p$ nontrivially.
Since $h([x]) = \lambda([x])$ for all these curves and $\lambda$ preserves these properties, we have  $h([t_c(l)]) = \lambda ([t_c(l)])$.
 
For $f = \sigma_i$, where $i \in \{1, \dots, n-1\}$, we set $L_f = \CC_2$.
We know that $\lambda ([x]) = h ([x])$ for all $x \in L_f$, and we have $\sigma_i (x) = x$ for all $x \in \CC_2 \setminus \{ d_i \}$. 
So, it remains to show that $h ([ \sigma_i (d_i) ]) = \lambda ([ \sigma_i (d_i) ])$.
To do this, we use the curve $t_7$ depicted in Figure \ref{fig-new} (vi).
The curve $t_7$ is the unique curve up to isotopy disjoint from $a_{g-4}, a_{g-3}, c_{g-5}, d_{i+1}, s_1, m_1, \dots, m_{i-1}$ which intersects $c_{g-3}$ nontrivially. 
Since $h([x]) = \lambda([x])$ for all these curves and $\lambda$ preserves these properties, we have $h([t_7]) = \lambda ([t_7])$.
The curve $\sigma_i (d_i)$, which is shown as $o_i$ in Figure \ref{fig-new} (vi), is the unique curve up to isotopy disjoint from $d_{i-1}, d_{i+1}, t_7$ which intersects $d_i$ nontrivially.
Since $h([x]) = \lambda([x])$ for all these curves and $\lambda$ preserves these properties, we have $h([\sigma_i (d_i)]) = \lambda ([\sigma_i (d_i)])$.
 
For $f=y$ we set $L_f = \CC_1$.
We know that $\lambda ([x]) = h ([x])$ for all $x \in L_f$, and we have $y (x) = x$ for all $x \in \CC_1 \setminus \{ a_1, p_3 \}$. 
Moreover, we have $y(a_1) = t_c(l)$, hence, by the above, $h ([ y(a_1) ]) = \lambda ([ y(a_1) ])$.
It remains to show that $h ([ y(p_3) ]) = \lambda ([ y(p_3) ])$.
The curve $y(p_3)$ is depicted in Figure \ref{fig-new} (vii).
Observe that $y(p_3)$ is the unique curve up to isotopy disjoint from $a_1, a_2, a_3, a_5, b_5, c_5, s, t_c(l)$ which intersects $l$ nontrivially.
Since $h([x]) = \lambda([x])$ for all these curves and $\lambda$ preserves these properties, we have $h ([ y(p_3) ]) = \lambda ([ y(p_3) ])$.
 
For $f = \xi_1$ we set $L_f = \CC_3$.
We know that $h ([x]) = \lambda ([ x ])$ for all $x \in \CC_3$, and we have $\xi_1(x) = x$ for all $x \in \CC_3 \setminus \{ c\}$.
So, it remains to show that $h ([ \xi_1(c) ]) = \lambda ([\xi_1 (c) ])$.
The curve $\xi_1 (c)$ is drawn as the curve $o$ in Figure \ref{fig-new} (viii).
We consider the curve $t_8$ depicted in Figure \ref{fig-new} (viii).
We have $t_8 = v_{(g-2)/2, n-1}$, hence, by Lemma 3.2, $h ([ t_8 ]) = \lambda ([ t_8 ])$.
The curve $\xi_1 (c)$ is the unique curve up to isotopy disjoint from $a,c,s,t_8$ which intersects $p$ nontrivially.
Since $h([x]) = \lambda([x])$ for all these curves and $\lambda$ preserves these properties, we have $h ([ \xi_1 (c) ]) = \lambda ([ \xi_1 (c) ])$.
 
For $f = \xi_2$ we set $L_f = \xi_2^{-1}(\CC_4)$.
We know that $h([x]) = \lambda ([x])$ for all $x \in f(L_f) = \CC_4$, and we have $\xi_2^{-1} (x) = x$ for all $x \in \CC_4 \setminus \{c\}$.
So, it remains to show that $h ([ \xi_2^{-1}(c) ]) = \lambda ([\xi_2^{-1} (c) ])$.
But $\xi_2^{-1}(c) = \xi_1 (c)$, hence, by the above, $h ([ \xi_2^{-1}(c) ]) = \lambda ([\xi_2^{-1} (c) ])$.
\end{proof}
  
\begin{theorem} \label{A} If $g \geq 6$ and $g$ is even, then $h([x]) = \lambda([x])$ for every vertex $[x] \in \mathcal{T}(N)$.\end{theorem}

\begin{proof} Let $\mathcal{X}= \mathcal{C}_0 \cup \mathcal{B}_0 \cup \mathcal{B}_1 \cup \mathcal{B}_2 \cup $
$\big( \bigcup _{f \in G} (L_f \cup f(L_f) \big ))$. For each vertex $x \in \mathcal{T}(N)$ there exists $r \in Mod_N$ and a
vertex $y \in \mathcal{X}$ such that $r(y)=x$. By Lemma \ref{A-even} and Lemma \ref{prop1}, we know that $h([y]) = \lambda([y])$ for each
$y \in \mathcal{X}$. Define $\mathcal{X}_n$ by induction on $n$. Let $\mathcal{X}_1 = \mathcal{X}$. For $n \geq 2$,
let $\mathcal{X}_n = \mathcal{X}_{n-1} \cup (\bigcup _{f \in G} (f(\mathcal{X}_{n-1}) \cup f^{-1}(\mathcal{X}_{n-1})))$.
We observe that $\mathcal{T}(N) = \bigcup _{n=1} ^{\infty} \mathcal{X}_n$. We prove that $h([x]) = \lambda([x])$ for all
$x \in \mathcal{X}_{n}$ by induction on $n$. We know it for $n=1$. Assume that $n \geq 2$ and $h([x]) = \lambda([x])$ for all
$x \in \mathcal{X}_{n-1}$. Let $f \in G$. There exists $h_f \in Mod_N$ such that $h_f([x]) = \lambda([x])$ for all
$x \in f(\mathcal{X}_{n-1})$. But $f(L_f) \subset \mathcal{X}_{n-1} \cap f(\mathcal{X}_{n-1})$. This implies $h_f = h$
since $f(L_f)$ has trivial stabilizer. Similarly, there exists $h'_f \in Mod_N$ such that $h'_f([x]) = \lambda([x])$ for all
$x \in f^{-1}(\mathcal{X}_{n-1})$. But $L_f \subset \mathcal{X}_{n-1} \cap f^{-1}(\mathcal{X}_{n-1})$ and $L_f$ 
has trivial stabilizer. Hence, $h'_f = h$. So, $h([x]) = \lambda([x])$ for each $x \in \mathcal{X}_n$. Since
$\mathcal{T}(N) = \bigcup _{n=1} ^{\infty} \mathcal{X}_n$, we have $h([x]) = \lambda([x])$ for all $x \in \mathcal{T}(N)$.\end{proof}\\
 
\begin{figure}[htb]
\begin{center}
\hspace{.1cm} \epsfxsize=3.2in \epsfbox{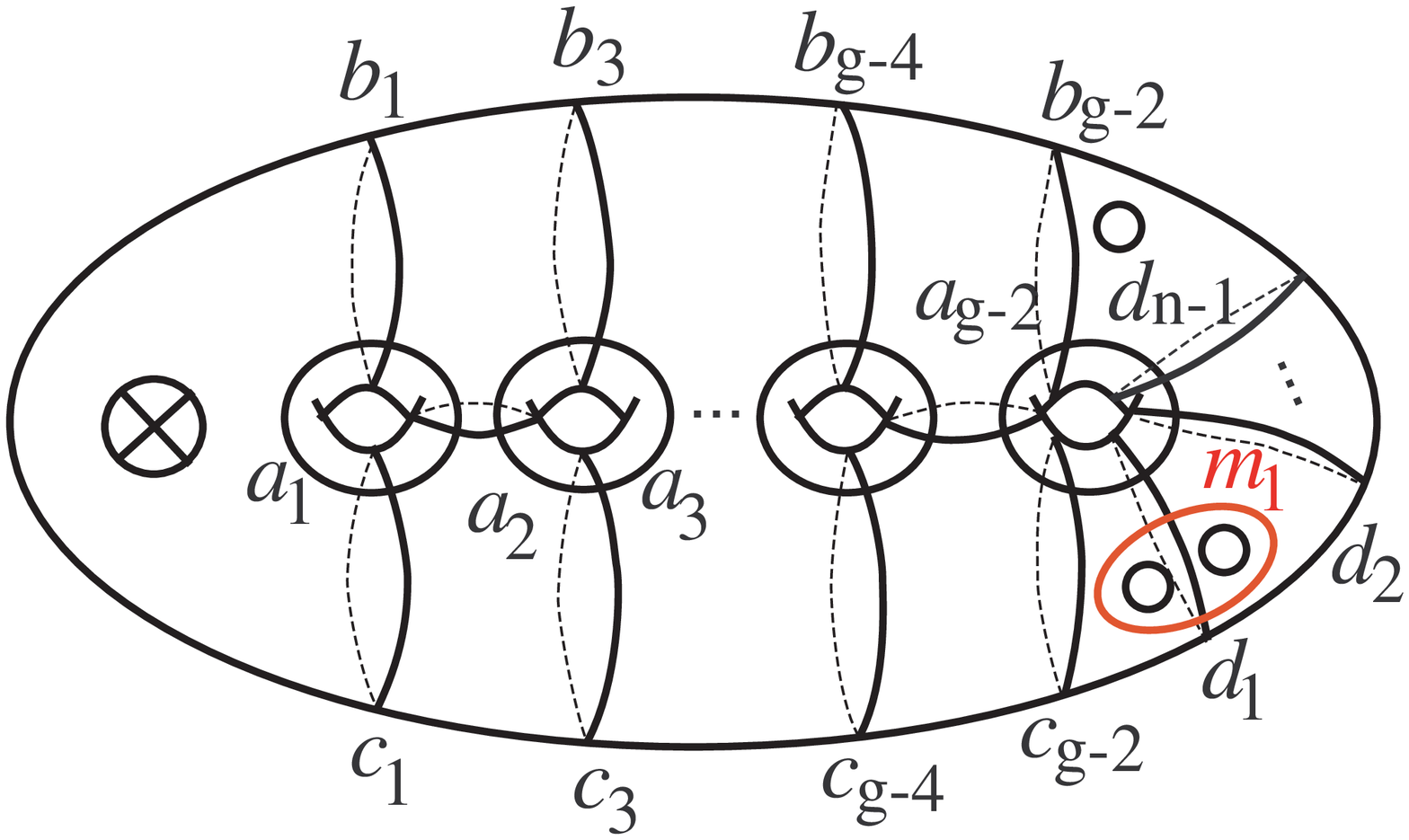} \hspace{-1.3cm} \epsfxsize=3.2in \epsfbox{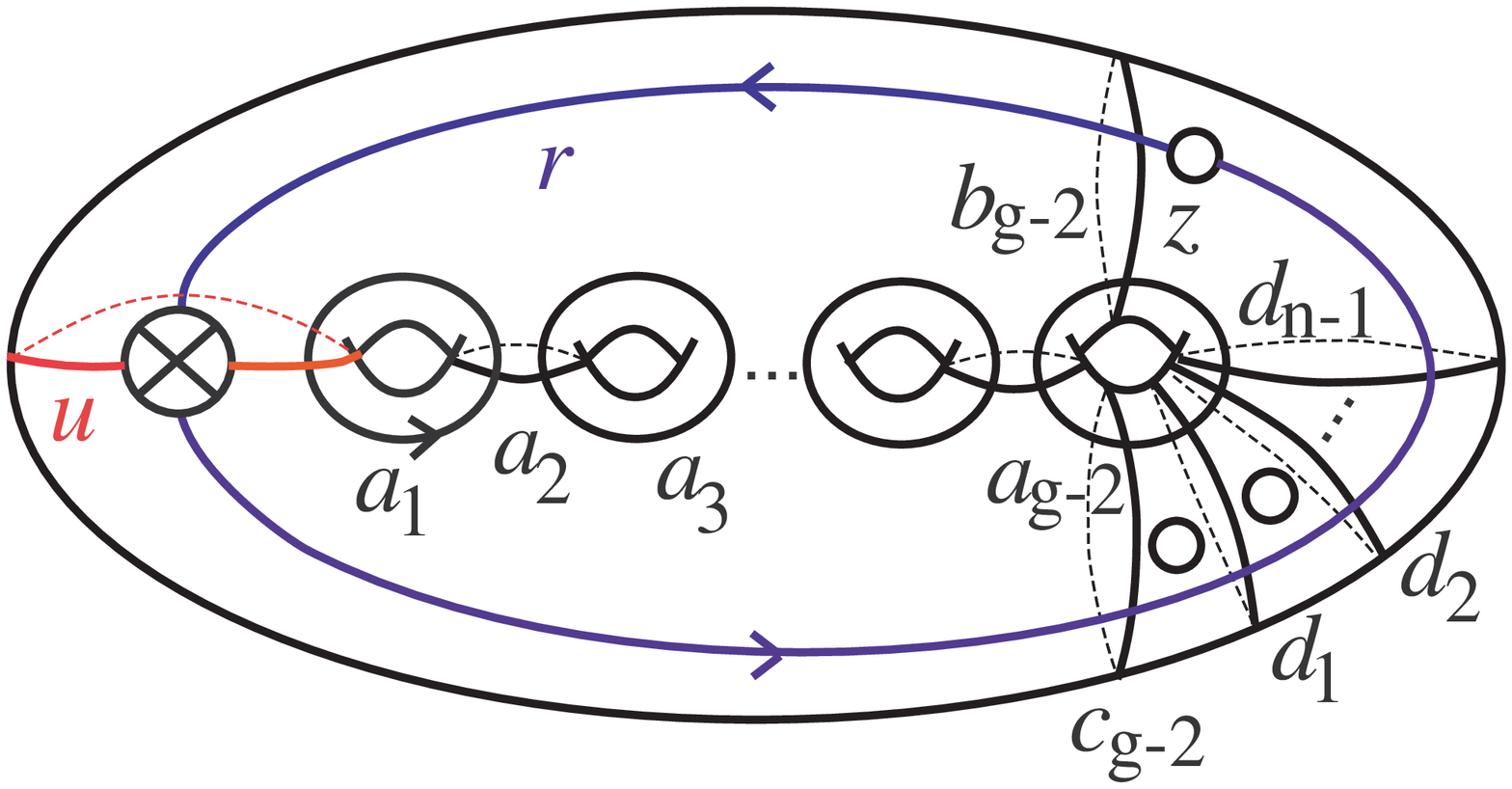}

\hspace{-1cm} (i) \hspace{6.4cm} (ii)

\hspace{.1cm} \epsfxsize=3.2in \epsfbox{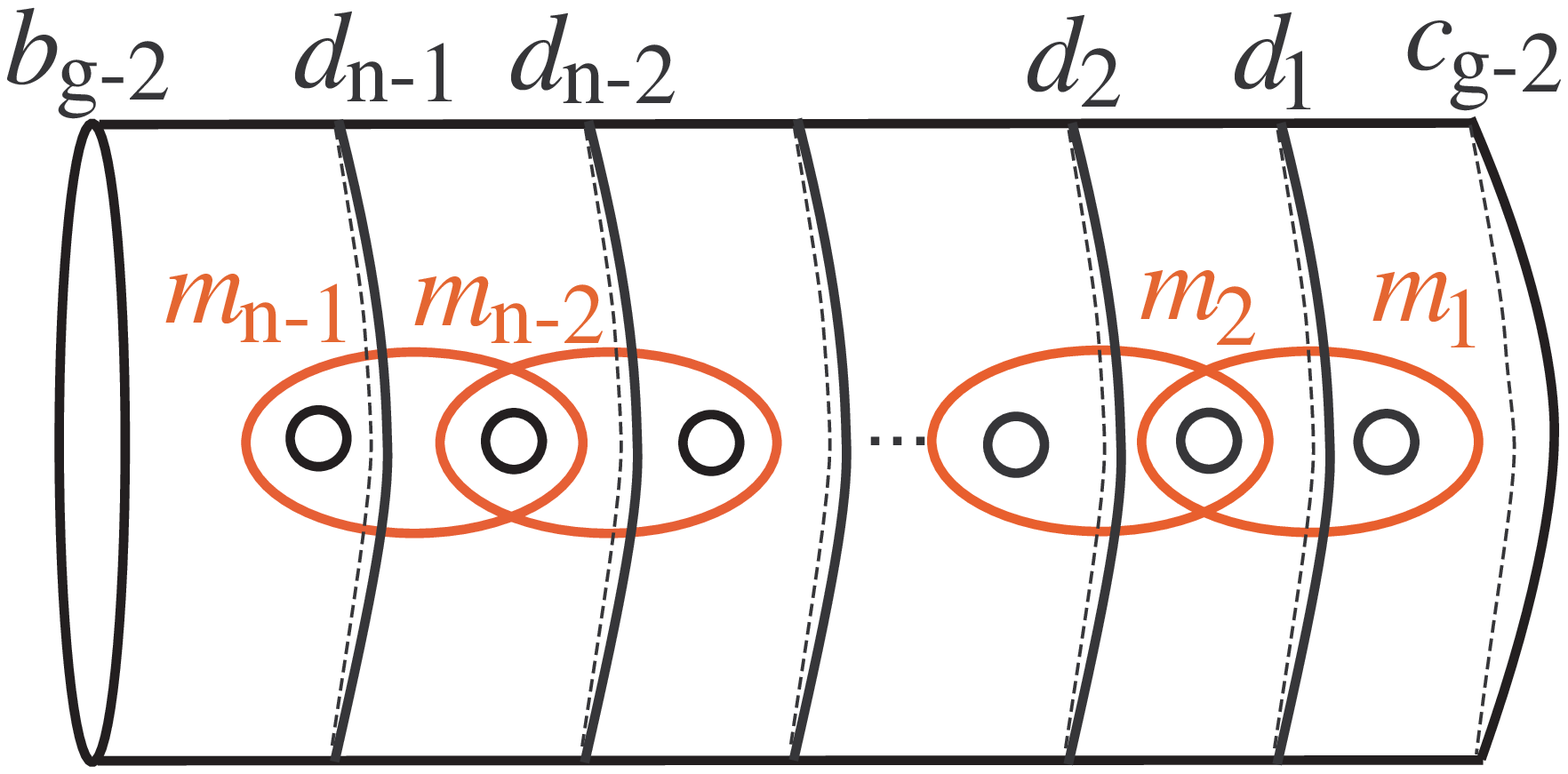}  \hspace{-1.3cm} \epsfxsize=3.2in \epsfbox{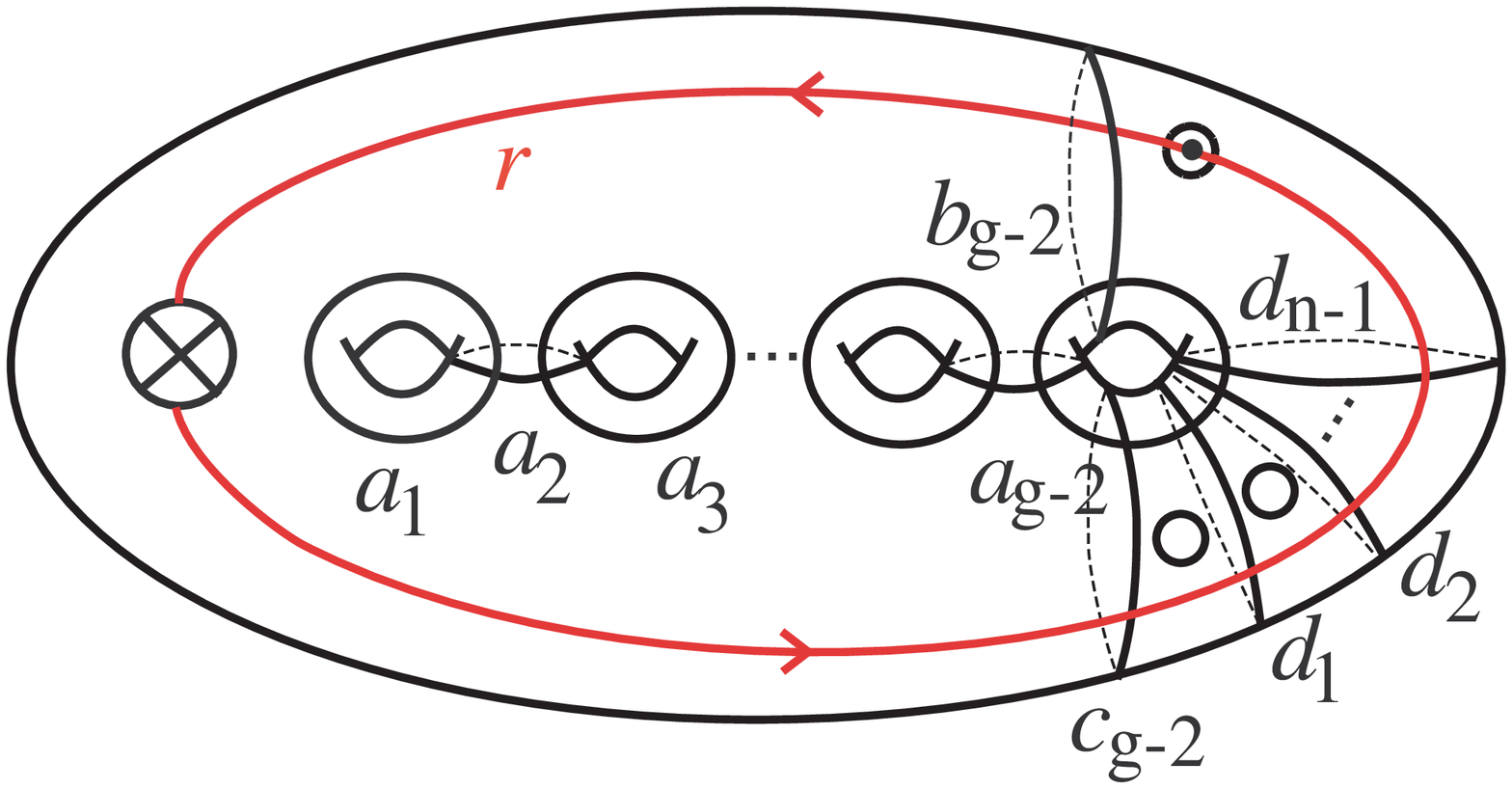}

\hspace{-1cm} (iii)  \hspace{6.2cm} (iv)

\caption{Curve configuration XV} \label{fig36}
\end{center}
\end{figure}
 
In the following lemma we consider the curves given Figure \ref{fig36} (i)-(iii).
Let $\sigma_i$ be the half twist along $m_i$. Let $y$ be the crosscap slide of $u$ along $a_1$. Let $\xi$ be the boundary slide of $z$ along $r$.
The proof of the lemma is similar to the proofs of Theorem 4.13 and Theorem 4.14 given by Korkmaz in \cite{K2}. We change only one loop in his generating set for the fundamental group of $R$ where $R$ is the surface obtained by gluing a disk along one boundary component of $N$ as shown in Figure \ref{fig36} (iv). We use the one sided curve given in Figure \ref{fig36} (iv) instead of the one sided curve given in his proof and consider the loops based at the center of the disk. The proof then follows similar.

\begin{lemma}
\label{prop2} If $g \geq 5$ and $g$ is odd, then $Mod_N$ is generated by $\{t_x: x \in \{a_1, a_2,$ $ \cdots, a_{g-2}, b_1,
b_3, \cdots, b_{g-2}, c_1, c_3, \cdots, c_{g-2}, d_1, d_2, \cdots, d_{n-1}\} \} \cup \{\sigma_1, \sigma_2, \cdots, \sigma_{n-1}, y, \xi\}$.\end{lemma} 

Let $G = \{t_x: x \in \{a_1, a_2,$ $ \cdots, a_{g-3}, b_1, b_3, \cdots, b_{g-3}, c_1, c_3, \cdots, 
c_{g-3}, d_1, d_2, \cdots,$ $ d_{n-1}\} \}$ $ \cup \{\sigma_1, \sigma_2, \cdots, \sigma_{n-1}, y, \xi\}$ where the curves are as shown in Figure \ref{fig36} (i)-(iii).

\begin{figure}[htb]
\begin{center}
\hspace{-0.4cm} \epsfxsize=2.7in \epsfbox{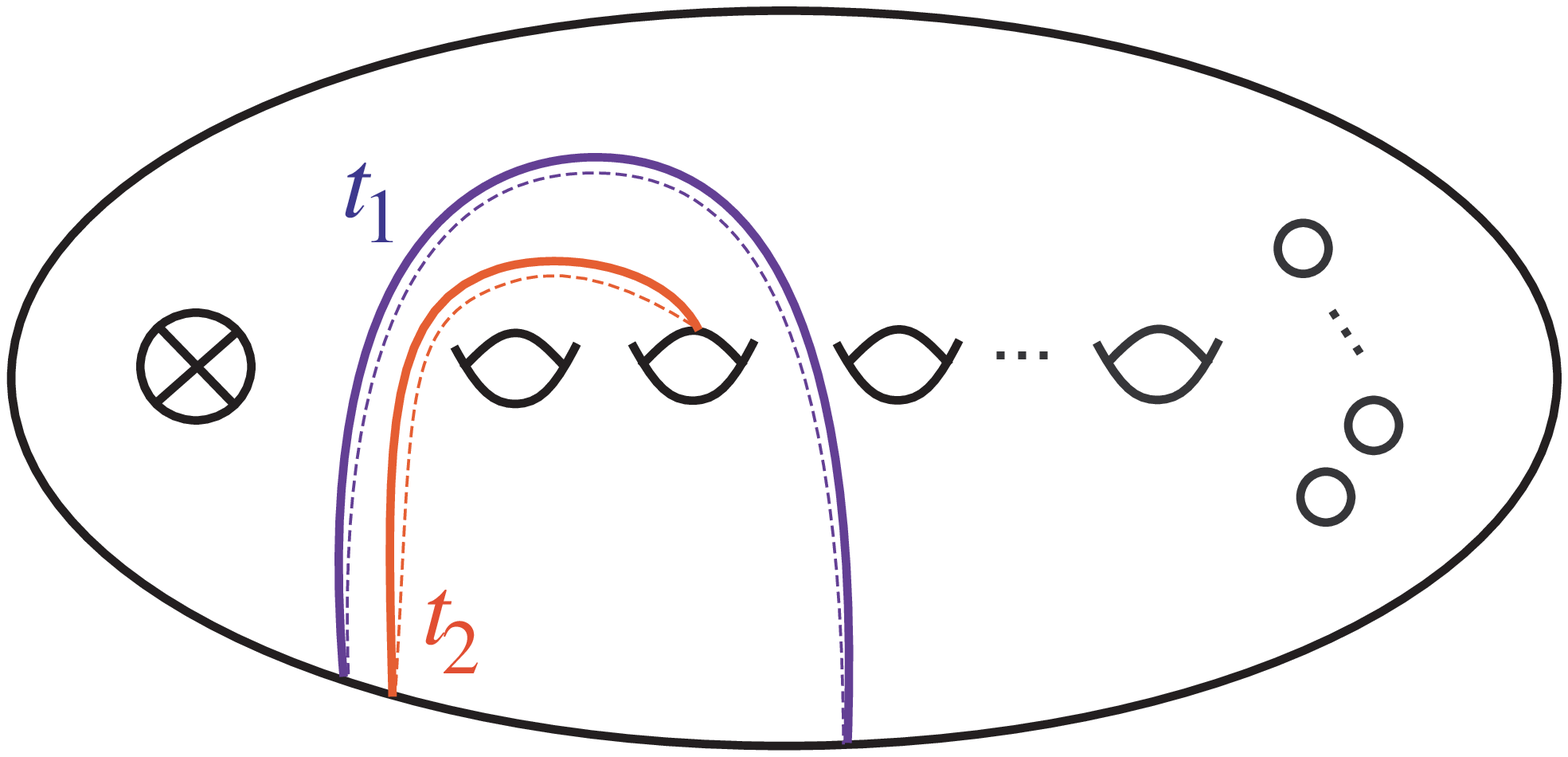} \hspace{0.2cm} \epsfxsize=2.7in \epsfbox{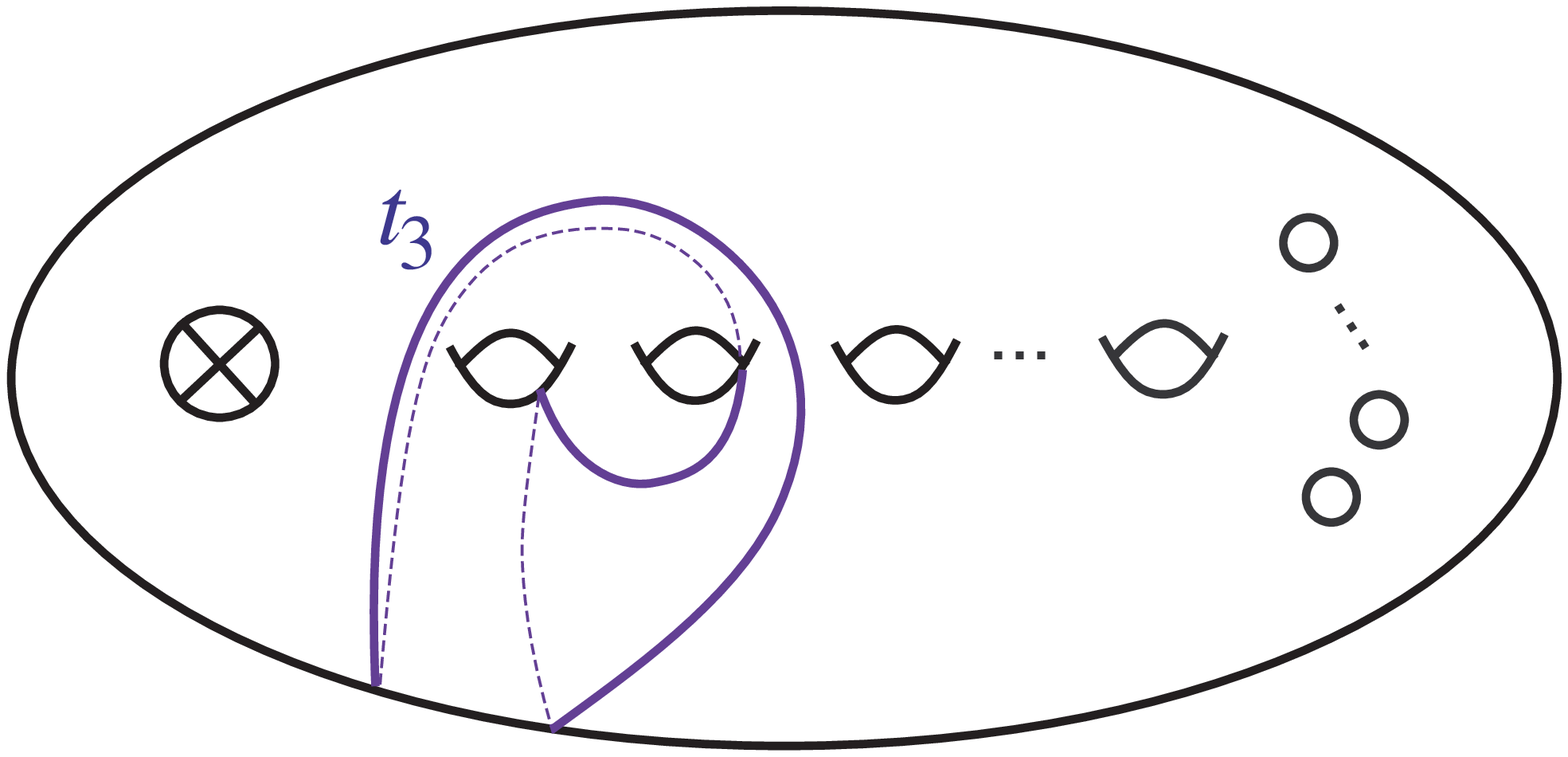} 

\hspace{-1cm} (i) \hspace{6.6cm} (ii)

\hspace{-0.4cm} \epsfxsize=2.7in \epsfbox{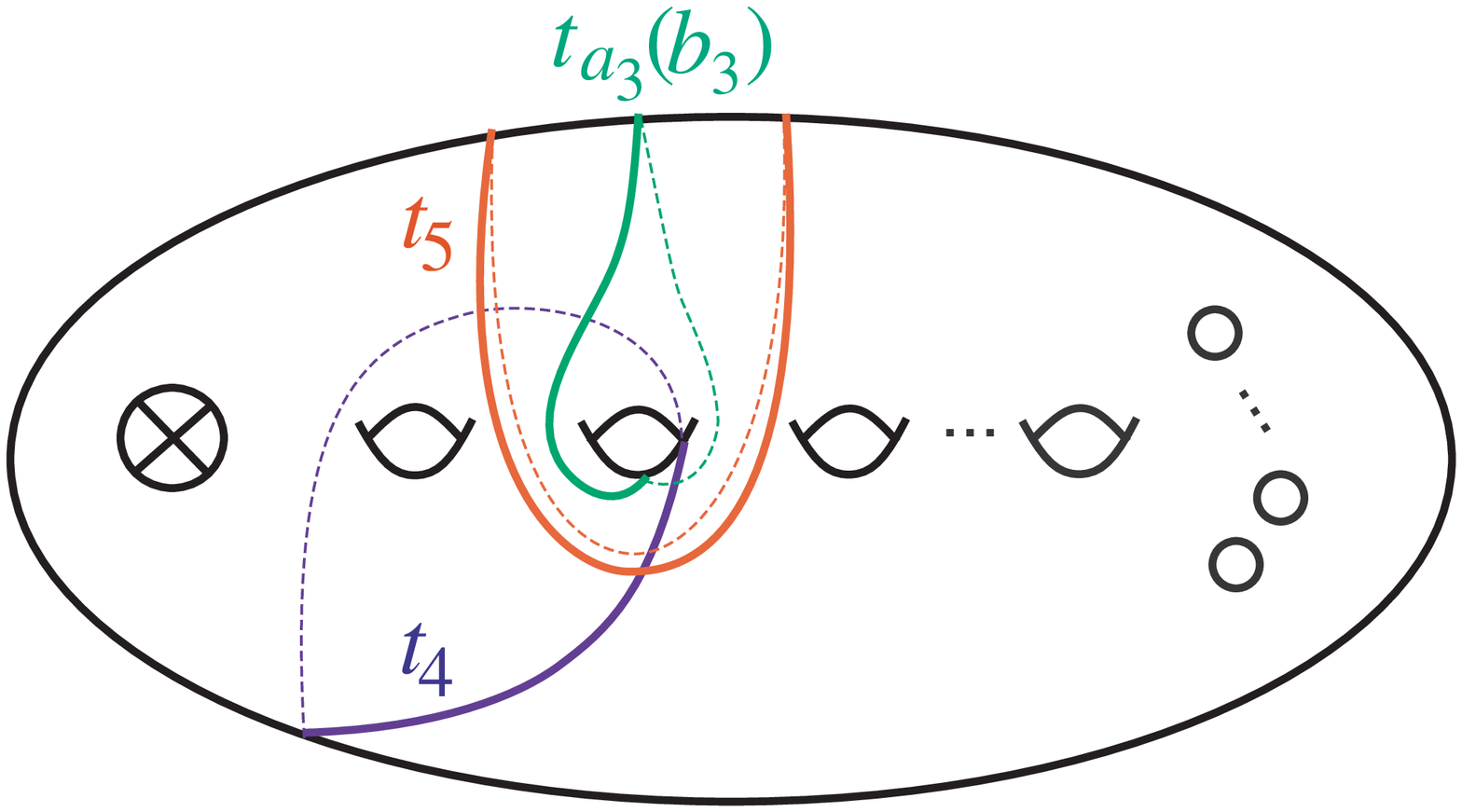} \hspace{0.2cm}  \epsfxsize=2.7in \epsfbox{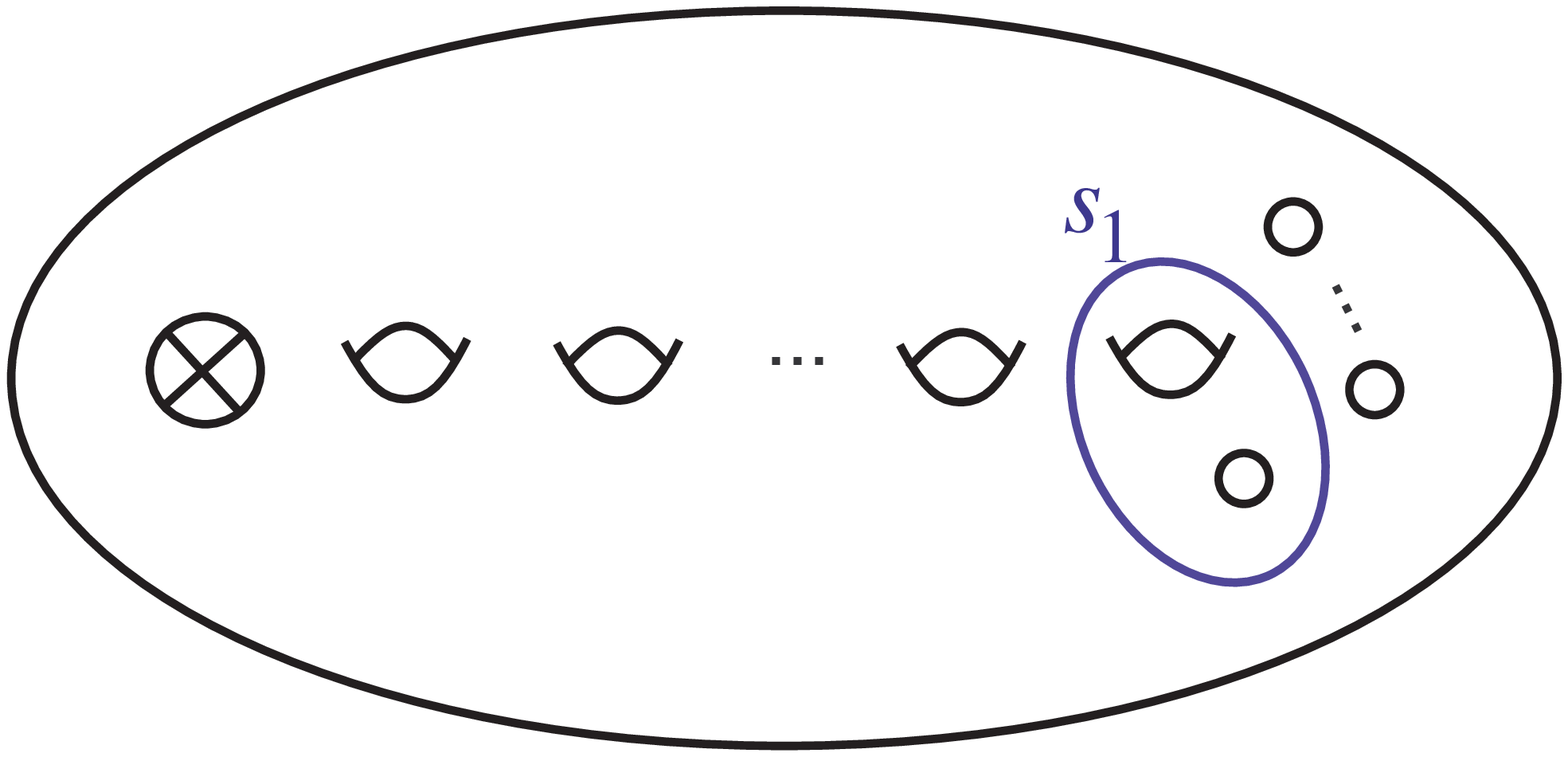}  

\hspace{-1cm} (iii)   \hspace{6.4cm} (iv)

\hspace{-0.4cm} \epsfxsize=2.7in \epsfbox{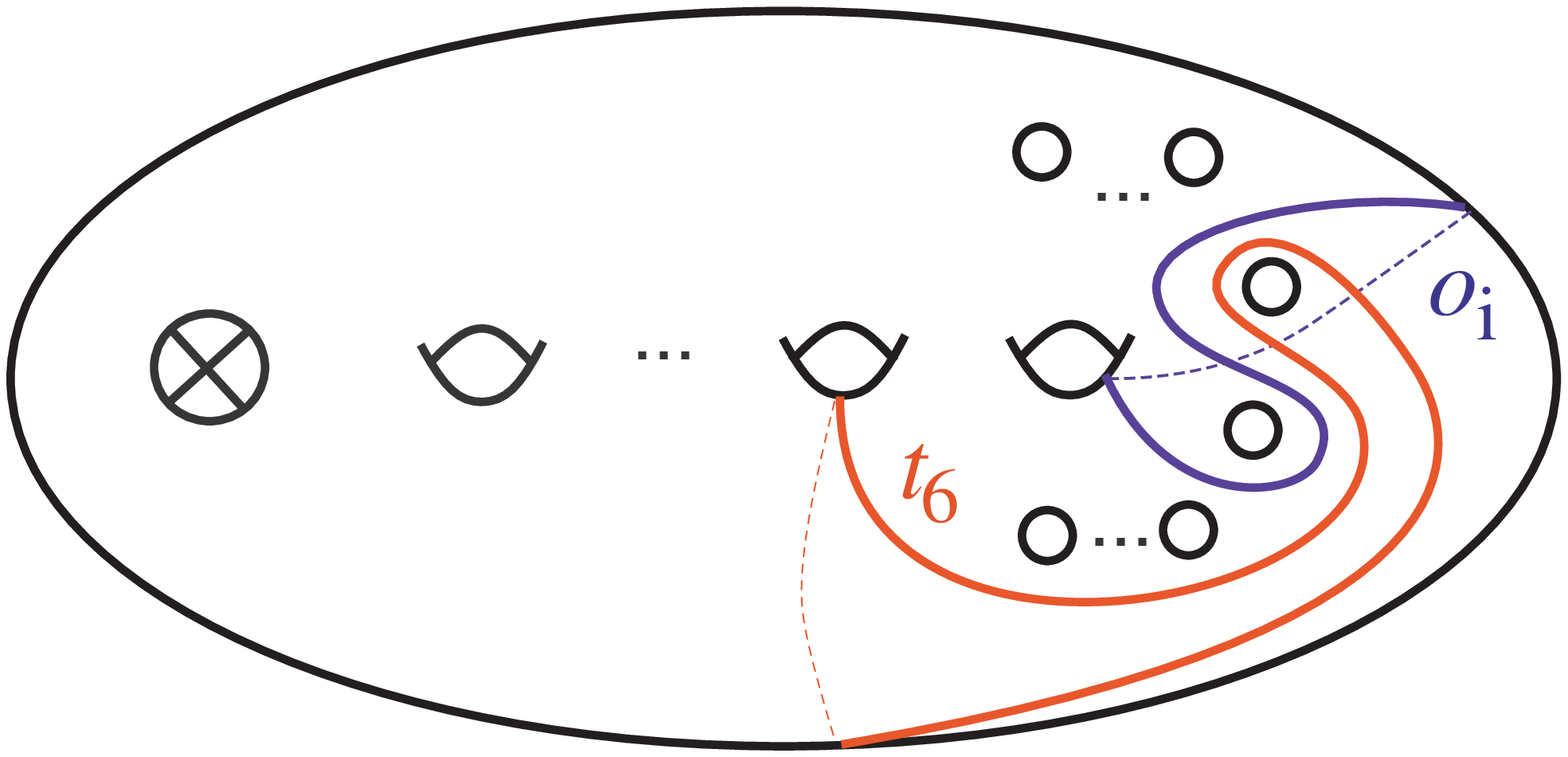}  

\hspace{-1cm} (v)  

\caption{Curve configuration XVI} \label{fig-new2}
\end{center}
\end{figure}

\begin{lemma}\label{prop2} Suppose $g \geq 5$ and $g$ is odd. $\forall \ f \in G$, $\exists$ a set $L_f \subset \mathcal{T}(N)$
such that $L_f$ has trivial stabilizer and $\lambda([x])= h([x])$ $\ \forall \ x \in L_f \cup f(L_f)$.\end{lemma}

\begin{proof} We assume that $g \ge 7$. 
The case $g=5$ can be proven in a similar way.
We consider the collections of curves $\CC_1, \CC_2, \CC_3$ of Lemma 3.7.
We first prove that $h ([x]) = \lambda ([x])$ for all $x \in \CC_1 \cup \CC_2 \cup \CC_3 = \{ a_1, \dots, a_{g-2}, b_1, \dots, b_{g-2}, c_1, \dots, c_{g-2}, d_1, \dots,
\allowbreak
d_{n-1}, p_3,s_2, k\}$.
We already know by Lemma 2.17 that $h ([x]) = \lambda ([x])$ for all $x \in \{ a_1, \dots, a_{g-2}, 
\allowbreak
b_1, \dots, b_{g-2}, c_1, \dots, c_{g-2}, d_1, \dots, d_{n-1},k\}$.
The curve $p_3$ is the unique curve up to isotopy disjoint from $a_1, a_2, a_3, a_5, b_5, c_5$, which intersects $b_1$ and $c_1$ once, which intersects $l$ nontrivially, and which bounds a pair of pants together with $a_1$ and $a_3$.
Since $h([x]) = \lambda([x])$ for all these curves and $\lambda$ preserves these properties, we have $h ([p_3]) = \lambda ([p_3])$.
The curve $s_2$ is the unique curve up to isotopy disjoint from $a_1, a_3, a_5, b_1, b_5, c_1, c_5,$ $p_3$, which intersects $b_3$ nontrivially, and which is not isotopic to $a_3$.
Since $h([x]) = \lambda([x])$ for all these curves and $\lambda$ preserves these properties, we have $h ([s_2]) = \lambda ([s_2])$.
 
Now we show that $h ([t_{a_3} (b_3) ]) = \lambda ([t_{a_3} (b_3) ])$.
We proceed in the same way as in the proof of Lemma 3.9.
We consider the curves $t_1, t_2, t_3, t_4, t_5$ depicted in Figure \ref{fig-new2} (i)-(iii).
The curve $t_1$ is the unique curve up to isotopy disjoint from $a_1, a_2, a_3, a_5, b_5, c_1, c_3, c_5$ which intersects $b_1$ nontrivially.
Since $h([x]) = \lambda([x])$ for all these curves and $\lambda$ preserves these properties, we have $h ([t_1]) = \lambda ([t_1])$.
The curve $t_2$ is the unique curve up to isotopy disjoint from $a_1, a_2, b_3, c_1$ which intersects $b_1$ nontrivially.
Since $h([x]) = \lambda([x])$ for all these curves and $\lambda$ preserves these properties, we have $h ([t_2]) = \lambda ([t_2])$.
The curve $t_3$ is the unique curve up to isotopy disjoint from $a_2, c_1, p_3, t_1, t_2$ which intersects $c_3$ nontrivially. 
Since $h([x]) = \lambda([x])$ for all these curves and $\lambda$ preserves these properties, we have $h ([t_3]) = \lambda ([t_3])$.
The curve $t_4$ is the unique curve up to isotopy disjoint from $a_1, a_2, l, t_1,t_3$ which intersects $c_1$ nontrivially. 
Since $h([x]) = \lambda([x])$ for all these curves and $\lambda$ preserves these properties, we have $h ([t_4]) = \lambda ([t_4])$.
The curve $t_5$ is the boundary of a regular neighborhood of $a_3 \cup b_3$.
Since $h ([a_3]) = \lambda ([a_3])$ and $h ([b_3]) = \lambda ([b_3])$, by Lemma 2.6 and Lemma 2.7 we have $h ([t_5]) = \lambda ([t_5])$.
Finally, $t_{a_3} (b_3)$ is the unique curve up to isotopy disjoint from $t_4,t_5$ which intersects $b_3$ nontrivially. 
Since $h([x]) = \lambda([x])$ for all these curves and $\lambda$ preserves these properties, we have $h ([t_{a_3} (b_3) ]) = \lambda ([t_{a_3} (b_3) ])$.

We show using the same arguments as in the proof of Lemma 3.9 that, for all $x,y \in \{ a_1, \dots, 
\allowbreak
a_{g-2}, b_1, \dots, b_{g-2}, c_1, \dots, c_{g-2}, d_1, \dots, d_{n-1}, p_3\}$, we have $h ([ t_x (y)]) = \lambda ([t_x (y)])$.
For $f=t_x$, where $x \in  \{ a_1, \dots, a_{g-2}, b_1, \dots, b_{g-2}, c_1, \dots, c_{g-2}, d_1, \dots, d_{n-1}\}$, we set $L_f = \CC_1$. 
Then, by the above, $h ([y]) = \lambda ([y])$ for all $y \in L_f \cup f (L_f)$.
 
For $f = \sigma_i$, where $i\in \{1, \dots, n-1\}$, we set $L_f = \CC_1$.
We know that $ h([x]) = \lambda ([x])$ for all $x \in L_f$ and we have $f(x)=x$ for all $x \in \CC_1 \setminus \{d_i\}$.
It remains to prove that $h ([\sigma_i (d_i)]) = \lambda ([\sigma_i (d_i)])$.
We show using the same arguments as in the proof of Lemma 3.9 that $h ([s_1]) = \lambda ([s_1])$ and $h([t_6]) = \lambda ([t_6])$, where $s_1, t_6$ are the curves drawn in Figure \ref{fig-new2} (iv), (v). The curve $\sigma_i (d_i)$ is shown as $o_i$ in the figure. Again using the same arguments as in the proof of Lemma 3.9 we deduce that $h ([\sigma_i (d_i)]) = \lambda ([\sigma_i (d_i)])$.
  
For $f=y$ we set $L_f = \CC_2$.
We know that $h ([x]) = \lambda ([x])$ for all $x \in \CC_2$ and we have $y (x) = x$ for all $x \in \CC_2 \setminus \{ a_2,b_1,c_1\}$.
It remains to prove that $h ([y (x)]) = \lambda ([y (x)])$ for $x \in \{a_2,b_1,c_1\}$.
We have $y(a_2) = k$, hence $h ([ y(a_2)]) = \lambda ([y (a_2)])$.
We have $y(b_1) = c_1$, hence $h ([y (b_1)]) = \lambda ([y (b_1)])$.
The curve $y(c_1)$ is the unique curve up to isotopy disjoint from $a_3, b_3, c_1, c_3, k$ which intersects $a_2$ nontrivially. 
Since $h([x]) = \lambda([x])$ for all these curves and $\lambda$ preserves these properties, we have $h ([y (c_1)]) = \lambda ([y (c_1)])$.

For $f = \xi$ we set $L_f = \CC_3$.
We know that $h ([x]) = \lambda ([x])$ for all $x \in \CC_3$ and we have $f(x) = x$ for all $x \in \CC_3 \setminus \{c_1, \dots, c_{g-2}, d_1, \dots, d_{n-1}\}$.
Let $i \in \{1,3, \dots, g-4\}$.
Then $\xi (c_i)$ is the unique curve up to isotopy disjoint from $a_{i+1}, \dots, a_{g-3}, a_{g-2}, b_{g-2}, c_i, s_1, m_1, \dots,$ $m_{n-2}$ which intersects $c_{g-2}$ nontrivially. 
Since $h([x]) = \lambda([x])$ for all these curves and $\lambda$ preserves these properties, we have $h ([\xi (c_i)]) = \lambda ([\xi (c_i)])$.
If $n=1$, then $\xi (c_{g-2}) = b_{g-2}$, hence $h ([\xi (c_{g-2})]) = \lambda ([\xi (c_{g-2})])$.
Suppose that $n \ge 2$.
Then $\xi (c_{g-2})$ is the unique curve up to isotopy disjoint from $b_{g-2}, c_{g-2}, m_1, \dots, m_{n-2}, \xi (c_{g-4})$ which intersects $d_{n-1}$ nontrivially. 
Since $h([x]) = \lambda([x])$ for all these curves and $\lambda$ preserves these properties, we have $h ([\xi (c_{g-2})]) = \lambda ([\xi (c_{g-2})])$.
Let $i \in \{1, \dots, n-2\}$.
The curve $\xi (d_i)$ is the unique curve up to isotopy disjoint from $b_{g-2}, d_i, m_{i+1}, \dots, m_{n-2}, \xi (c_{g-4})$ which intersects $d_{n-1}$ nontrivially.
Since $h([x]) = \lambda([x])$ for all these curves and $\lambda$ preserves these properties, we have $h ([\xi (d_i)]) = \lambda ([\xi (d_i)])$.
We have $\xi (d_{n-1}) = b_{g-2}$, hence $h ([\xi (d_{n-1})]) = \lambda ([\xi (d_{n-1})])$.\end{proof}

\begin{theorem} \label{A} If $g \geq 5$ and $g$ is odd, then $h([x]) = \lambda([x])$ for every vertex $[x] \in \mathcal{T}(N)$.\end{theorem}

\begin{proof} The proof is similar to the even genus case.\end{proof}\\

{\bf Acknowledgements}\\

This work started during the authors' visit to ``Automorphisms of Free Groups" program in CRM, Barcelona. We thank the organizers, and
also CRM for the hospitality. The first author also thanks Peter Scott for helpful discussions about this project.

{\bf Elmas Irmak}

Bowling Green State University, Department of Mathematics and Statistics, Bowling

Green, 43403, OH, USA. 

e-mail: eirmak@bgsu.edu\\

{\bf Luis Paris}

IMB, UMR 5584, CNRS, Univ. Bourgogne Franche-Comt\'e, 21000 Dijon, France.

e-mail: lparis@u-bourgogne.fr
\end{document}